\documentclass{amsart}

\usepackage[left=1.2in, right=1.2in]{geometry}

\usepackage{amsmath,bbold,mathrsfs,amscd,lineno,hyperref,amsthm}
\usepackage{mathtools,afterpage}
\usepackage{tikz}
\usepackage{mathrsfs}
\usepackage{yfonts,amsbsy} 

\usepackage{amsfonts,subcaption}
\usepackage{amssymb}
\usepackage{enumitem}

\usepackage{tipa}

\usepackage[utf8]{inputenc}

\newcommand\xqed[1]{%
  \leavevmode\unskip\penalty9999 \hbox{}\nobreak\hfill
  \quad\hbox{#1}}
\newcommand\tqed{\xqed{$\triangle$}}

\usetikzlibrary{matrix, patterns,snakes,calc}
\usetikzlibrary{math,arrows,shapes}

\renewcommand{\sp}{{\ \ }}

\newcommand{\hide}[1]{}

\newcommand{\blambda}{{\boldsymbol\lambda}}
\newcommand{\btau}{{\boldsymbol\tau}}
\newcommand{\bkappa}{{\boldsymbol\kappa}}

\everymath{\displaystyle}

\newcommand{\upbullet}[1]{\overset{\bullet}{ #1}{}}
\newcommand{\upstar}[1]{\overset{\star}{ #1}{}}

\renewcommand{\Re}{\operatorname {Re}}

\newcommand{\FamG}{{\mathcal G}}
\newcommand{\FamH}{{\mathcal H}}
\newcommand{\FamP}{{\mathcal P}}
\newcommand{\FamQ}{{\mathcal Q}}

\newcommand{\ext}{{\operatorname{ext}}}
\renewcommand{\div}{{\operatorname{div}}}

\newcommand{\Fam}{{\mathcal F}}
\newcommand{\Width}{{\mathcal W}}

\newcommand{\epscasc}{{\mu}}
\newcommand{\opp}{{\text{opp}}}

\newcommand{\stab}{{\text{stab}}}

\newcommand\LL{{\mathcal L}}

\newcommand\NN{\mathcal{N}}

\makeatletter
\newcommand{\setword}[2]{%
  \phantomsection
  #1\def\@currentlabel{\unexpanded{#1}}\label{#2}%
}

\newcommand{\Mandel}{{\mathcal{M}}}
\newcommand\MM{\mathcal{M}}

\newcommand\TT{\mathcal{T}}
\newcommand\YY{\mathcal{Y}}
\newcommand\ZZ{\mathcal{Z}}

\newcommand\HH{\mathcal{H}}

\newcommand{\bDelta}{{\mathbb{\Delta}}}

\newcommand{\bC}{{\mathbf{C}}}

\newcommand\RR{\mathcal{R}}
\renewcommand\SS{\mathcal{S}}
\newcommand\XX{\mathcal{X}}

\newcommand\FF{\mathcal{F}}

\newcommand\intr{\operatorname{int}}

\newcommand\Jul{{\mathfrak J}}

\newcommand{\bN}{{\mathbf N}}

\newcommand{\bM}{{\mathbf M}}

\newcommand{\tF}{{\mathfrak F}}

\newcommand{\wL}{{\widehat L}}
\newcommand{\wJ}{{\widehat J}}
\newcommand{\wI}{{\widehat I}}
\newcommand{\wZ}{{\widehat Z}}

\newcommand{\bZ}{{\mathbf Z}}

\newcommand{\bK}{{\mathbf K}}

\newcommand{\bbn}{{\mathbf n}}

\newcommand{\bbk}{{\mathbf k}}
\newcommand{\bbs}{{\mathbf s}}

\newcommand{\bbt}{{\mathbf t}}

\newcommand{\Level}{\operatorname{Level}}

\newcommand{\bbg}{{\mathbf g}}

\newcommand{\str}{{\star}}

\newcommand{\pp}{\mathfrak{p}}
\newcommand{\ee}{\mathbf e}

\newcommand{\qq}{\mathfrak{q}}
\renewcommand{\tt}{\mathfrak{t}}

\newcommand{\rr}{\mathfrak{r}}
\newcommand{\kk}{\mathfrak{k}}

\newcommand{\length}{{\mathfrak l}}

\newcommand\inn{{\operatorname{inn}}}
\newcommand\parab{{\operatorname{par}}}

\newcommand\UU{\mathcal{U}}

\newcommand\BB{\mathcal{B}}
\renewcommand\AA{\mathcal{A}}

\newcommand\PP{\mathcal{P}}

\newcommand\ovl{\overline}

\newcommand\out{{\operatorname{out}}}

\newcommand{\grnd}{{\operatorname{grnd}}}
\newcommand{\GRND}{{\operatorname{GRND}}}
\newcommand{\well}{{\operatorname{well}}}

\newcommand\new{{\operatorname{new}}}
\newcommand\New{{\operatorname{New}}}
\newcommand\NEW{{\operatorname{NEW}}}

\newcommand\full{{\operatorname{full}}}

\newcommand\dual{{\operatorname{dual}}}

\newtheorem{claim}{Claim}

\newtheorem{claim3}{Claim}
\newtheorem{claim4}{Claim}
\newtheorem{claim5}{Claim}
\newtheorem{claim6}{Claim}
\newtheorem{claim7}{Claim}
\newtheorem{claim8}{Claim}

\newcommand{\wC}{\widehat{\mathbb{C}}}
\newcommand{\C}{\mathbb{C}}
\newcommand{\Q}{\mathbb{Q}}
\newcommand{\R}{\mathbb{R}}
\newcommand{\N}{\mathbb{N}}

\DeclareFontFamily{U}{mathb}{\hyphenchar\font45}
\DeclareFontShape{U}{mathb}{m}{n}{ <5> <6> <7> <8> <9> <10> gen * mathb <10.95> mathb10 <12> <14.4> <17.28> <20.74> <24.88> mathb12 }{}
\DeclareSymbolFont{mathb}{U}{mathb}{m}{n}
\DeclareFontSubstitution{U}{mathb}{m}{n}
\DeclareMathSymbol{\selfmap}{3}{mathb}{"FD}

\newcommand{\Z}{\mathbb{Z}}

\newcommand{\Circle}{\mathbb{S}^1}
\newcommand{\Disk}{\mathbb{D}}
\newcommand{\ovDisk}{{\overline \Disk}}
\newcommand{\ovZ}{{\overline Z}}

\newcommand{\Jbb}{{\mathfrak{J}}}
\newcommand{\Dbb}{{\mathfrak{D}}}

\newtheorem{assum}{Assumption}

\newtheorem{thm}{Theorem}[section]

\newtheorem{cor}[thm]{Corollary}
\newtheorem{lem}[thm]{Lemma}

\newtheorem{prop}[thm]{Proposition}
\newtheorem{rem}[thm]{Remark}

\newtheorem{loclmm}[thm]{Localization Lemma}
\newtheorem{snakelmm}[thm]{Snake Lemma}
\newtheorem{weldlmm}[thm]{Welding Lemma}
\newtheorem{caliblem}[thm]{Calibration Lemma}
\newtheorem{ampthm}[thm]{Amplification Theorem}

\newtheorem{squeelmm}[thm]{Squeezing Lemma}
\newtheorem{sneaklmm}[thm]{Sneaking Lemma}

\newtheorem{lairlmm}[thm]{Snake-Lair Lemma}

\newtheorem{mainthm:bounds}[thm]{Uniform Bounds Theorem}

\newtheorem{UniformBounds}[thm]{Uniform Bounds Theorem}
\newtheorem{Hedgehog}[thm]{Mother Hedgehog Theorem}
\newtheorem{Quasidisk}[thm]{Quasidisk Approximation Theorem}

\newtheorem{mainthm:geom_limits}[thm]{Geometric Limit Theorem}

\newtheorem{mainthm:sector_bounds}[thm]{Sector Bounds Theorem}
\newtheorem{mainthm:radial_str}[thm]{Star-Like Theorem}
\newtheorem{mainthm:psi_ql}[thm]{QL Siegel bounds Theorem}

\theoremstyle{remark}

\numberwithin{equation}{section}

\theoremstyle{definition}
\newtheorem{defn}[thm]{Definition}
\font\nt=cmr7

\def\be{\begin{equation}}

\newcommand{\filled}{{\mathcal {K} }}

\newcommand{\bnd}{{\mathrm{bnd}}}

\renewcommand{\sec}{\mathrm{sec}}

\newcommand{\ovlTheta}{{\overline \Theta}}

\newcommand{\di}{\partial}

\def\sm{\smallsetminus}

\newcommand{\diam}{\operatorname{diam}}
\newcommand{\dist}{\operatorname{dist}}

\renewcommand{\mod}{\operatorname{mod}}

\newcommand{\orb}{\operatorname{orb}}

\newcommand{\supp}{\operatorname{supp}}

\newcommand{\Fjord}{{\mathfrak F}}

\newcommand{\Penin}{{\mathfrak P}}

\newcommand{\CP}{{\operatorname{CP}}}

\newcommand{\CC}{{\mathcal C}}

\newcommand{\bdelta}{{\boldsymbol{ \delta}}}
\newcommand{\bvarepsilon}{{\boldsymbol{ \varepsilon}}}

\newcommand{\bchi}{{\boldsymbol{ \chi}}}

\newcommand{\RN}[1]{\MakeUppercase{\romannumeral#1}}%

\def\note#1
{\marginpar
{\nt $\leftarrow$
\par
\hfuzz=20pt \hbadness=9000 \hyphenpenalty=-100 \exhyphenpenalty=-100
\pretolerance=-1 \tolerance=9999 \doublehyphendemerits=-100000
\finalhyphendemerits=-100000 \baselineskip=6pt
#1}\hfuzz=1pt}


\DeclareFontFamily{U}{mathb}{\hyphenchar\font45}
\DeclareFontShape{U}{mathb}{m}{n}{ <5> <6> <7> <8> <9> <10> gen * mathb <10.95> mathb10 <12> <14.4> <17.28> <20.74> <24.88> mathb12 }{}
\DeclareSymbolFont{mathb}{U}{mathb}{m}{n}
\DeclareFontSubstitution{U}{mathb}{m}{n}
\DeclareMathSymbol{\righttoleftarrow}{3}{mathb}{"FD}

\begin{document}

\title[Uniform \emph{a priori} bounds for neutral renormalization]{Uniform \emph{a priori} bounds  \\ for neutral renormalization.}
\author{Dzmitry Dudko}
\author{Mikhail Lyubich}

\begin{abstract} We prove {\em uniform ``pseudo-Siegel'' a priori bounds} for Siegel disks of bounded type that give a  uniform control of oscillations of their boundaries in all scales. As a consequence, we construct the {\em Mother Hedgehog} controlling the postcritical set for any quadratic polynomial with a neutral periodic point and show that this hedgehog has a star-like structure. Pseudo-Siegel bounds imply \emph{uniform a priori bounds} of the Sector Renormalization, which gives an opportunity to extend Siegel/Pacman Renormalization Theory and Near-Parabolic Renormalization Theory to all near-neutral quadratic polynomials. Various applications beyond quadratic polynomials are also underway.
\end{abstract}
\maketitle

\setcounter{tocdepth}{1}

\tikzset{%
  block/.style    = {draw, thick, rectangle, minimum height = 3em,
    minimum width = 3em},
  sum/.style      = {draw, circle, node distance = 2cm}, 
  input/.style    = {coordinate}, 
  output/.style   = {coordinate} 
}


\tableofcontents

\section{Introduction}

Local dynamics near a neutral fixed point, and a closely related dynamical theory of circle homeomorphisms,
is a classical story going back to Poincar\'e, Fatou,  and Julia.
It followed up in the next two decades  with breakthroughs by Denjoy (1932)
and  Siegel (1942) on the linearization of circle diffeomorphisms and local maps, respectively. In the 1950-60s, the Kolmogorov–Arnold–Moser (KAM) theory emerged resulting in the reenvision of the near-rotation phenomenon in mathematics and physics. In the mid-1970s, Renormalization Ideas were introduced into Dynamics by  Feigenbaum, Coullet and Tresser and led, in particular, to numerous conjectures in Low-Dimensional Dynamics including the nature of the KAM \emph{small divisor problem}.

The theory of analytic circle diffeomorphisms and local theory for neutral holomorphic germs
and quadratic polynomials received an essentially complete treatment in the
second half of the last century in the work by Arnold (in the KAM framework), Bruno,
Herman,  Yoccoz, and Perez-Marco. About at the same time (1980--90s) a global and semi-local theory for neutral Siegel quadratic polynomials $f_\theta: z\mapsto e^{2\pi i \theta} z + z^2$ with rotation numbers $\theta$  of {\em bounded type}  was designed on the basis of the  {\em Douady-Ghys surgery}. Renormalization Theory of Siegel maps, also initiated by physicists, was mathematically designed by McMullen in~\cite{McM1} in the mid-1990s. In the late-2010s, the hyperbolicity of Siegel Renormalization of bounded type was established in the framework of Pacman Renormalization~\cite{DLS} providing tools to study near-Siegel maps~\cite{DL}.

A remarkable progress in understanding Near-Neutral Complex Dynamics came in the 2000s, when Inou and Shishikura established  {\em uniform a priori bounds}
for quadratic polynomials $f_\theta$ with rotation numbers of {\em high type} (near-parabolic perturbative regime).  This theory found numerous applications, from constructing examples of Julia sets of positive area by Buff-Cheritat (2000s, \cite{BC}) and Avila-Lyubich (2010s, \cite{AL-posmeas})) to a complete
description, for high type rotation numbers,
of the topological structure of the  {\em Mother Hedgehogs} 
that capture the semi-local dynamics of  neutral quadratic polynomials
(Shishikura-Yang,  Cheraghi (2010s)). See \S\ref{ss:HistRetro} and~\S\ref{ss:MLC} for a more detailed historical account.


In this paper, we lay down a foundation for the renormalization  theory based upon \emph{almost-invariant objects}, called \emph{pseudo-Siegel disks}, that ``hide'' various geometric irregularities; compare with Item~\ref{Item:wZ} in~\S\ref{ss:MLC}. In the almost invariant framework, we will establish {\em uniform ``pseudo-Siegel'' a priori bounds} for neutral quadratic polynomials $f_\theta$ with artbitrary rotation numbers and show that the postcritical set of every such $f_\theta$ is ``small'': it lies within the compact {\em Mother Hedgehog} $H_\theta$ so that the restriction $f_\theta\colon H_\theta\selfmap $ is a homeomorphism. Pseudo-Siegel bounds can be transferred into other forms of \emph{a priori bounds}~\cite{DL:sector bounds, DL:psi bounds} making the theory more complete and compatible with the previous developments; see~\S\ref{ss:pseudo-Siegel bounds} for the summary.


Ideas of the current paper helped to confirm the MLC at the classical Feigenbaum parameter~\cite{DL:Feigen} -- historically, the most important parameter for the development of the Dynamical Renormalization Theory;  see Remark~\ref{remark: going deep}.  Applications to McMullen's problem on Sierpinski hyperbolic components and to the construction of Herman curves have already appeared in~\cite{DL:HypComp,Lim}. More applications of the theory developed in the current paper are underway; see~\S\ref{ss:models}.

Our proof is based upon analysis of {\em degenerating} Siegel disks of bounded type. The {\em Near-Degenerate Regime and its principles}, in the quadratic-like renormalization context,
were originally designed by Jeremy Kahn \cite{K},
with a key analytic tool, the {\em Covering Lemma}, 
 appeared in \cite{KL}. They serve as an entry point for our paper. One of the major subtleties of our situation is that Siegel disks of bounded type \emph{do not} have uniformly bounded qc-geometry since they may develop long fjords in all scale. (Otherwise, Cremer and even parabolic points would not have existed.) To deal with this problem, we design a \emph{regularization machinery} of filling-in the fjords to gradually turn Siegel disks $\overline Z_f$ into almost-invariant pseudo-Siegel disks $\big(\wZ_f^n\big)_{n\ge -1}$ so that $\wZ_f^{-1}$ are uniform qc disks. Every $\wZ_f^n$ is almost invariant up to $\qq_{n+1}$ iterations so that $\bigcap_{n\ge -1} \wZ^n_f=H_f$ is the Mother Hedgehog controlling the postcritical set. Moreover, $\wZ^{\kk(n+1)}_f$ become uniform qc disks $\wZ^{-1}_{f_{n+1}}$ after applying the $(n+1)$-fold sector renormalization $f\mapsto f_{n+1}$; see Item~\ref{into:main_cor:5} in~\S\ref{ss:pseudo-Siegel bounds}.

 We remark  in conclusion that the resulting pseudo-Siege bounds are the first \emph{non-perturbative} \emph{a priori} bounds dealing with non-locally connected (e.g., Cremer) Julia sets -- this removes arguably the main conceptual obstacle towards the \emph{full} MLC; see~\S\ref{ss:MLC} for a detailed summary of the developments in the direction of the MLC conjecture.

\subsection{Results} Due to the Douady-Ghys surgery, 
for a Siegel map $f=f_\theta$ of bounded type, the dynamics on the
Siegel disk $Z$, all the way up to  its boundary $\partial Z$, is  qc conjugate to 
the rigid rotation by $\theta$, which provides us with  the
{\em rotation combinatorial model} for $f|\, \di Z$.

\begin{figure}
\begin{tikzpicture}
     \begin{scope}[scale=0.8, every node/.style={scale=0.8} ]
  \node at (0,0) {   \includegraphics[width=8cm]{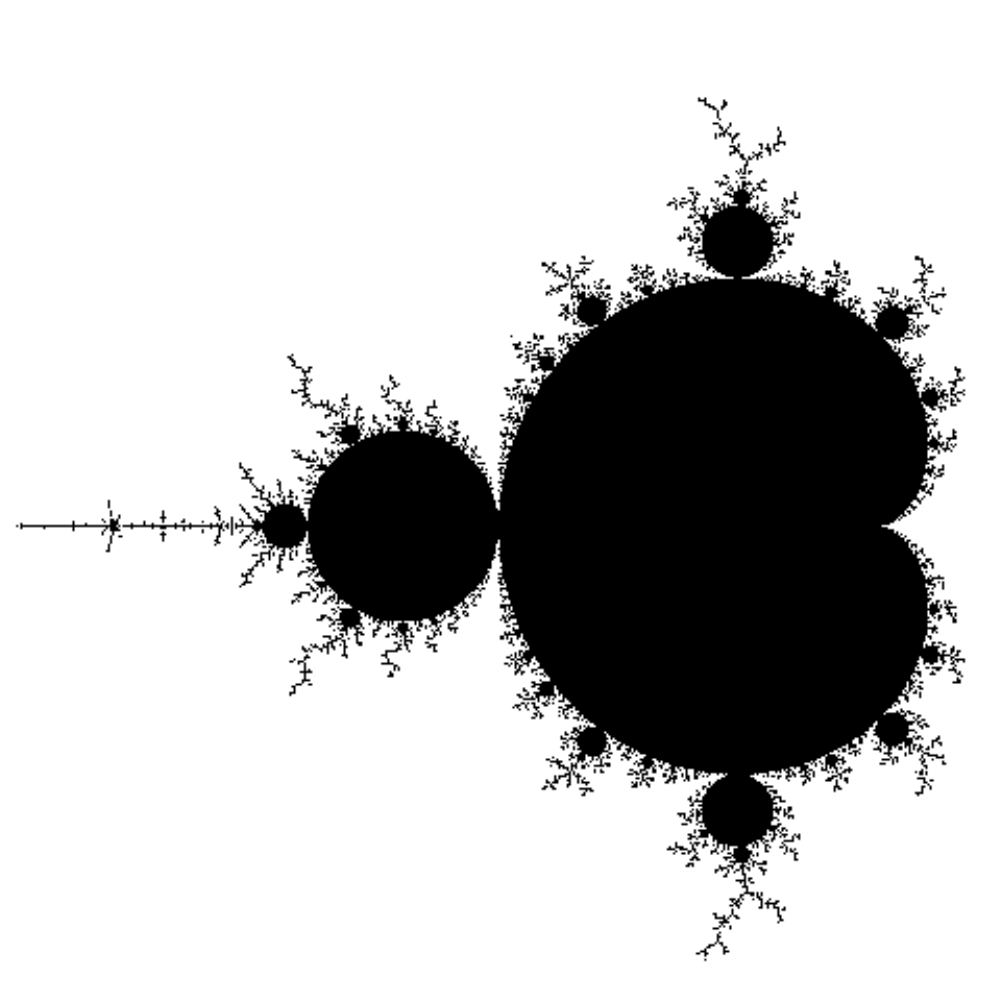}};

  \begin{scope}[shift={(-2.6,-4.8)},scale=1.4,every node/.style={scale=1.4}]
 \node[scale=0.8] at (1,1) {\includegraphics[width=2cm]{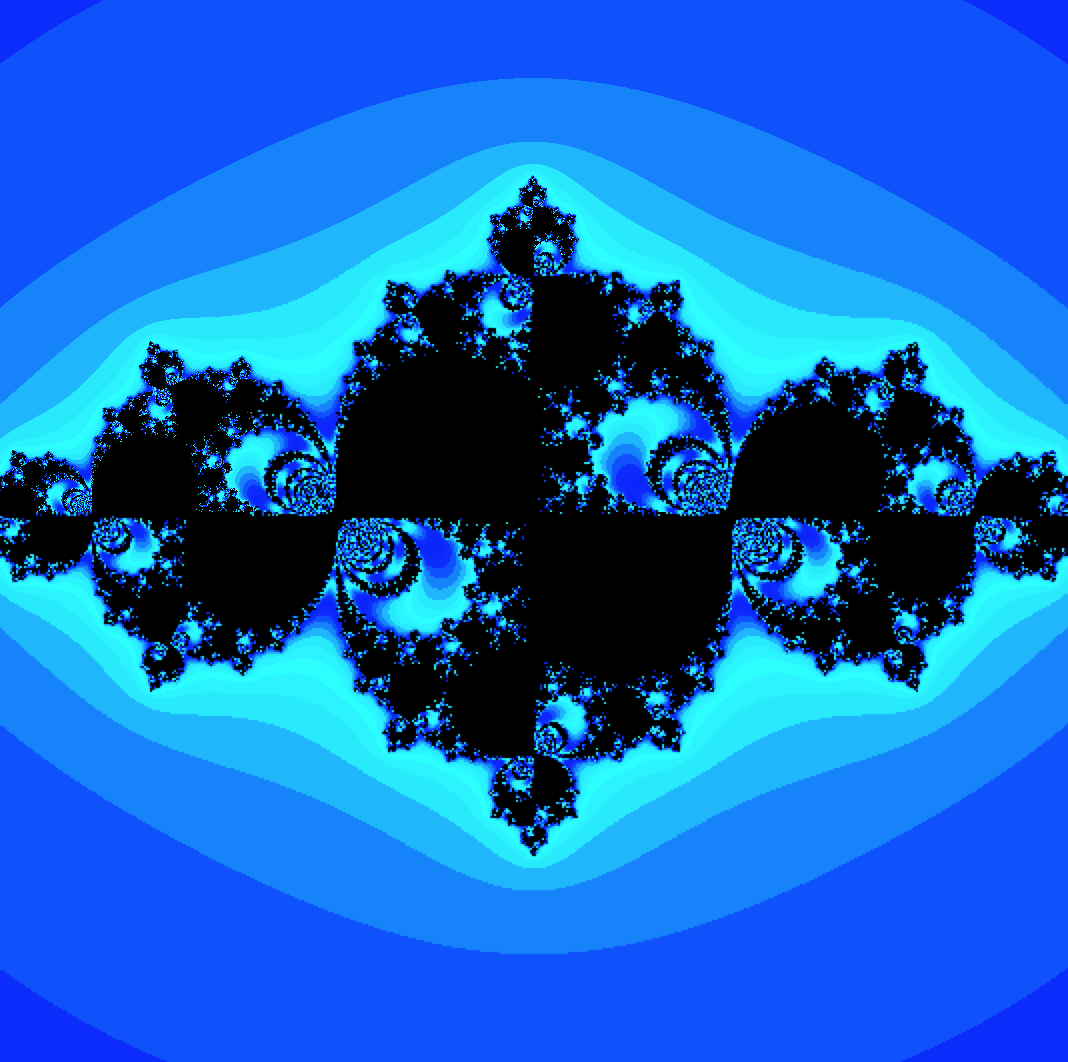}};
  \draw[red, line width=0.3mm] (0, 0) --(0, 2) --(2, 2) --(2, 0)--(0,0);

\draw[red,dashed, line width=0.3mm,shift={(0.47,0.8)}] (-0.15,-0.15) --(-0.15, 0.6) --(0.6,0.6) --(0.6, -0.15)--(-0.15,-0.15);
 
 \coordinate (yr) at (0.3,1.2);
 
\coordinate  (w1) at (1.6,2);
\end{scope}  
  \draw [red, line width=0.3mm] (w1) -- (0.02,-0.1);

\begin{scope}[shift={(-6,-4.8)},scale=1.4,every node/.style={scale=1.4}]
 \node[scale=0.8] at (1,1) {\includegraphics[width=2cm]{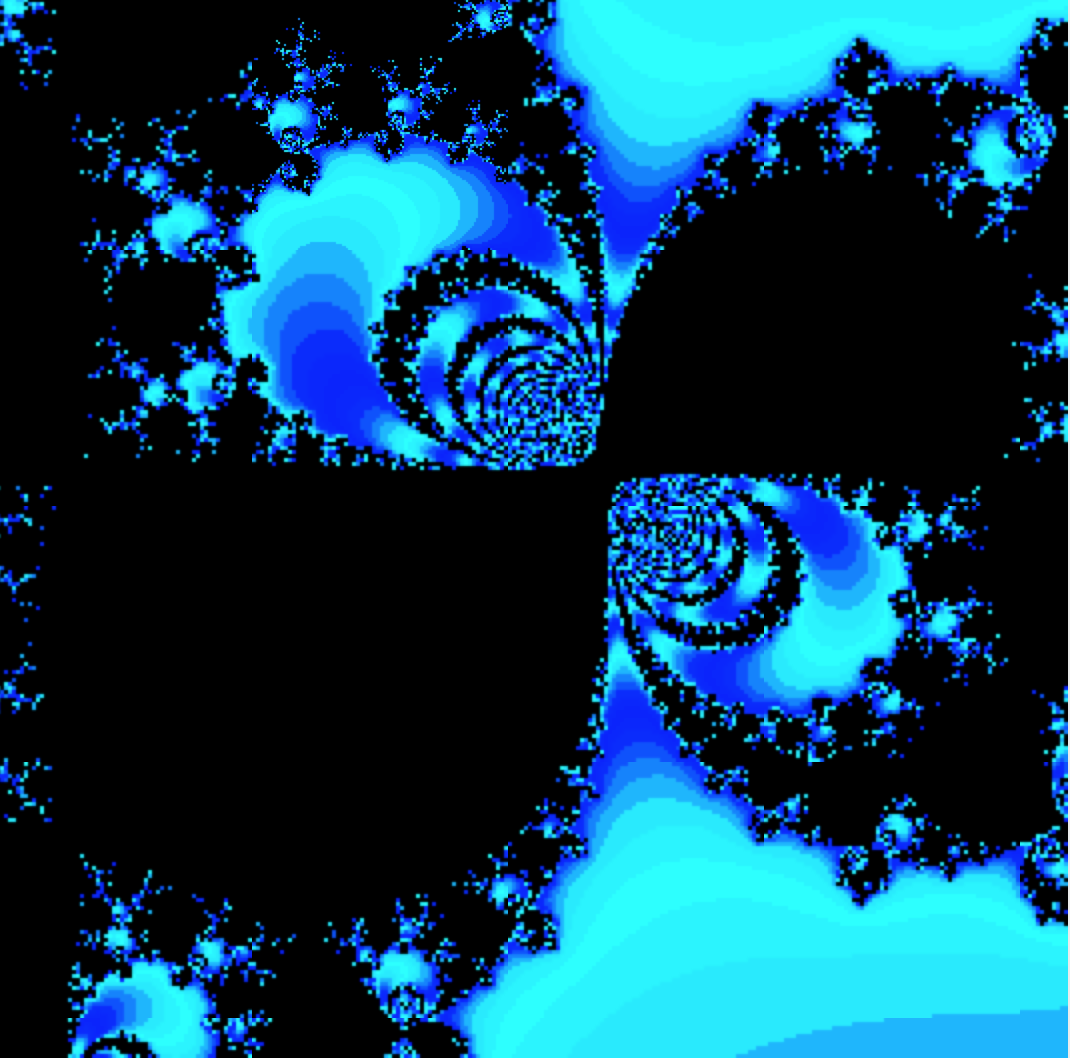}};
  \draw[red, line width=0.3mm] (0, 0) --(0, 2) --(2, 2) --(2, 0)--(0,0);

\coordinate  (w1) at (2,1.6);
\end{scope}  
  \draw [red, dashed, line width=0.3mm] (w1) -- (yr);

\begin{scope}[shift={(4.2,-1.5)},scale=1.4,every node/.style={scale=1.4}]
 \node[scale=0.8] at (1,1) {\includegraphics[width=2cm]{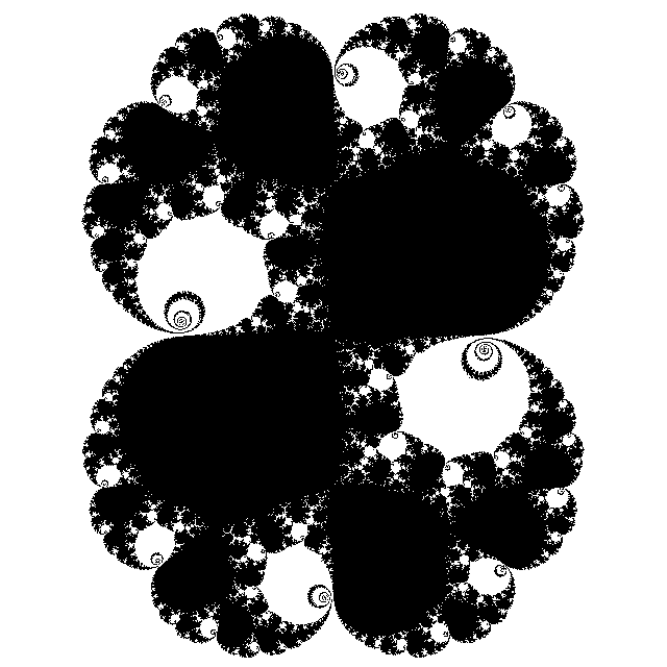}};
  \draw[red, line width=0.3mm] (0, 0) --(0, 2) --(2, 2) --(2, 0)--(0,0);

   \draw[red, dashed,line width=0.3mm] (0.85, 0.85) --(1.75, 0.85) --(1.75, 1.75) --(0.85, 1.75)--(0.85,0.85);
   \coordinate (xt) at (1.25, 0.85);

\node[white,scale=0.7] at (0.8,1.1){$\alpha$};

\coordinate  (w1) at (0,1.25);
\end{scope}  
  \draw [red, line width=0.3mm] (w1) -- (3.05,-0.18);

\begin{scope}[shift={(4.2,-4.8)},scale=1.4,every node/.style={scale=1.4}]
 \node[scale=0.8] at (1,1) {\includegraphics[width=2cm]{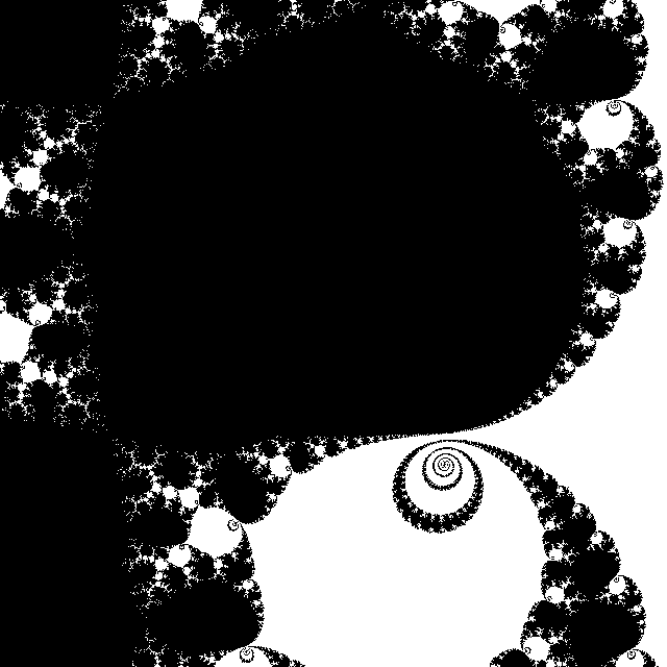}};
  \draw[red, line width=0.3mm] (0, 0) --(0, 2) --(2, 2) --(2, 0)--(0,0);

\coordinate  (w1) at (1.5,2);
\end{scope}  
  \draw [red, dashed, line width=0.3mm] (w1) -- (xt);

\begin{scope}[shift={(-3.8,1.3)},scale=1.4,every node/.style={scale=1.4}]
 \node[scale=0.8] at (1,1) {\includegraphics[width=2cm]{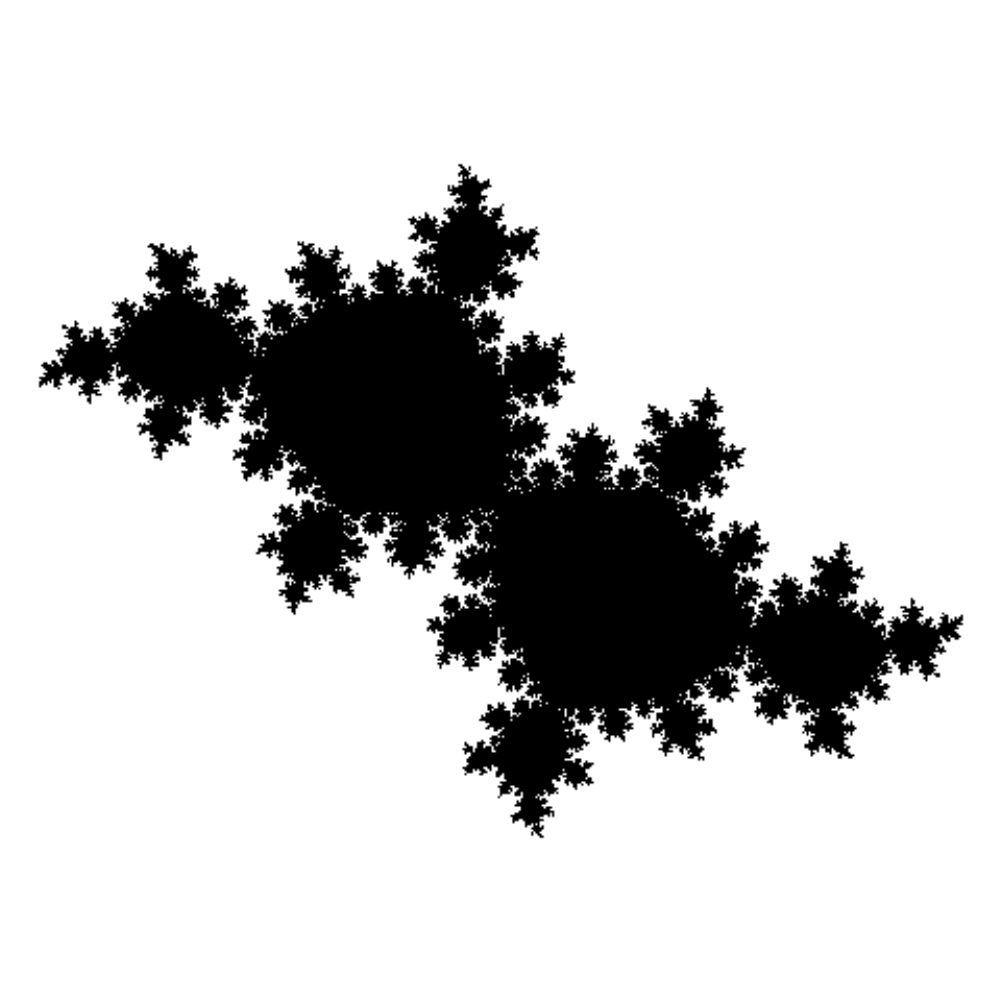}};
  \draw[red, line width=0.3mm] (0, 0) --(0, 2) --(2, 2) --(2, 0)--(0,0);

   

\coordinate  (w1) at (2,0.4);
\end{scope}  
  \draw [red, line width=0.3mm] (w1) -- (1.15,1.6);

\begin{scope}[shift={(4.2,1.8)},scale=1.4,every node/.style={scale=1.4}]
 \node[scale=0.8] at (1,1) {\includegraphics[width=2cm]{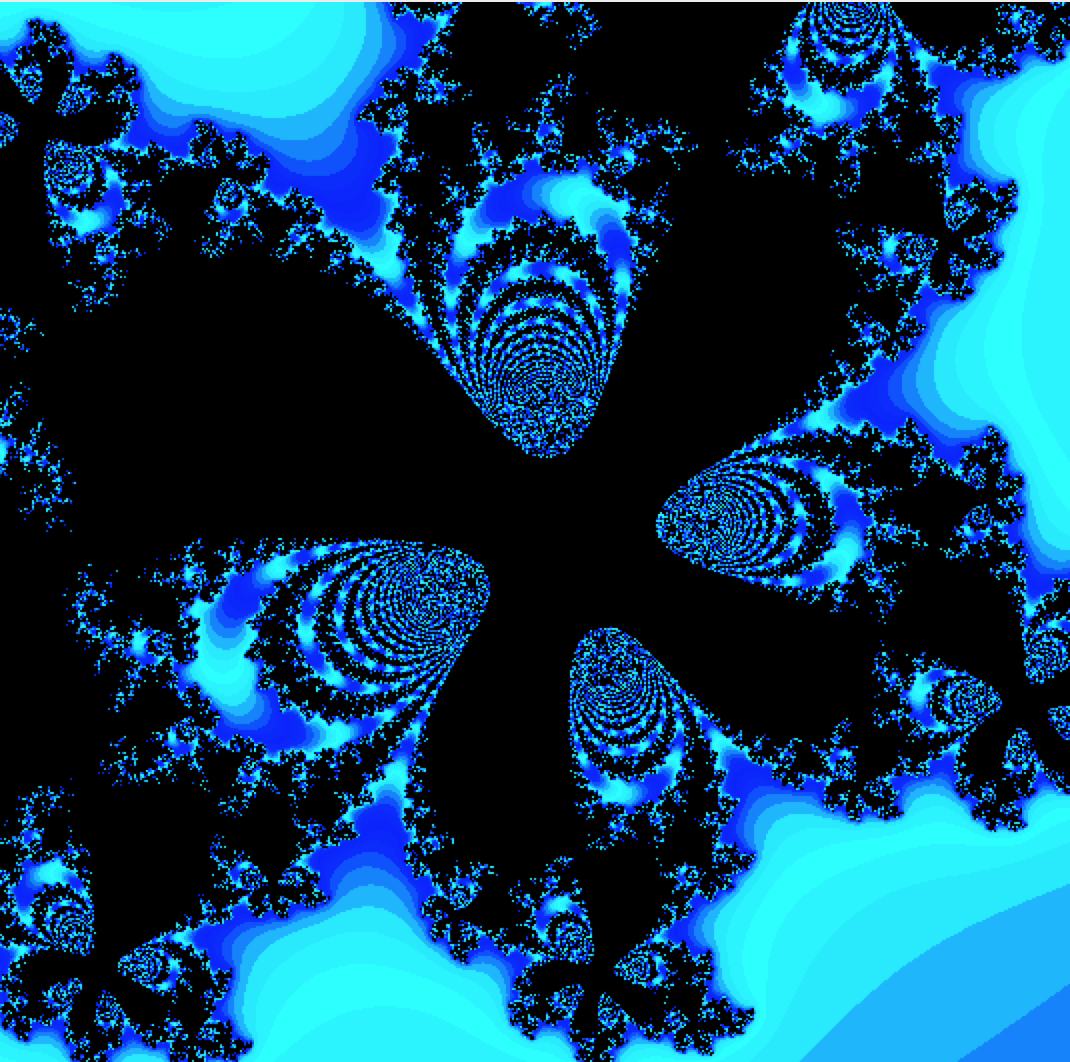}};
  \draw[red, line width=0.3mm] (0, 0) --(0, 2) --(2, 2) --(2, 0)--(0,0);


\coordinate  (w1) at (0.3,0.0);
\end{scope}  
  \draw [red, line width=0.3mm] (w1) -- (3.03,1.33);
    \end{scope}

  \end{tikzpicture}   
    
\caption{Different types of Siegel disks: golden-ratio (top-left),  near-Basilica (bottom-left), near-cauliflower (bottom-right), near-$1/4$-Rabbit (top-right).}   
\label{fig:Man + S disks}
\end{figure}
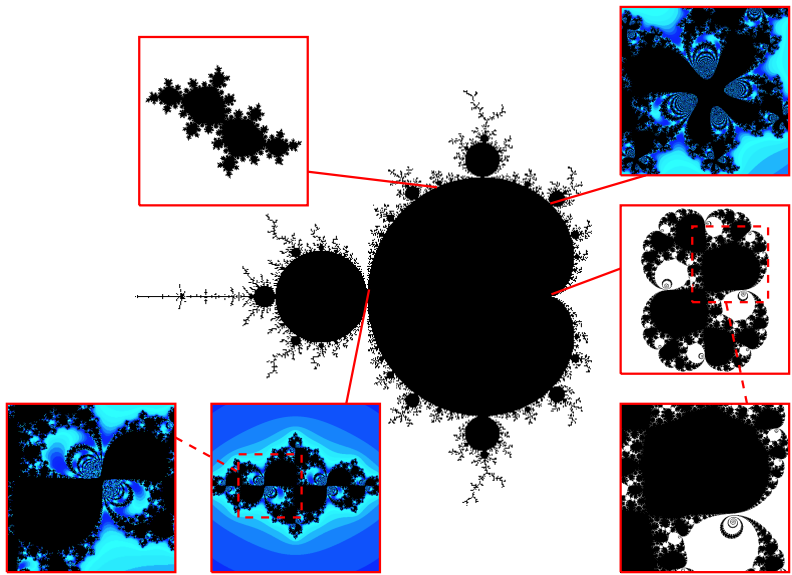

Let $\pp_n/\qq_n$ be the continued fraction
approximands for $\theta$, so for any $x\in \di Z$, $f^{q_n} x$ are
the closest combinatorial returns of $\orb x$ back to $x$. A {\em combinatorial
  interval} $I\equiv I_f^n(x) \subset \di Z$ of level $n$ is  the combinatorially shortest  interval bounded by $x$ and
$f^{q_n} x$. For a combinatorial interval $I\subset \di Z$, we let
$\widetilde I= 3I\supset I$ be the enlargement  of $I$ by  two attached combinatorial intervals.  

Given a combinarorial interval $I\subset \di Z$, let us consider the
family $\FF^+_3(I)$ of curves $\gamma\subset \wC\sm Z$  connecting $I$ to 
points of $\di Z\sm \widetilde I$. The {\em external modulus} $\Width^+_3(I)$ is the {\em extremal width} (i.e., the inverse of the extremal length)  of the family $\FF^+_\lambda(I)$.

\begin{UniformBounds}
\label{thm:unform bounds}
  There exists an absolute constant $\bK$ such that
  $\Width^+_3(I)\leq \bK$ for all Siegel quadratic polynomials $f=f_\theta$ of bounded type and all combinatorial intervals $ I=I_f^n(x)$.
\end{UniformBounds}


A {\em hull} $Q\subset \C$ is a compact connected full set.
The {\em  Mother Hedgehog}~\cite{Ch} for a neutral polynomial $f_\theta$
is an invariant hull containing both the fixed point $0$ and the critical
point $c_0(f)\coloneqq -e^{2\pi i \theta}/2$.

\begin{Hedgehog}
\label{thm:hedgehog}
  Any neutral quadratic polynomial $f=f_\theta$, $\theta\not\in \Q$,
  has a Mother Hedgehog $H_f\ni c_0(f)$ such that $f\colon  H_f \to H_f$ is a homeomorphism.
\end{Hedgehog}

The last theorem is a consequence of the following result:

\begin{Quasidisk}
\label{thm:quasidisk}
   There exists an absolute constant $\bK$ such that for any Siegel
   quadratic polynomial $f$ of bounded type there exists a $\bK$-quasidisk
    $\wZ_f\supset  \overline Z_f$ such that $f|\, \wZ_f$ is injective. 
 \end{Quasidisk}   






\subsection{Variations of uniform bounds} \label{ss:pseudo-Siegel bounds}  Let us formulate  a more technical statement representing the main result of the paper. Let $\pp_n/\qq_n\approx \theta$ be the best approximations of an irrational $\theta\in \R/\Z$ starting with $\qq_0=1$. Then the $f_\theta^{\qq_n}\equiv f_\theta^{\qq_n[\theta]}$ are the closest combinatorial returns enumerated so that $f_\theta^{\qq_0}=f$. Let us introduce the following concept:

\begin{defn}[Pseudo-Siegel bounds] \label{defn:ps Siegel bounds}Consider a quadratic polynomials $f_\theta\colon z\mapsto e^{2\pi i \theta} z + z^2$. We say that $f_\theta$ has \emph{pseudo-Siegel bounds} if it has a sequence of disks  
\begin{equation}
\label{eq:dfn:wZ^m:intro}
\wZ^{-1}\equiv \wZ^{-1}_\theta\supseteq \wZ^{0}_\theta\supseteq \wZ^1_\theta\supseteq \dots \supseteq H_{\theta}\ \ni 0, \sp\sp f\mid \wZ^m_\theta \text{ is injective, }\sp\sp \bigcap_{m\ge -1} \wZ^m_\theta =H_{\theta},
\end{equation}
where $H_{\theta}\equiv H_{f_\theta}$ is the Mother Hedgehog from Theorem~\ref{thm:hedgehog}, such that (see also~\S\ref{sss:uniform wZ bounds})
\begin{enumerate}[label=\text{(\Alph*)},font=\normalfont,leftmargin=*]
\item\label{intro:prop:A} $\wZ^{m}_\theta$ is a pseudo-Siegel disk almost invariant under $f_\theta^{\qq_{m+1}}$;
\item\label{intro:prop:B} $\partial \wZ^m_\theta$ has a nest of tilings $\TT\left(\wZ^m_\theta\right)$ with essentially bounded geometry independent of $m$, see Remark~\ref{rem:TT:wZ^m}. 
\end{enumerate}
If the geometric bounds of $\TT\left(\wZ^m_\theta\right)$ in~\ref{intro:prop:B} are independent of $\theta$, then we say that $f_\theta$ has \emph{uniform pseudo-Siegel bounds}.
\end{defn}

The main outcome of the paper is the establishment of uniform pseudo-Siegel bounds for bounded type rotation numbers. Since the bounds are uniform, they persist for all rotation numbers; see Theorem~\ref{thm:quasidisk} for illustration. Pseudo-Siegel disks $\wZ^m_\theta$ can be viewed as \emph{matings} of rotational and unicritical circle dynamics; see~\S\ref{sss:uniform wZ bounds}. We remark that uniform bounds in~\ref{intro:prop:B} can only be stated for almost-invariant objects because the inner geometry of $\wZ^m$ can degenerate or even disappear in the limit.

In~\cite{DL:sector bounds}, we transfer uniform pseudo-Siegel bounds to other forms of precompactness used in Siegel/Pacman Renormalization Theory and in Near-Parabolic Renormalization Theory and clarify the behavior of pseudo-Siegel disks $\wZ^m_\theta $ for $m\ge 0$ at $\theta$ of unbounded type. 

 \begin{mainthm:sector_bounds}[{\cite{DL:sector bounds}}] \label{thm:intro:sector bounds}There exists a compact Sector Renormalization operator $\RR_\sec$ for all neutral quadratic polynomials. 
 
 Moreover, Properties~\ref{into:main_cor:1} -- \ref{into:main_cor:6} stated below hold. 
\end{mainthm:sector_bounds}

 The Sector Renormalization was originally designed to study Local Dynamics around an indifferent fixed point~\cite[\S3.1]{D-Siegel}, ~\cite{Y} with various ideas going back to the theory of Ecalle-Voronin's invariants. For high type rotation numbers, Inou and Shishikura~\cite{IS} extended the sector renormalization into a semi-local compact operator with a geometric control of the critical orbit; see~\S\ref{ss:HistRetro} for more details. Often, sectorial bounds are stated in the framework of cylinder renormalization which is more canonical,~\cite{Ya-posmeas}. 
 
 Sectorial Renormalization constructed in~\cite{DL:sector bounds} comes along with the following structure:
\begin{enumerate}[label=\text{(\Roman*)},font=\normalfont,leftmargin=*]
\item \label{into:main_cor:1} (appropriately specified) pseudo-Siegel disks $\wZ^m_\theta$  depend, in the $C^0$-topology of
\[ \text{RiemannMap}_\theta\colon \wC\setminus \Disk\to \wC\setminus \intr(\wZ^m_\theta),\hspace{1cm}  \infty\mapsto \infty, \sp\sp 1\mapsto c_0(\theta),\]
 uniformly continuously in the topology coming from the compactification 
\[\theta\in \ovlTheta \ \coloneqq \ \big\{\theta= [0;a_1,a_2,\dots] \mid a_i \in \N_{\ge 1}\cup \{\infty\}\big\}\  \supset\ \Theta\coloneqq [\R\setminus \Q]/\Z.\]
\end{enumerate} 

\noindent We remark that $\ovlTheta$ is a natural compactification of the set of irrational rotation numbers $\Theta$.  Property~\ref{into:main_cor:1} implies:
 \begin{enumerate}[label=\text{(\Roman*)},font=\normalfont,leftmargin=*,start=2]
\item \label{into:main_cor:2} $\wZ^m_\theta$ exist  by continuity for all $\theta\in \ovlTheta$ and are almost invariant under $f^{\qq_{m+1}[\theta]}_\theta$ as in ~\ref{intro:prop:A} and~\ref{intro:prop:B}, where  $\left(f^{\qq_{m+1}[\theta]}_\theta\right)_{m\ge -1}$ evolves into Lavaurs-Epstein towers~\cite{La,Ep} at $\theta\in \ovlTheta\setminus \Theta$.
\item \label{into:main_cor:3} the Mother Hedgehog $H_{\theta}$ depends uniformly continuous on $\theta\in \ovlTheta$ in the Hausdorff topology. If $\theta\in \ovlTheta\setminus \Theta$, then $H_\theta$ is the Mother Hedgehog of the associated Lauvars-Epstein tower.
\item \label{into:main_cor:4}The set of closest returns \[
 \left\{\left.\left(f_\theta^{\qq_n[\theta]}\right)_{n\ge 0}\ \right| \ \theta\in \Theta \  \right\} \] is precompact up to linear conjugacy at the critical value in the following sense: for $n(i)\to  \infty$, any sequence $f^{\qq_{n(i)}[\theta(i)]}_{\theta(i)}$ has a converging subsequence $f^{\qq_{n(i_t)}[\theta(i_t)]}_{\theta(i_t)}\to F$, where $F\colon W_F\to \C$ is a $\sigma$-proper transcendental map. The convergence is uniform on compact subsets of $W_F$.
\end{enumerate} 
\noindent Transcendental limit as in~\ref{into:main_cor:4} were previously established for bounded type parameters $\Theta_\bnd$ in~\cite[Theorem~8.1]{McM1} and for the unstable manifolds associated with $\Theta_\bnd$ in~\cite[Theorem 5.1]{DLS}. The latter became a central tool in Pacman Renormalization Theory~\cite{DLS,DL}  in bypassing a technical challenge that semi-local maps under consideration are \emph{not} branched coverings. We believe that Transcendental Dynamics as in~\ref{into:main_cor:4} will allow us to extend Pacman Renormalization Theory (starting with the hyperbolicity result) to all rotation numbers; see also~\S\ref{ss:models} and Remark~\ref{rem:fill hors}.

A key ingredient for Properties~\ref{into:main_cor:1} -- \ref{into:main_cor:4} is the uniform continuity of $\wZ^{-1}[\RR^m_\sec f_\theta]$ in the spirit of Theorem~\ref{thm:quasidisk}; the key relation between $\wZ^m_{f_\theta}$ and the sector renormalization operator $\RR_\sec$:
\begin{enumerate}[label=\text{(\Roman*)},font=\normalfont,leftmargin=*,start=5]
\item \label{into:main_cor:5}  Under the renormalization change of variables representing $f_\theta \leadsto f_{n,\theta}=\RR^n_\sec f_\theta$, where $\theta\in \Theta$, the pseudo-Siegel disk $\wZ^{m+\kk(n)}_{f_\theta}$ becomes $\wZ^{m}_{f_{n,\theta}}$ -- a pseudo-Siegel disk almost invariant under $f_{n,\theta}^{\qq_{m+1}}$. Here $\kk(n,\theta)\ge n$ refers to the count of ``non-negligible'' (with respect to $\RR_\sec$) renormalization levels. Moreover, the above Properties~\ref{into:main_cor:1} -- \ref{into:main_cor:4} with appropriate adjustments hold for the sector renormalizations $\{f_{n,\theta}\}$ of neutral quadratic polynomials.
\end{enumerate}

The sectorial renormalization change of variables in Theorem~\ref{thm:intro:sector bounds} is expanding in the angular direction. Since the ``full'' non-escaping set of sectororial changes of variables is $H_\theta$, we deduce (compare with~\cite{RRRS}) that

\begin{enumerate}[label=\text{(\Roman*)},font=\normalfont,leftmargin=*,start=6]
\item \label{into:main_cor:6} For every $\theta\in{ (\R\setminus \Q)/\Z}$, there is a set $Q_\theta\subseteq \R/\Z$ invariant under the rotation $\phi \mapsto \phi+\theta$ such that the Mother Hedgehog $H_{\theta}$ is star-like (or a bouquet of curves): \[H_{f_\theta}= \bigcup_{\phi\in Q_\theta} I_\phi\sp\sp\sp \text{ with }\sp \sp\sp f(I_\phi)=I_{\phi+\theta},\] where $I_\phi$ is a closed simple arc emerging from the $\alpha$-fixed point. The arcs $I_\phi$ are pairwise disjoint away from the $\alpha$. If $Q_\theta\not= \R/\Z$, then $\alpha(f)$ is a Cremer point. If $Q_\theta=\R/\Z$, then $Z_f\coloneqq\intr H_f$ is the Siegel disk of $f$.
\end{enumerate}

\noindent For high-type parameters $\theta$, there are much finer models of $H_{f_\theta}$ depending on the arithmetic properties of $\theta$, see~\S\ref{ss:HistRetro}. We expect that these models can be now justified for all $\theta$, see~\S\ref{ss:models}.

In \cite{DL:psi bounds}, we establish pseudo-Siegel bounds for ``$\psi^\bullet$ quadratic-like maps'' so that the bounds depend (with an explicit estimate) only on the external modulus. Here, we state a simplified version for quadratic-like maps:

\begin{mainthm:psi_ql} [\cite{DL:psi bounds}, extension of Theorem~\ref{thm:unform bounds}] \label{thm:ql verions} Let $g\colon X\to Y$ be a quadratic-like map with bounded-type Siegel disk $Z_g$ at its $\alpha$-fixed point. Denote by \[K_g=\Width_\bullet(g)\coloneqq \Width\big(Y\setminus \overline Z_g\big)\] the degeneration of $g$ around its Siegel disk. Then the degeneration $K_g$ is equidistributed among combinatorial intervals as follows. If $I=[x,g^{\qq_{n+1}[g]}x]\subset \partial Z_g$ is a level $n$ combinatorial interval, then 
\begin{equation}
\label{eq:thm:ql verions}\Width^+_3(I) =O\left(\frac{K_g}{\qq_{n+1}}+1\right),
\end{equation}
where $\Width^+_3(I)$ measures the width of curves in $Y\setminus Z_g$ from $I$ to $\partial Y\sqcup \overline Z_g\setminus (3I)$.

In particular, if $K_g=0(1)$ is bounded, then~\eqref{eq:thm:ql verions} takes form $\Width^+_3(I) =O(1)$ as in Theorem~\ref{thm:unform bounds}.
\end{mainthm:psi_ql}

With appropriate modifications (compare with~\cite{DL:psi bounds} and~\cite{DL:HypComp}), Theorem~\ref{thm:intro:sector bounds} together with Properties~\ref{into:main_cor:1} -- \ref{into:main_cor:6} hold for $\psi^\bullet$-quadratic-like maps and transfer various results from quadratic-polynomials to various families of rational maps. (Sector Renormalization can also  substitute the Douady-Hubbard Straightening Theorem whenever it is unavailable.)

\subsection{Historical retrospective}
\label{ss:HistRetro}
As we have already mentioned, the local theory for neutral holomorphic germs, quadratic polynomials, and circle diffeomorphisms was thoroughly developed by Arnold, Herman, Yoccoz and Perez-Marco in the second half of the last century.
In particular, Yoccoz showed that the {\em Bruno's linearization condition}
is sharp for germs and quadratic polynomials \cite{Y}, while Perez-Marco introduced a topological object,
a {\em hedgehog} that greatly clarified the local structure of non-linearizable
Cremer maps~\cite{PM}.

Another line of thought was related to the {\em quasiconformal surgery}  machinery
introduced to the field by Sullivan, Douady and Hubbard in the early 1980s.
Namely, the {\em Douady-Ghys surgery} (see \cite{D-Siegel}) led  to a precise topological model for the Julia set of a neutral Siegel quadratic
polynomial $f_\theta: z\mapsto e^{2\pi i \theta} z+ z^2$ with rotation number $\theta$ of bounded type.  In particular, it allowed Petersen to justify local connectivity of
the corresponding Julia set \cite{Pe}. Renormalization Theory of Siegel  maps first appeared in the work by physicists (see \cite{Wi, MN,MP}) and then was mathematically developed in~\cite{McM1,Ya-posmeas, DLS, DL} and others. The theory deals with universal small-scale geometry of Siegel disks and near-Siegel families.


The Douady-Ghys surgery is based upon {\em real a priori bounds} for
critical circle maps proved by Swiatek \cite{Sw} and Herman \cite{H}.
 ({\em A priori bounds} mean a uniform geometric control of a
 system(s) under consideration in all dynamical scales.) 
 The Swiatek-Herman  bounds were promoted to {\em compex a priori bounds}
 by de Faria for bounded combinatorics \cite{dF} and by Yampolsky in general  
\cite{Ya}. 
 
Yet another direction was the theory of {\em parabolic implosion} designed by Doaudy and Lavaurs in the 1980s. Parabolic implosion (iterated twice) was used in the mid-1990s by Shishikura~\cite{Sh} in showing that the Hausdorff dimension of the boundary of the Mandelbrot set is equal to two.  A breakthrough in the parabolic renormalization by Inou and Shishikura in the mid-2000s provided us with {\em uniform a priori bounds} for rotation numbers $\theta$ of {\em high type} \cite{IS}; see also Cheritat \cite{Che:nearpar} for a computer-free argument. In short, \emph{a priori} bounds of~\cite{IS} allow us to iterate, in a controlled way, parabolic implosion infinitely many times.  Besides applications mentioned above (to producing Julia sets of
positive area \cite{BC,AL-posmeas} and to the description of the Mother
Hedgehogs \cite{ShY,Ch2}), the
Inou-Shishikura bounds were instrumental in the
proof of the \emph{Marmi-Moussa-Yoccoz Conjecture} for rotation numbers of high type  \cite{ChC} and in the description of the \emph{measurable dynamics} on the Julia sets of positive measure in the Inou-Shishikura class \cite{Ch1,ACh}. It also provided an opening to a partial description of the global topological structure of Cremer Julia sets, which have been viewed as most mysterious objects in holomorphic dynamics~\cite{BBCO}.

\subsection{Applications and further perspective}
\label{ss:models} With our {\em a priori} bounds in hands,
the above theory (see~\S\ref{ss:HistRetro}) for {\em  high type} (near-parabolic) and for {\emph bounded type} (near-Siegel) rotation numbers seems to be readily  extendable to quadratic polynomials with  {\em arbitrary}
rotation numbers. (The first steps have been implemented in this paper and in \cite{DL:sector bounds}, see~\S\ref{ss:pseudo-Siegel bounds} for a summary.) In particular, it would yield:

\noindent $\bullet$
{\em Siegel disks of quadratic polynomials are Jordan disks.} 

\noindent $\bullet$
{\em The Mother Hedgehogs from Theorem~\ref{thm:hedgehog} (see also~\ref{into:main_cor:6} in \S\ref{ss:pseudo-Siegel bounds}) satisfy the Cheraghi Trichotomy into Jordan disks, hairy
  Jordan disks, or  Cantor bouquets of curves depending on whether
the rotation number is Herman, Bruno but not Herman, or not Bruno.}

Another possible application of our {\em a priori bounds}
is a construction of a  {\em hyperbolic full renormalization horseshoe} for all
rotation numbers simultaneously containing the pacman renormalization periodic points \cite{McM3,Ya-posmeas,DLS} and the Inou-Shishikura horseshoe of high type \cite{IS}. Such a structure has a potential to reformulate Near-Neutral Dynamics as \emph{Transcendental Dynamics on the unstable manifolds} of the horseshoe; see also Item~\ref{into:main_cor:4} and the followup discussion in~\S\ref{ss:pseudo-Siegel bounds}. We refer to~\cite[\S 1.4]{DL} for a historical account on the interplay between Neutral and Transcendental Dynamics.

\begin{rem}[Full  Horseshoes] \label{rem:fill hors}A full renormalization horseshoe had been previously constructed for real-symmetric quadratic-like maps in~\cite{L-regul vs stoch}, for analytic unicritical circle maps in~\cite{Ya:hors}, and for real analytic unimodular maps in~\cite{ALM}. In every case, a horseshoe structure had strong consequences.
\end{rem}

Applications beyond quadratic family of the machinery designed in this paper
include:

\noindent $\bullet$
 An advance in {\em McMullen's Conjecture} (going back to the 1990s) in~\cite{DL:HypComp}: Sierpinski hyperbolic components of disjoint type (i.e.,when  all critical points are simple and attracted by different periodic cycles) are bounded in the moduli space of all rational functions.

\noindent $\bullet$
  Construction of bounded-type {\em Herman curves} by collapsing bounded-type Herman rings~\cite{Lim}. Herman curves are invariant rotational curves in the Julia set with chaotic dynamics on both sides.

And last but not least, uniform  bounds for the neutral renormalization give a control of the
satellite quadratic-like renormalization that are relevant to the
MLC Conjecture:

\subsection{Connections to the MLC conjecture}\label{ss:MLC}
 The MLC Conjecture on the local connectivity of the Mandelbrot set
$\Mandel$, put forward by Douady and Hubbard in the mid 1980s \cite{DH:Orsay}, is one of the central questions in Dynamical Systems guiding the development of contemporary Holomorphic Dynamics.  The classical Fatou Conjecture on the density of hyperbolicity in the complex quadratic family would be just one of many important consequences.

In a parallel development, Renormalization and
Universality  phenomena were discovered in the mid-1970s by Feigenbaum (the parameter plane) and independently by Coullet and Tresser (the dynamical plane) for the period-doubling real renormalization. It is intimately related to the renormalization phenomenon in the Quantum Field Theory and Statistical Mechanics; its discovery opened up a new universality paradigm in Dynamical Systems.  Quadratic-Like (``ql'')  theory of Douady and Hubbard \cite{DH:ql} laid down a conceptual frame for QL-Renormalization in Holomorphic Dynamics.  The exploration of this renormalization theory was initiated in the foundational work by Sullivan in the late 1980s - early 1990s. In particular, he established ql \emph{a priori} bounds for real quadratic polynomials of bounded type~\cite{S:Berkeley,S}. The theory was further developed by McMullen \cite{McM2,McM3}, and completed (in the real symmetric case) by the second author in the second half of the 1990s \cite{L-feigen}.

Around 1990, Yoccoz proved MLC at any parameter $c\in \Mandel$ that is
not infinitely ql renormalizable \cite{Hub:Yoccoz, Mil},
thus linking the problem tightly to the QL Renormalization Theory. 
In the following years the result was extended, by the second author, to a class of infinitely renormalizable parameters (of ``high type")  by means of ``generalized ql renormalization" that controls intermediate scales between ql renorm levels \cite{L-acta}.
It led to important applications to Real Dynamics culminating in the \emph{Regular or Stochastic Dichotomy} in the real quadratic family \cite{L-aster,L-regul vs stoch}. A key ingredient of the proof was a construction of the full renormalization horseshoe, see Remark~\ref{rem:fill hors}.

Consequently, the following metaprinciple was accepted:
\begin{enumerate}[label=\text{(\Roman*)},font=\normalfont,leftmargin=*]
\item\label{MLC:Step1} The MLC should follow from an appropriate form of \emph{a priori} bounds.
\end{enumerate}
One expects that \emph{a priori} bounds for \emph{all} parameters in a combinatorial class $\mathcal C\subset \mathcal M$ should imply that all maps in $\mathcal C$ are qc conjugate. Since the boundary $\partial \MM$ is qc rigid, this will imply that $\partial C=\mathcal C=\{c\}$ is a singleton, which implies the MLC at $c$.

In the primitive case, it is anticipated that
\begin{enumerate}[label=\text{(\Roman*)},font=\normalfont,leftmargin=*,start=2]
\item \label{MLC:Step2} Quadratic-Like bounds should hold for all infinitely ql-renormalizable parameters of primitive type.
\end{enumerate}
In the late 2000s, the Near-Degenerate Regime developed in~\cite{K, KL, KL2, molecules} led to a substantial progress in Conjecture~\ref{MLC:Step2}: the MLC was established \
in the {\em anti-molecule} case,. i.e.,
for combinatorics that stay ``$\varepsilon$-away'' from the main molecule of $\Mandel$.

In the satellite case, the situation is quite different:
\begin{enumerate}[label=\text{(\Roman*)},font=\normalfont,leftmargin=*,start=3]
\item \label{MLC:Step3} Quadratic-like bounds generally fail for the ql-Renormalization of satellite type~\cite[\S2]{Mil}: namely, if the Julia set is not locally connected at the critical point $c_0$ (the ``non-JLC phenomenon''), then small Julia sets do not shrink around $c_0$ and the ql-bounds fail.
\end{enumerate}
As the satellite copies of $\Mandel$ are attached to the main hyperbolic component $\HH$ of $\Mandel$, the {\em Satellite QL-Renormalization} evolves on $\partial \HH$ into the {\em Neutral Renormalization}, which is a subject in its own right with rich history, see~\S\ref{ss:HistRetro}. 

 As we have already mentioned, a landmark in the Near-Neutral Renormalization occurred when the Inou-Shishikura sectorial \emph{a priori} bounds~\cite{IS} were established in the Near-Parabolic Regime. The sectorial bounds are compatible with the non-JLC phenomenon and convincingly explain the compatibility of the MLC with Obstacle~\ref{MLC:Step3}; see~\cite{CS,Ch3}. Taking into account a parallel development in the Near-Siegel Renormalization~\cite{DLS, DL}, we conjecture:
 
  \begin{enumerate}[label=\text{(\Roman*)},font=\normalfont,leftmargin=*,start=4]
\item \label{MLC:Step4} Sectorial IS-like \emph{a priori} should hold for all infinitely ql-renormalizable parameters of satellite type and are sufficient for Task~\ref{MLC:Step1} for satellite parameters.
\end{enumerate}
Namely, if sectorial bounds are satisfied for \emph{all} maps in a combinatorial class $\CC\subset \MM$, then a variant of the Pullback Argument similar to~\cite[Corollary 4.7]{DLS} will produce qc conjugacies between \emph{all} maps in $\CC$.

 Note that the Near-Parabolic Theory is perturbative and hence is not (easily) amenable for the general rotational combinatorics. It also seems that the near-degenerate regime cannot be directly applied to the sector renormalization because the involved maps are \emph{not} branched coverings. 
These issues are addressed in  our current paper whose outcome can be roughly summarized as follows:

\begin{enumerate}
\item[\setword{$\big(\wZ\big)$}{Item:wZ}] Almost-invariant pseudo-Siegel disks $\wZ^m$ ``hide'' non-JLC issues. Near-degenerate tools are used to establish {\em pseudo-Siegel bounds}  for $\partial \wZ^m$; see~\S\ref{ss:pseudo-Siegel bounds}. They  yield {\em Sectorial bounds} by analyzing in detail the inner geometry of $\wZ^m$; see Theorem~\ref{thm:intro:sector bounds}.
\end{enumerate}
We believe that the theory of almost-invariant pseudo-Siegel disks can be developed for infinitely renormalizable quadratic polynomials of satellite type by enveloping their postcritical sets into  regularized ``pseudo-Siegel flowers''. This should imply the Satellite Case of the MLC conjecture. We also remark that the MLC at bounded type parameters was recently established in~\cite{DL:Feigen} by extending Kahn's argument~\cite{K}. This does not leave much room for the MLC to fail.

Let us conclude with our current formulation of the MLC Problem:
\begin{enumerate}[label=\text{(\Roman*)},font=\normalfont,leftmargin=*,start=5]
\item \label{MLC:Step5} Conjecturally, pseudo-Siegel bounds as in Item~\ref{Item:wZ} can be established for the Satellite Renormalization and then interpolated with primitive ql bounds established in~\cite{molecules}. This would complete Task~\ref{MLC:Step1} and confirm the full MLC conjecture.
\end{enumerate}

A popular version of the MLC story has been recently presented in a Quanta Magazine article \href{https://www.quantamagazine.org/the-quest-to-decode-the-mandelbrot-set-maths-famed-fractal-20240126/}{``The Quest to Decode the Mandelbrot Set, Math’s Famed Fractal''} \cite{Quest}.

\subsection{Quick outline of the proof}\label{ss:outline} Let us first give an informal description of Siegel disks degenerations. Let us denote by 
\[    \theta = [0;a_1, a_2,\dots, a_n ,a_{n+1}, \dots ],\sp\sp\sp\ a_n \le M_\theta.
\] the rotation number of $f$. Its Siegel disk $Z_f$ is a $K_\theta$-quasidisk by the Douady-Ghys surgery. Assume that $a_1=1$. If we start increasing $a_{n+1}$ with the remaining $a_i$ fixed, then $Z_f$ will be developing parabolic fjords towards the $\alpha$-fixed point on the renormalization scale $n$, see Figure~\ref{fig:fjords}. These fjords can approach $\alpha$ arbitrary close. Cremer points are obtained by developing fjords in many scales so that in the limit the $\alpha$-fixed point is not an interior point of the filled Julia set. In the paper, we will justify that this ``star-like'' degeneration is the only possible degeneration of bounded type Siegel disks. We will work in the near-degenerate regime where wide rectangles impose non-crossing constrains on the geometry. (One may call it ``$1.5$-dimensional real dynamics''.)

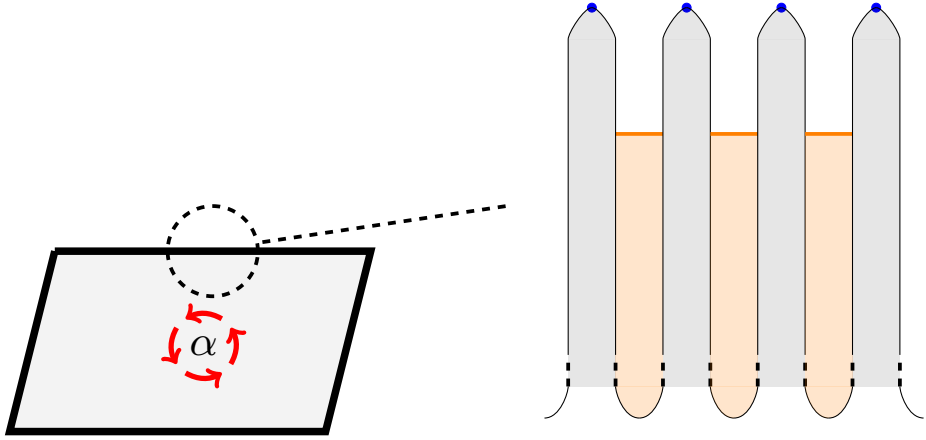
\begin{figure}
\begin{tikzpicture}[scale=0.6, every node/.style={scale=0.6}]

\draw[line width =0.1 cm,fill, fill opacity =0.05] (0,0) -- (7,0) -- (6,-4)--(-1,-4)--(0,0)--(0.,0);

\draw[dashed,line width =0.05 cm] (3.5,0) circle (1cm)
(4.6,0.2)  -- (10,1) ;

\begin{scope}[shift={(-5.7,3.4)}]
\draw[shift={(8,-6.5)},scale=2, line width=0.7mm,red] 
  (0.8,0.3) edge[->,bend right] (0.8,0.7)
 (0.2,0.7) edge[->,bend right] (0.2,0.3)
 (0.3,0.2) edge[->,bend right] (0.7,0.2)
 (0.7,0.8) edge[->,bend right] (0.3,0.8); 
 
 \node[shift={(7.5,-7.)},scale=3,] at (1.5,1.5) {\small $\alpha$}; 
\end{scope}

\begin{scope}[shift={(14,-3)},scale=0.7,xscale =0.5]

\coordinate (dd) at (3.5,7);

\begin{scope}
\filldraw[blue,xscale =2] (0,12) circle (0.14 cm);

\draw[orange, line width=0.5mm,shift={(-6,0)}]  (1.5,8) --(4.5,8) ;
\draw[opacity=0, fill=orange, fill opacity =0.2,shift={(-6,0)}] 
(1.5,8) --(4.5,8)--(4.5,0)--(1.5,0);

\draw[opacity=0, fill=orange, fill opacity =0.2] 
(-4.5,0) 
 .. controls  (-4.5, -1+0.8) and (-4,-1) .. 
(-3,-1) 
 .. controls (-2,-1) and (-1.5, -1+0.8) .. 
 (-1.5,0);

\draw (-3,-1)  .. controls (-2,-1) and (-1.5, -1+0.8) .. 
 (-1.5,0);
\draw[dashed,line width=0.5mm]  (-1.5,0)--(-1.5,1);
  
\draw (-1.5,1)  .. controls (-1.5,1) and (-1.5, 3) .. 
 (-1.5,11) 
  .. controls (-1.5,11.2) and (-0.5, 12) .. 
 (0,12);

\draw[xscale=-1] (-3,-1) 
 .. controls (-2,-1) and (-1.5, -1+0.8) .. 
 (-1.5,0);
\draw[xscale=-1,dashed,line width=0.5mm]  (-1.5,0)--(-1.5,1);
  
\draw [xscale=-1](-1.5,1)
  .. controls (-1.5,1) and (-1.5, 3) .. 
 (-1.5,11) 
  .. controls (-1.5,11.2) and (-0.5, 12) .. 
 (0,12);
 \end{scope}
 
 \begin{scope}[shift={(6,0)}]
 \filldraw[blue,xscale =2] (0,12) circle (0.14 cm);

\draw[orange, line width=0.5mm,shift={(-6,0)}]  (1.5,8) --(4.5,8) ;
\draw[opacity=0, fill=orange, fill opacity =0.2,shift={(-6,0)}] 
(1.5,8) --(4.5,8)--(4.5,0)--(1.5,0);

\draw[opacity=0, fill=orange, fill opacity =0.2] 
(-4.5,0) 
 .. controls  (-4.5, -1+0.8) and (-4,-1) .. 
(-3,-1) 
 .. controls (-2,-1) and (-1.5, -1+0.8) .. 
 (-1.5,0);

\draw (-3,-1) 
 .. controls (-2,-1) and (-1.5, -1+0.8) .. 
 (-1.5,0);
\draw[dashed,line width=0.5mm]  (-1.5,0)--(-1.5,1);
  
\draw (-1.5,1)  .. controls (-1.5,1) and (-1.5, 3) .. 
 (-1.5,11) 
  .. controls (-1.5,11.2) and (-0.5, 12) .. 
 (0,12);

\draw[xscale=-1] (-3,-1) 
 .. controls (-2,-1) and (-1.5, -1+0.8) .. 
 (-1.5,0);
\draw[xscale=-1,dashed,line width=0.5mm]  (-1.5,0)--(-1.5,1);
  
\draw [xscale=-1](-1.5,1)
  .. controls (-1.5,1) and (-1.5, 3) .. 
 (-1.5,11) 
  .. controls (-1.5,11.2) and (-0.5, 12) .. 
 (0,12);
 \end{scope}

\begin{scope}[shift={(-6,0)}]
\filldraw[blue,xscale =2] (0,12) circle (0.14 cm);

\draw (-3,-1) 
 .. controls (-2,-1) and (-1.5, -1+0.8) .. 
 (-1.5,0);
\draw[dashed,line width=0.5mm]  (-1.5,0)--(-1.5,1);
  
\draw (-1.5,1)  .. controls (-1.5,1) and (-1.5, 3) .. 
 (-1.5,11) 
  .. controls (-1.5,11.2) and (-0.5, 12) .. 
 (0,12);

\draw[xscale=-1] (-3,-1) 
 .. controls (-2,-1) and (-1.5, -1+0.8) .. 
 (-1.5,0);
\draw[xscale=-1,dashed,line width=0.5mm]  (-1.5,0)--(-1.5,1);
  
\draw [xscale=-1](-1.5,1)
  .. controls (-1.5,1) and (-1.5, 3) .. 
 (-1.5,11) 
  .. controls (-1.5,11.2) and (-0.5, 12) .. 
 (0,12);
 \end{scope} 
 
 \begin{scope}[shift={(12,0)}]
 \filldraw[blue,xscale =2] (0,12) circle (0.14 cm);

\draw[orange, line width=0.5mm,shift={(-6,0)}]  (1.5,8) --(4.5,8) ;
\draw[opacity=0, fill=orange, fill opacity =0.2,shift={(-6,0)}] 
(1.5,8) --(4.5,8)--(4.5,0)--(1.5,0);

\draw[opacity=0, fill=orange, fill opacity =0.2] 
(-4.5,0) 
 .. controls  (-4.5, -1+0.8) and (-4,-1) .. 
(-3,-1) 
 .. controls (-2,-1) and (-1.5, -1+0.8) .. 
 (-1.5,0); 
 
\draw (-3,-1) 
 .. controls (-2,-1) and (-1.5, -1+0.8) .. 
 (-1.5,0);
\draw[dashed,line width=0.5mm]  (-1.5,0)--(-1.5,1);
  
\draw[] (-1.5,1)  .. controls (-1.5,1) and (-1.5, 3) .. 
 (-1.5,11) 
  .. controls (-1.5,11.2) and (-0.5, 12) .. 
 (0,12);

\draw[xscale=-1] (-3,-1) 
 .. controls (-2,-1) and (-1.5, -1+0.8) .. 
 (-1.5,0);
\draw[xscale=-1,dashed,line width=0.5mm]  (-1.5,0)--(-1.5,1);
  
\draw [xscale=-1](-1.5,1)
  .. controls (-1.5,1) and (-1.5, 3) .. 
 (-1.5,11) 
  .. controls (-1.5,11.2) and (-0.5, 12) .. 
 (0,12);
 \end{scope}

\draw[opacity=0,fill ,fill opacity=0.1] 
 (-1.5,11) 
  .. controls (-1.5,11.2) and (-0.5, 12) .. 
 (0,12)
  .. controls (0.5, 12) and (1.5,11.2)   ..
  (1.5,11)  ;

\draw[shift={(-6,0)} ,opacity=0,fill ,fill opacity=0.1] 
 (-1.5,11) 
  .. controls (-1.5,11.2) and (-0.5, 12) .. 
 (0,12)
  .. controls (0.5, 12) and (1.5,11.2)   ..
  (1.5,11)  ;
  
  \draw[shift={(6,0)} ,opacity=0,fill ,fill opacity=0.1] 
 (-1.5,11) 
  .. controls (-1.5,11.2) and (-0.5, 12) .. 
 (0,12)
  .. controls (0.5, 12) and (1.5,11.2)   ..
  (1.5,11)  ;
  \draw[shift={(12,0)} ,opacity=0,fill ,fill opacity=0.1] 
 (-1.5,11) 
  .. controls (-1.5,11.2) and (-0.5, 12) .. 
 (0,12)
  .. controls (0.5, 12) and (1.5,11.2)   ..
  (1.5,11) ;
  
 \draw[opacity=0,fill ,fill opacity=0.1,shift={(-9,0)} ] (1.5,0)-- (1.5,11)--(4.5,11)--(4.5,0);

  \draw[opacity=0,fill ,fill opacity=0.1,shift={(-3,0)} ] (1.5,0)-- (1.5,11)--(4.5,11)--(4.5,0);

  \draw[opacity=0,fill ,fill opacity=0.1,shift={(3,0)} ] (1.5,0)-- (1.5,11)--(4.5,11)--(4.5,0);
  \draw[opacity=0,fill ,fill opacity=0.1,shift={(9,0)} ] (1.5,0)-- (1.5,11)--(4.5,11)--(4.5,0);


  \end{scope}


\end{tikzpicture}

\caption{For a rotation number $[0;1,\dots, 1,a_{n+1},1,\dots]$, the Siegel disk develops parabolic fjords on scale $n$ towards the $\alpha$-fixed point as $a_{n+1}\to \infty$.  The critical points of $f^{\qq_{n+1}}$ (blue) are beacons (on the top of) parabolic peninsulas. After adding appropriately truncated parabolic fjords (orange) to the Siegel disk, the resulting pseudo-Siegel disk is ``quasi-invariant'' up to $\qq_{n+1}$ iterates.}
\label{fig:fjords}
\end{figure}

In Section~\ref{s:par fjords} we will justify (in the near-degenerate regime) that parabolic fjords have translational geometry -- reminiscent to the Fatou coordinates for near-parabolic maps. Combined with Calibration Lemma~\ref{lmm:CalibrLmm}, this will imply that the critical points of $f^{\qq_{n+1}}$ are ``beacons'' (i.e., on the top) of the level-$n$ parabolic peninsulas.  

A pseudo-Siegel disk $\wZ^m$ is constructed by adding to $\overline Z$ all truncated parabolic fjords on scales $\ge m$, see Figure~\ref{fig:fjords}. We will show that $\wZ^m$ is quasi-invariant (in particular, injective) for all $f^i$ with $i\le \qq_{m+1}$. Moreover, the pseudo-Siegel disk $\wZ_f=\wZ^{-1}$ is uniformly qc (Theorem~\ref{thm:quasidisk}).

  Theorem~\ref{thm:beau:part U: c quasi line} establishes certain beau bounds to control the inner geometry of $\wZ^m$. The bounds imply, in particular, that errors do not accumulate under the regularization \[\overline Z \leadsto
\dots \leadsto \wZ^{m+2}\leadsto \wZ^{m+1} \leadsto \wZ^{m}\leadsto \dots\leadsto \wZ^{-1}=\wZ_f\]
Furthermore, the outer geometry of the Siegel disk is almost unaffected under $\overline Z  \leadsto \wZ^m$ -- most of the outer harmonic measure of $\overline Z$ sits on the tops of peninsulas, see~\S\ref{ss:WellGrounded}. In other words, a random walk in $\wC\setminus \overline Z$ starting at $\infty$ is unlikely to enter any truncated parabolic fjord of any level if the truncation is chosen sufficiently deep.

To control the outer geometry and its interaction with the inner geometry of $\wZ^m$, we will introduce the following degeneration parameters. For an interval $I\subset \partial \wZ^m$, we denote by $\lambda I\subset \partial \wZ^m$ the $\lambda$-rescaling of $I$ in the linearized coordinates of $\partial Z$. Then $\Width_\lambda(I)$ is the width of the family of curves connecting $I$ to $\partial \wZ^m \setminus (\lambda I)$. Similarly, $\Width^+_\lambda(I)$ is the width of the outer family of curves (i,e., in $\wC\setminus Z$) connecting $I$ to $\partial \wZ^m \setminus (\lambda I)$. If $\Width_\lambda(I) = K\gg 1$, then iterating Snake Lemma~\ref{simplmm:SnLmm:wZ}, we can eventually find $J$ with $\Width^+_{\lambda}(J)\succeq K$ and $|J|\le |I|$, where ``$|\ |$'' denotes the length in the linearized coordinates of $\partial Z$.

If for a combinatorial interval $I\subset \partial \wZ^m$ on the renormalization level $m$ we have $\Width^+_\lambda(I)=K\ge \bK$ for an absolute threshold $\bK\gg 1$, then the Covering Lemma allows us to spread the associated degeneration around $\partial \wZ^m$. Then Snake-Lair Lemma~\ref{lem:Hive Lemma} finds a bigger degeneration: there will exist a combinatorial interval $J_2\subset \partial \wZ^n$ for some $n>m$ such that $\Width^+_\lambda(J_2)\ge 2K$ and $|J_2|\le |J|$. Proceeding by induction, we obtain a sequence of shrinking intervals $J_n$ such that $\Width^+(J_n) \ge 2^nK$ eventually contradicting that $\overline Z$ is a $K_\theta$-quasidisk. This establishes Theorem~\ref{thm:unform bounds}. 

Theorem~\ref{thm:quasidisk} is obtained by justifying a universal combinatorial bound for the truncation depth. It allows us to control Hausdorff limits of $\overline Z_{f_\theta}$ as $\theta$ approaches any irrational number; the resulting limits are Mother Hedgehogs (Theorem~\ref{thm:hedgehog}).

The main part of the paper is organized as follows. In Part~\ref{part:near rotat} we will show that near-rotation domains and parabolic fjords are coarse qs-equivalent to rotations of the unit disk -- see Proposition~\ref{prop:part U: c quasi line} and Theorems \ref{thm:beau:part U: c quasi line},~\ref{thm:par fjords}. In Part~\ref{part:psSiegDisk}, we will introduce pseudo-Siegel disks, show that they quasi-behave as uniformly bounded Siegel disks, establish Snake Lemma~\ref{simplmm:SnLmm:wZ}. Corollary~\ref{cor:regul} states that either the regularization $\wZ^{m+1}\leadsto \wZ^{m}$ is possible or there is a much bigger degeneration on some scale $\ge m+1$. In Part~\ref{part:CovCalibr} we will prove Theorem~\ref{thm:SpreadingAround} (application of the Covering Lemma followed by Snake-Lair Lemma~\ref{lem:Hive Lemma}) and Calibration Lemma~\ref{lmm:CalibrLmm}; they say that if the outer geometry of $\wZ^{m+1}$ is sufficiently degenerate on scale ``between $m$ and $m+1$'', then this degeneration can be ``calibrated'' on level $m+1$ combinatorial interval. The main theorems are proven in Part~\ref{part:conclus}. For the readers convenience, in the beginning of each section we provide its detailed outline.


\begin{rem}
\label{remark: going deep}
Let us note that our inductive argument goes in the opposite direction compared with the primitive quadratic-like renormalization theory~\cite{K,KL2}. Indeed, we show that high degeneration on a certain level implies even higher degeneration on a deeper rather than shallower level.

The idea that degeneration can also be accounted (with amplification) on a deeper level helped us to develop an extension to Kahn's argument \cite{K} and establish the MLC at the classical Feigenbaum parameter \cite{DL:Feigen}.
\end{rem}

\subsection{Main notations and conventions}
We state here our main notations and conventions; see~\S\ref{s:PrepNot} for more details. We denote by
 \begin{itemize}
 \item $\Theta_\bnd=\{\theta=[0;a_1,a_2,\dots] \mid a_i \le M_\theta\}$ the set of bounded (irrational) rotation numbers;
 
 \item $\ee(\theta)\coloneqq e^{2\pi i \theta}$;

\item $Z$ the Siegel disk of $f=f_\theta$ for $\theta\in \Theta_\bnd, \sp 0<\theta<1$;

\item $c_0, c_1\in  \partial Z$ the critical point and critical value of $f$; 

\item more generally: $c_n\coloneqq (f\mid \partial Z)^n (c_0)$;

 
\item $h\colon (Z,\alpha)\to (\Disk,0)$ the conformal conjugacy between $f$ and $z\mapsto e(\theta) z$;

\item $|I|\coloneqq |h(I)|_{\R/\Z}$ the (combinatorial) Euclidean length of an interval $I\subset \partial Z$;

\item $\theta_n\in (-1/2,1/2)$ the rotation number of $f^{\qq_n}\mid Z$ and $\ell_n\coloneqq |\theta_n|$ the length of a  \emph{level $n$ combinatorial} interval $[x,f^{\qq_n}(x)]\subset \partial Z$, where $\pp_n/\qq_n\approx \theta$ are best approximations;

\item $x \boxplus \nu \coloneqq h^{-1}(h(x)+\nu)$, where $x\in \partial Z$, $\nu\in \R$, and $h(x)+\nu \in \R/\Z\simeq \partial \Disk$;

\item $\Width_\lambda(I)\coloneqq \Width(\Fam_\lambda(I))$ and $\Width^+_\lambda(I)\coloneqq \Width(\Fam^+_\lambda(I))$ the width of the full and outer families measuring degeneration of $\overline Z$ at an interval $I\subset \partial Z$, see~\S\ref{ss:FullOutFam};

\item ``$<$'' denotes a clockwise orientation on $\partial Z$; 

\item for an interval $I\subset \partial Z$ and $x,y\in I$ we write $x<y$ rel $I$ if $x$ is on the left of $y$ in $I$, i.e.~$\partial Z\setminus I , x,y$ are clockwise oriented;

\item for a pair of disjoint intervals $I,J\subset \partial Z$ we define $\lfloor I, J \rfloor\coloneqq I\cup L\cup J$, where $L$ is the complementary interval between $I,J$ so that $I<L <J$; in most cases $L$ will be the shortest interval between $I$ and $J$;

 \item  $x\oplus y =(x^{-1}+y^{-1})^{-1},\sp x,y>0$ the harmonic sum -- see the Gr\"otzsch inequality~\eqref{eq:Grot};

 \item $\gamma \#\beta $ the concatenation of curves $\gamma$ and $\beta$.
 \end{itemize}
 By default, curves are considered up reparametrization and are usually parameterized by the unit interval $[0,1]$. We say a curve $\gamma\colon (0,1)\to \wC\setminus (A\cup B)$ \emph{connects} $A$ and $B$ if \[\lim_{\tau\to 0}\gamma(\tau)=\gamma(0)\in A\sp\sp\text{ and }\sp\sp \lim_{\tau\to 0}\gamma(1-\tau)=\gamma(1)\in B.\]

A \emph{Jordan} disk is a closed or open topological disk bounded by a Jordan curve.

We write \[f(x)\preceq g(x) \sp\sp \text {if}\sp  f(x)\le C g(x),\sp\sp f(x),g(x)>0\] 
for an absolute constant $C>0$. Similarly:
\[f(x)\preceq_\kappa g(x) \sp\sp \text {if}\sp  f(x)\le C_\kappa g(x),\sp\sp f(x),g(x)>0\] 
for a constant $C_\kappa>0$ depending on $\kappa$. The big $O(\sp )$ notation describes at most linear dependence on the argument: $O(f(x))\preceq f(x)$. Similarly, $O_\kappa(f(x))\preceq_\kappa f(x)$ is at most linear dependence on $f(x)$ with a constant depending on $\kappa$.

We will write ``$A\gg B$'' to assume that $A$ is sufficiently bigger than $B$. Similarly, ``$A\gg_\kappa B$'' means that $A$ is sufficiently big than $B$ depending on a parameter $\kappa$.

We will often need to truncate laminations $\Fam,\FamG,\FamH$ by removing buffers of certain sizes. We will use upper indices  ``$\new,\New,\NEW$'' to denote new truncated families with the convention
\[\Fam \supset \Fam^\new \supset \Fam^\New \supset \Fam^\NEW.\]

Slightly abusing notations, we will often identify a lamination with its support.

A \emph{vertical} curve of a rectangle $\RR$ is a curve that becomes vertical after conformal identification $\RR$ with a standard Euclidean rectangle.

{\bf Acknowledgment.} The first author was partially supported by Simons Foundation grant of the IMS, the ERC grant ``HOLOGRAM,'' and the NSF grant DMS $2055532$. The second author has been partly supported by the NSF, the Hagler and Clay Fellowships, the Institute for Theoretical Studies at ETH (Zurich), and MSRI (Berkeley).

We are grateful to Davoud Cheraghi and Mitsuhiro Shishikura for fruitful  discussions of potential applications of our results that motivated \cite{DL:sector bounds}. 

We are grateful to Yusheng Luo for multiple discussions regarding subtleties of the $\psi^\bullet$-ql formalism~\cite{DL:psi bounds} in relation to~\cite{DL:HypComp}.

The results of  this paper were announced at the Fields Institute Symposium celebrating Artur
Avila's Fields medal (November 2019), and during the SCGP Renormalization program in December 2020 (see the mini-course on \href{https://scgp.stonybrook.edu/video/results.php?event_id=281}{http://scgp.stonybrook.edu/video/results.php?event\_id=281}).

\part{Rotational geometry}
\label{part:near rotat}

\section{Preparation}
\label{s:PrepNot}
 We fix a neutral quadratic polynomial $f:z\mapsto \ee(\theta)z+z^2$ with bounded type rotation number $\theta\in (0,1)$; i.e.~the $\alpha$-fixed point of $f$ has multiplier \[\ee(\theta)=e^{2\pi i \theta}, \sp\sp \theta\in \Theta_\bnd.\]

\subsection{Siegel Disk}
\label{ss:SiegDisk}
Let us denote by $Z$ the Siegel disk of $f$. Recall that $\overline Z$ is a qc disk because  $\theta\in \Theta_\bnd.$ Consider a Riemann map \[h\colon \overline Z\to \overline \Disk, \sp\sp h(\alpha)=0\] conjugating $f\mid \overline Z$ to $z\mapsto \ee(\theta)z$.

The \emph{(combinatorial) length} of an interval $I\subset \partial Z\simeq \R/\Z$ is defined by 
\[ 
|I| \coloneqq |h(I)|_{\R/\Z}  \in  (0,1).\]
Similarly, the \emph{(combinatorial) distance} \[\dist(x,y)\coloneqq \dist_{\R/\Z}(h(x),h(y))\in [0,1/2],\sp\sp x,y\in \partial Z\] is defined. It also induces the distance between subsets of $\partial Z$.

Given $t\in \R/\Z$ and $x\in \partial Z$, we set 
\[x\boxplus t = h ^{-1} \big( \ee(t) h(x)\big),\]
i.e.~$x\boxplus t$ is $x$ rotated by angle $t$. We have
\[ f(x)=x\boxplus \theta ,\sp\sp \text{ for }x\in\partial Z.\]

\subsubsection{Closest returns of $ f| \partial Z$}
\label{sss:closest returns}
Let $\theta=[0;a_1,a_2,\dots ]$ be the continued fraction expansion and consider the sequence of the best approximations of $\theta$
\[\pp_n/\qq_n\coloneqq \left\{ 
\begin{array}{cc}
[0;a_1,a_2,a_3,\dots, a_n] & \text{if }a_1>1\\{} 
[0;1,a_2,a_3,\dots, a_{n+1}] & \text{if }a_1=1
\end{array}
\right.,\]and set $\qq_0\coloneqq 1$. Then $f^{\qq_0}=f,f^{\qq_1},f^{\qq_2},\dots$ is the sequence of the closest returns of $f\mid \partial Z$; i.e.
\[ \dist \big( f^i (x),x\big) >\dist \big( f^{\qq_n} (x),x\big)\eqqcolon \length_n, \sp x\in \partial Z\sp\sp \text{ for all }i<\qq_n.\]
For $n\ge 0$, we specify $\theta_n\in (-1/2,1/2)$ so that
\[ f^{\qq_n} (x)= x\boxplus \theta_n, \sp\sp x\in \partial Z.\]
In particular, $\theta_0=\theta$ if $\theta <1/2$ and $\theta_0=\theta-1$ otherwise. By construction, $\length_n=|\theta_n|$. The sequence $\theta_n$ is alternating: $\theta_n\theta_{n+1}<0$ -- reflecting the fact that $\theta$ is between $\pp_{n}/\qq_{n}$ and $ \pp_{n+1}/\qq_{n+1}.$ Since $\length_{n}\ge \length_{n+1}+\length_{n+2}$ and $\length_{n+2}<\length_{n+1}$, we have
\begin{equation}
\label{eq:length decrease}
\length_{n+2} < \length_n/2.
\end{equation}

\subsubsection{Intervals} Consider two points $x,y\in \partial Z$. Unless stated otherwise, we denote by $[x,y]$ the shortest closed interval of $\partial Z$ between $x$ and $y$. The corresponding  open interval is denoted by $(x,y)$. Most of the intervals will be closed.

For an interval $I\subset \partial Z$, we denote by $I^c=\overline {\partial Z\setminus I}$ its complement.

We denote by ``$<$'' the clockwise order on $\partial \Disk$ and on $\partial Z$. Given two non-equal points $a,b$ in an interval $I $ with $ |I|<1/2$, we say that $a$ is on the \emph{left} of $b$, and write $a<b$, if $I^c, a,b$ have the clockwise order. This convention is consistent with drawing intervals on the upper side of $\partial Z$, see Figures~\ref{Fg:ShrinkingLmm:Disk}, \ref{fig:Width Width^+}. (Note that $x\mapsto x\boxplus \varepsilon$ is a counterclockwise rotation for a small $\varepsilon>0$.)

Given intervals $I,J\subset \partial  Z$, we define $\lfloor I,J\rfloor \subset \partial  Z$ to be the interval $I\cup L\cup J$, where $L$ is the complementary interval of $I,J$ (i.e.,~a component of $\partial Z\setminus (I\cup J)$) specified so that $I< L <J$ with respect to the clockwise order. In other words, $\lfloor I,J\rfloor$ is the shortest interval containing $I\cup J$ such that $I<J$ in $\lfloor I,J\rfloor$. In most cases, $\lfloor I,J\rfloor$ will be the shortest interval containing $I,J$.

Given an interval $I\subset \partial Z$ and $\lambda \ge 1$, we define
\begin{equation}
\label{eq:dfn:lambdaI}
\lambda I \coloneqq \{x\in \partial Z:\sp  \dist(x , I) \le (\lambda-1)|I|/2\}
\end{equation}
to be the $\lambda$-rescaling of $I$ with respect to its center.

\subsubsection{Combinatorial intervals}
For $n\ge0,$ a \emph{combinatorial level $n$ interval} is an interval $I\subset \partial Z$ with length $\length_n$. It has the form \[I=[x,f^{\qq_n}(x)]=[x,x \boxplus\theta_n] \sp\sp\text{ where }x\in \partial Z. \]
Since the sequence $\theta_n$ is alternating, we have:

\begin{lem}
\label{lem:FirstLand}
The return time of points in a level $n\ge 0$ combinatorial interval $I=[x,x \boxplus\theta_n]$ is at least $\qq_{n+1}$:
\begin{equation}
f^i(y)\not\in I \sp\sp \text{ for }\sp  y\in (x,x \boxplus\theta_n),\sp  i< \qq_{n+1}. 
\end{equation}\qed
\end{lem}

As a consequence, the commuting pair 
\begin{equation}
\label{eq:ren:com_pair}
\big(f^{\qq_{n+1}}\mid [f^{\qq_n}(x),x],\sp  f^{\qq_{n}}\mid [x,f^{\qq_{n+1}}(x)]\big)
\end{equation}  
realizes the first return map to $[f^{\qq_n}(x), f^{\qq_{n+1}}(x)]$. The renormalization theory of circle maps is often set up using commuting pairs.

\subsubsection{Renormalization tilings} \label{sss:renorm til}
Given $x\in \partial Z$ and $n\ge 0$, we denote by $\Jbb_n(x)$ the associated \emph{renormalization tiling of level $n$}:  
\begin{equation}
\label{eq:ren_til}
 \bigcup _{i=0}^{\qq_{n+1}-1} f^i [f^{\qq_n}(x),x] \sp\cup \sp \bigcup _{i=0}^{\qq_{n}-1} f^i [x,f^{\qq_{n+1}}(x)]. 
\end{equation}
We also set $\Jbb_n\coloneqq \Jbb_n(c_0)$, where $c_0$ is the free critical point in $\partial Z$. 

Note that most of the intervals in~\eqref{eq:ren_til} are in the orbit of $[f^{\qq_n}(x),x]$ (explaining the subindex $n$  in $\Jbb_n(x)$). Moreover, level $n$ intervals in $\Jbb_n$ form an almost tiling, with gaps being intervals of level $n+1$:
\begin{lem}
\label{lem:Jbb_n}
Level $n+1$ combinatorial intervals in~\eqref{eq:ren_til} are disjoint.  
\end{lem}
\begin{proof}
It is a well known statement that easily follows by induction. In $\Jul_0$, there is a single level $1$ interval and $\qq_1\ge 2$ level $0$ intervals. The tiling $\Jbb_{n+1}$ is obtained from $\Jbb_{n}$  by decomposing every level $n$ interval into level $n+1$ intervals and a single level $n+2$ interval either on the level or on the right depending on the parity of $n$; i.e. level $n+2$ intervals are disjoint in $\Jbb_{n+1}$.
\end{proof}

For $n>m$, we say that level $n$ combinatorial intervals are on \emph{deeper scale} than level $m$ combinatorial intervals, while level $m$ combinatorial intervals are on \emph{shallower scale} than level $n$ combinatorial intervals.

\subsubsection{Spreading around a combinatorial interval}\label{sss:spread around} Consider a combinatorial level $n$ interval $I$. We say that the intervals \[\{ f^{i}(I)\mid i\in \{0,1,\dots, \qq_n-1\}\}\] are obtained by \emph{spreading around $I$}. We enumerate these intervals counterclockwise starting with $I=I_0$
\begin{equation}
\label{eq:spred_around:I} I_0=I,\sp I_1=f^{i_1}(I), \dots, I_{\qq_n-1}=f^{i_{\qq_n-1}-1}(I),\sp\sp\sp i_j\in\{1,2,\dots,\qq_n-1\}.
\end{equation}

It follows from Lemma~\ref{lem:Jbb_n} that either $I_i$ is attached to $I_{i+1}$ or there is a level $n+1$ combinatorial complementary interval between $I_{i}$ and $I_{i+1}$.

\subsubsection{Diffeo-tilings}
\label{sss:diff tilings} For $n\ge -1$, we denote by $\CP_n=\CP_n(f)=\CP(f^{\qq_{n+1}})$ the set of critical points of $f^{\qq_{n+1}}$.
 The \emph{diffeo-tiling $\Dbb_n$ of level $n$} is the partition of $\partial Z$ induced by $\CP_n$: every interval in $\Dbb_n$ is the closure of a component of $\partial Z \setminus \CP_n$. For $n=-1$, the tiling consists of a single ``interval'' viewed as $[c_0,c_0\boxplus 1]$.

 For $n\ge 0$, every point in $\CP_n$ is an endpoint of an interval in $\Jbb_n(c_{-\qq_n+1})$; we see that $\Dbb_n$ is an enlargement of $\Jbb_n(c_{-\qq_{n+1}+1})$. By Lemma~\ref{lem:Jbb_n}, every interval in $\Dbb_n$ has length either $\length_n$ or $\length_n+\length_{n+1}$. 
 
 Enumerating counterclockwise intervals in $\Dbb_n$ as $P_0, P_1,\dots, P_{\qq_{n+1}}$, we have $f^{i\qq_{n+1}}(P_k)\approx P_{k+ i\pp_{n+1} }$, where the ``rotational error'' with respect to the combinatorial length is small if $\length_{n+1}\ll \length_n$.

Let us set 
\begin{equation}
\label{eq:dfn:filled_n}
\filled_n\coloneqq f^{-\qq_{n+1} } (\overline Z),
\end{equation}
Every interval of $\Dbb_n$ is between two components, called \emph{limbs}, of $\filled_n\setminus \overline Z$.

For  $T\in \Dbb_n$, we write $T'\coloneqq T\cap f^{\qq_{n+1}}(T)\subset T$ so that $f^{\qq_{n+1}}\colon T'\boxminus \theta_{n} \to T'$ is a homeomorphism. If $n=-1$, then $T'$ is the longest interval connecting $c_1$ and $c_0$.

The \emph{nest of diffeo-tilings} is 
\begin{equation}
\label{eq:diffeo nest}
\Dbb\coloneqq [\Dbb_n]_{n\ge -1}
\end{equation}

It follows from~\eqref{eq:length decrease} that
\begin{lem}
\label{lem:Dbb_n:Dbb_n+2}
Every interval of $\Dbb_n$ consists of at least $2$ intervals of $\Dbb_{n+2}$.\qed.
\end{lem}

\subsubsection{Fjords}
\label{sss:fjords}

Consider an interval $T=[v,w]$ in $\Dbb_m$, $m\ge -1$. Our standing assumption is that $f^{\qq_{m+1}}\mid T$ moves points clockwise towards $w$, see Figure~\ref{fig:par rect}; i.e., we assume that $ v<w$ and $\theta_{m+1} <0$, compare with~\eqref{eq:T:orientat ass}. Let us set $v'=v\boxplus \theta_{m+1}$ and $w'=w\boxplus \theta_{m+1}$. We specify 
\[[v',w'] = f^{\qq_{m+1}}(T)\subset \partial Z ,\sp \sp\sp T'\coloneqq[v',w]\subset T,\sp \sp\sp\text{ and}\]
\begin{itemize}
\item the \emph{level $m$-fjord} $\Fjord_m(T)$ on $T$ to be the closed disk bounded by $T$ and by the hyperbolic geodesic $\ell_{[v,w]}^{(m)}$ of $\C\setminus \filled_m$ with endpoints $\{v,w\}$ such that $\ell_{[v,w]}^{(m)}$ is homotopic to $[v,w]$ in $\C\setminus \intr(\filled_{m})$;
 \item the \emph{fjord} $\Fjord[v',w']$ on $[v',w']$ to be the closed disk bounded by $[v',w']$ and the hyperbolic geodesic $\ell_{[v',w']}$ of $\C\setminus \overline Z$ with endpoints in $\{v',w'\}$ such that $\ell_{[v',w']}$ is homotopic to $[v',w']$ in $\C\setminus Z$.
\end{itemize}

Similarly, $\Fjord_m(J)\subset \Fjord_m(T)$ and $\Fjord(S)\subset \Fjord[v',w']$ are defined for all subintervals $J\subset T$ and $S\subset [v',w']$. 

Since $f^{\qq_{m+1}}\colon \C\setminus \filled_m\to \C\setminus \overline Z$ is a covering map, we have:

\begin{lem}
\label{lem:Fjord_m(T)}
For the fjord $\Fjord_m(T)$ as above, $f^{\qq_{m+1}}\mid \Fjord_m(T)$ is injective:
\[ f^{\qq_{m+1}}\colon\ \Fjord_m(T)\ \overset{1:1}{\longrightarrow} \ \Fjord[v',w']. \]
We denote by $f^{-\qq_{m+1}}_T\colon \Fjord[v',w'] \overset{1:1}{\longrightarrow} \Fjord_m(T)$ the inverse branch.

  If $J\subset T'$ is a subinterval such that \[ f^{-t\qq_{m+1}}_T(J)\subset T'\sp\text{ for all }\ t\in \{ 0,1,\dots, s-1\},\] 
\[\text{then: }\sp\sp \sp  f_T^{ -s\qq_{m+1}}\big(\Fjord(J)\big)\subset\Fjord_m(T);\]
i.e., $\Fjord(J)$ can be iteratively lifted under $f^{ s\qq_{m+1}}$. 
\qed
\end{lem}  

We say that a fjord is \emph{parabolic} if it contains a wide parabolic rectangles as in Section~\ref{s:par fjords}. Theorem~\ref{thm:par fjords} will describe the geometry of parabolic fjords.

\subsection{Inner geometry of $\overline Z$: Log-Rule}
\label{ss:GeomZ}
 For disjoint intervals $I,J\subset\partial Z$,  the \emph{inner family} $\Fam^-(I,J)=\Fam_{\ovZ}(I,J)$ is the family of all curves in $\ovZ$ connecting $I,J$; see also~\S\ref{sss:can rect}. Its width $\Width^-(I,J)$ can be explicitly computed:

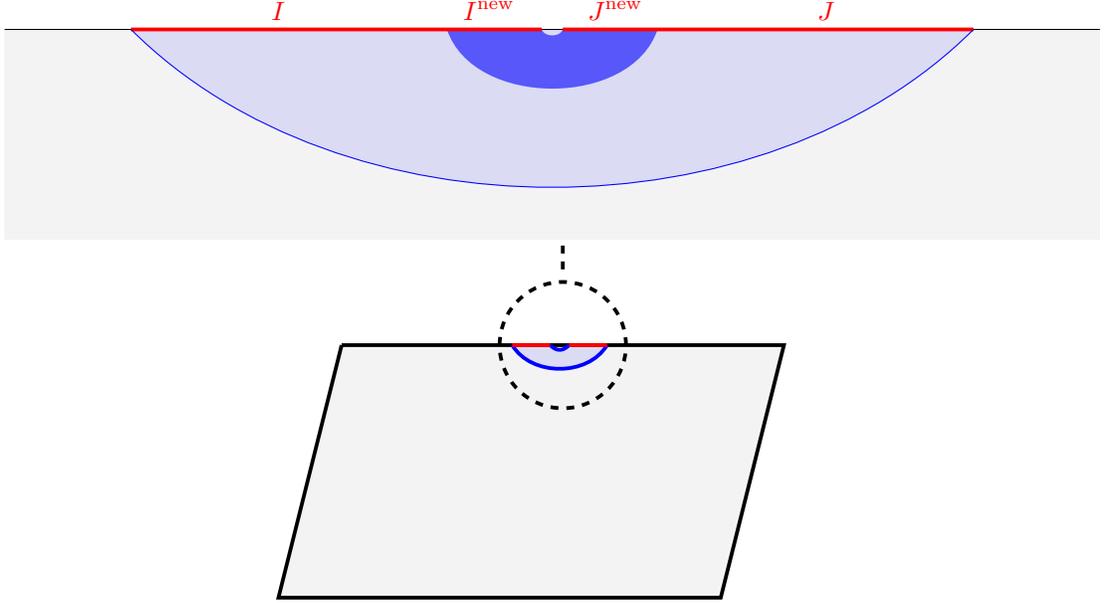
\begin{figure}[t!]
\[\begin{tikzpicture}[scale=1.4]

\begin{scope}[shift={(-2,-3)},scale=0.6]
\draw[line width =0.05 cm,fill, fill opacity =0.05] (0,0) -- (7,0) -- (6,-4)--(-1,-4)--(0,0)--(0.,0);

\draw [blue,line width=0.05 cm] (2.7,0)
 .. controls (3,-0.5) and (3.9,-0.5) ..
 (4.2,0)
(3.6,-0.) 
 .. controls (3.5,-0.1) and (3.4,-0.1) ..
(3.3,-0.);

\draw [opacity=0,fill = blue, fill opacity =0.1] 
(2.7,0)
 .. controls (3,-0.5) and (3.9,-0.5) ..
 (4.2,0)
  .. controls (4,-0) and (3.8,-0.) ..
(3.6,-0.) 
 .. controls (3.5,-0.1) and (3.4,-0.1) ..
(3.3,-0.);

\draw[red, line width=0.05 cm]
 (4.2,0) --(3.6,-0.) 
(2.7,0)-- (3.3,-0.);

\draw[dashed,line width =0.05 cm] (3.5,0) circle (1cm)
(3.5,1.2) -- (3.5,1.7);

\end{scope}

\node[red,above] at(2.6,0){$J$}; 
\node[red,above] at(-2.6,0){$I$};

\draw (-5.2,0) -- (5.2,0);

\draw[opacity=0, fill, fill opacity=0.05] (-5.2,0) -- (5.2,0)--(5.2,-2) -- (-5.2,-2);

\node[red,above] at(0.6,0){$J^\new$}; 
\node[red,above] at(-0.6,0){$I^\new$};

\draw[opacity=0, fill =blue, fill opacity=0.1 ]
(0.1,0)
 .. controls (0.075, -0.075) and (-0.075,-0.075) ..
 (-0.1,0)
 .. controls (-0.5, 0) and (-0.8,-0.) ..
(4,0)
 .. controls (2, -2) and (-2,-2) ..
 (-4,0);

\draw[opacity=0, fill =blue, fill opacity=0.6 ]
(0.1,0)
 .. controls (0.075, -0.075) and (-0.075,-0.075) ..
 (-0.1,0)
 .. controls (-0.5, 0) and (-0.8,-0.) ..
(-1,0)
 .. controls (-0.75, -0.75) and (0.75,-0.75) ..
 (1,0);

\draw[red,line width=0.5mm]  (-4,0) --(-0.1,0)
(4,0) --(0.1,0);

 \draw[blue]
(4,0)
 .. controls (2, -2) and (-2,-2) ..
 (-4,0);

\end{tikzpicture}\]
\caption{Illustration to the Localization Property: the width $\Fam^-(I,J)$ is within $\Fam^-(I^\new,J^\new)$ up to $O(\log \lambda)$, where $I^\new, J^\new$ is an innermost subpair.}
\label{Fg:ShrinkingLmm:Disk}
\end{figure}

\begin{lem}[Log-Rule]
\label{lmm:W^- I J}
Consider intervals $I,J\subset \partial Z$. If $\dist (I,J)\le \min\{|I|,|J|\}$, then
\begin{equation}
\label{eq:1:lmm:W^- I J}
\Width^-(I,J) \asymp \log \frac{\min\{|I|,|J|\}}{\dist (I,J)} +1;
\end{equation}
otherwise 
\begin{equation}
\label{eq:2:lmm:W^- I J}
\Width^-(I,J) \asymp \left( \log \frac{\dist (I,J)}{\min\{|I|,|J|\}} +1\right)^{-1}.
\end{equation}
\end{lem}
\noindent We will later establish a similar Log-Rule to near-rotation domains (Proposition~\ref{prop:part U: c quasi line}  and Theorem \ref{thm:beau:part U: c quasi line}) and to parabolic fjords (Theorem~\ref{thm:par fjords}).

\begin{proof}
Since $\ovZ$ and $\ovDisk$ are conformally identified, it is sufficient to prove the lemma for $\ovDisk$. Observe first that for $A,B\subset \partial \Disk$
\begin{equation}
\label{eq:HypRect is 1}
\Width^-(A,B) \asymp \Width(\RR(A,B))\asymp 1 \sp\sp\text{ if }\dist(A,B) \asymp \min\{|A|,|B|\},
\end{equation} 
where $\RR(A,B)$ is the geodesic rectangle, see \S\ref{sss:GeodRect}. Indeed, the condition $\dist(A,B) \asymp \min\{|A|,|B|\}$ implies that the cross-ratio of $4$ endpoints of $A, B$ is comparable to $1$. Applying a M\"obius transformation, we can assume that all $4$ intervals (i.e.,~$A,B$ and two complementary intervals between $A,B$) have comparable length, and the claim follows by compactness.

Suppose $\dist (I,J)\le \min\{|I|,|J|\}$. We also assume that $|I|\le |J|$, and we present $I$ and $J$ as concatenations \[I=I_1\#I_2\#\dots \# I_n\sp\sp\text{ and }\sp\sp J_1\# J_2\#\dots \# J_n,\]
where $n\asymp  \log \frac{\min\{|I|,|J|\}}{\dist (I,J)} +1$, such that 
\[\dist (I_k,J_k)\asymp |I_k|\asymp \dist (I_k,J) .\]
By the Parallel Law, we obtain:
\[ \Width^-(I,J)\le \sum_{k=1}^n \Width(I_k,J) \preceq n,\]
\[ \Width^-(I,J)\ge \sum_{k=1}^n \Width(\RR(I_k,J_k)) \succeq n.\]
This proves~\eqref{eq:1:lmm:W^- I J}. If $\dist (I,J)\ge \min\{|I|,|J|\}$, then $\Width^-(I,J)=1/\Width^-(A,B)$, where $A,B$ are complementary intervals between $I,J$;~i.e.,~\eqref{eq:2:lmm:W^- I J} follows from \eqref{eq:1:lmm:W^- I J}. \end{proof}

\begin{rem}[Splitting Argument]
\label{rem:SplitArg} Note that in the proof of Lemma~\ref{lmm:W^- I J} we established \eqref{eq:1:lmm:W^- I J} and \eqref{eq:2:lmm:W^- I J} from~\eqref{eq:HypRect is 1} by splitting $I$ and $J$ into an appropriate number of intervals. We will call it the \emph{Splitting Argument} -- this argument will be used several times later.
\end{rem}

\subsubsection{Localization Property}
\label{sss:LocProp} Consider a pair $I,J \subset \partial D$, where $D\subset \C$ is a closed Jordan disk. A subpair $I^\new\subset I,\sp J^\new\subset J$ is called \emph{innermost} if 
\[ I\setminus I^\new < I^\new< J^\new< J\setminus J^\new \sp\sp\sp\text{ in }\lfloor I,J\rfloor.\]

Lemma~\ref{lmm:W^- I J} implies the following localization property, see Figure~\ref{Fg:ShrinkingLmm:Disk}.  For a pair of intervals $I,J$  with $|\lfloor I, J\rfloor| \le 1 - \frac{1}{\lambda}\min\{|I|,|J|\}$, define $I^\new \subset I$ and $J^\new \subset J$ to be the closest innermost subpair such that 
\[ |I^\new| =|J^\new| =\frac 1 \lambda \min\{|I|,|J|\}.\]
Then most of the width of $\Fam^-(I,J)$ is in $\Fam^-(I^\new, J^\new)$: 
\[ \Width^-(I^\new, J) + \Width^-(I,J^\new)= O(\log \lambda).\]

\subsubsection{Squeezing Property}
\label{sss:SqueezProp} A counterpart to the localization property is the following squeezing property which also follows from Lemma~\ref{lmm:W^- I J}.

There is a constant $C>0$ such that the following holds. Suppose $I, J\subset \partial Z$ is a pair of intervals such that
\[ \Width^{-}(I,J) \ge C\log \lambda,\sp\sp \lambda>2.\]
Then 
\[\dist(I,J)\le \frac 1 \lambda \min\{|I|,|J|\}.\]

We will later generalize Localization and Squeezing Properties to pseudo-Siegel disks, see~\S\ref{ss:wZ:local lmm}.

\subsection{Outer geometry of $\overline Z$} In this section we will define $\Width_\lambda (I)=\Width(\Fam_\lambda(I)),$ $\Width^+_\lambda (I)=\Width(\Fam^+_\lambda(I))$ and other quantities to measure degeneration of $\overline Z$, see Figure~\ref{fig:Width Width^+}. 


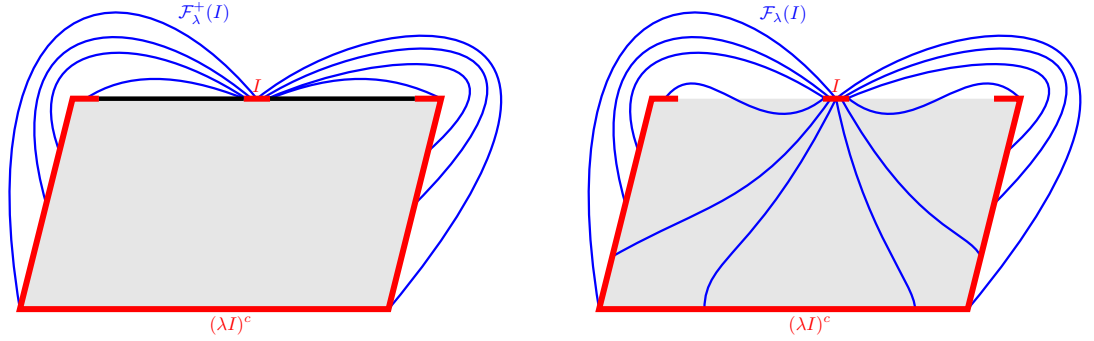
\begin{figure}

\begin{tikzpicture}[scale=0.7,every node/.style={scale=0.7}]

\draw[opacity=0,fill, fill opacity =0.1] (0,0) -- (7,0) -- (6,-4)--(-1,-4)--(0,0)--(0.,0);

\draw[line width =0.03 cm,blue]
(3.26,0)  .. controls (2, 0.5) and (1,0.5) ..
(0.25,0)
(3.3,0)  .. controls (1, 1.5) and (-1,1) ..
(-0.25,-1)
(3.4,0)  .. controls (1, 2) and (-1.5,1.5) ..
(-0.5,-2)
(3.5,0)  .. controls (1, 3) and (-2,2) ..
(-1,-4)
(3.74,0)  .. controls (5, 0.5) and (6,0.5) ..
(7,0)
(3.7,0)  .. controls (6, 1) and (9,1) ..
(6.75,-1)
(3.6,0)  .. controls (6, 1.5) and (10,1.5) ..
(6.5,-2)
(3.5,0)  .. controls (6, 2) and (11,2) ..
(6,-4)
;

\node[above,blue] at (2.5,1.3) {$\Fam^+_\lambda(I)$};

\node[above,red] at (3.5,0) {$I$};
\node[above,below,red] at (3,-4) {$(\lambda I)^c$};

\draw[line width =0.08 cm,red] (6.5,0) -- (7,0) -- (6,-4)--(-1,-4)--(0,0)--(0.5,0)
(3.25,0)  -- (3.75,0);
\draw[line width =0.06 cm] (3.75,0)--(6.5,0)
(0.5,0)--(3.25,0);

\begin{scope}[shift={(11,0)}]
\draw[opacity=0,fill, fill opacity =0.1] (0,0) -- (7,0) -- (6,-4)--(-1,-4)--(0,0)--(0.,0);

\draw[line width =0.03 cm,blue]
(3.26,0)  .. controls (2, -1) and (1,1) ..
(0.25,0)
(3.3,0)  .. controls (1, 1.5) and (-1,1) ..
(-0.25,-1)
(3.4,0)  .. controls (1, 2) and (-1.5,1.5) ..
(-0.5,-2)
(3.4,0)  .. controls (2, -2) and (1,-2) ..
(-0.75,-3)
(3.5,0)  .. controls (1, 3) and (-2,2) ..
(-1,-4)
(3.5,0)  .. controls (2, -3) and (1,-3) ..
(1,-4)

(3.74,0)  .. controls (5, -1) and (6,1) ..
(7,0)
(3.7,0)  .. controls (6, 1) and (9,1) ..
(6.75,-1)
(3.6,0)  .. controls (6, 1.5) and (10,1.5) ..
(6.5,-2)
(3.6,0)  .. controls (5, -2.5) and (6,-2.5) ..
(6.25,-3)
(3.5,0)  .. controls (6, 2) and (11,2) ..
(6,-4)
(3.5,0)  .. controls (4, -2.5) and (5,-3.5) ..
(5,-4)
;
\node[above,blue] at (2.5,1.3) {$\Fam_\lambda(I)$};

\node[above,red] at (3.5,0) {$I$};
\node[above,below,red] at (3,-4) {$(\lambda I)^c$};

\draw[line width =0.08 cm,red] (6.5,0) -- (7,0) -- (6,-4)--(-1,-4)--(0,0)--(0.5,0)
(3.25,0)  -- (3.75,0);

\end{scope}

\end{tikzpicture}

\caption{Parameters measuring degeneration of the Siegel disk: $\Width_\lambda(I)$ is the width of the full family of curves connecting $I$ and $[\lambda I]^c$ (right), while $\Width^+_\lambda(I)$ is the width of outer family (left).}
\label{fig:Width Width^+}
\end{figure}

\subsubsection{Full and outer families} \label{ss:FullOutFam} Recall~\eqref{eq:dfn:lambdaI} that $\lambda I$ denotes the rescaling of $I\subset \partial Z$ by $\lambda$ with respect to the center of $I$. Recall also that $[\lambda I]^c$ denotes the complement of $\lambda I$ in $\partial Z$. 

Given disjoint intervals $I,J\subset \partial Z$ and $\lambda\ge 2$, we denote by
\begin{itemize}
\item $\Fam(I,J)$ the family of curves in $\wC\setminus (I\cup J)$ connecting $I$ and $J$;
\item $\Width (I,J)=\Width (\Fam(I,J))$ the extremal width of $\Fam(I,J)$; 
\item $\Fam_\lambda(I)\coloneqq \Fam\left(I, [\lambda I]^c \right)$;
\item $\Width_\lambda(I) =\Width(\Fam_\lambda (I))$;
\item $\Fam^+(I,J)$ the family of curves in $\C\setminus Z$ connecting $I$ and $J$;
\item $\Width^+(I,J)=\Width (\Fam^+(I,J))$ the extremal width of $\Fam^+(I,J)$; 
\item $\Fam^+_\lambda(I)\coloneqq \Fam^+\left(I, [\lambda I]^c \right)$;
\item $\Width^+_\lambda(I) =\Width(\Fam^+_\lambda (I))$.
\end{itemize}
We call $\Fam$ and $\Fam^+$ the \emph{full} and \emph{outer} families respectively.

We say that an interval $I$ is
\begin{itemize}
\item \emph{$[K,\lambda]$-wide} if $\Width_\lambda(I)\ge K$, and
\item  \emph{$[K,\lambda]^+$-wide} if $\Width^+_\lambda(I)\ge K$.
\end{itemize}

Clearly, for every $K>1$ and $\lambda\ge 2$, there is $K_\lambda>1$ such that if $\ovZ$ is a $K$-quasidisk, then $\Width(\Fam_\lambda (I))\le K_\lambda$ for every $I\subset \partial Z$.  See~\S\ref{ss:NestOfTiling} for a converse statement.

For a closed Jordan disk $D\subset \C$ and disjoint intervals $I,J\subset \partial D$, the objects \[\Fam_D^-(I,J),\sp \Width_D^-(I,J), \sp \Fam_D^+(I,J), \sp \Width_D^+(I,J),\sp  \Fam_D(I,J),\sp  \Width_D(I,J)\] are defined in the same way as in the Siegel case $D=\overline Z$ (see also~\S\ref{sss:can rect}). We say that a rectangle $\RR$ is \emph{based on an interval $I\subset \partial D$} if 
\begin{equation}
\label{eq:dfn:based on} \RR\subset \wC\setminus \intr D \sp\sp\text{ and }\sp\sp \partial ^h\RR \subset I. 
\end{equation}

\subsubsection{External and diving families}
\label{sss:ext div famil} 

Consider an interval $I\subset \partial Z$. Recall from~\eqref{eq:dfn:filled_n} that $\filled_m\coloneqq f^{-\qq_{m+1} } (\overline Z)$. A curve $\gamma$ in $\Fam_\lambda^+(I)$ is called 
\begin{itemize}
\item \emph{external} (rel $\filled_m$) if $\gamma$ minus its endpoints is in $\wC\setminus \filled_m$, and 
\item \emph{diving} (rel $\filled_m$) otherwise;
\end{itemize}
i.e., diving curves submerge into limbs of $\filled_m$.

We denote by $\Fam^+_{\ext,m}(I,J)$ and $\Fam^+_{\div,m}(I,J)$ the subfamilies of $\Fam_{\lambda}^+(I,J)$ consisting of external and diving curves respectively. As usual, we write 
\[\Width^+_{\ext,m}(I,J)=\Width\left(\Fam^+_{\ext,m}(I,J)\right)\sp\sp\text{ and }\sp\sp \Width^+_{\div,m}(I,J)=\Width\left(\Fam^+_{\div,m}(I,J)\right).\]
The families $\Fam^+_{\lambda,\ext,m}(I),\Fam^+_{\lambda,\div,m}(I)$ are defined accordingly.

\begin{lem}
\label{lmm:ext famil}
Consider an interval $T=[a,b]\in \Dbb_m$ in the diffeo-tiling~\S\ref{sss:diff tilings} and let $L_a,L_b$ be the limbs of $\filled_m$ attached to $a, b$. Consider intervals $I\subset T$ and $J\subset \partial Z$. 

 \emph{\bf \setword{(A)}{Cl:A:lmm:ext famil}} We have:
\begin{equation}
\label{eq:lmm:ext famil}
 \Width^+(I,J)=\Width^+_{\ext,m}(I,J)+\Width^+_{\div,m}(I,J)-O(1).
\end{equation}

 \emph{\bf \setword{(B)}{Cl:B:lmm:ext famil}} There are laminations $\FamG_\ext\subset \Width^+_{\ext,m}(I,J), \sp \FamG_\div\subset \Fam^+_{\div,m}(I,J)$ consisting of at most two rectangles each such that
\begin{equation}
\label{eq:lmm:ext famil:2}
 \Width(\FamG_{\ext})=\Width^+_{\ext,m}(I,J)-O(1),\sp\sp\sp  \Width(\FamG_{\div})=\Width^+_{\div,m}(I,J)-O(1).
\end{equation}

  \emph{\bf \setword{(C)}{Cl:C:lmm:ext famil}} Moreover, we can assume that $\FamG_\ext$ consists of rectangles $\RR_a,\RR_b\subset \C\setminus Z$ with $\partial^{h,0}\RR_a,\partial^{h,0}\RR_b\subset I$ such that every curve in $\Fam(\RR_a)$ intersects $L_a$ before intersecting $\filled_m\setminus L_a$ while every curve in $\Fam(\RR_b)$ intersects $L_b$ before intersecting $\filled_m\setminus L_b$. 

  \emph{\bf \setword{(D)}{Cl:D:lmm:ext famil}} If $J\subset T^c$, then \[\Width^+_{\ext,m}(I,J)\le \Width^+_{\ext,m}(I,T^c) \le 2\sp\text{ and } \sp \Width^+(I,J)=\Width^+_{\div,m}(I,J)+O(1).\]

\end{lem}
\begin{proof}
 Present $J$ as a concatenation of intervals $J_x\cup J_y\cup J_z$ such that $J_x,J_z\subset T$ while $J_y\subset T^c$;  we allow some of the intervals to be empty. Consider the canonical rectangle $\RR$ of $\Fam^+(I,J)$, see~\S\ref{sss:can rect}. Then $\RR$ splits into a union of genuine subrectangles $\RR_1\cup \RR_2\cup \RR_3\cup \RR_4\cup \RR_5$, where some of them can be empty, such that
\begin{itemize}
\item $\RR_1\subset \Fam^+_{\ext,m}(I,J_x)$ and $\RR_5\subset \Fam^+_{\ext,m}(I,J_z)$;
\item every $\gamma\in \Fam(\RR_2)$ intersects $L_a$ before intersecting $\filled_m\setminus L_a$; 
\item similarly, every $\gamma\in \Fam(\RR_4)$ intersects $L_b$ before intersecting $\filled_m\setminus L_b$; 
\item every $\gamma\in \Fam(\RR_3)\subset \Fam^+(I,J_y)$ is either disjoint from $L_a\cup L_b$ or it intersects $\filled_m\setminus (L_a\cup L_b)$ before intersecting $L_a\cup L_b$.
\end{itemize}
In particular, $\RR_2,\RR_4\subset \Fam^+_{\div, m}(I,J)$.

Since the harmonic measure in $(\wC\setminus \filled_m,\infty) $ of $\partial L_a,\partial L_b$ is bigger than the harmonic measure of $T$, we have $\Width(\RR_3)\le 2$. Indeed, if $\RR'_3$ is the restriction of $\RR_3$ to $\wC\setminus \filled_m$, then $\RR'_3$ is blocked in $\wC\setminus \filled_m$ by two of its shifts emerging from $L_a$ and $L_b$.

 By removing $O(1)$ buffers, we can assume that $\RR_2,\RR_4\subset \C\setminus Z$. By~\eqref{eq:width: concat}, we have:
\[\Width(\RR) =\Width(\RR_1)+\Width(\RR_2)+\Width(\RR_4)+\Width(\RR_5)-O(1).\]
\end{proof}

\subsubsection{Univalent push-forward}\label{sss:univ push forw}  Let $T=[a,b]\subset \Dbb_m$ be an interval in the diffeo-tiling. Consider a rectangle $\RR\subset \wC\setminus Z$ with $\partial^{h,0} \RR\subset T$ and $\partial^{h,1} \RR\subset \partial Z$. Let us fix an iterate $f^i$ with $i\le \qq_{m+1}$. Below we will show that after removing $O(1)$-curves in $\RR$, the remaining family of curves can be restricted to $\C\setminus f^{-i}( \overline Z)$ and then univalently pushed-forward under $f^i$, see~\eqref{eq:sss:univ push forw}.

 Let $\RR'$ be the restriction (compare to~\S\ref{sss:short subcurves}) of $\RR$ to $\wC\setminus f^{-i}( \overline Z)$: the family $\RR'$ consist of the subarcs $\gamma'$ of the $\gamma\in\Fam(\RR)$ such that every $\gamma'$ is the shortest subarc of $\gamma$ connecting $\partial^{h,0}\RR $ to $\partial f^{-i}( \overline Z)\setminus \partial^{h,0} \RR$. The $\infty$ splits $\RR'$ into at most two families of curves $\RR'_{1},\RR'_{2}$ such that the leftmost and rightmost curve of $\RR'_i$ bound a rectangle $\widetilde \RR'_{i}$ in $\C\setminus  f^{-i}( \overline Z)$ (see~\S\ref{sss:App:Lam}), where the latter surface is viewed as an annulus. The left-right order in $\RR'_i$ is inherited from $\RR$.
 
 Let $\widetilde \RR^\new_{i}$ be the rectangles obtained from $\widetilde \RR'_{i}$ by removing the $1$-buffers on every side. Let $\RR^\new_{i}$ be the family of all curves in $\RR'_{i}$ that are in $\widetilde \RR^\new_{i}$. We set $\RR^\new\coloneqq \RR^\new_1\cup \RR^\new_2$ and observe that $\Width(\RR^\new)\ge \Width(\RR)-8$. It follows from Lemma~\ref{lem:univ push} applied to $\bigcup_i\widetilde \RR^\new_i$ that  the map 
\begin{equation}
\label{eq:sss:univ push forw} f^i \colon \RR^\new \ \overset{1:1}\longrightarrow\  f^i (\RR^\new)\subset \wC\setminus Z
\end{equation}
is injective.  We will refer to ~\eqref{eq:sss:univ push forw} as the \emph{univalent push-forward} of $\RR$. Again, by convenience, $f^i (\RR^\new)$ can be replaced in applications by the rectangle bounded by the leftmost and rightmost curves in $f^i (\RR^\new)\subset \wC\setminus Z$, see~\S\ref{sss:App:Lam}.

\subsubsection{$\Fam(I^+,J^+)$-families} For a closed Jordan disk $D\subset \C$, consider two disjoint intervals $I,J\subset \partial D$. Let us view $\wC\setminus (I\cup  J)$ as a Riemann surface; with respect to this Riemann surface both $I, J$ have two sides: the \emph{outer sides} $I^+,J^+$ and the \emph{inner sides} $I^-,J^-$. Ignoring the endpoints of $I,J$, a curve $\gamma\colon (0,1)\to \wC\setminus (I\cup J)$ \emph{lands} at $\gamma(1)\in I^+$ if \[\gamma [1-\varepsilon, 1)\subset \wC\setminus Z\sp \forall \varepsilon >0\sp\sp \text{ and }\sp \lim_{\tau\to 1} \gamma(\tau)=\gamma(1). \] 
Similarly, the landing at $I^-$ is defined. Let
\begin{itemize}
\item $\Fam (I^+,J^+)$ be the family of curves in $\wC\setminus (I\cup J)$ connecting $I^+$ and $J^+$;
\item $\Width (I^+,J^+)=\Width (\Fam(I^+,J^+))$ be the extremal width of $\Fam (I^+,J^+)$.
\end{itemize}

The \emph{central arc} in $\Fam^-_D(I,J)$ is the curve $\ell\in \Fam^-_D(I,J)$ that splits $D$, viewed as a rectangle with horizontal sides $I,J$, into two genuine subrectangles of equal width. 


\begin{lem}[Trading $\Fam(I,J)$ into $\Fam(I^+,J^+)$]
\label{lmm:W into W(I+,J+)}
Consider a closed Jordan disk $D\subset \C$ and a family $\Fam(I,J)=\Fam_D (I,J)$ for $I,J\subset \partial D$.  Let $A,B\subset \partial D$ be two complementary intervals between $I$ and $J$ and let $\widetilde I\supset I$ and $  \widetilde J\supset J$ be thickenings of $I,J$ so that $\widetilde I, \widetilde J$ are disjoint intervals of $\partial D$. Set 
\[ C\coloneqq \Width^-(A,B) +\Width^-(I,\widetilde I ^c )+\Width^-(J,\widetilde J ^c) .\]
Let $\RR\subset \Fam(I,J)$ be a lamination.

Then there are intervals $\wI, \wJ \subset \partial Z$ with $I\subset \wI\subset \widetilde I$ and $J\subset \wJ\subset \widetilde J$ such that there is a restriction $\FamG$ of a sublamination of $\RR$ with
\begin{itemize}
\item  $\FamG\subset \Fam\big(\wI^+, \wJ^+\big)$;
\item $\Width(\RR|\FamG)=\Width(I,J) - O(C)$, see~\eqref{eq:sss:Restr of Lam};
\item $\FamG$ is disjoint from the central arc in $\Fam^-(I,J)$.
\end{itemize}

 In particular, by taking $\RR$ to be the vertical family of  $\Fam(I,L)$, see~\S\ref{sss:can rect}, we obtain $\Width(\wI^+, \wJ^+)\ge\Width(I,J) - O(C)$. 
\end{lem}
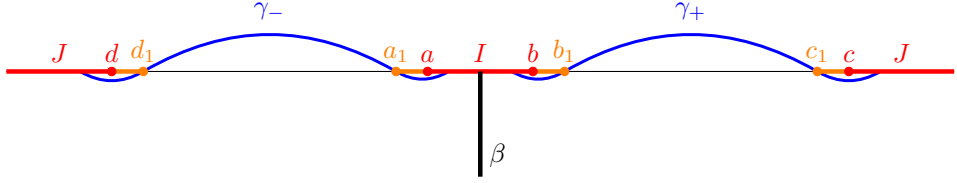
\begin{figure}[t!]
\begin{tikzpicture}[scale=1.4]

\draw (-4.5,0) -- (4.5,0);

\draw[blue,line width=0.4mm]   
(0.3,0) edge[bend right] (0.8,0)
(0.8,0) edge[bend left] (3.2,0) 
(3.2,0)edge[bend right] (3.8,0);

\node[blue,above] at (2,0.4) {$\gamma_+$};

\draw[blue,line width=0.4mm]   
(-0.3,0) edge[bend left] (-0.8,0)
(-0.8,0) edge[bend right] (-3.2,0) 
(-3.2,0)edge[bend left] (-3.8,0);

\node[blue,above] at (-2,0.4) {$\gamma_-$};

\draw[red,line width=0.6mm] (-0.5,0) -- (0.5,0)
(-4.5,0) -- (-3.5,0)
(4.5,0) -- (3.5,0);
\node[above,red] at (0,0) {$I$};

\draw[line width=0.6mm] (0,0) -- (0,-1);
\node[above right ] at (0,-1) {$\beta$};

\node[above,red] at (-4,0) {$J$};
\node[above,red] at (4,0) {$J$};

\draw[orange,line width=0.6mm]
(-0.5,0)--(-0.8,0)
(0.5,0)--(0.8,0)
(-3.2,0)--(-3.5,0)
(3.2,0)--(3.5,0);


\filldraw[red] (0.5,0) circle (0.04 cm);
\node[above, red] at (0.5,0) {$b$};

\filldraw[orange] (0.8,0) circle (0.04 cm);
\node[above, orange] at (0.8,0) {$b_1$};

\filldraw[red] (-0.5,0) circle (0.04 cm);
\node[above, red] at (-0.5,0) {$a$};

\filldraw[orange] (-0.8,0) circle (0.04 cm);
\node[above, orange] at (-0.8,0) {$a_1$};

\filldraw[red] (3.5,0) circle (0.04 cm);
\node[above, red] at (3.5,0) {$c$};

\filldraw[orange] (3.2,0) circle (0.04 cm);
\node[above, orange] at (3.2,0) {$c_1$};

\filldraw[red] (-3.5,0) circle (0.04 cm);
\node[above, red] at (-3.5,0) {$d$};

\filldraw[orange] (-3.2,0) circle (0.04 cm);
\node[above, orange] at (-3.2,0) {$d_1$};
\end{tikzpicture}
\caption{The curves $\gamma_-$ and $\gamma_+$ specify the intervals $\widehat I=[a_1,b_1]$ and $\widehat J=[c_1,d_1]$. Here $I=[a,b]$ and $ J=[c,d]$.}
\label{Fg:gamma_- gammma_+}
\end{figure}
\begin{proof}
 Let $\FamG^\new$ be the lamination obtained from $\FamG$ by removing all leaves $\ell$ satisfying one of the following property:
\begin{itemize}
\item $\ell$ intersects the central arc $\beta$ of  $\Fam^-(I,J)$;
\item $\ell$ contains a subarc in $\Fam^-(I,\widetilde I^c)$;
\item $\ell$ contains a subarc in $\Fam^-(J,\widetilde J^c)$.
\end{itemize}
By assumption, $\Width(\FamG^\new) = \Width(\FamG)- O(C)$.

Since curves in $\FamG^\new$ do not intersect $\beta$, they possess a left-right order. Denote by $\gamma_-,\gamma_+$ the leftmost and rightmost arc in $\FamG$. We assume that $\gamma_-<\gamma_+<\beta$ with respect to the clockwise order around $I$, see Figure~\ref{Fg:gamma_- gammma_+}. Write $I=[a,b], J=[c,d]$, where $a<b<c<d$.  Intersecting $\gamma_-,\gamma_+$ with $\partial D^+$, we obtain the intervals $\widehat I =[a_1,b_1], \widehat J=[c_1,d_1]$ as follows (Figure~\ref{Fg:gamma_- gammma_+}):
\begin{itemize}
\item If $\gamma_-$ starts in $I^+$, then $a_1\coloneqq a$; otherwise $a_1$ is the first intersection of $\gamma_-$ with $\partial D\setminus I$.
\item If $\gamma_+$ starts in $I^+$, then $b_1\coloneqq b$; otherwise $b_1$ is the first intersection of $\gamma_+$ with $\partial D\setminus I$.
\item If $\gamma_-$ ends at $J^+$, then $d_1\coloneqq d$; otherwise $d_1$ is the last intersection of $\gamma_-$ with $\partial D\setminus J$.
\item If $\gamma_+$ ends at $J^+$, then $c_1\coloneqq c$; otherwise $c_1$ is the last intersection of $\gamma_+$ with $\partial D\setminus J$.
\end{itemize} 
By construction, $\widehat I\subset I$ and $\widehat J\subset J$. Restricting $\FamG^\new$ to the family $\Fam(\widehat I, \widehat J)$, we obtain a required lamination $\FamG^\New\subset \Fam(\widehat I^+, \widehat J^+).$ 
\end{proof}

\subsubsection{$\Fam^\circ_L(I,J)$-families}
\label{sss:Fam^circ}
Consider a pair $I,J\subset \partial D$ of disjoint intervals, and let $L\subset \partial D$ be one of the complementary intervals between $I$ and $J$. We define 
\begin{itemize}
\item $\Fam^\circ _L(I,J)$ to be the set of curves $\gamma\in \Fam (I^+,J^+)$ such that $\gamma$ is disjoint from $\partial D\setminus (I\cup L\cup J)$;
\item $\Width^\circ _L(I,J)\coloneqq \Width \left(\Fam^\circ _L(I,J) \right)$.
\end{itemize}

If $I<L<J$, then we write:
\[ \Fam^\circ (I,J) \coloneqq \Fam^\circ_L(I,J)\sp\sp\text{ and }\sp\sp \Width^\circ (I,J) \coloneqq \Width^\circ_L(I,J).\]

\begin{lem}[Snakes in $\Fam^\circ \setminus \Fam^+$]
\label{lem:Fam^circ:R}Consider a closed topological disk $D$, intervals $I,J\subset\partial D$, and $\Fam_L^\circ(I,J)$ as above, if 
\[K =\Width^\circ(I,J)-\Width^+_L(I,J)\gg 1,\] then $\Fam^\circ_L(I,J)$ contains a rectangle $\RR$, called a \emph{snake}, with $\Width(\RR)=K-O(1)$ such that every vertical curve of $\RR$ intersects $L$. 
\end{lem}

\begin{proof}
Let $\RR$ be the canonical rectangle of $\Fam^\circ_L(I,J)$, \S\ref{sss:can rect}; i.e. the semi-closed rectangle realizing the width between $I^+,J^+$ in the open topological disk $\wC\setminus L^c$. Let $\gamma\in \Fam(\RR)$ be the unique vertical curve intersecting $L$ such that $\gamma\subset \wC\setminus \intr D$. Then $\gamma$ splits $\RR$ into two rectangles $\RR^\out\subset \Fam^+(I,J)$ and $\RR^\inn\subset \Fam_L^\circ(I,J)$, where the latter rectangle submerges into $D$. By Lemma~\ref{lem:splitting rectangle:2}, $\Width(\RR^\out)=\Fam^+(I,J)-O(1)$; this will imply that $\Width(\RR^\inn)=K-O(1)$.
\end{proof}

\subsection{Series Decompositions}
\label{sss:SerDecomp for F circ}
 In this subsection we will discuss how to take restrictions (compare with~\S\ref{sss:short subcurves}) of families submerging into topological disks. For a closed topological disk $D$ consider a lamination $\RR$ in $\Fam^\circ_L(I,J)\setminus \Fam^+(I,J)$; i.e.~every curve in $\RR$ intersects $L$.  We assume that $I<L<J$ is the order of intervals in $\partial D$. Let us introduce a topological decomposition for $\RR$. 

We will use the inner/outer order on curves in $\RR$, see \S\ref{sss:inner outer order}: the innermost curve in $\RR$ is the closest curve to $( L^c )^-$ in $\wC\setminus L^c$, where $L^c=\partial D \setminus L$.

Let $\beta\colon [0,1]\to \wC, \sp \beta(0)\in I,\sp \beta(1)\in J$ be the outermost curve of $\RR$; i.e. all other curves of $\RR$ are between $\beta$ and $(L^c)^-$. Let $x=\beta(t)\in \beta \cap L$ be the first (for the smallest $t$) intersection between $L$ and $\beta$.  


 \begin{figure}[t!]
\[\begin{tikzpicture}[scale=1.4]

\node[red,above] at(4.6,0){$J$}; 
\node[red,above] at(-4.6,0){$I$};

\draw (-5.2,0) -- (5.2,0);

\draw[red,line width=1mm] (5,0)--(4,0)
(-5,0)--(-4,0);

\draw [] (-4.5,0)
 .. controls (-3, 1.5) and (-0.5,1.5) ..
 (0,0);

\draw [] (4.5,0)
 .. controls (3, 1.5) and (0.5,1.5) ..
 (0,0);
\node[above] at  (-2.5, 1.)  {$\beta$};

 \draw[red] (-3,0)
 .. controls (-2.5, 1) and (-1,1) ..
 (-0.5,0);
 
 \node[red] at (-2.1, 0.5){$\gamma_a$};
 
  \draw[red,shift={(-4,0)}]
(0.5,0)
.. controls (0.5+0.2, -0.3) and (1-0.2,-0.3) ..
(1,0); 
\draw[red,shift={(-4.5,0)}]
(0.5,0)
.. controls (0.5+0.2, 0.3) and (1-0.2,0.3) ..
(1,0); 
 
\draw[orange]
(0.5,0)
.. controls (0.2, -0.5) and (-0.2,-0.5) ..
(-0.5,0);

\draw[]
(0.5,0)
.. controls (0.5+0.2, 0.3) and (1-0.2,0.3) ..
(1,0);

\draw[scale =2]
(0.5,0)
.. controls (0.2, -0.5) and (-0.2,-0.5) ..
(-0.5,0);

\draw[shift={(-2,0)}]
(0.5,0)
.. controls (0.5+0.2, 0.3) and (1-0.2,0.3) ..
(1,0);

\draw[scale =3,orange]
(0.5,0)
.. controls (0.2, -0.5) and (-0.2,-0.5) ..
(-0.5,0);

\node[below,orange] at (0,-1.1) {$\gamma_b^d$};

\draw[red,shift={(1,0)}]
(0.5,0)
.. controls (0.5+0.2, 0.3) and (1-0.2,0.3) ..
(1,0); 
\draw[red,shift={(1.5,0)}]
(0.5,0)
.. controls (0.5+0.2, -0.3) and (1-0.2,-0.3) ..
(1,0); 
\draw[red,shift={(2,0)}]
(0.5,0)
.. controls (0.5+0.2, 0.3) and (1-0.2,0.3) ..
(1,0); 
 \draw[red,shift={(2.5,0)}]
(0.5,0)
.. controls (0.5+0.2, -0.3) and (1-0.2,-0.3) ..
(1,0); 
\draw[red,shift={(3,0)}]
(0.5,0)
.. controls (0.5+0.2, 0.3) and (1-0.2,0.3) ..
(1,0); 
 
 \node[below,red ] at (3.2,-0.2) {$\gamma_b$};
 
\filldraw (0,0) circle (0.04 cm);
\node[above right] at (0,0) {$x$};
\node[below,orange] at (0,0.05) {$\gamma_a^d$};

\end{tikzpicture}\]
\caption{The subcurves of $\gamma$: $\gamma_a$ (red), $\gamma^d_a$ (orange),  $\gamma^d_b$ (orange), $\gamma_b$ (red). Note that $\gamma_a$ and $\gamma_b$ are disjoint. }
\label{Fig:prf:lmm:W^circ to W^+:2}
\end{figure}
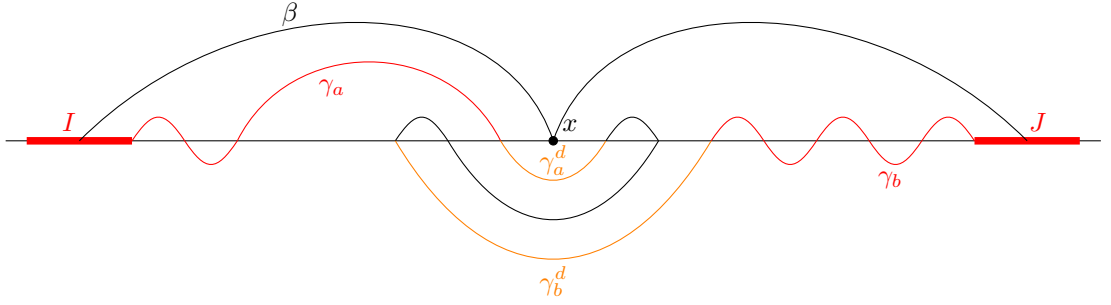

Consider a vertical curve $\gamma\colon [0,1]\to \C$ in $\RR$ with $\gamma(0)\in I$ and $\gamma(1)\in  J$. Since $x$ is the first entry of $\beta$ into $\lfloor x, J \rfloor$, we obtain that the first entry of $\gamma$ into $\lfloor x, J \rfloor$ is from $\intr D$. And since $\gamma$ starts and ends at $I^+$ and $J^+$ respectively, we can define (see Figure~\ref{Fig:prf:lmm:W^circ to W^+:2})
\begin{itemize} 
\item $\gamma(a_2)\in \partial D$ to be the first intersection of $\gamma$ with $\lfloor x, J\rfloor $;
\item $\gamma(a)\in  \partial D$ to be the last before $a_2$ intersection of $\gamma$ with $\lfloor I, x\rfloor$;
\item $\gamma_a$ to be the subcurve of $\gamma$ between $I$ and $\gamma(a)$;
\item $\gamma_{a}^d$ to be the subcurve of $\gamma$ between $\gamma(a)$ and $\gamma(a_2)$;
\item $\gamma(b_2)\in  \partial D$ to be the last intersection of $\gamma$ with $\lfloor I, x\rfloor$; 
\item $\gamma(b)\in  \partial D$ to be the first after $b_2$ intersection of $\gamma$ with  $\lfloor x, J\rfloor $,
\item $\gamma_{b}^{d}$ to be the subcurve of $\gamma$ between $\gamma(b_2)$ and $\gamma(b)$; and
\item $\gamma_b$ to be the subcurve of $\gamma$ between $\gamma(b)$ and $J$.
\end{itemize}

We say that
\begin{itemize}
\item $\gamma_a^d\subset D$ is the \emph{first passage} of $\gamma$ under $x$;
\item  $\gamma_a$ is the subcurve of $\gamma$ \emph{before} $\gamma_a^d$;
\item $\gamma_b^d\subset D$ is the \emph{last passage} of $\gamma$ under $x$;
\item $\gamma_b$ is the subcurve of $\gamma$ \emph{after} $\gamma_b^d$.
\end{itemize}

Clearly, $\gamma_a$ and $\gamma_b$ are disjoint because $a_2,b_2 $ are between $a$ and $b$. Also $\gamma_a\cup \gamma_b$ is disjoint from $\gamma_a^d\cup \gamma_b^d$. The curves $\gamma_a^d, \gamma_b^d$ may or may not coincide.

\begin{rem} Since $x$ is the first intersection of $\beta$ with $L$, the curve $\beta$ is outside of $D$ before it reaches $x$. After $x$, the curve $\beta$ may have a complicated intersection pattern with $\partial D$. For example, $\beta$ may pass under $x$ to intersect the left interval of $L\setminus \{x\}$; but then $\beta$ must go back under $x$
and intersect the right interval of $L\setminus \{x\}$. 
\end{rem}

Let us specify the following laminations
 \[
\widetilde \Fam_a \coloneqq \{ \gamma_a\mid \gamma\in \RR\}, \sp\sp\sp\widetilde \Fam_b\coloneqq \{\gamma_b\mid \gamma\in \RR\}
,\]
 \[
 \Gamma_a \coloneqq \{ \gamma^d_a\mid \gamma\in \RR\}, \sp\sp\sp \Gamma_b\coloneqq \{\gamma^d_b\mid \gamma\in \RR\}
,\]
\begin{equation}
\label{eq:Gamma}
 \Gamma\coloneqq \Gamma_a\cup \Gamma_b=\{\gamma_a^d\mid \gamma\in \RR\} \cup  \{\gamma_b^d\mid \gamma\in \RR\}.
\end{equation}
Then $\RR$ consequently overflows $\widetilde \Fam_a, \Gamma, \widetilde \Fam_b$.

 Let $\ell$ be the lowest curve in $\Gamma$ with respect to $x$; i.e.~$\ell$ separates $\Gamma$ from $L^c$ in $D$. We denote by $J_a\subset L$ the interval between the left endpoint of $\ell$ and $x$ and we denote by $I_b\subset L$ the interval between $x$ and the right endpoint of $\ell$. Define  
\begin{equation}
\label{eq:F_a:dfn}
\Fam_a=\{\gamma'\mid \gamma' \text{ is the first shortest subcurve of }\gamma\in \widetilde \Fam_a \text{ connecting }I^+, J^+_a\}\end{equation} 
to be the restriction of $\widetilde \Fam_a$ to $ \Fam(I,J_a)$ 
-- compare to~\S\ref{sss:short subcurves}; and
\begin{equation}
\label{eq:F_b:dfn}
\Fam_b=\{\gamma'\mid \gamma' \text{ is the first shortest subcurve of }\gamma\in \widetilde \Fam_b \text{ connecting }I^+_b, J^+\}
\end{equation}
to be the restriction of $\widetilde \Fam_b$ to $ \Fam(I_b,J)$.   Since curves in $\widetilde \Fam_a\sqcup \widetilde \Fam_b$ are disjoint from $\ell $, we obtain
\[ \Fam_a\subset  \Fam^\circ(I,J_a),\sp\sp\sp \Fam_b\subset \Fam^\circ(I_b,J);\]
in particular, curves in $\Fam_a, \Fam_b$ land at $J_a^+, I_b^+$ respectively. 

We summarize:

\begin{lem}
\label{lem:SerDecomp}
A lamination $\RR \subset \Fam^\circ_L(I,J)\setminus \Fam^+(I,J)$ as above consequently overflows the pairwise disjoint laminations 
\[\Fam_a\subset \Fam^\circ (I,J_a),\sp \sp \Gamma\subset \Fam^-(J_a,I_b),\sp \sp \Fam_b\subset \Fam^\circ (I_b, J).\]
defined in~\eqref{eq:F_a:dfn},~\eqref{eq:Gamma},~\eqref{eq:F_b:dfn} respectively.
\[\begin{tikzpicture}
\draw (-6,0) -- (6,0);
\draw[line width =0.8mm,red ] (-5.5,0)--(-3,0);
\node[below,red ] at (-4.2,0){$I$};
\draw[line width =0.8mm,blue ] (5.5,0)--(3,0);
\node[below,blue ] at (4.2,0){$J$};

\draw[line width =0.8mm,red ] (-1,0)--(-0.,0);
\node[below,red ] at (-0.9,0){$J_a$};

\draw[line width =0.8mm,blue ] (1,0)--(0,0);
\node[below,blue ] at (0.9,0){$I_b$};

\filldraw (0,0) circle (0.08 cm);
\node[above] at (0,0.1){$x$};

\draw[line width =1.2mm,brown ] (-0.6,0)
.. controls (-0.3, -0.5) and (0.3, -0.5)..
(0.6,0);
\node[brown,below] at (0,-0.4){$\Gamma$};

\draw [line width =1.2mm,red ] (-4,0)
 .. controls (-3.5, 0.5) and (-2.5,0.5) ..
 (-2,0)
 .. controls (-1.8, -0.5) and (-1.6,-0.5) ..
 (-1.4,0)
  .. controls (-1.2, 0.5) and (-0.8,0.5) ..
 (-0.6,0);
 \node[above,red] at (-1.7, 0.3) {$\Fam_a$};

\draw [line width =1.2mm,blue ] (4,0)
 .. controls (3.5, 0.5) and (2.5,0.5) ..
 (2,0)
 .. controls (1.8, -0.5) and (1.6,-0.5) ..
 (1.4,0)
  .. controls (1.2, 0.5) and (0.8,0.5) ..
 (0.6,0);
 \node[above,blue] at (1.8, 0.3) {$\Fam_b$};

\end{tikzpicture}
\]
\end{lem}\qed

\subsection{Snake Lemma}
\label{sss:SnakeLmm:simpl toplog}


\begin{figure}[t!]

\begin{tikzpicture}[scale=0.8,every node/.style={scale=0.8}]

\begin{scope}[shift={(11.6,0)}]
\draw[line width =0.05 cm,fill, fill opacity =0.05] (0,0) -- (7,0) -- (6,-4)--(-1,-4)--(0,0)--(0.,0);

\draw [blue,line width=0.1 cm] (2.7,0)
 .. controls (3, 0.5) and (3.4,-0.2) ..
 (3.5,-0.2)
 .. controls (3.6, -0.2) and (4,0.5) ..
 (4.2,0);

\draw[dashed,line width =0.05 cm] (3.5,0) circle (1cm)
(3.5,1.2) -- (3.5,2.5);

\end{scope}

\begin{scope}[line width =0.04 cm,shift={(15,4)},scale=1.5]

\draw (-6,0) -- (6,0);

\draw[opacity=0,fill, fill opacity =0.05] (-6,0) -- (6,0)-- (6,-1)--(-6,-1);

\draw[red,line width=1mm] (5.5,0)--(4,0)
(-5.5,0)--(-4,0);

\node[below,red] at (4.8, 0) {$J$};
\node[below,red] at (-4.8, 0) {$I$};

\draw [blue] (5.5,0)
 .. controls (3, 2.5) and (0.5,1.5) ..
 (0,0);
 
 \draw [blue] (5.2,0)
 .. controls (3, 2.2) and (0.5,1.2) ..
 (0.3,0);

\draw[opacity=0, fill=blue, fill opacity=0.2]
(0,0)
.. controls (0.5,1.5) and (3, 2.5) ..
(5.5,0)
.. controls (5.4,0) and (5.3,0) ..
(5.2,0)
 .. controls (3, 2.2) and (0.5,1.2) ..
 (0.3,0);

  \draw [blue] (-5.2,0)
 .. controls (-3, 2.2) and (-0.5,1.2) ..
 (-0.3,0);

\draw[opacity=0, fill=blue, fill opacity=0.2]
(0,0)
.. controls (-0.5,1.5) and (-3, 2.5) ..
(-5.5,0)
.. controls (-5.4,0) and (-5.3,0) ..
(-5.2,0)
 .. controls (-3, 2.2) and (-0.5,1.2) ..
 (-0.3,0);
 
 \draw[red,line width =0.1cm]  (-0.3,0)-- (0.3,0);
\node[ below left ,red] at (-0,0){$J_1$};
 \node[ below right ,red] at (-0,0){$I_1$};
 
 \node[blue, above] at (-2,1.5) {$\RR^\new_\ell$};
  \node[blue, above] at (2,1.5) {$\RR^\new_\rho$};

\draw[blue] (4,0)
 .. controls (2.5, 1) and (2,0.5) ..
 (1.5,0);

\filldraw[blue] (1.5,0) circle (0.04cm)
 (-1.5,0) circle (0.04cm);

\node[blue,above] at (3,0.5) {$\gamma_\rho$};
\node[blue,above] at (-3,0.5) {$\gamma_\ell$};
\node[blue,above] at (0,-0.8) {$\gamma'$};

\draw[blue]
(1.5,0)
.. controls (1, -1) and (-1,-1) ..
(-1.5,0);

\draw [blue] (-5.5,0)
 .. controls (-3, 2.5) and (-0.5,1.5) ..
 (0,0);

\draw[blue] (-4,0)
 .. controls (-2.5, 1) and (-2,0.5) ..
 (-1.5,0);

\filldraw (0,0) circle (0.04 cm);
\node[below] at (0,0) {$x$};

\end{scope}

\end{tikzpicture}

\caption{Illustration to Snake Lemma~\ref{simplmm:SnLmm:Z}: if a ``snake'' with width $K\gg 1$ submerges, then either $\RR^\new_\ell$ or $\RR^\new_\rho$ has width ${2K-O(\log \lambda)}$.}
\label{Fig:SnakeLmm:illustr}
\end{figure}
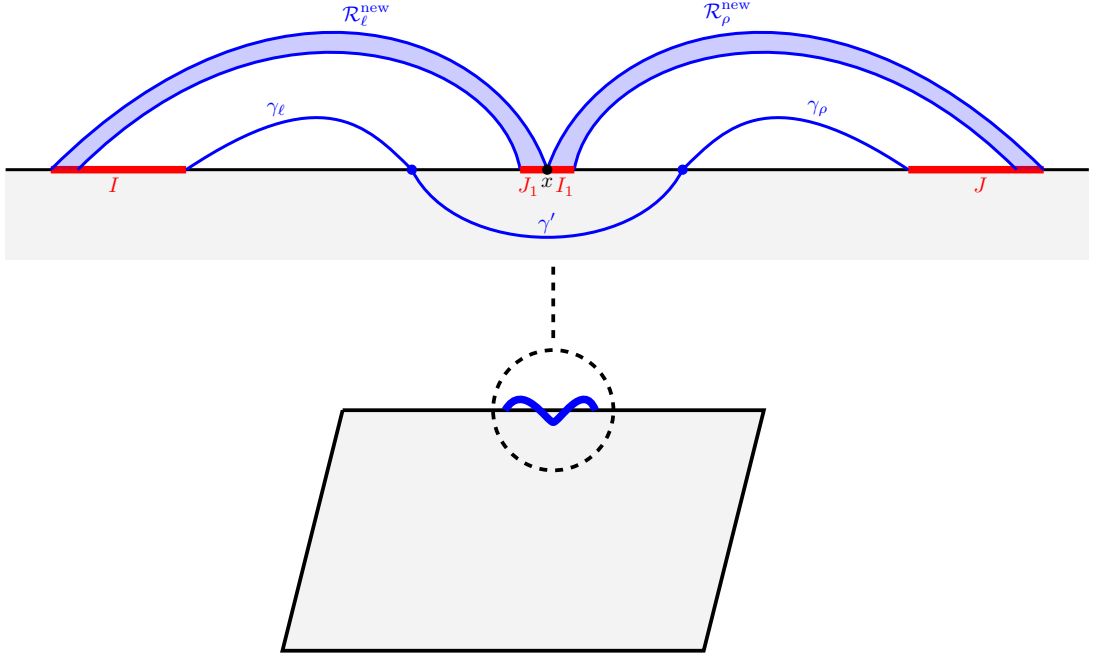

The following lemma allows us to control submergence of $\Fam^\circ(I,J)$ into $Z$ (see~\S\ref{sss:Fam^circ}). The Snake Lemma for pseudo-Siegel disks will be proven as Lemma~\ref{simplmm:SnLmm:wZ}.

\begin{snakelmm}[See Figure~\ref{Fig:SnakeLmm:illustr}]
\label{simplmm:SnLmm:Z}
Let $I,J\subset \partial Z$ be a pair of intervals with $\lfloor I, J \rfloor <1/2$ and let $L\coloneqq \lfloor I, J \rfloor\setminus (I\cup J)$ be the complementary interval between $I,J$ with $I<L<J$. Set
\[K\coloneqq \Width^\circ_L(I,J) - \Width^+(I,J).\]

 If $K\gg \log \lambda$ with $\lambda>2$, then there are intervals \[J_1,I_1\subset L,\sp\sp |J_1|< \frac{\dist(I, J_1)}{\lambda},\sp |I_1|<\frac{\dist( I_1,J)}{\lambda},\sp\sp I< J_1<I_1<J\] such that 
\begin{equation}
\label{eq:SnakeLmm:simpl}  \Width^\circ_{L_a} (I,J_1) \oplus\Width^\circ_{L_b} (I_1,J)\ge K - O(\log \lambda),   
\end{equation}
where $L_a,L_b\subset L$ are the complementary intervals between $I,J_1$ and $I_1,J$ respectively:
\[\begin{tikzpicture}
\draw (-6,0) -- (6,0);
\draw[line width =0.8mm,red ] (-5.5,0)--(-3,0);
\node[below,red ] at (-4.2,0){$I$};
\draw[line width =0.8mm,blue ] (5.5,0)--(3,0);
\node[below,blue ] at (4.2,0){$J$};

\draw[line width =0.8mm,red ] (-1,0)--(-0.3,0);
\node[below,red ] at (-0.7,0){$J_1$};

\draw[line width =0.8mm,blue ] (1,0)--(0.3,0);
\node[below,blue ] at (0.7,0){$I_1$};

\draw [line width =1.2mm,red ] (-4,0)
 .. controls (-3.5, 0.5) and (-2.5,0.5) ..
 (-2,0)
 .. controls (-1.8, -0.5) and (-1.6,-0.5) ..
 (-1.4,0)
  .. controls (-1.2, 0.5) and (-0.8,0.5) ..
 (-0.6,0);
 \node[above,red] at (-1.7, 0.3) {$\Fam^\circ_{L_a} (I,J_1)$};

\draw [line width =1.2mm,blue ] (4,0)
 .. controls (3.5, 0.5) and (2.5,0.5) ..
 (2,0)
 .. controls (1.8, -0.5) and (1.6,-0.5) ..
 (1.4,0)
  .. controls (1.2, 0.5) and (0.8,0.5) ..
 (0.6,0);
 \node[above,blue] at (1.8, 0.3) {$\Fam^\circ_{L_b} (I_1,J)$};

\end{tikzpicture}
\]
\end{snakelmm}

\noindent In particular, either $ \Width^\circ_{L_a} (I,J_1) $ or $\Width^\circ_{L_b} (I_1,J)$ has width $\ge 2K-O(\log \lambda)$. 

The Snake Lemma is a consequence of the Localization Property~\S\ref{sss:LocProp} applied to Series Decomposition~\S\ref{sss:SerDecomp for F circ}.

\begin{proof}
 Let $\RR\subset \Fam^\circ_L(I,J)\setminus \Fam^+(I,J)$ with $\Width(\RR)=K-O(1)$ be a rectangle (a snake) from Lemma~\ref{lem:Fam^circ:R} realizing $K$. Apply Series Decomposition~\S\ref{sss:SerDecomp for F circ} to $\RR$, we obtain that $\Fam(\RR)$ consequently overflows the laminations \[\Fam_a\subset \Fam^\circ (I,J_a),\sp \sp \Gamma\subset \Fam^-(J_a,I_b),\sp \sp \Fam_b\subset \Fam^\circ (I_b, J).\]
By the Localization Property~\S\ref{sss:LocProp}, $J_a, I_b$ contain an innermost subpair $J_1,I_1$ such that \[ |I_1|, \ |J_1|\le \frac 1 {2\lambda} \{|I_a|,\ |J_b|\}\]
and up to $O(\log \lambda)$-width the family $\Fam^-(J_a,I_b)$ is in  $\Fam^-(J_1,I_1)$:
\[\Width^-(J_a\setminus J_1,\sp I_b) +\Width^-(J_a,\sp I_b\setminus I_1)=O(\log \lambda)\]
 
Let $\RR^\new$ be the lamination obtained from $\RR$ by removing all $\gamma\in \Fam(\RR)$ with $\gamma_a^d\not \in \Fam^-(J_1,I_1)$ or  $\gamma_b^d\not \in \Fam^-(J_1,I_1)$. Then $\Width(\RR^\new)=K-O(\log \lambda)$.

Applying Series Decomposition~\S\ref{sss:SerDecomp for F circ}  to $\RR^\new$, we obtain that $\Gamma^\new\subset \Fam^-(J_1,I_1)$; i.e.~$J^\new_a\subset J_1$ and $I^\new_b\subset I_1$.

\end{proof}

\subsection{Trading $\Fam$ into $\Fam^+$}

\begin{cor}[Trading $\Width^\circ$ into $\Width^+$]
Under the assumptions of Lemma~\ref{simplmm:SnLmm:wZ}, there is an interval $I^\new\subset L$ such that $\Width^+_\lambda(I^\new)\succeq K$.
\end{cor}
\begin{proof}
It follows from~\eqref{eq:SnakeLmm:simpl} and $K\gg \log \lambda$ that either $\Width^\circ_{L_a} (I,J_1) $ or $\Width^\circ_{L_b} (I_1,J)$ has width $\ge 2K-O(\log \lambda)\ge \frac 7 4 K$. Assume that  $\Width^\circ_{L_a} (I,J_1) \ge \frac 74 K$. Since $\Fam^+_\lambda (J_1)\supset \Fam^+_{L_a} (I,J_1)$, either $\Width^+_\lambda(J_1)\ge \frac 1 5 K$ or $ \Width^\circ(I,J_1) - \Width^+(I,J_1)\ge \frac32 K$; in the latter case, we can again apply the Snake Lemma and construct intervals $I_2,J_2$ such that
\[ \Width^\circ (I,J_2) \oplus\Width^\circ (I_2,J_1)\ge \frac 32K - O(\log \lambda),\sp\sp\sp\text{ where }\] 
 \[\Fam^+ (I,J_2)\subset \Fam^+_{\lambda}(J_2)\sp\sp\text{ and }\sp\sp \Fam^+ (I_2,J_1)\subset \Fam^+_{\lambda}(I_2).\]
 Applying induction, we either find an interval $I^\new$ with $\Width^+_\lambda(I^\new)\ge  \frac{3^n}{2^n5}K$, or construct an infinite sequence of shrinking intervals $I^\new_n,J_n^\new, L^\new_n$ with \[\Fam^\circ_{L_n^\new}(I^\new_n, J^\new_n)\ge \frac{3^n}{2^n}K\sp\sp\sp |L^\new_n|\ge \min\{|I^\new_n|,\sp |J^\new_n|\}.\] 
Such an infinite sequence does not exist because $\overline Z$ is a (non-uniform) qc disk. 
\end{proof}

\section{Near-Rotation Systems}
\label{s:NearRotatSystem}

For $r>0$, we denote by $|x-y|_r$ the Euclidean distance between $x,y$ on the circle $\R/r=\R /(r\Z)$. We also write $|x-y| = |x-y|_1$, which is consistent with the combinatorial distance introduce for $\partial \Disk, \partial Z$.

Fix $\epscasc>0$. A \emph{$\epscasc$-near-rotation} system with rotation number $\pp/\qq \in \Q$ is $\tF_\qq=\big( f^t \colon U\to U_t\big)_{0
\le t \le \qq}$ such that (see Figure~\ref{Fig:NearRotDom})
\begin{itemize}
\item[(A)] $\overline U$ and $\overline U_t$ are closed Jordan disks;
\item[(B)] $f^t \colon U\to U_t$ is conformal for $t\le \qq$;
\item[(C)] $\partial U$ is a cyclic (clockwise or counterclockwise) concatenation of simple arcs: 
\begin{equation}
\label{eq:part U:unit intervals}
\partial U =L_{0}\# L_1\# L_2\#\dots\# L_{\qq-1};
\end{equation}
\item[(D)] for every $k$ there is an annulus $A_k$ with $\mod(A_k)\ge \epscasc$ such that 
\begin{itemize}
\item[(D1)] the bounded component $B_k$ of $\C\setminus A_k$ compactly contains 
$L_k$ as well as all $f^t (L_{k-\pp t})$ for $t \le \qq$, and
\item[(D2)] $A_k$ is disjoint from $A_i$ for $|i-k|_\qq>1$.
\end{itemize}
\end{itemize}

In other words, $f^{t}$ maps $L_k$ approximately onto $L_{k + t\pp}$ so that $L_k,L_{k + t\pp}\subset B_k$; this ``error'' is controlled (surrounded) by $A_k$. Let us write $\widetilde A_k=A_k\cup B_k$ -- the filling-in of $A_k$; and 
 \[\widetilde U\coloneqq U\cup \bigcup_{k=0}^{\qq-1} B_k.\]

We call the $L_i$ \emph{unit intervals} of $\partial U$ and we call $\overline U$ a $\epscasc$-\emph{near-rotation domain}.

\subsubsection{Motivation and outline} 
\label{sss:outline:s:NearRotatSystem}
Recall that $\theta\approx \pp_n/\qq_n$ (see \S\ref{sss:closest returns}) and $f$ rotates the diffeo-tiling $\Dbb_n$ by approximately $\pp_{n+1}/\qq_{n+1}$, see \S\ref{sss:diff tilings}. If $\length_{n+1}\ll \length_n$, then the ``rotational error'' is small with respect to the combinatorial metric of $\partial Z$.  However, with respect to the conformal geometry of $\C$, the rotational error will be big due to parabolic fjords, see Figure~\ref{fig:fjords}. To deal with this issue, we will approximate $\overline Z$ by a $\bdelta$-near-rotation domain $Z^n$ with a universal $\bdelta>0$ and add all such $Z^m$ to $\overline Z$, see \S\ref{s:Welding}.

For $\Circle=\partial \Disk$, the Euclidean metric on $\R/\Z\simeq \Circle$ is a unique invariant metric under all rotations $z\mapsto \ee(\phi)z$. For near-rotation domains we have  almost-rotations $f^i\mid U, i \le\qq$, where the error is controlled by annuli $A_i$ with $\mod A_i\ge \epscasc>0$. It is natural to expect that as $\epscasc$ is fixed and $\qq\to \infty$, ``almost invariant metrics'' on $\partial U$ converge to the Euclidean metric on $\R/\Z\simeq \partial \Disk$ after a conformal uniformization. In this section, we will prove a slightly weaker statement: universal Log-Rule holds for $\partial U$ on scales $\gg_{\mu} 1/\qq$,  see Theorem~\ref{thm:beau:part U: c quasi line}. These beau bounds will imply that the error does not increase during iterative construction of pseudo-Siegel disks $\dots \leadsto\wZ^{m+1}\leadsto \wZ^{m} \leadsto\wZ^{m-1}\leadsto\dots$, see Remark~\ref{rem:error does not accum}.

 Theorem~\ref{thm:beau:part U: c quasi line} is proven using the Shift Argument (Figure~\ref{Fg:lmm:shift}): if there is an unexpected wide rectangle, then its appropriate shift will have a substantial cross-intersection with itself contradicting Non-Crossing Principle~\S\ref{sss:non cross princ}. From this, the Log-Rule in Theorem~\ref{thm:beau:part U: c quasi line} are established in the same way as in Lemma~\ref{lmm:W^- I J}. The main subtlety is that shifted curves can sneak through the $B_k$. We will first establish estimates on scale $\ge 40/\qq$ with an error depending on $\epscasc$ (Proposition~\ref{prop:part U: c quasi line}), then we will upgrade them to universal estimates on scale $\gg_{\epscasc} 1/\qq$.

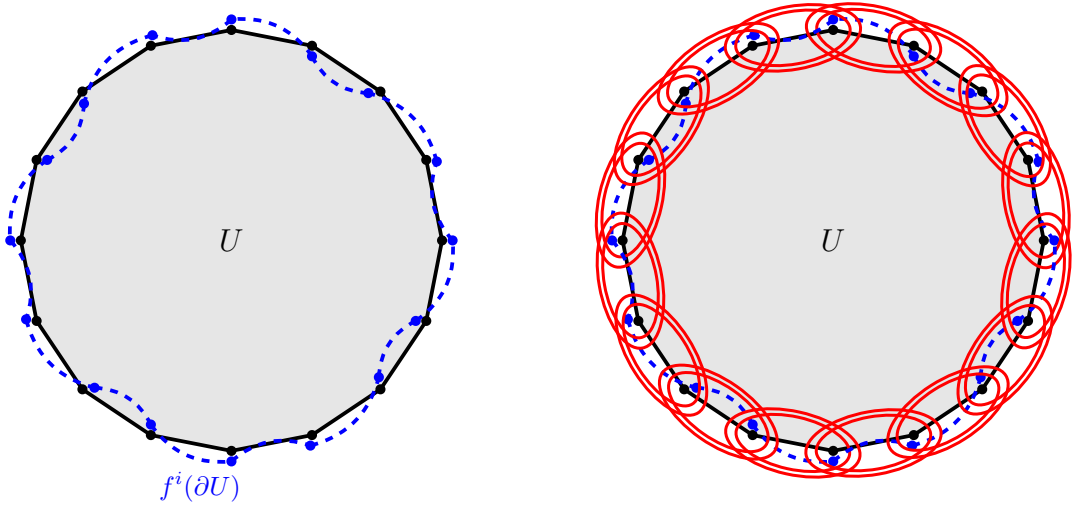
\begin{figure}
\begin{tikzpicture}[line width=0.05cm]
\begin{scope}[shift={(8,-6)},scale=0.7]

\draw[fill, fill opacity=0.1] (4,0)  --(3.7,1.53)-- (2.83, 2.83)--(1.53,3.7)--(0,4)--(-1.53,3.7)--(-2.83, 2.83)--(-3.7,1.53)--(-4,0)--(-3.7,-1.53)--(-2.83, -2.83)
--(-1.53,-3.7)--(0,-4)--(1.53,-3.7)--(2.83, -2.83)--(3.7,-1.53)--(4,0);

\node at (0,0) {\Large $U$};

\node[blue,below]at (-0.6,-4.2) {$f^i(\partial U)$};

\filldraw  (4,0)  circle (0.06cm)
(3.7,1.53)  circle (0.06cm)
(2.83, 2.83)  circle (0.06cm)
(1.53,3.7)  circle (0.06cm)
(0,4) circle (0.06cm)
(-1.53,3.7) circle (0.06cm)
(-2.83, 2.83) circle (0.06cm)
(-3.7,1.53) circle (0.06cm)
(-4,0)  circle (0.06cm)
(-3.7,-1.53)  circle (0.06cm)
(-2.83, -2.83)  circle (0.06cm)
(-1.53,-3.7)  circle (0.06cm)
(0,-4) circle (0.06cm)
(1.53,-3.7) circle (0.06cm)
(2.83, -2.83) circle (0.06cm)
(3.7,-1.53) circle (0.06cm);

\draw[blue,dashed] 
(4.2,0)  edge [ bend left ] (3.9,1.5)
 (3.9,1.5) edge [ bend right ]  (2.6, 2.8)
(2.6, 2.8) edge [bend left] (1.53,3.5) 
(1.53,3.5) edge[bend right](0,4.2);

\filldraw [blue] (4.2,0)  circle (0.06cm)
 (3.9,1.5) circle (0.06cm)
(2.6, 2.8)circle (0.06cm)
(1.53,3.5)circle (0.06cm);



\begin{scope}[rotate=90]
\draw[blue,dashed] 
(4.2,0)  edge [ bend left ] (3.9,1.5)
 (3.9,1.5) edge [ bend right ]  (2.6, 2.8)
(2.6, 2.8) edge [bend left] (1.53,3.5) 
(1.53,3.5) edge[bend right](0,4.2);

\filldraw [blue] (4.2,0)  circle (0.06cm)
 (3.9,1.5) circle (0.06cm)
(2.6, 2.8)circle (0.06cm)
(1.53,3.5)circle (0.06cm);


\end{scope}

\begin{scope}[rotate=180]
\draw[blue,dashed] 
(4.2,0)  edge [ bend left ] (3.9,1.5)
 (3.9,1.5) edge [ bend right ]  (2.6, 2.8)
(2.6, 2.8) edge [bend left] (1.53,3.5) 
(1.53,3.5) edge[bend right](0,4.2);

\filldraw [blue] (4.2,0)  circle (0.06cm)
 (3.9,1.5) circle (0.06cm)
(2.6, 2.8)circle (0.06cm)
(1.53,3.5)circle (0.06cm);

\end{scope}

\begin{scope}[rotate=270]
\draw[blue,dashed] 
(4.2,0)  edge [ bend left ] (3.9,1.5)
 (3.9,1.5) edge [ bend right ]  (2.6, 2.8)
(2.6, 2.8) edge [bend left] (1.53,3.5) 
(1.53,3.5) edge[bend right](0,4.2);

\filldraw [blue] (4.2,0)  circle (0.06cm)
 (3.9,1.5) circle (0.06cm)
(2.6, 2.8)circle (0.06cm)
(1.53,3.5)circle (0.06cm);


\end{scope}
\end{scope}

\begin{scope}[shift={(16,-6)},scale=0.7]

\draw[fill, fill opacity=0.1] (4,0)  --(3.7,1.53)-- (2.83, 2.83)--(1.53,3.7)--(0,4)--(-1.53,3.7)--(-2.83, 2.83)--(-3.7,1.53)--(-4,0)--(-3.7,-1.53)--(-2.83, -2.83)
--(-1.53,-3.7)--(0,-4)--(1.53,-3.7)--(2.83, -2.83)--(3.7,-1.53)--(4,0);

\node at (0,0) {\Large $U$};

\filldraw  (4,0)  circle (0.06cm)
(3.7,1.53)  circle (0.06cm)
(2.83, 2.83)  circle (0.06cm)
(1.53,3.7)  circle (0.06cm)
(0,4) circle (0.06cm)
(-1.53,3.7) circle (0.06cm)
(-2.83, 2.83) circle (0.06cm)
(-3.7,1.53) circle (0.06cm)
(-4,0)  circle (0.06cm)
(-3.7,-1.53)  circle (0.06cm)
(-2.83, -2.83)  circle (0.06cm)
(-1.53,-3.7)  circle (0.06cm)
(0,-4) circle (0.06cm)
(1.53,-3.7) circle (0.06cm)
(2.83, -2.83) circle (0.06cm)
(3.7,-1.53) circle (0.06cm);

\draw[blue,dashed] 
(4.2,0)  edge [ bend left ] (3.9,1.5)
 (3.9,1.5) edge [ bend right ]  (2.6, 2.8)
(2.6, 2.8) edge [bend left] (1.53,3.5) 
(1.53,3.5) edge[bend right](0,4.2);

\filldraw [blue] (4.2,0)  circle (0.06cm)
 (3.9,1.5) circle (0.06cm)
(2.6, 2.8)circle (0.06cm)
(1.53,3.5)circle (0.06cm);


 \draw[red,rotate around={10:(3.85,0.765)},line width=0.04cm] (3.85,0.765) ellipse (0.5cm and 1.1cm)
 (3.85,0.765) ellipse (0.6cm and 1.3cm);
 \draw[rotate=22.5,red,rotate around={10:(3.85,0.765)},line width=0.04cm] (3.85,0.765) ellipse (0.5cm and 1.1cm)
 (3.85,0.765) ellipse (0.6cm and 1.3cm);

 \draw[rotate=45,red,rotate around={10:(3.85,0.765)},line width=0.04cm] (3.85,0.765) ellipse (0.5cm and 1.1cm)
 (3.85,0.765) ellipse (0.6cm and 1.3cm);
  \draw[rotate=45+22.5,red,rotate around={10:(3.85,0.765)},line width=0.04cm] (3.85,0.765) ellipse (0.5cm and 1.1cm)
 (3.85,0.765) ellipse (0.6cm and 1.3cm);

\begin{scope}[rotate=90]
\draw[blue,dashed] 
(4.2,0)  edge [ bend left ] (3.9,1.5)
 (3.9,1.5) edge [ bend right ]  (2.6, 2.8)
(2.6, 2.8) edge [bend left] (1.53,3.5) 
(1.53,3.5) edge[bend right](0,4.2);

\filldraw [blue] (4.2,0)  circle (0.06cm)
 (3.9,1.5) circle (0.06cm)
(2.6, 2.8)circle (0.06cm)
(1.53,3.5)circle (0.06cm);

 \draw[red,rotate around={10:(3.85,0.765)},line width=0.04cm] (3.85,0.765) ellipse (0.5cm and 1.1cm)
 (3.85,0.765) ellipse (0.6cm and 1.3cm);
 \draw[rotate=22.5,red,rotate around={10:(3.85,0.765)},line width=0.04cm] (3.85,0.765) ellipse (0.5cm and 1.1cm)
 (3.85,0.765) ellipse (0.6cm and 1.3cm);

 \draw[rotate=45,red,rotate around={10:(3.85,0.765)},line width=0.04cm] (3.85,0.765) ellipse (0.5cm and 1.1cm)
 (3.85,0.765) ellipse (0.6cm and 1.3cm);
  \draw[rotate=45+22.5,red,rotate around={10:(3.85,0.765)},line width=0.04cm] (3.85,0.765) ellipse (0.5cm and 1.1cm)
 (3.85,0.765) ellipse (0.6cm and 1.3cm);
 
\end{scope}

\begin{scope}[rotate=180]
\draw[blue,dashed] 
(4.2,0)  edge [ bend left ] (3.9,1.5)
 (3.9,1.5) edge [ bend right ]  (2.6, 2.8)
(2.6, 2.8) edge [bend left] (1.53,3.5) 
(1.53,3.5) edge[bend right](0,4.2);

\filldraw [blue] (4.2,0)  circle (0.06cm)
 (3.9,1.5) circle (0.06cm)
(2.6, 2.8)circle (0.06cm)
(1.53,3.5)circle (0.06cm);

 \draw[red,rotate around={10:(3.85,0.765)},line width=0.04cm] (3.85,0.765) ellipse (0.5cm and 1.1cm)
 (3.85,0.765) ellipse (0.6cm and 1.3cm);
 \draw[rotate=22.5,red,rotate around={10:(3.85,0.765)},line width=0.04cm] (3.85,0.765) ellipse (0.5cm and 1.1cm)
 (3.85,0.765) ellipse (0.6cm and 1.3cm);

 \draw[rotate=45,red,rotate around={10:(3.85,0.765)},line width=0.04cm] (3.85,0.765) ellipse (0.5cm and 1.1cm)
 (3.85,0.765) ellipse (0.6cm and 1.3cm);
  \draw[rotate=45+22.5,red,rotate around={10:(3.85,0.765)},line width=0.04cm] (3.85,0.765) ellipse (0.5cm and 1.1cm)
 (3.85,0.765) ellipse (0.6cm and 1.3cm);
 
\end{scope}

\begin{scope}[rotate=270]
\draw[blue,dashed] 
(4.2,0)  edge [ bend left ] (3.9,1.5)
 (3.9,1.5) edge [ bend right ]  (2.6, 2.8)
(2.6, 2.8) edge [bend left] (1.53,3.5) 
(1.53,3.5) edge[bend right](0,4.2);

\filldraw [blue] (4.2,0)  circle (0.06cm)
 (3.9,1.5) circle (0.06cm)
(2.6, 2.8)circle (0.06cm)
(1.53,3.5)circle (0.06cm);

 \draw[red,rotate around={10:(3.85,0.765)},line width=0.04cm] (3.85,0.765) ellipse (0.5cm and 1.1cm)
 (3.85,0.765) ellipse (0.6cm and 1.3cm);
 \draw[rotate=22.5,red,rotate around={10:(3.85,0.765)},line width=0.04cm] (3.85,0.765) ellipse (0.5cm and 1.1cm)
 (3.85,0.765) ellipse (0.6cm and 1.3cm);

 \draw[rotate=45,red,rotate around={10:(3.85,0.765)},line width=0.04cm] (3.85,0.765) ellipse (0.5cm and 1.1cm)
 (3.85,0.765) ellipse (0.6cm and 1.3cm);
  \draw[rotate=45+22.5,red,rotate around={10:(3.85,0.765)},line width=0.04cm] (3.85,0.765) ellipse (0.5cm and 1.1cm)
 (3.85,0.765) ellipse (0.6cm and 1.3cm);
 
\end{scope}
\end{scope}

\end{tikzpicture}

\caption{Illustration to near-rotation systems: $\partial U$ is a finite concatenation of arcs, and $f^i (\partial U)$ (blue) is approximately $\partial U$ rotated by $i \pp  /\qq$, where the error is controlled by a system of annuli (right side). We say that $\overline U$ is \emph{almost invariant} under $f^\qq$.}
\label{Fig:NearRotDom}
\end{figure}

\subsection{Standard intervals of $\partial U$} 
\label{ss:stand inter}
A \emph{discrete interval} \[S\subset\{0,1,2\dots, \qq-1\} \simeq \Z/\qq\] with \emph{length} $b$ is a finite subset $\{a,a+1,a+2,\dots,a+b -1\}$ of $\Z/\qq$ consisting of consecutive numbers. Set 
\[ L_S\coloneqq \bigcup_{s\in S}L_s,\sp \sp\sp 
B_S\coloneqq \bigcup_{s\in S}B_s. \] By construction, $f^t (L_S) \subset  B_{S +  \pp t} $ for $t\le \qq$,  
where $S + j=\{s+ j \mid s\in S\}$.

For $r>1$, we define the rescaling of $S$ with respect to its center as
\[ r S \coloneqq  \left\{n\in \Z/\qq :  \left| n- a -(b-1)/2\right|_\qq \le r|S|_\qq/2 \right\}.\]
Let $(rS)^c \coloneqq \{0,1,\dots, \qq-1\} \setminus (rS)^c$ be the complement of $rS$. Similar to~\S\ref{ss:GeomZ}, we define:
\begin{itemize}
\item $\Fam^-(L_V,L_W)$ to be the family of curves in $U$ connecting $L_V$ and $L_W$;
\item $\Width^-(L_V,L_W)=\Width(\Fam^-(L_V,L_W))$;
\item $\Fam^-_r(L_S)=\Fam^-(L_S, L_{(rS)^c})$;
\item $\Width^-_r(L_S)=\Width^-(L_S, L_{(rS)^c})$.
\end{itemize} 
We say that an interval $L_S$ is \emph{$[K,r]^-$-wide} if $\Width_r(L_S)\ge K$. 

We call an interval $L_S\subset \partial U$ \emph{standard}. Any interval $ I\subset \partial U$ can be approximated from above or below by a standard interval with an error within $L_a\cup L_b$ for some $a,b \in \{0,1,\dots, \qq-1\}$.

\subsection{Inner geometry of $U$} The following lemma is a corollary of Lemma~\ref{lem:Rect in U}.

\begin{lem}
\label{lem:1/eps buffer}
Consider a rectangle
\[\RR\subset f^i(\ovl U)\subset\widetilde U,\sp\sp \partial ^h \RR\subset f^i(\partial U), \sp\sp \sp i\le \qq\]
\[\text{ with } \sp\partial^{h,0} \RR\subset B_V,\sp \partial^{h,1} \RR\subset B_W, 
\] where $V$ and $W$ are discrete intervals.  After removing two $C'_\epscasc\coloneqq 1/\epscasc$-wide buffers from $\RR$, the new rectangle $\RR^\new$ is disjoint from $B_s$ for every $s\in \Z/\qq$ at distance at least $3$ from $V\cup W$.\qed
\end{lem}

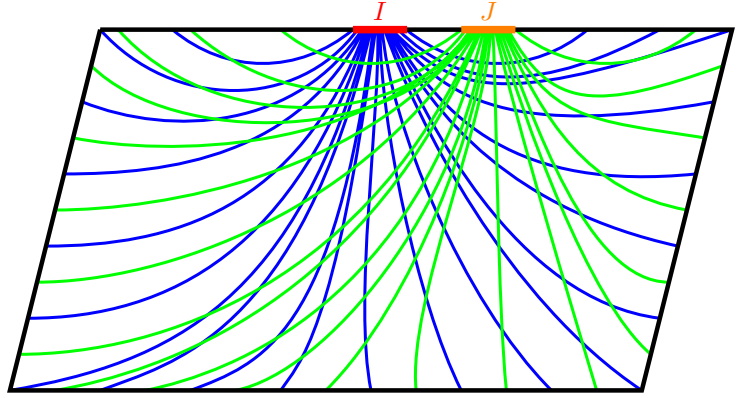
\begin{figure}[t!]
\[\begin{tikzpicture}

\begin{scope}[shift={(-2,-3.6)},scale =1.2]

\draw[blue,line width=0.4mm]
 (2.8,0) 
.. controls (2.2,-0.5) and (1.4,-0.5)..
(0.8,0)
 (2.85,0) 
.. controls (2.2,-0.9) and (0.8,-0.9)..
(0,0)
 (2.9,0) 
.. controls (2.2,-1.1) and (0.8,-1.2)..
(-0.2,-0.8)
 (2.92,0) 
.. controls (2.2,-1.4) and (0.8,-1.6)..
(-0.4,-1.6)
 (2.94,0) 
.. controls (2.2,-2) and (0.8,-2.4)..
(-0.6,-2.4)
 (2.96,0) 
.. controls (2.2,-2.4) and (0.8,-3.2)..
(-0.8,-3.2)
(2.98,0) 
.. controls (2.2,-3) and (0.8,-3.7)..
(-1,-4)
(3,0) 
.. controls (2.2,-3.4) and (0.8,-3.7)..
(0,-4)
(3.025,0) 
.. controls (2.6,-3.4) and (1.8,-3.7)..
(1,-4)
(3.05,0) 
.. controls (2.6,-3.4) and (2.2,-3.7)..
(2,-4)
(3.075,0) 
.. controls (3.1,-1) and (2.8,-3)..
(3,-4)
(3.1,0) 
.. controls (3.1,-1) and (3.5,-3)..
(4,-4)
(3.125,0) 
.. controls (3.2,-1) and (4,-3)..
(5,-4)
(3.15,0) 
.. controls (3.5,-1) and (4.5,-3)..
(6,-4)
(3.17,0) 
.. controls (3.6,-1) and (4.6,-3)..
(6+0.2,-4+0.8)
(3.19,0) 
.. controls (3.7,-1) and (4.6,-2)..
(6+0.4,-4+1.6)
(3.21,0) 
.. controls (3.8,-1) and (4.6,-1.8)..
(6+0.6,-4+2.4)
(3.23,0) 
.. controls (3.8,-0.8) and (4.6,-1.2)..
(6+0.8,-4+3.2)
(3.23,0) 
.. controls (4,-0.8) and (4.6,-0.8)..
(7,0)
(3.3,0) 
.. controls (4,-0.7) and (4.9,-0.7)..
(6.2,0)
(5.4,0)
.. controls (4.8,-0.5) and (4,-0.5)..
(3.39,0);

\draw[green,line width=0.4mm]
 (4,0) 
.. controls (3.4,-0.5) and (2.6,-0.5)..
(2,0)
(4.05,0) 
.. controls (3.4,-0.9) and (2,-0.9)..
(1,0)
(4.1,0) 
.. controls (3.4,-1.15) and (1.2,-1.15)..
(0.2,0)
(4.12,0) 
.. controls (3.4,-1.3) and (0.8,-1.3)..
(-0.1,-0.4)
(4.14,0) 
.. controls (3,-1.4) and (0.7,-1.4)..
(-0.3,-1.2)
(4.16,0) 
.. controls (3,-1.6) and (0.7,-2)..
(-0.5,-2)
(4.18,0) 
.. controls (3,-2) and (0.7,-2.8)..
(-0.7,-2.8)
(4.2,0) 
.. controls (3,-2.8) and (0.7,-3.6)..
(-0.9,-3.6)
(4.225,0) 
.. controls (3,-3) and (0.7,-3.8)..
(-0.5,-4)
(4.25,0) 
.. controls (3.5,-3) and (1.5,-3.8)..
(0.5,-4)
(4.275,0) 
.. controls (4,-2) and (2.5,-3.8)..
(1.5,-4)
(4.3,0) 
.. controls (4.2,-1) and (3.5,-3)..
(2.5,-4)
(4.325,0) 
.. controls (4.3,-1) and (3.4,-3)..
(3.5,-4)
(4.35,0) 
.. controls (4.4,-1) and (4.4,-3)..
(4.5,-4)
(4.375,0) 
.. controls (4.6,-1) and (5.2,-3)..
(5.5,-4)
(4.4,0) 
.. controls (4.7,-1) and (5.4,-3)..
(6+0.1,-4+0.4)
(4.43,0) 
.. controls (4.8,-1) and (5.4,-2.8)..
(6+0.3,-4+1.2)
(4.46,0) 
.. controls (4.9,-1) and (5.4,-1.8)..
(6+0.5,-4+2)
(4.49,0) 
.. controls (4.95,-1) and (5.4,-1.)..
(6+0.7,-4+2.8)
(4.52,0) 
.. controls (5,-0.9) and (5.5,-0.9)..
(6+0.9,-4+3.6)
(6.6,0)
.. controls (6,-0.5) and (5.2,-0.5)..
(4.6,0);

\draw[line width =0.06 cm] (0,0) -- (7,0) -- (6,-4)--(-1,-4)--(0,0)--(0.,0);

\draw[red,line width=0.1 cm]  (2.8,0) --(3.4,0);

\node[above, red] at (3.1, 0) {$I$};

\node[above, orange] at (4.3, 0) {$J$};

\draw[orange,line width=0.1 cm]  (4,0) --(4.6,0);

\end{scope}

\end{tikzpicture}\]
\caption{Illustration to Lemma~\ref{lem:near transl:est:eps}: since $\Fam_{20} (I)$ (blue) crosses its shift $\Fam_{20}(J)$ (green), $\Fam_{20} (I)$ is not wide.}
\label{Fg:lmm:shift}
\end{figure}

\begin{lem}
\label{lem:near transl:est:eps}
Set $C_\epscasc\coloneqq 30+ 2C'_\epscasc =30+2/\epscasc$. There are no  $[C_\epscasc,20]^-$-wide intervals $L_S$, where $1\le |S|_\qq \le \qq/40$.
\end{lem}
\begin{proof}
 Suppose $I=L_S$ is such a $[C_\epscasc,20]^-$-wide interval. Let $\RR,\sp  \partial^{h,0}\RR=I$ be the canonical rectangle of $\Fam^-_{20}(L_S)$, see~\S\ref{sss:can rect}. We will construct below a shift $\RR_J$ of $\RR$ so that $\RR,\RR_J$ have substantial cross-intersection,  see Figure~\ref{Fg:lmm:shift}. 

Fix $k\in \N$ such that $S+k$ has $\Z/\qq$-distance at least $3$ from $S\cup [20 S]^c$. Let $j < \qq $ be so that $\pp j= k$ in $\Z/\qq$. Define \[J\coloneqq f^{j}(I)\subset B_{S+k}\sp\sp \text{ and }\sp\sp 
\RR_J\coloneqq f^j (\RR).
\]

Let $\RR_J^\new$ be the rectangle obtained from $\RR_J$ by removing $5$-buffers. By Lemma~\ref{lem:vet boundar} (with $n=1$), we can remove from $\RR$ and $\RR_J^\new$ buffers with width less than $5$ so that the new rectangles $\RR^\NEW$ and $\RR_J^\NEW$ have disjoint vertical boundaries. Let $V$ be the minimal discrete interval such that $\partial^{h,1}\RR^\NEW \subset L_V$. By construction:
\begin{equation}
\label{eq:lem:near transl:est:eps}
\Width\left(\RR^\NEW_J\right)\ge 10 +2/\mu, \sp\partial^{h,0}\RR_J^\NEW \subset B_{S+k},\sp \partial^{h,1}\RR_J^\NEW\subset B_{V+k}, 
\end{equation}
\[
 f^j\left(\RR^\NEW\right)\supset \RR_J^\new \supset \RR^\NEW_J.\]
 Since $\RR^\NEW\subset \overline U$ with $\partial ^h \RR^\NEW \subset \partial U$, we can choose a vertical boundary component $\beta \in \{\partial ^{v,\ell} \RR^\NEW, \partial ^{v,\rho} \RR^\NEW \}$ that separate $L_{S+k}$ from $\RR^\NEW\setminus \beta$; i.e., $L_{S+k}$, $\RR^\NEW \setminus \beta$  are in different components of $\overline U\setminus \beta $. Suppose $\beta$ starts in $L_a$ and ends at $L_b$. 

By construction, $\{a,b\}$ has distance at least $3$ from $[S+k]\cup [V+k]$. Since the horizontal boundary of $\RR_J^\NEW$ is within $B_{S+k}$ and $B_{V+k}$  (see~\eqref{eq:lem:near transl:est:eps}), by Lemma~\ref{lem:1/eps buffer}, the rectangle $\RR_J^\NEW$ has a vertical curve $\gamma$ disjoint from $B_a\cup B_b$, i.e. $\gamma$ is disjoint from $\beta\cup B_a\cup B_b $. This is a contradiction as the endpoints of $\gamma$ are separated by $\beta\cup B_a\cup B_b $ in $\widetilde U$.
\end{proof}

\subsection{Coarse bounds for Near-Rotation domains} We now extend the estimates from Lemma~\ref{lmm:W^- I J} to near-rotation domains. Let us rescale the distance on $\partial U$ by $1/\qq$: 
\[ |I|\coloneqq \frac 1 \qq |I|_\qq,\sp\sp \dist(I,J)\coloneqq \frac 1 \qq\dist_{\qq}(I,J),\sp\sp\text{ for } I=L_V,\  J=L_W\subset \partial U,
\] 
and we choose any continuous extension of the distance function $\dist (\sp,\sp )$ to all point in $\partial U$. The objects \[\Fam^-(I,J)=\Fam_{\overline U}^-(I,J),\sp\sp  \Width^-(I,J)=\Width^-_{\overline U}(I,J)\] for intervals $I,J\subset \partial U$ are defined in ~\S\ref{ss:FullOutFam}.

\begin{prop}[Coarse bounds]
\label{prop:part U: c quasi line}
The following \emph{Log-Rule} holds for intervals \[I,J\subset \partial U\sp\sp\sp\text{ such that }\sp |I|,\ |J|,\ \dist(I,J) \ge 40/\qq .\] If $\dist (I,J)\le \min\{|I|,|J|\}$, then
\begin{equation}
\label{eq:1:lmm:nrd:W^- I J}
\Width^-(I,J) \asymp_\epscasc \log \frac{\min\{|I|,|J|\}}{\dist (I,J)} +1;
\end{equation}
otherwise 
\begin{equation}
\label{eq:2:lmm:nrd:W^- I J}
\Width^-(I,J) \asymp_\epscasc \left( \log \frac{\dist (I,J)}{\min\{|I|,|J|\}} +1\right)^{-1}.
\end{equation} 
\end{prop}

\begin{cor}[$\partial U$ is a coarse quasi-line]
\label{cor:part U: c quasi line}
 Choose a homeomorphism \[h\colon \partial U\to \Circle=\R/\Z, \hspace{1cm} 
  h(L_i) = [i/\qq,(i+1)/\qq].\] Let $I,J\subset \partial U$ be two intervals with ${\min\{|I|, |J|, \dist (I,J)\} \ge 40/\qq}$. Then 
\[ \Width^-_{U}(I,J)  \asymp_\epscasc \Width^-_{\Disk} (h(I), h(J)).\]
\end{cor}\qed

\begin{proof}[Proof of Proposition~\ref{prop:part U: c quasi line}]

\begin{claim3}
\label{cl:1:prop:part U: c quasi line}
Suppose $I,J$ are intervals with 
\begin{equation}
\label{eq:cl:1:prop:part U: c quasi line}
 \min \{|I| , |J| \} \asymp \dist(I,J) \sp\sp \text{ and } \sp |I|, |J|, \dist(I,J) \ge 40/\qq.
\end{equation}
Then $\Width^-(I,J)  \asymp_\epscasc 1$.
\end{claim3}
\begin{proof}
We can approximate $I$ from above by a concatenation of standard intervals \[\widetilde I=I_1\#I_2\#\dots \# I_n\sp\sp \text{such that }\sp \Fam^-(I,J)\subset \bigcup_{k=1}^n \Fam^-_{20}(I_k),\sp\sp \widetilde I\setminus I\subset L_a\cup L_b,\] where $n$ depends on the constant representing ``$\asymp$'' in~\eqref{eq:cl:1:prop:part U: c quasi line}. 
Using Lemma~\ref{lem:near transl:est:eps} and Parallel Law~\eqref{eq:ParLaw}, we obtain
\begin{equation}
\label{eq:prf:cl:1:prop:part U: c quasi line}
\Width^-(I,J)\le \Width_{20}^-(I_1)+\Width_{20}^-(I_2)+\dots + \Width_{20}^-(I_n)\preceq_{\epscasc} 1.
\end{equation}

Let $X, Y$ be the connected components of $\partial U\setminus (I\cup J)$. We have:
\[ \min \{|X| , |Y| \} \asymp \dist(X,Y) \sp\sp \text{ and } \sp |X|, |Y|, \dist(X,Y) \ge 40/\qq.
\] 
Repeating the above argument for $X, Y$, we obtain:
\[\left(\Width^-(I,J)\right)^{-1}=\Width^-(X,Y) \preceq_{\epscasc} 1,\sp\sp \text{i.e.}\sp \Width^-(I,J) \succeq_{\epscasc}1.\]
Therefore, $\Width^-(I,J)\asymp_\epscasc 1$.
\end{proof}

The proposition follows from Claim~\ref{cl:1:prop:part U: c quasi line} by applying the Splitting Argument, see Remark~\ref{rem:SplitArg}.
\end{proof}

\begin{rem}
\label{rem:prop:part U: c quasi line}
We note that the comparison ``$\asymp_\epscasc$''  in~\eqref{eq:1:lmm:nrd:W^- I J} and~\eqref{eq:2:lmm:nrd:W^- I J} depends only on the constant $C_\epscasc$ from Lemma~\ref{lem:near transl:est:eps} -- this lemma was used only in~\eqref{eq:prf:cl:1:prop:part U: c quasi line}. The constant $C_\epscasc$ depends only on the constant $C'_\epscasc=1/\epscasc$ from Lemma~\ref{lem:1/eps buffer}. In \S\ref{ss:BeauCoarBnds}, we will improve Lemma~\ref{lem:1/eps buffer} and obtain beau coarse-bounds on scale $\gg_\epscasc 1/\qq $.
\end{rem}

\subsection{Beau coarse-bounds for near-rotation domains}
\label{ss:BeauCoarBnds} Let us start by improving Lemma~\ref{lem:1/eps buffer} on scale $\gg_\epscasc 1/\qq$:

\begin{lem}
\label{lem:1/eps buffer:impr} There is a constant $T_\epscasc >1$ such that the following holds. Consider a rectangle
\[\RR\subset f^i(\ovl U)\subset\widetilde U,\sp\sp \partial ^h \RR\subset f^i(\partial U), \sp\sp \sp i\le \qq\]
\[\text{ with } \sp\partial^{h,0} \RR\subset B_V,\sp \partial^{h,1} \RR\subset B_W, 
\]  where $V$ and $W$ are discrete intervals. After removing two $1$-buffers from $\RR$, the new rectangle $\RR^\new$ is disjoint from $B_s$ for every $s\in \Z/\qq$ with $\dist_\qq(s, V\cup W)\ge T_\epscasc$.
\end{lem}

\begin{proof}
Follows essentially from Proposition~\ref{prop:part U: c quasi line} because $B_s$ is protected  by a wide family $\Fam ^-(L_G,L_H)$, where $G,H$ are discrete intervals separating $s$ from $V\cup W$. A slight complication is that $\RR$ is a rectangle in $\widetilde U$ and not in $U$.

Suppose $T_\epscasc$ is sufficiently big. There is a sequence of pairs of discrete intervals $G_i,H_i\subset \Z/\qq$ such that all $G_i,H_i$ have pairwise distances at least three, $|G_i|= |H_i|=\dist(G_i,H_i)$, every pair $G_i, H_i$ separates $s$ from $V\cup W$, and $n=n( T_\epscasc)$ is big. 

Using Proposition~\ref{prop:part U: c quasi line}, we can choose a subfamily $\Fam_i$ in $\Fam^-(G_i,H_i)$ such that  $\Width(\Fam_i)\asymp_\epscasc 1$ and such that the $\Fam_i$  are pairwise disjoint.

For every $g\in G_i$, the set of  vertical curves in $\RR$ that intersect $B_g$ forms a buffer of $\RR$ by Lemma~\ref{lem:buffer:R O}; let us choose $g_i\in G_i$ such that the buffer is maximal. Similarly, we choose $h_i\in H_i$ such that the buffer of curves in $\RR$ intersecting $B_{h_i}$ is maximal. Then for every $i$ and every curve $\gamma\in \Fam$ intersecting $B_s$ either
\begin{itemize}
\item[(1)] $\gamma$ intersects $B_{g_i}$; or
\item[(2)] $\gamma$ intersects $B_{h_i}$; or
\item[(3)] $\gamma $ intersects every curve in $\Fam_i$.
\end{itemize}
The modulus of curves in $\Fam$ satisfying (1), (2), and (3) is $\preceq _\epscasc 1$ because $B_{g_i} ,B_{h_i}$ are separated from $V\cup W\cup \{s\}$ by $A_{g_i} ,A_{h_i}$. Therefore, the modulus of  vertical curves in $\RR$ intersecting $B_s$ is $\preceq_\epscasc 1/n $. Since $n$ is big, the lemma follows.
\end{proof}

\begin{lem}
\label{lem:near transl:est:eps:impr}
There is a universal constant $C>0$ and a constant $T_\epscasc>0$ depending on $\epscasc$ such that there are no  $[C,5]^-$-wide intervals $L_S$ with $|S|\ge T_\epscasc$.
\end{lem}
\begin{proof}
Follows from Lemma~\ref{lem:1/eps buffer:impr} in the same way as Lemma~\ref{lem:near transl:est:eps} follows from Lemma~\ref{lem:1/eps buffer}.
\end{proof}

\begin{thm}[Beau coarse-bounds]
\label{thm:beau:part U: c quasi line} Let $\tF_\qq=\big( f^t \colon U\to U_t\big)_{0
\le t \le \qq}$ be a $\epscasc$-near rotation domain. 
There is a constant $T_\epscasc>1$ depending on $\epscasc$ such that the following holds. The following \emph{universal Log-Rule} holds for intervals \[I,J\subset \partial Z\sp\sp\sp\text{ such that }\sp |I|,\ |J|,\ \dist(I,J) \ge T_\epscasc/\qq .\] If $\dist (I,J)\le \min\{|I|,|J|\}$, then
\begin{equation}
\label{eq:1:lmm:beau:nrd:W^- I J}
\Width^-(I,J) \asymp \log \frac{\min\{|I|,|J|\}}{\dist (I,J)} +1;
\end{equation}
otherwise 
\begin{equation}
\label{eq:2:lmm:beau:nrd:W^- I J}
\Width^-(I,J) \asymp \left( \log \frac{\dist (I,J)}{\min\{|I|,|J|\}} +1\right)^{-1}.
\end{equation}

\end{thm}
\begin{proof}
Follows from Lemma~\ref{lem:near transl:est:eps:impr} in the same way as Proposition~\ref{prop:part U: c quasi line} follows from Lemma~\ref{lem:near transl:est:eps} -- see Remark~\ref{rem:prop:part U: c quasi line}.
\end{proof}

In Theorem~\ref{thm:wZ:shallow scale}, we will extend beau coarse-bounds for pseudo-Siegel disks.

\section{Parabolic fjords}
\label{s:par fjords}
In this section we will describe the geometry inside a fjord $\Fjord_m(T)$ from \S\ref{ss:pullbacks curves}, where $T\in \Dbb_m, \sp m\ge -1$ is an interval in the diffeo-tiling, see~\S\ref{sss:diff tilings}. In particular, we establish a Log-Rule for wide ``parabolic'' rectangles in $\C\setminus \filled_m$ based on $T$. By removing $1$-buffer, such rectangles are in $\Fjord_m(T)$.

We recall that $|T|\in \{ \ell_m, \ell_m+\ell_{m+1}\}$ and $T'\coloneqq T\cap f^{\qq_{m+1}}(T)$. If $m=-1$, then $T=[c_0, c_0\boxplus 1]\simeq \partial Z$ and $T'$ is the longest interval between $c_1$ and $c_0$.

Recall from~\eqref{eq:dfn:based on} that a rectangle $\RR$ is based on $T'$ if $\RR\subset \wC\setminus Z$ and $\partial^h\RR\subset T'$.  We assume that $\partial ^{h,0}\RR< \partial ^{h,1}\RR$ in $T'$ so that $\lfloor \partial \RR \rfloor \subset T'$. 

Let us denote by $\dist_T(\ )$ the ``internal'' metric on $T$ induced by $\dist(\   )$. Namely, if $m>-1$, then we set $\dist_T(x,y)\coloneqq \dist(x,y)$. For $m=-1$ and $x,y\not= c_0$, we define $\dist_T(x,y)$ to be the length of the interval $(x,y)$ that does not contain $c_0$.  In other words, we view $T$ as $(c_0,c_0\boxplus 1)$ with the induced Euclidean metric.

 A rectangle based on $T'$ or $T$ is called \emph{parabolic} (of level $m$) if
\begin{equation}
\label{eq:parab rect}
\dist_{T}(\partial^{h,0}\RR,  \partial^{h,1}\RR ) \ge 6\min \{ | \partial^{h,0}\RR|,\sp | \partial^{h,1}\RR| \} +3\length_{m+1}
\end{equation}
i.e.~the gap between $ \partial^{h,0}\RR$ and $\partial^{h,1}\RR$ is bigger than the minimal horizontal side of $\RR.$ We say that a parabolic rectangle $\RR$ is \emph{balanced} if $|\partial^{h,0} \RR|=|\partial^{h,1} \RR|$. 

Let us assume that 
\begin{equation}
\label{eq:T:orientat ass} T=[v,w], \sp v<w,\sp\sp \theta_{m+1} <0,\sp\sp T'=[v',w],\sp\sp\text{ where }v'=v\boxplus \theta_{m+1},
\end{equation}
i.e.~$f^{\qq_{m+1}}\mid T$ moves points clockwise towards $w$, see Figure~\ref{fig:par rect}.  The case $\theta_{m+1}>0$ is equivalent. For a parabolic rectangle $\RR$ based on $T'$ we will often write
\begin{equation}
\label{eq:part RR:notat}
\partial^{h,0} \RR=[a,b],\sp \partial^{h,1} \RR=[c,d], \sp\sp\text{ where }\sp a<b<c<d.
\end{equation}

\begin{figure}[t!]
\begin{tikzpicture}[ line width=0.5mm ] 
\draw (-6.5,0)--(-4.5,0)
(-2.5,0)--(2.5,0)
(4.5,0) --(6.5,0);
\draw[red, ] (-4.5,0) --(-2.5,0)
(4.5,0) --(2.5,0);

\node[red,below] at(-3.5, 0) {$\partial ^{h,0} \RR$};
\node[red,below] at(3.5, 0) {$\partial ^{h,1} \RR$};

\filldraw (-6,0)  circle (0.5mm);
\node[below] at (-6,0) {$v$};

\filldraw (-5.5,0)  circle (0.5mm);
\node[below] at (-5.5,0) {$v'$};

\filldraw (6,0)  circle (0.5mm);
\node[below] at (6,0) {$w$};

\filldraw (-4.5,0)  circle (0.5mm);
\node[below] at (-4.5,0) {$a$};

\filldraw (-2.5,0)  circle (0.5mm);
\node[below] at (-2.5,0) {$b$};

\filldraw (4.5,0)  circle (0.5mm);
\node[below] at (4.5,0) {$d$};

\filldraw (2.5,0)  circle (0.5mm);
\node[below] at (2.5,0) {$c$};

\draw[blue] (-4.5,0)  .. controls (-2,2.5) and (2,2.5) .. (4.5,0); 
\draw[blue] (2.5,0) .. controls (1.5,1.5) and (-1.5,1.5) ..  (-2.5,0); 

\draw[opacity = 0, fill=blue, fill opacity=0.05] (-4.5,0)  .. controls (-2,2.5) and (2,2.5) .. (4.5,0)
 .. controls (4,0) and (3.5,0) .. 
(2.5,0) .. controls (1.5,1.5) and (-1.5,1.5) ..  (-2.5,0); 

\node[blue] at(0,1.4) {$\RR$};

\draw (-6,0.2) edge[->,red, bend left=40] node[above]{$f^{\qq_{m+1}}$}(-5.5,0.2);

\end{tikzpicture}
\caption{A parabolic rectangle $\RR$ on $T'=[v',w]$, following Notations~\eqref{eq:T:orientat ass} and~\eqref{eq:part RR:notat}.} 
\label{fig:par rect}
\end{figure}
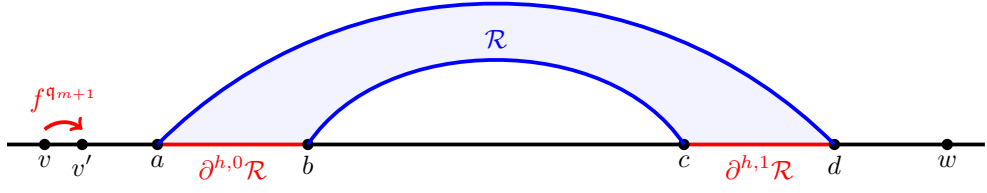

Following~\S\ref{sss:ext div famil}, we say that a parabolic rectangle $\RR$ is \emph{external} rel $\filled_m$ if $\intr\RR\subset \wC\setminus \filled_m$. The following theorem (Item~\ref{thm:par fjords:part1}) says that if there is at least one external wide parabolic rectangle, then the width of all such wide rectangles is calculated according to the Log-Rule on a certain $T_\parab$. Items~\ref{thm:par fjords:part2} and \ref{thm:par fjords:part3} state the Log-Rule for intervals in $T_\parab$.

\begin{thm}[Log-rule for $\Fjord_m(T)$]
\label{thm:par fjords}
Assume that there is at least one sufficiently wide external parabolic rectangle rel $\filled_m$ based on $T$ (see~\eqref{eq:T:orientat ass}). Then $T$ contains a subinterval 
 \[T_\parab=[x,y]\subset T',\sp\sp v'<x<y<w ,\sp\sp \dist(x,v)\le\dist(y,w)\] with the following properties.

 \emph{\bf \setword{(\RN{1})}{thm:par fjords:part1}} If $\RR$ is any external parabolic rectangle based on $T$ with  $\Width(\RR)\gg 1$, then $\RR$  contains a balanced parabolic subrectangle 
\[ \RR^\new \subset\RR,\sp\sp \partial^h\RR^\new \subset T_\parab, \sp\sp\Width(\RR^\new)= \Width(\RR)-O(1)\]   
 such that $x< \partial ^{h,0}\RR^\new < \partial ^{h,1}\RR^\new <y$, $\sp\dist(x,\partial^{h,0}\RR^\new )=\dist(\partial^{h,1}\RR^\new,y ),$
\begin{equation}
\label{eq:1:thm:par fjords}
 \Width(\RR^\new) \asymp \log \frac{|\partial^{h,0}\RR^\new|}{\dist ( \partial^{h,0}\RR^\new ,v)}=\log \frac{|\partial^{h,1}\RR^\new|}{\dist ( \partial^{h,0}\RR^\new ,v)}.
\end{equation}

 \emph{\bf \setword{(\RN{2})}{thm:par fjords:part2}}  If $I,J\subset T_\parab,  \sp x<I<J<y$ are two intervals with
\[ \dist(x,I) \asymp \dist(J,y) \preceq |I| \asymp |J|\preceq \dist_{T}(I,J) \]
and $|I|, |J|,\dist_{T}(I,J)\ge \length_{m+1}$, then
\begin{equation}
\label{eq:2:thm:par fjords}
 \Width^+(I,J) -O(1)=\Width_{\ext,m}^+(I,J) \asymp \log^+ \frac{\min\{|I|,|J|\}}{\dist ( v,I)}+1.
\end{equation}

 \emph{\bf \setword{(\RN{3})}{thm:par fjords:part3}}  If $I,J\subset T_\parab,  \sp x<I<J<y$ are two intervals with
\[ |\lfloor I,J \rfloor | <\frac 1 2 |T_\parab| \sp\text{ and } \sp \min\{|I|, |J|\} \succeq\dist(I,J) >3\length_{m+1}\]
and $|I|, |J|\ge \length_{m+1}$, then
\begin{equation}
\label{eq:3:thm:par fjords}
\Width^+(I,J) -O(1)=\Width_{\ext,m}^+(I,J)\asymp \log^+ \frac{\min\{|I|,|J|\}}{\dist ( I,J)}
 +1 .
\end{equation}
\end{thm}

Theorem~\ref{thm:par fjords} will be proven in \S\ref{ss:prf:thm:par fjords}. We remark that narrow families based on $T$ can be estimated by evaluating their dual families using Theorem~\ref{thm:par fjords}. 

\subsubsection{Explanation and Motivation}
\label{sss:outline:s:par fjords} Theorem~\ref{thm:par fjords} says that Siegel disks develop fjords in a controllable way. Roughly, as Figures~\ref{fig:fjords}  and \ref{fig:parabol_fjords} illustrate, fjords are vertical strips towards the $\alpha$ fixed point and wide parabolic rectangles are horizontal. After conformal uniformization, $f^{q_{m+1}}\mid \text{fjord}$ becomes a ``quasi-rotation'' of the unit disk; Theorem~\ref{thm:par fjords} describes the geometry of this quasi-rotation. Under the conformal uniformization ``$\text{fjord}\to \Disk$'', the hyperbolic geodesic $\ell_{x,y}\subset \partial (\text{fjord})$ connecting $x,y$ will get the Euclidean length $\asymp \dist(v,x)$ -- this explains $v$ in the estimates of Theorem~\ref{thm:par fjords}; taking this into account, the Log-Rules in Theorem~\ref{thm:par fjords} are similar to the Log-Rule in Lemma~\ref{lmm:W^- I J}.


The central theme of this section is designing ``shifts'' for rectangles based on $T$; after that the proof of Theorem~\ref{thm:par fjords} is similar to Lemma~\ref{lmm:W^- I J}. Shifts towards $v$ (pullbacks) are relatively easy and follow from Lemma~\ref{lem:Fjord_m(T)}: external parabolic rectangles ``$\RR\subset \wC\setminus \intr \filled_m$'' are in $\Fjord_m(T)$ after removing $1$-buffer, see~\S\ref{ss:pullbacks curves}. Shifts towards $w$ (push-forwards) are more delicate because curves may hit $\filled_m$. Since pullbacks are well-defined, we can choose the closest to $v'$ outermost external parabolic geodesic rectangle $\RR_\out$ with a certain fixed width; then we set $T_\parab\coloneqq[x,y]$ to be the complementary interval between $\partial^{h,0}\RR_\out $ and $\partial^{h,1}\RR_\out$. Thanks to the ``protection'' by $\RR_\out $,  rectangles based on $[x,y]$ can be efficiently shifted towards $y$ using $f^{\qq_{m+1}}$, see~\S\ref{ss:push forw curves}. We note that our arguments are not local: the global branched structure of $f^{\qq_{m+1}}$ is essential in the proofs. 

At the end of the section, we will formulate Lemma~\ref{lem:central cond} dealing with the ``non-central'' case  $\dist_T(y,w)\gg \dist_T(v,x)$. It will be used in Corollary~\ref{cor:regul}, Case~\ref{case:2:thm:regul} of the Welding Lemma. See also Remark~\ref{rem:thm:regul}.

\subsubsection{Theorem~\ref{thm:par fjords} after assuming final results of the paper}
\label{sss:thm:par fjords:post factum} Let us remark that \emph{post factum}, assuming the final results of the paper established in Part~\ref{part:conclus}, the interval $T_\parab$ in Theorem~\ref{thm:par fjords} can be taken to be $T'$. Indeed, writing $\bK=O(1)$, where $\bK$ is the absolute constant in Theorem~\ref{eq:cond on theta}, Lemma \ref{lem:est for bN} implies~\eqref{eq:2:thm:par fjords} if $I$ and $J$ are attached to $v'$ and $w$ respectively. This allows us to set $T_\parab\coloneqq T'$.

\subsubsection{Remark about parabolic fjords and unicritical circle maps}\label{sss:rem:crit circle maps and fjords} It follows from~\S\ref{sss:thm:par fjords:post factum} that the geometry of level $m+1$ diffeo-tiling $\Dbb_{m+1}$ in $T$ is that of unicritical circle maps: with respect to the ``external'' harmonic measure of $\big(\wC\setminus \overline Z,\infty \big)$ 
\begin{itemize}
\item the intervals in $\Dbb_{m+1}\cap T$ that are not far from $v,w$ have external harmonic measures comparable to that of $T$;
\item the external harmonic measures of intervals in $\Dbb_{m+1}\cap T$ that are deep in $T$ shrink according to the Log-rule. 
\end{itemize}

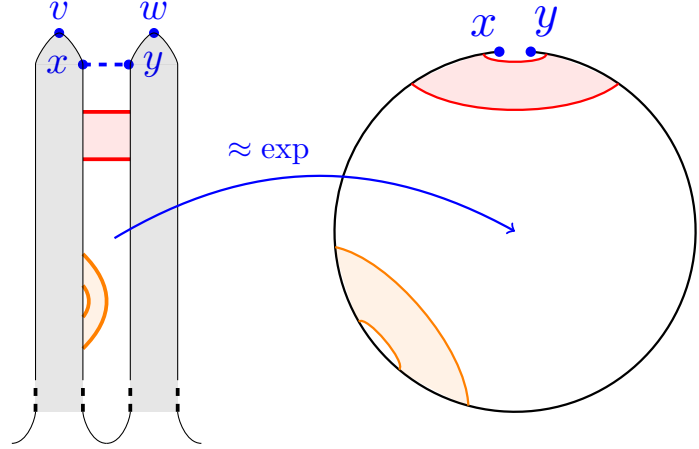
\begin{figure}
\begin{tikzpicture}[scale=0.6, every node/.style={scale=0.6}]

\begin{scope}[shift={(14,-3)},scale=0.7,xscale =0.5]

\coordinate (dd) at (3.5,7);

\begin{scope}
\filldraw[blue,xscale =2] (0,12) circle (0.14 cm);

\node[above,blue,scale=2.5 ] at (0,12) {$w$};
\node[above,blue,scale=2.5 ] at (-6,12) {$v$};

\draw[blue, line width=0.5mm,shift={(-6,0)}]  (1.5,11) edge[dashed] (4.5,11) ;

\filldraw[blue, shift={(-4.5,11)},xscale=2] (0,0) circle (0.14 cm);

\filldraw[blue, shift={(-1.6,11)},xscale=2] (0,0) circle (0.14 cm);

\node[left,blue,scale=2.5 ] at (-4.5,11) {$x$};
\node[right,blue,scale=2.5 ] at (-1.6,11) {$y$};

\draw[red, line width=0.5mm,shift={(-6,0)}]  (1.5,9.5) -- (4.5,9.5) ;
\draw[red, line width=0.5mm,shift={(-6,0)}]  (1.5,8) -- (4.5,8) ;

\draw[red, line width=0.5mm,shift={(-6,0)},opacity=0, fill=red, fill opacity=0.1]  (1.5,9.5) -- (4.5,9.5) --(4.5,8)--(1.5,8) ;


\draw[orange, line width=0.5mm] 
(-4.5,5)
.. controls (-2.5,4) and (-2.5,3) .. 
(-4.5,2)
(-4.5,4)
.. controls (-4,3.7) and (-4,3.3) .. 
(-4.5,3);

\draw[opacity=0, fill= orange, fill opacity=0.1] 
(-4.5,5)
.. controls (-2.5,4) and (-2.5,3) .. 
(-4.5,2)
.. controls (-4.5,2) and (-4.5,3) .. 
(-4.5,3)
.. controls (-4,3.3) and (-4,3.7)  .. 
(-4.5,4);

\coordinate (xx) at (-2.5,5.5);


\draw (-3,-1)  .. controls (-2,-1) and (-1.5, -1+0.8) .. 
 (-1.5,0);
\draw[dashed,line width=0.5mm]  (-1.5,0)--(-1.5,1);
  
\draw (-1.5,1)  .. controls (-1.5,1) and (-1.5, 3) .. 
 (-1.5,11) 
  .. controls (-1.5,11.2) and (-0.5, 12) .. 
 (0,12);

\draw[xscale=-1] (-3,-1) 
 .. controls (-2,-1) and (-1.5, -1+0.8) .. 
 (-1.5,0);
\draw[xscale=-1,dashed,line width=0.5mm]  (-1.5,0)--(-1.5,1);
  
\draw [xscale=-1](-1.5,1)
  .. controls (-1.5,1) and (-1.5, 3) .. 
 (-1.5,11) 
  .. controls (-1.5,11.2) and (-0.5, 12) .. 
 (0,12);
 \end{scope}

\begin{scope}[shift={(-6,0)}]
\filldraw[blue,xscale =2] (0,12) circle (0.14 cm);

\draw (-3,-1) 
 .. controls (-2,-1) and (-1.5, -1+0.8) .. 
 (-1.5,0);
\draw[dashed,line width=0.5mm]  (-1.5,0)--(-1.5,1);
  
\draw (-1.5,1)  .. controls (-1.5,1) and (-1.5, 3) .. 
 (-1.5,11) 
  .. controls (-1.5,11.2) and (-0.5, 12) .. 
 (0,12);

\draw[xscale=-1] (-3,-1) 
 .. controls (-2,-1) and (-1.5, -1+0.8) .. 
 (-1.5,0);
\draw[xscale=-1,dashed,line width=0.5mm]  (-1.5,0)--(-1.5,1);
  
\draw [xscale=-1](-1.5,1)
  .. controls (-1.5,1) and (-1.5, 3) .. 
 (-1.5,11) 
  .. controls (-1.5,11.2) and (-0.5, 12) .. 
 (0,12);
 \end{scope}

\draw[opacity=0,fill ,fill opacity=0.1] 
 (-1.5,11) 
  .. controls (-1.5,11.2) and (-0.5, 12) .. 
 (0,12)
  .. controls (0.5, 12) and (1.5,11.2)   ..
  (1.5,11)  ;

\draw[shift={(-6,0)} ,opacity=0,fill ,fill opacity=0.1] 
 (-1.5,11) 
  .. controls (-1.5,11.2) and (-0.5, 12) .. 
 (0,12)
  .. controls (0.5, 12) and (1.5,11.2)   ..
  (1.5,11)  ;

 \draw[opacity=0,fill ,fill opacity=0.1,shift={(-9,0)} ] (1.5,0)-- (1.5,11)--(4.5,11)--(4.5,0);

  \draw[opacity=0,fill ,fill opacity=0.1,shift={(-3,0)} ] (1.5,0)-- (1.5,11)--(4.5,11)--(4.5,0);


  \end{scope}

\begin{scope}[shift={(22,1)},line width=0.3mm]
\draw (4*0.08715574274,4*0.99619469809) arc (85:-265:4);

\filldraw[blue]  (4*0.08715574274,4*0.99619469809)  circle (0.09 cm);
\filldraw[blue]  (-4*0.08715574274,4*0.99619469809)  circle (0.09 cm);

\node[blue,above,scale=3] at (4*0.17364817766,4*0.98480775301)  {$y$};
\node[blue,above,scale=3] at (-4*0.17364817766,4*0.98480775301)  {$x$};

\draw[red] (4*0.17364817766,4*0.98480775301)
.. controls (0.7, 3.7) and (-0.7,3.7) ..
(-4*0.17364817766,4*0.98480775301);

\draw[red] (-4*0.57357643635,4*0.81915204428)
.. controls (-1.3, 2.5) and (1.3,2.5) ..
(4*0.57357643635,4*0.81915204428);

\draw[opacity=0, fill=red,fill opacity=0.1] (4*0.17364817766,4*0.98480775301)
.. controls (0.7, 3.7) and (-0.7,3.7) ..
(-4*0.17364817766,4*0.98480775301)
.. controls  (-4*0.2588190451, 4*0.96592582628) and (-4*0.5, 4*0.86602540378) ..
 (-4*0.57357643635,4*0.81915204428)
.. controls (-1.3, 2.5) and (1.3,2.5) ..
(4*0.57357643635,4*0.81915204428)
.. controls (4*0.5, 4*0.86602540378) and  (4*0.2588190451, 4*0.96592582628) ..
(4*0.17364817766,4*0.98480775301);

\begin{scope}[rotate=130]
\draw[orange] (4*0.17364817766,4*0.98480775301)
.. controls (0.7, 3.7) and (-0.7,3.7) ..
(-4*0.17364817766,4*0.98480775301);

\draw[orange] (-4*0.57357643635,4*0.81915204428)
.. controls (-1.3, 2.5) and (1.3,2.5) ..
(4*0.57357643635,4*0.81915204428);

\draw[opacity=0, fill=orange,fill opacity=0.1] (4*0.17364817766,4*0.98480775301)
.. controls (0.7, 3.7) and (-0.7,3.7) ..
(-4*0.17364817766,4*0.98480775301)
.. controls  (-4*0.2588190451, 4*0.96592582628) and (-4*0.5, 4*0.86602540378) ..
 (-4*0.57357643635,4*0.81915204428)
.. controls (-1.3, 2.5) and (1.3,2.5) ..
(4*0.57357643635,4*0.81915204428)
.. controls (4*0.5, 4*0.86602540378) and  (4*0.2588190451, 4*0.96592582628) ..
(4*0.17364817766,4*0.98480775301);

\end{scope}

\draw[blue] (xx) edge[->,bend left] node[shift={(-1,0)},above,scale=2] {$\approx \exp$}  (0,0);

\end{scope}


\end{tikzpicture}

\caption{Two types of wide rectangles in a parabolic fjord; compare with Figures~\ref{fig:Man + S disks} and \ref{fig:fjords}. The red (parabolic) rectangle has a small combinatorial distance towards $x$ and $y$ relative its horizontal sides, while the orange rectangle has a small distance between its horizontal sides. On the right picture, the distance between $x,y$ should be thought as $\asymp \dist(v,x)$ -- explaining ``$v$'' in\eqref{eq:1:thm:par fjords} and~\eqref{eq:2:thm:par fjords}.}
\label{fig:parabol_fjords}
\end{figure}

\subsection{Pullbacks in fjords}
\label{ss:pullbacks curves} In this subsection, we will first establish a few modifications of  Lemma~\ref{lem:Fjord_m(T)} and then obtain a few consequences.

We say that a parabolic rectangle $\RR$ based on $T$ is \emph{non-winding} if every vertical curve in $\RR$ is homotopic in $\C\setminus Z$ to a curve in $T$; i.e.~vertical curves in a non-winding parabolic rectangle do not go around $\infty$.  For a non-winding parabolic rectangle $\RR$, we will write \[\upbullet{ \RR}\coloneqq \RR\cup \overline O,\]
where $O$ is the bounded component of $\C\setminus \big(\overline  Z\cup \partial^{v,\inn} \RR\big)$; i.e.~$O$ is the component between $\RR$ and $T$.

\begin{lem}[From $\Fam^+_{\ext,m}$ to non-winding]
\label{lem:injecti of diffeo rect}
Let $\RR$ be an external parabolic rectangle based on $T$ with $ \Width(\RR)>1.$ Let $\RR^{\new}$ be the rectangle obtained from $\RR$ by removing the outer $1$-buffer. Then $\RR^{\new},\sp f^{\qq_{m+1}}(\RR^\new)$ are non-winding and $f^{\qq_{m+1}}\mid \upbullet{\RR}^\new$ is injective. 
\end{lem}
\begin{proof}
 By Lemma~\ref{lem:hyp rectangles}, after removing the outermost  $1/2$-buffer from $\RR$ the new rectangle $\RR_2$ is bounded from outside by the hyperbolic geodesic $\gamma_T$ of $\wC\setminus \filled_m$ connecting the endpoints of $T$. Since the harmonic measure of $(\wC\setminus \filled_T, \infty)$ is less than $2^{-\qq_{m+1}}\le \frac12$, we obtain that $\gamma_T\cup T$ bounds an open disk in $\C$; i.e., $\RR_2\subset \C$. By Lemma~\ref{lem:hyp rectangles}, after removing the outermost $1/2$-buffer from $\RR_2$, the new rectangle $\RR^\new$ is within the fjord $\Fjord_m(T)$ from see~\S\ref{sss:fjords}. We obtain that $\upbullet{\RR}^\new\subset \Fjord_m(T)$ and, by Lemma~\ref{lem:Fjord_m(T)}, $f^{\qq_{m+1}}\mid \upbullet{\RR}^\new$ is injective. The image $f^{\qq_{m+1}}(\RR^\new)$ is non-winding.
\end{proof}

\begin{lem}[Pullbacks]
\label{lem:Rect:pullback}
Let $\RR$ be a parabolic non-winding rectangle on $T'$. Then the pullback of $\RR$ along $f^{\qq_{m+1}}\colon T'\boxminus \theta_{m+1}\to T'$ is a parabolic non-winding rectangle on $T$.
\end{lem}
\begin{proof}
For every vertical curve $\ell\in \RR$ there is a homotopy $\tau$ in $\C\setminus \big(Z\cup T'^c\big)$ between $\ell$ and a curve $\bar \ell\subset T'$, where $T'^c=\partial Z\setminus T'$. This homotopy $\tau$ lifts under $f^{\qq_{m+1}}$ to a homotopy between $\bar \ell\boxminus \theta_m\subset T$ and a curve $\ell_1$; all such curves $\ell_1$ form a parabolic non-winding rectangle $\RR_1$ which is the pullback of $\RR$ along $f^{\qq_{m+1}}\colon T'\boxminus \theta_{m+1}\to T'$.
\end{proof}

\begin{lem}
\label{lem:left balanc}
Let $\RR$ be a parabolic non-winding rectangle based on $T'$. Then $\RR$ contains a parabolic non-winding subrectangle $\RR^\new$ with 
\begin{equation}
\label{eq:lem:left balanc}
|\partial^{h,0}\RR^\new |\le |\partial^{h,1}\RR^{\new}|\sp\sp\text{ and }\sp \sp\Width(\RR^\new) \ge \Width(\RR)-2,
\end{equation}
where we assume Notations~\eqref{eq:T:orientat ass} and~\eqref{eq:part RR:notat}.
\end{lem}
\begin{proof}
Assume that $|\partial^{h,0}\RR |> |\partial^{h,1}\RR|$. Present \[\partial ^{h,0}\RR=I_1\#I_2,\sp\sp I_1\le I_2,\sp\sp \sp |I_1|=|\partial^{h,1}\RR|,\]
and let $\RR_2$ be the subrectangle of $\RR$ consisting of vertical curves emerging from $I_2$. We claim that $\Width(\RR_2)\le 2$, this implies the lemma by removing $\RR_2$ from $\RR$.

The claim follows from the Shift Argument~\S\ref{sss:shift argum}. Let $k$ be the smallest integer such that $k\ge |I_1|/\length_{m+1}$. Pulling back $\RR_2$ under $f^{\qq_{m+1}k}$, we obtain a rectangle $\RR'_2$ linked with $\RR_2$, see~\eqref{eq:LinkCond}:
$\partial^{h,0} \RR_2 < \partial ^{h,1}\RR'_2 <\partial^{h,1}\RR_2.$ Lemma~\ref{lem:shift argum} completes the proof.
\end{proof}

\begin{lem}[Upper Log-Rule]
\label{lem:contr:part^h,1}
Let $\RR$ be a parabolic non-winding rectangle based on $T'$ with $\Width(\RR)>2$. Assume Notations~\eqref{eq:T:orientat ass} and~\eqref{eq:part RR:notat}. Then  $|\partial^{h,0}\RR|>  1$ and
\begin{equation}
\label{eq:lem:contr:part^h,1}
\Width(\RR) \ \preceq\  \log \frac{|\partial^{h,0}\RR |}{\dist(v, \partial^{h,0}\RR)} +1,
\end{equation} 
i.e.,~$\dist(v, \partial^{h,0}\RR) $ is small compared to $|\partial^{h,0}\RR|$ if $\Width(\RR)$ is big.
\end{lem}
\begin{proof}
Follows from the Shift Argument~\S\ref{sss:shift argum} and Lemma~\ref{lem:Rect:pullback}. If $|\partial^{h,0}\RR|< 1$, then the pullback $\RR_1$ of $\RR$ under $f^{\qq_{m+1}}$ would be linked to $\RR$ -- impossible because $\Width(\RR)>2$.

By Lemma~\ref{lem:left balanc}, there is a parabolic subrectangle  $\RR^\new\subset \RR$ satisfying~\eqref{eq:lem:left balanc}; thus $|\partial^{h,0}\RR^\new|<\dist_{T'}(\partial^{h,0}\RR^\new,\partial^{h,1}\RR^\new)+\length_{m+1}$. Let $k\ge 1$ be the integer part of $\dist(v,\partial^{h,0}\RR^\new )/\length_{m+1}$. Decompose $\partial^{h,0}\RR^\new$ into the concatenation of closed intervals \[I_1\#I_2\#\dots \# I_n, \sp\sp v'\le I_1\le I_2\le \dots \le I_n<w\]
such that \[ |I_1|=k\length_{m+1},\sp |I_2|=2k\length_{m+1}, \dots , |I_{n-1}|=2^{n-1}k \length_{m-1}, \sp |I_{n}|\le 2^{n}k \length_{m-1}.\]
We claim that $n\ge \Width(\RR^\new)/2$ -- this will imply the lemma.

Let $\RR_t$ be the subrectangle of $\RR^\new$ consisting of vertical curves connecting $I_t$ and $\partial^{h,1} \RR^\new$. Then $\RR_t$ is linked to its pullback $\RR'_t$ under $f^{\qq_{m+1} 2^{t-1} k}$, see~\eqref{eq:LinkCond}: $ \partial ^{h,0} \RR'_t < \partial^{h,0} \RR_t < \partial ^{h,1} \RR'_t.$ By Lemma~\ref{lem:shift argum}, $\Width(\RR_t)\le 2$ for every $t$. By the Parallel Law \S\ref{sss:ParLaw},  $n\ge \Width(\RR^\new)/2\ge \Width(\RR)/2+1$.
\end{proof}

The following lemma is a counterpart to Lemma~\ref{lem:injecti of diffeo rect}.
\begin{lem}
\label{lem:zero wind are diffeo} Let $\RR$ be a parabolic  non-winding rectangle based on $T'$. If $\Width(\RR)>2$, then after removing the outermost $2$-buffer from $\RR$, we obtain an external parabolic rectangle $\RR^\new$ with $ f^{\qq_{m+1}}\left(\upbullet{\RR}^\new\right)\subset \upbullet\RR$. In particular, $f^{\qq_{m+1}}\mid \upbullet\RR^\new$ is injective.
\end{lem}
\begin{proof}
As before, we assume Notations~\eqref{eq:T:orientat ass} and~\eqref{eq:part RR:notat}.  Consider the pullback $\RR_1$ of $\RR$ under $f^{\qq_{m+1}}$ (see Lemma~\ref{lem:Rect:pullback}); clearly, $\intr(\RR_1)\subset \C\setminus \intr\filled_m$. We claim that after removing the outermost $2$-buffer from $\RR$, the new rectangle $\RR^\new$ is within $\upbullet \RR_1$. This would imply the lemma because $f^{\qq_{m+1}}\mid \upbullet\RR_1$ is injective.

Denote by $\XX$ the outermost $1$-buffer of $\RR$, and denote by $\YY$ the outermost $1$-buffer of $\RR\setminus \XX$. We have $\RR=\XX\cup \YY\cup \RR^\new$. Let $\XX_1\subset \RR_1$ be the pullback of $\XX$ under $f^{\qq_{m+1}}$. 
Since the distance between $\partial ^{h,0} \RR$ and $\partial ^{h,1}\XX$ is bigger than $\length_{m+1}$ (see~\eqref{eq:parab rect}), we have \[\partial ^{h,0} \XX_1<\partial^{h,0}\YY\cup \partial^{h,0}\RR^\new < \partial^{h,1} \XX_1;\]
thus at most $1$-wide part of $\YY\cup \RR^\new$ can cross $\XX_1$, see \S\ref{sss:non cross princ}. Hence $\RR^\new\subset \RR_1$.
\end{proof}

Combined with Lemma~\ref{lem:left balanc}, we obtain:

\begin{cor}
If $\RR$ is a parabolic non-winding rectangle on $T'$ with $\Width(\RR)\ge 5$, then $|\partial ^{h,0}\RR|,|\partial ^{h,1}\RR|\ge 1$.\qed
\end{cor}

\subsection{Push-forwards in fjords}
\label{ss:push forw curves} As before, we assume Notations~\eqref{eq:T:orientat ass},~\eqref{eq:part RR:notat}.

Let us select a simple arc $\delta\subset \wC\setminus  Z$ connecting a point in $\delta(0)\in \partial Z\setminus (v,w)$ to $\delta(1)=\infty$ such that $\delta$ is disjoint from $\partial Z\setminus \delta(0).$ Then $\Delta\coloneqq \wC\setminus (\overline Z\cup \delta)$ is an open topological disk. 

For a rectangle $\RR$ based on $ T$, we will define below the push-forward $\RR_k$ of $\RR$ under $f^{\qq_{m+1}k}$ assuming that $\dist (\partial ^{h,1}\RR,w)>k\length_{m+1}$. The result $\RR_k$ will be a lamination in $\Delta$.

Let us orient all vertical curves in $\RR$ from $\partial^{h,1} \RR$ to $\partial^{h,0} \RR $:
\begin{equation}
\gamma(0)\in \partial^{h,1}\RR,\sp  \gamma(1)\in \partial^{h,0}\RR\sp \sp\text{ for }\sp\sp [\gamma\colon [0,1]\to \wC\setminus Z]\in \RR.
\end{equation}

Let $\Delta_{-k}$ be the component of $f^{-\qq_{m+1}k}(\Delta)$ attached to $[v,w \boxminus k\theta_{m+1} ]\subset T$. For a vertical curve $\ell\colon [0,1]\to \wC$ in $\RR$, let $t_k^\ell >0$ be the first moment such that $\ell\big(t^\ell_k\big)\in \partial \Delta_{-k}$. 
We define 
\[\RR'_k\coloneqq \{\ell\mid [0,t_k^\ell] \sp \text{ for }\sp \ell\in \RR\},\sp\sp \RR_k\coloneqq f^{\qq_{m+1}k }(\RR'_k).\]
In other words, $\RR'_k$ is the restriction (see~\S\ref{sss:short subcurves}) of $\RR$ to $\Delta_{-k}$ and $\RR_k$ is the appropriate conformal image of $\RR'_k$. 
 We say that the curve $\ell\mid [0,t_k^\ell]$ in $\RR'_k$ and its image $f^{\qq_{m+1}k}\big(\ell\mid [0,t_k^\ell] \big)$ in $\RR_k$ is of
\begin{itemize}
\item {Type \RN{1}} if $t_k^\ell=1$,
\item {Type \RN{2}} if $t_k^\ell<1$ but $f^{\qq_{m+1}k} \circ \ell\big(t_k^\ell\big)\in T$;
\item {Type \RN{3}} otherwise.
\end{itemize}
We denote by $\RR^{\RN{1}}_k,\ \RR^{\RN{2}}_k,\ \RR^{\RN{3}}_k$ the sublaminations of $\RR_k$ consisting of Type~\RN{1}, \RN{2}, \RN{3}  curves respectively. Similarly are defined the sublaminations $\RR'^{\RN{1}}_k,\ \RR'^{\RN{2}}_k,\ \RR'^{\RN{3}}_k$ of  $\RR'^k$. Since $\RR$ overflows $\RR'_k$, we have
\begin{equation}
\label{eq:RR_k:monot}
\Width(\RR)\le \Width(\RR'_k) = \Width(\RR_k).
\end{equation}


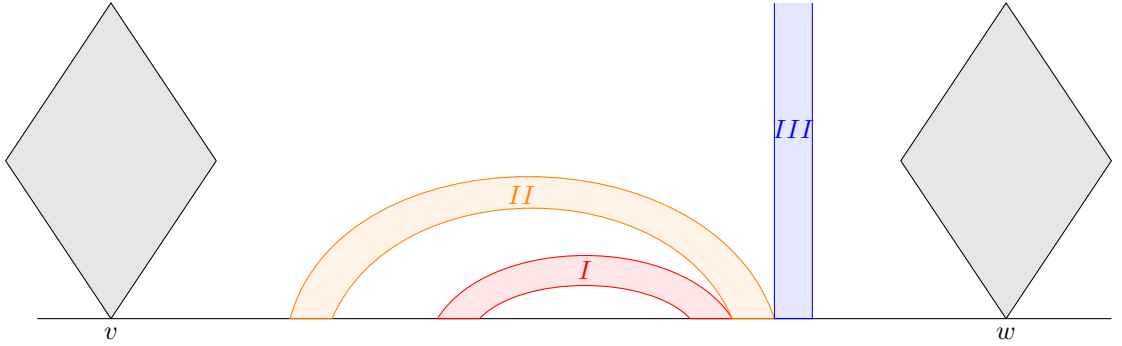
\begin{figure}[t!]
\[\begin{tikzpicture}[scale=1.4]

\draw[shift={(2,0)},fill=black, fill opacity=0.1] (0,0) -- (1,1.5 )--(0,3)--(-1,1.5)--(0,0);

\node[below] at (2,0) {$w$};

\draw (-7.2,0) -- (3,0);

\draw[shift={(-6.5,0)},fill=black, fill opacity=0.1] (0,0) -- (1,1.5 )--(0,3)--(-1,1.5)--(0,0);

\node[below]at (-6.5,0) {$v$};

\begin{scope}[shift= {(-2,0)},scale=0.2]

\draw [red,fill=red , fill opacity =0.1] (7,0)
 .. controls (4.5, 4) and (-4.5,4) ..
 (-7,0)
 .. controls (-7, 0) and (-5,0) ..
  (-5,0)
 .. controls (-3, 2.1) and (3,2.1) ..
 (5,0)
 .. controls (5, 0) and (7,0) ..
 (7,0) ;
 
\node[red] at (0,2.3){$\RN{1}$}; 
\end{scope}

\begin{scope}[shift= {(-2,0)},scale=0.2]

\draw [orange,fill=orange , fill opacity =0.1] (9,0)
 .. controls (6, 9) and (-11.5,9) ..
 (-14,0)
 .. controls (-14, 0) and (-12,0) ..
  (-12,0)
 .. controls (-9, 7) and (4,7) ..
 (7,0)
 .. controls (7, 0) and (9,0) ..
 (9,0) ;
 
\node[orange] at (-3,5.9){$\RN{2}$}; 
\end{scope}

\begin{scope}[shift ={(-0.2,0)},scale=0.3]
\draw[yscale=4,blue, fill=blue, fill opacity=0.1] 
(1.2,2.5) 
-- 
(1.2,0)
--
 (0,0) 
 --
(0,2.5); 

\node[blue] at (0.6,6){$\RN{3}$};

\end{scope}

\end{tikzpicture}\] 
\caption{Types $\RN{1}$ (red), $\RN{2}$ (orange), and $\RN{3}$ (blue) curves in $\RR_k$.}
\label{Fg:Fam:I II III}
\end{figure}
 
\begin{lem}
\label{lem:3 Types}
In $\wC\setminus Z$, the lamination $\RR^{\RN{2}}_k$ separates $\RR^{\RN{1}}_k$ from $\RR^{\RN{3}}_k$ and from $\{v,w\}$; i.e., $\RR^{\RN{3}}_k\cup \{v,w\}$ and $\RR^{\RN{1}}_k$ are in different components of $\wC\setminus (Z\cup \gamma)$ for every $\gamma\in \RR^{\RN{2}}_k$, see Figure~\ref{Fg:Fam:I II III}.
\end{lem}
\begin{proof}
Consider the preimage $T_{-k} \coloneqq f^{-\qq_{m+1} k} (T) \cap \partial \Delta_{-k}$ of $T$ under ${f^{\qq_{m+1} k} \colon\overline \Delta_{-k} \to \overline \Delta}$. Observe that $T_{-k}$ contains $v$ but not $w$. The point $v$ splits $T_{-k}$ into two intervals $T^b_{-k}$ and $T^a_{-k}$, we assume that $T^a_{-k}\subset T$ while $T^b_{-k}$ is disjoint from $\partial Z$. Then \begin{itemize}
\item $\RR'^{\RN{1}}_k$ is the sublamination of $\RR'_k$ landing at $T^a_{-k}$,
\item $\RR'^{\RN{2}}_k$ is the sublamination of $\RR'_k$ landing at $T^b_{-k}$,
\item $\RR'^{\RN{3}}_k$ is the sublamination of $\RR'_k$ landing at $\partial \Delta_{-k}\setminus T_{-k}$.
\end{itemize}
The lemma now follows from the observation that in $\Delta_{-k}$, the lamination $\RR'^{\RN{2}}_k$ separates $\RR'^{\RN{1}}_k$ from $\RR'^{\RN{3}}_k$ and $w\boxminus \theta_{m+1}k$.
\end{proof}

Let $\XX$ be an external parabolic rectangle based on $T'$ with $\Width(\XX)\ge 10$. 
 Let $P\subset T'$ be the complementary interval between $\partial ^{h,0}\XX$ and $\partial ^{h,1}\XX$. We say that $P$ is \emph{protected} by $\XX$. 
\begin{lem}[Push-forwards]
\label{lem:push forw in par int}
Let $P$ be an interval protected by an external parabolic rectangle $\XX$ as above. If $\RR$ is a parabolic rectangle based on $P$ such that \[\partial^{h} \RR\boxplus i\length_{m+1}\subset P \sp\sp\text{ for all }\sp\sp i\in\{0,1,2,\dots, k\},\]
then after removing the $1$-outermost buffer, the rectangle $\RR^\new$ has univalent push-forwards:
\begin{equation}
\label{eq:lem:push forw in par int}
 f^{\qq_{m+1}i}(\RR^\new)\subset \upbullet \XX\sp\sp\text{ for all }\sp\sp i\in\{0,1,2,\dots, k\}.
 \end{equation}
\end{lem}
\begin{proof}
Let us choose $\delta$ to be disjoint from $\upbullet \XX$ and let $\XX^\new$ be the rectangle obtained by removing the outermost $5$-buffer from $\XX$. By Lemmas~\ref{lem:injecti of diffeo rect} and~\ref{lem:zero wind are diffeo}, $f^{\qq_{m+1}}$ is injective on $ \upbullet \XX^\new$ and:  
\begin{equation}
\label{eq:0:prf:lem:push forw in par int}
f^{\qq_{m+1}} \left( \upbullet \XX^\new \right)\subset \upbullet \XX\sp\sp\sp \text{ hence }\sp\sp\sp  \intr\upbullet \XX^\new\subset \Delta_{-1}.
\end{equation}

Note that $\Width(\XX^\new) \ge 5$. Let us prove by induction that  
\begin{equation}
\label{eq:prf:lem:push forw in par int}
 \Width\left(\RR_i^{\RN{1}}\right)\ge \Width(\RR)-4/5.
\end{equation} 
for all $i\le k$. This will imply~\eqref{eq:lem:push forw in par int} because at most $\frac 1 {5}$-wide family of $\RR_i^{\RN{1}}$ can cross the protection $\XX^\new$.

It follows from~\eqref{eq:0:prf:lem:push forw in par int} that
\begin{equation}
\label{eq:1:lem:push forw in par int}
\RR^{\RN{1}}_{i+1} \supseteq\big\{f^{\qq_{m+1}} (\gamma)\sp \mid \sp \gamma\in \RR^{\RN{1}}_{i}\sp \text{ and }\sp \gamma \text{ is disjoint from } \partial \Delta_{-1}\setminus T\big\}.
\end{equation} 

If $\Width(\RR^{\RN{1}}_{i})> \Width(\RR)-3/5$, then at most $1/5$ curves in $\RR^{\RN{1}}_{i}$ can cross the protection $\XX^\new$ and  hit $\partial \Delta_{-1}\setminus T$. We obtain that $\Width(\RR^{\RN{1}}_{i+1})> \Width(\RR)-4/5$.

Assume now that \[\Width(\RR^{\RN{1}}_{i})\le  \Width(\RR)-3/5 \sp\sp \text{hence }\sp\sp \Width(\RR^{\RN{2}}_{i})+\Width(\RR^{\RN{3}}_{i})\ge 3/5,\]
by~\eqref{eq:RR_k:monot}. Since $\RR^{\RN{3}}_{i}$ crosses $\XX^\new$, we obtain $\Width(\RR^{\RN{3}}_{i})\le 1/5$ and $\Width(\RR^{\RN{2}}_{i})\ge 2/5$. At most $1/5$ curves in $\RR^{\RN{1}}\cup \RR^{\RN{2}}$ can cross $\XX^\new$ and hit $\partial \Delta_{-1}\setminus T$; and all such curves must be in $\RR^{\RN{2}}$ -- they are outermost by Lemma~\ref{lem:3 Types}. We obtain that all curves in $\RR^{\RN{1}}_{i}$ are inside $\upbullet \XX^\new$ and 
$\Width(\RR^{\RN{1}}_{i+1})\ge \Width(\RR^{\RN{1}}_{i})$.
\end{proof}

\subsection{Proof of Theorem~\ref{thm:par fjords}}
\label{ss:prf:thm:par fjords}

\begin{lem}
\label{lem:balance rectangle}
Let $\RR$ be a parabolic non-winding rectangle based on $T'$ with $\Width(\RR)>50$. Then $\RR$ contains a parabolic non-winding balanced geodesic rectangle $\RR^\new$ with $\Width(\RR^\new)\ge \Width(\RR)-25$.
\end{lem}
\begin{proof}
We assume Notations~\eqref{eq:T:orientat ass} and~\eqref{eq:part RR:notat}. Let $\RR^\new$ be the rectangle obtained from $\RR$ by removing the outermost $18$-buffer $\XX$. Then Lemma~\ref{lem:push forw in par int} (push-forwards) is applicable in $\upbullet \RR^\new$. 

Choose the maximal intervals $I\subset \partial^{h,0}\RR^\new $ and $J\subset \partial^{h,1}\RR^\new $ so that the geodesic rectangle $\RR(I,J)$ is in $\RR^\new$. By Lemma~\ref {lem:hyp rectangles},  $\RR(I,J)$ contains most of the width of $\RR^\new$: we have $\Width(\RR(I,J))\ge \Width(\RR)-20$.

Assume that $|I|>|J|$. Repeating the proof of Lemma~\ref{lem:left balanc}, we present 
\[ I=I_1\#I_2,\sp\sp I_1\le I_2,\sp\sp \sp |I_1|=|J|.\]
 Let $k$ be the smallest integer such that $k\ge |I_1|/\length_{m+1}$.  Since the geodesic rectangle $\RR(I_2,J)$ is linked to its pullback under $f^{\qq_{m+1}k}$, we have $\Width(\RR(I_2,J))\le 2$; hence $\Width(\RR(I_1,J))\ge \Width(\RR)-25$ and $\RR(I_1,J)$ is required.

Assume that $|I|<|J|$. We present 
\[ J=J_2\#J_1,\sp\sp J_2\le J_1,\sp\sp \sp |J_1|=|I|.\]
 Let $k$ be the smallest integer such that $k\ge |J_1|/\length_{m+1}$.  After removing the outermost  $1$-buffer from the geodesic rectangle $\RR(I,J_2)$, we obtain a rectangle linked to its push-forward under $f^{\qq_{m+1}k}$ (Lemma~\ref{lem:push forw in par int}).
 We have $\Width(\RR(I,J_2))\le 3$; hence $\Width(\RR(I,J_1))\ge \Width(\RR)-25$ and $\RR(I,J_1)$ is required
\end{proof}

\begin{proof}[Proof of Theorem~\ref{thm:par fjords}] For $\zeta\in T'$ with $\dist_T(v',\zeta)<|T'|/10$, define $\zeta_{10}\in T'$ so that $\dist_T(v',\zeta_{10})=10 \dist_T(v',\zeta)$.

Assume there exists a sufficiently wide external parabolic rectangle  $\ZZ$ based on $T'$. As before, we assume $\partial^{h,0}\ZZ <\partial^{h,1}\ZZ$ in $T'$. By Lemma~\ref{lem:balance rectangle}, we can assume that $\XX$ is balanced. By Lemma~\ref{lem:contr:part^h,1}, $|\partial ^{h,0}\ZZ|$ is small compared to $\dist(v, \partial ^{h,0} \ZZ)$ and, moreover, by removing an innermost buffer, we can assume that the distance between  $\partial ^{h,0}\XX$ and $\partial ^{h,1}\XX$ is $10$-times bigger than $|\partial ^{h,0}\XX|+\dist(v', \partial ^{h,0}  \XX)$. Therefore, the existence of $\XX$ implies that we can define the shortest interval
\begin{equation}
\label{eq:prf:thm:par fjords:defn S}
S\coloneqq [v', \zeta]\subset T'\sp\sp\text{ such that }\sp\sp\Width^+_{\ext, m}(S, [\zeta_{10},w])= 500,\sp\sp |S|> \length_{m+1},
\end{equation} where the condition $|S|> \length_{m+1}$ is guaranteed by Lemma~\ref{lem:contr:part^h,1}. By Lemmas ~\ref{lem:injecti of diffeo rect} and \ref{lem:balance rectangle}, we can select two disjoint parabolic balanced non-winding geodesic rectangles based on $T'$ satisfying
\[\XX, \YY\subset \Fam^+_{\ext, m}(S,[\zeta_{10},w]), \sp\sp \XX\subset \YY^\bullet \setminus \YY\sp\sp\text{ with }\sp \Width(\XX)\ge 400,\sp \Width(\YY)\ge 10. \] 
We set, see Figure~\ref{fig:notat: x tx ty y}: 
\begin{equation}
\label{thm:par fjords:prf:1} \widetilde T_\parab=[\tilde x,\tilde y] \coloneqq \lfloor\partial^h \XX\rfloor \subset T',\sp\sp\text{ and }\sp\sp  T_\parab=[x,y] \coloneqq \widetilde T_\parab\setminus \partial^h \XX,\end{equation}
where $\tilde x< x <  y < \tilde y $ in $T'$.
Since $\YY$ protects $\widetilde T_\parab$, wide rectangles based on $\widetilde T_\parab$ can be push-forward. By Lemma~\ref{lem:contr:part^h,1}, $|\partial^{h,0}\XX|\succeq \dist (v',\partial^{h,0}\XX)$ hence
\begin{enumerate}[label=\text{(\Alph*)},font=\normalfont,leftmargin=*]
\item \label{prop:A:prf:thm:par fjords}$\dist_{T'}(v', z) \asymp \dist_{T'}(\tilde x, z) $ for all $z\in T_\parab$.
\end{enumerate}
Since $\Width(\XX)=400$ but $\Width^+_{\ext, m}(S, [\zeta_{10},w])= 500$, we obtain 
\begin{enumerate}[label=\text{(\Alph*)},font=\normalfont, start =2,leftmargin=*]
\item \label{prop:B:prf:thm:par fjords} $\Width^+_{\ext, m}([v',x], [\tilde y,w])\le  \Width^+_{\ext, m}([v',\zeta], [\tilde y,w])\le 102$.
\end{enumerate}

\begin{figure}

\includegraphics[width=11cm]{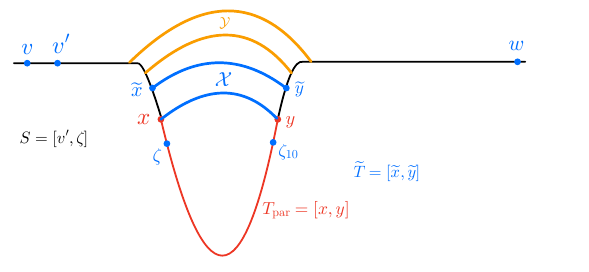}
\caption{Intervals $T_\parab=[x,y]$, $\widetilde T=[\tilde x, \tilde y]$, and $S=[v',\zeta]$.}
\label{fig:notat: x tx ty y}
\end{figure}

\begin{claim8}[Upper Log-Rule]
\label{cl:1:thm:par fjords}
For an interval $I\subset T_\parab$ with \[\dist(I, \{\tilde x, \tilde y\})\succeq |I|\sp\sp\text{ and }\sp\sp |I|\ge \length_{m+1}\] we have $\Width_3^+(I) \preceq 1.$
\end{claim8}
\begin{proof}
Recall from Lemma~\ref{lem:contr:part^h,1} that $\dist(I, \{\tilde x, \tilde y\})\ge \dist(\tilde x,x)\ge \length_{m+1}$. By splitting $I$ into finitely many intervals (depending on the constant representing ``$\succeq$''), it is sufficient to assume that $\dist(I, \{\tilde x, \tilde y\})\ge |I|+\length_{m+1}$ or $|I|=\length_{m+1}$. Write $I=[a,b]$ with $\tilde x< a<b< \tilde y$. Let us present $\Fam_3^+(I)$ as $\Fam_-\sqcup \Fam_+\sqcup \Fam'$ where
\begin{itemize}
\item $\Fam'$ consists of curves crossing $\XX$;
\item $\Fam_-$ consists of curves in $\upbullet \XX$ connecting $I$ and $[\tilde x, a]\cap (3I)^c$; 
\item $\Fam_+$ consists of curves in $\upbullet \XX$ connecting $I$ and $[b, \tilde y]\cap (3I)^c$.
\end{itemize}
Clearly, $\Width(\Fam') \le 1/\Width(\XX)$.  We will estimate the width of $\Fam_-,\Fam_+$ using the Shift Argument.  Let $\RR_-\subset \Fam_-$ and $\RR_+\subset \Fam_+$ be the canonical rectangles; i.e.~$\Width(\RR_-)=\Width(\Fam_-)$ and $\Width(\RR_+)=\Width(\Fam_+)$. Let $k$ be the smallest integer such that $k\ge |I|/\length_{m+1}$. Then $\RR_+$ is linked to its pullback under $f^{-k\qq_{m+1}}$ implying that $\Width(\RR_+)\le 2$. Since $\widetilde T_\parab$ is protected by $\YY$, the rectangle $\RR_-^\new$ obtained by removing the outermost $1$-buffer from $\RR_-$ is linked to its push-forward under $f^{k\qq_{m+1}}$ (Lemma~\ref{lem:push forw in par int}); this implies $\Width(\RR^-)\le 3$.
\end{proof}

\begin{claim8} 
\label{cl:2:thm:par fjords}
If $I, J\subset T_\parab,\sp  I<J$ are two intervals with \[\frac 12|\lfloor I,J\rfloor | < |T_\parab|,\sp\sp\min\{|I|,|J|\}\asymp \dist(I,J),\sp\sp |I|, |J|,\dist(I,J)\ge \length_{m+1},\] then $\Width^+(I,J)\asymp 1$. 
\end{claim8}
\begin{proof}
Assume $|I|\le |J|$. Let $L\subset T_\parab$ be the complementary interval between $I,J$. Applying Claim~\ref{cl:1:thm:par fjords}, and subdividing if necessary $I$ and $L$ into finitely many intervals we obtain
\[\Width^+(I,J)\preceq 1\sp\sp\text{ and }\sp\sp \left(\Width^+(I,J)\right)^{-1}= \Width^+(L,  \lfloor I, J\rfloor^c) \preceq 1.\]
Therefore, $\Width^+(I,J)\asymp 1$.
\end{proof}

Item {\bf \ref{thm:par fjords:part3}} of Theorem~\ref{thm:par fjords} follows from Claims~\ref{cl:1:thm:par fjords} and~\ref{cl:2:thm:par fjords} using the Splitting Argument, see Remark~\ref{rem:SplitArg}, where $\Width^+(\ )-O(1)=\Width^+_{\ext, m}(\ )$ holds because of the protection.

\begin{claim8} 
\label{cl:3:thm:par fjords}
If $I, J\subset T_\parab, I<J$ are two intervals with
\begin{equation}
\label{eq:cl:3:thm:par fjords}
 \dist_{T'}(\tilde x,I) \asymp \dist_{T'}(J,\tilde y) \asymp |I| \asymp |J|\preceq \dist_{T'}(I,J)
\end{equation}
and $|I|, |J|,\dist_{T'}(I,J)\ge \length_{m+1}$. Then $\Width^+(I,J)\asymp 1$.
\end{claim8}
\begin{proof}
The property $\Width^+(I,J)\preceq  1$ follows from Claim~\ref{cl:1:thm:par fjords} by splitting, if necessary, $I$ into finitely many intervals.

Denote by $L\subset T_\parab$ the complementary interval between $I$ and $J.$ Let us show that that the dual family $\FamG=\Fam^+(L , \lfloor I, J \rfloor^c)$ satisfies $\Width(\FamG)\preceq 1$; this will imply the claim. Denote by $N\subset [\tilde x, \tilde y]$ the interval between $\tilde x $ and $\lfloor I, J \rfloor$ and by $M\subset [\tilde x, \tilde y]$ the interval between $\lfloor I, J \rfloor$ and $\tilde y $. As in the proof of Claim~\ref{cl:1:thm:par fjords}, we decompose $\FamG$ as $\FamG'\sqcup \FamG_-\sqcup \FamG_+$ where
\begin{itemize}
\item $\FamG'$ consists of curves crossing $\XX$;
\item $\FamG_-$ consists of curves in $\upbullet \XX$ connecting $N$ and $L$;
\item $\FamG_+$ consists of curves in $\upbullet \XX$ connecting $L$ and $M$.
\end{itemize}

Let $\RR_-\subset \FamG_-$ and $\RR_+\subset \FamG_+$ be the canonical rectangles; i.e. $\Width(\RR_-)=\Width(\FamG_-)$ and $\Width(\RR_+)=\Width(\FamG_+)$. Set 
\[ \tau\coloneqq \min\{\dist_{T'}(\tilde x, I), |I|,  |J|, \dist_{T'}( J, \tilde y)\}-\length_{m+1};\]
if $\tau< \length_{m+1},$ then replace $\tau\coloneqq \length_{m+1}$. We decompose $N$ and $M$ into finitely many intervals $\cup_i N_i$ and $\cup_i M_i$ so that $|N_i|, |M_i| \le \tau$ for all $i$. The number of intervals depends on the constants representing ``$\asymp$'' and ``$\preceq$'' in~\eqref{eq:cl:3:thm:par fjords}. 

Denote by $\RR_{-,i}\subset \RR_-$ the subrectangle consisting of vertical curves landing at $N_i$. Similarly, $\RR_{+,i}\subset \RR_+$ is the subrectangle consisting of vertical curves landing at $M_i$. Define $k$ to be the smallest integer such that $k\length_{m+1}\ge \tau$.
Then $\RR_{-,i}$ is linked to its push-forward under $f^{k\qq_{m+1}}$ (Lemma~\ref{lem:push forw in par int}); i.e.~$\Width(\RR_{-,i})\le 3$. And $\RR_{+,i}$ is linked to its pullback under $f^{k\qq_{m+1}}$; i.e.~$\Width(\RR_{+,i})\le 2$.
\end{proof}

By Property~\ref{prop:A:prf:thm:par fjords} from the beginning of the proof, we can replace $\dist_{T'}(\tilde x,I)$ in  Claim~\ref{cl:3:thm:par fjords} with $\dist_{T'}(v',I)$. Item {\bf \ref{thm:par fjords:part2}}  of Theorem~\ref{thm:par fjords} follows from Claims~\ref{cl:1:thm:par fjords} and~\ref{cl:3:thm:par fjords} using the Splitting Argument, see Remark~\ref{rem:SplitArg};

\begin{claim8} 
\label{cl8:piush forw}
Assume there is an $s\in \N$ such that $\dist_{T'}(y,\tilde y) < s\length_{m+1 }< \dist(\tilde y,w)-\length_{m+1}.$ Define 
$\Upsilon_s\coloneqq f^s[ y,\tilde y] $ and note that $\Upsilon_s$ is between $\Upsilon_0\coloneqq [y,\tilde y]$ and $w$. Then 
\begin{equation}
\label{eq:cl8:piush forw}
\Width^+_{\div,m}\big(\Upsilon_s,\  [x,w]^c\big)\ge 100.
\end{equation}
\end{claim8}
\begin{proof}
See Figure~\ref{Fg:lem:H_i G_i+1} for illustration. Let $\XX_s$ be the push-forward of $\XX$ under $f^{\qq_{m+1} s}$ as in \S\ref{ss:push forw curves}. By construction, $\XX_s$ is lamination containing tree types of curves (Figure~\ref{Fg:Fam:I II III}). Moreover, every curve in $\XX_s$ emerges from $\Upsilon _s$ and land at $[x\boxplus \theta_{m+1}s,w]^c$. We will show below that $\Width(\XX_s\setminus \Fam^+_{\div,m}\big(\Upsilon_s,\  [x,w]^c\big)\le 300$; this will imply that $\Width(\XX_s\cap \Fam^+_{\div,m}\big(\Upsilon_s,\  [x,w]^c\big)) \ge 100$; i.e.,~\eqref{eq:cl8:piush forw} holds.

If $y> x\boxplus s \theta_{m+1}$ in $T$, then \[[x\boxplus \theta_{m+1}s,w]^c=  T^c \cup [v,v']\cup [v',x] \cup [x,y]\]
and the required statement follows from the following estimates:
\begin{itemize}
\item $\Width^+_{\ext, m}(\Upsilon_s,T^c)\le 2$ by  Lemma~\ref{lmm:ext famil}, Item~\ref{Cl:D:lmm:ext famil}; 
\item $\Width^+_{\ext,m} (\Upsilon_s,[v',x])\le 102$ by Property~\ref{prop:B:prf:thm:par fjords} from the beginning of the proof;
\item $\Width^+_{\ext,m} (\Upsilon_s,[v,v'])\le 103$  because, otherwise, after removing the outermost $1$-buffer, Lemma~\ref{lem:injecti of diffeo rect} shifts the associated rectangle to $\Fam^+_{\ext,m}([v',v'\boxplus\theta_{m+1}], \Upsilon_s\boxplus \theta_{m+1})$ -- contradicting Property~\ref{prop:B:prf:thm:par fjords};
\item at most $1/\Width(\XX)<1$ curves in $\XX_s$ land at $[x,y]$ (after crossing $\XX$);
\end{itemize}

Assume that $y< x\boxplus s \theta_{m+1}$ in $T$; then 
\[[x\boxplus \theta_{m+1}s,w]^c=  T^c \cup [v,v']\cup [v',x] \cup [x,y]\cup [y, x\theta_{m+1}s].\]
Recall~\eqref{eq:parab rect} that $|[x,y]|> 6|[y,\tilde y]|$. Therefore,  
\begin{itemize}
\item  $[y,x\boxplus s \theta_{m+1}]$ and $\Upsilon_s=[y,\tilde y]\boxplus s \theta_{m+1}$ are combinatorially far from each other and are separated by $\Fam^+_{\div,m}\big(\Upsilon_{s-t},\  [x,w]^c\big)$ for some $\Upsilon_{s-t}$ between $y$ and $\Upsilon_s$; hence $\Fam^+([y, x\boxplus s \theta_{m+1}], \Upsilon_s)\le 1$. 
\end{itemize}
\end{proof}

\begin{figure}[t!]
\[\begin{tikzpicture}[scale=1.4]

\draw[shift={(2,0)},fill=black, fill opacity=0.1] (0,0) -- (1,1.5 )--(0,3)--(-1,1.5)--(0,0);

\node[below] at (2,0) {$w$};

\draw[shift={(-6,0)},fill=black, fill opacity=0.1] (0,0) -- (1,1.5 )--(0,3)--(-1,1.5)--(0,0);

\node[below] at (-6,0) {$v$};

\node[below,red] at (-0.3,0) {$\Upsilon_s$};
\node[above,red] at (-0.3,0.8) {$\XX_s$};

\draw (-7.,0) -- (3,0);

\begin{scope}[shift= {(-4,0)},scale=0.2]

\draw[red,line width=1mm] (5,0)--(3,0)
(-5,0)--(-3,0);

\draw [red] (5,0)
 .. controls (3.5, 3) and (-3.5,3) ..
 (-5,0);

\draw [red] (3,0)
 .. controls (1.8, 1.1) and (-1.8,1.1) ..
 (-3,0);

 \coordinate (A0) at (4,2) {};
\node[red,below] at(4,0){$[y,\tilde y]$}; 
\node[red,below] at(-4,0){$[\tilde x , x]$};

\node[red] at (0,3){$\XX$}; 
\end{scope}

\begin{scope}[shift ={(-2.5,0)},scale=0.3]

\coordinate (A1) at (-0.5,1.6); 

\coordinate (A2) at (1.5,1.6); 

\draw[red, line width=0.6mm] (0,0) --(1,0);

\draw[red, fill=red, fill opacity=0.1] 
(1,2.5) 
 .. controls  (1.3,1.7) and (1-0.2, 0.8)  .. 
(1,0)
.. controls  (0.8,0) and (0.2, 0)  ..
 (0,0) 
 .. controls (-0.2, 0.8) and (0.3,1.7) .. 
(0,2.5);

\end{scope}

\draw (A0) edge[bend left, ->] (A1);

\begin{scope}[shift ={(-1.5,0)},scale=0.3]

\coordinate (A3) at (-0.5,1.6); 

\coordinate (A4) at (1.5,1.6); 

\draw[red, line width=0.6mm] (0,0) --(1,0);

\draw[red, fill=red, fill opacity=0.1] 
(1,2.5) 
 .. controls  (1.3,1.7) and (1-0.2, 0.8)  .. 
(1,0)
.. controls  (0.8,0) and (0.2, 0)  ..
 (0,0) 
 .. controls (-0.2, 0.8) and (0.3,1.7) .. 
(0,2.5);

\end{scope}

\draw (A2) edge[bend left, ->] (A3);

\begin{scope}[shift ={(-0.5,0)},scale=0.3]

\coordinate (A5) at (-0.5,1.6); 

\coordinate (A6) at (1.5,1.6); 
\coordinate (A7) at (2.8,1.6);

\draw[red, line width=0.6mm] (0,0) --(1,0);

\draw[red, fill=red, fill opacity=0.1] 
(1,2.5) 
 .. controls  (1.3,1.7) and (1-0.2, 0.8)  .. 
(1,0)
.. controls  (0.8,0) and (0.2, 0)  ..
 (0,0) 
 .. controls (-0.2, 0.8) and (0.3,1.7) .. 
(0,2.5);

\end{scope}

\draw (A4) edge[bend left, ->] (A5);

\draw (A6) edge [bend left, ->] (A7);

\node[right,red] at (A7){$\dots$};
\end{tikzpicture}\]
\caption{Laminations $\XX_s$ are push-forwards of the rectangle $\XX$.}
\label{Fg:lem:H_i G_i+1}
\end{figure}
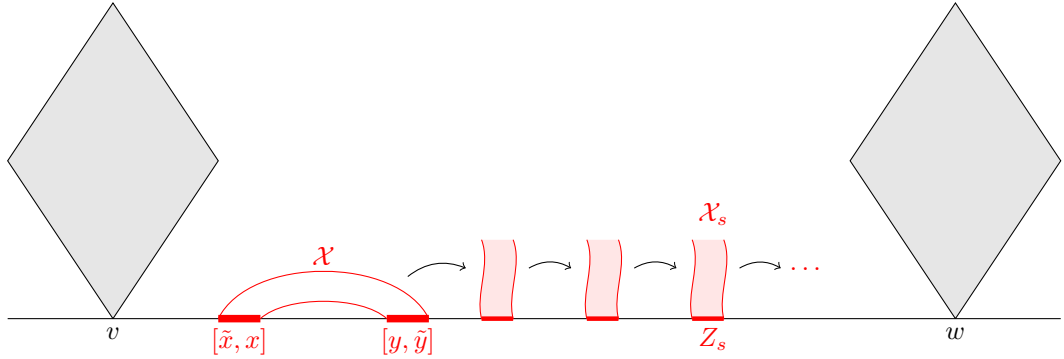

We will later need the following fact:
\begin{cor}
\label{cor:non-central disbalance}
For every $\lambda \gg 1$, the following holds. If $\dist_{T'}( y, w)\gg_\lambda \dist_{T'}(v, x)$, then there is an interval $J\subset [\widetilde y, w]$ (see Figure~\ref{Fg:lem:H_i G_i+1}) such that \[\Width^+_{\lambda, \div, m} (J)\succeq_\lambda \frac{\dist_{T'}(y,w)}{\dist_{T}(v,x )}.\]
\end{cor}
\begin{proof}
We will use the notations of Claim~\ref{cl8:piush forw}. Set $k$ to be the smallest integer bigger than $\dist_{T'}(y,\tilde y)/\length_{m+1}$.  Since $\dist_{T'}( y, w)\gg_\lambda \dist_{T}(v, x)$, we can find 
\[J\coloneqq \lfloor \Upsilon_s ,\Upsilon_{s+kj}\rfloor \supset \Upsilon_s\sqcup \Upsilon_{s+k}\sqcup \dots \sqcup \Upsilon_{s+jk}\sp\sp\sp j\asymp _\lambda \frac{\dist_{T'}(y,w)}{\dist_{T}(v,x )}\]
so that $\lambda J\subset [x,w]$. By~\eqref{eq:cl8:piush forw} and the Parallel Law, we have $\Width^+_{\lambda, \div, m} (J)\ge 90 j.$
\end{proof}

Let us prove Item {\bf \ref{thm:par fjords:part1}}. By Lemmas~\ref{lem:contr:part^h,1} and~\ref{lem:balance rectangle}, $\RR$ contains a balanced non-winding geodesic subrectangle $\RR_1$ with 
\[\Width(\RR_1)=\Width(\RR)-O(1)\sp\sp\text{ and }\sp\sp \partial ^{h,0}\XX<\partial^{h,0}\RR_1.\] 
Consider $J\coloneqq \lfloor \partial \RR_1\rfloor \setminus \partial ^h \RR_1$. Using Claim~\ref{cl8:piush forw} and its notations, $J$ contains neither $\Upsilon_0$ nor $\Upsilon_s$ for $s$ satisfying Claim~\ref{cl8:piush forw}. We deduce that $J\subset \lfloor\partial ^h \XX\rfloor$. By removing a $2$-buffer, we obtain that the new rectangle $\RR_1^\new$ is disjoint from $\XX$. 

It follows from Claim~\ref{cl:1:thm:par fjords} that $\RR_1^\new$ contains a balanced geodesic subrectangle $\RR_2$ such that $\Width(\RR_2)\ge \Width(\RR_1^\new)-O(1)$ and
\[ \dist(x,\partial^{h,0}\RR_2) \asymp \dist(\partial^{h,1}\RR_2,y) \preceq |\partial^{h,0}\RR_2| \asymp |\partial^{h,1}\RR_2| \]
 Item {\bf \ref{thm:par fjords:part2}} is now applicable for $\RR_2$.
\end{proof}

\subsection{Exponential Boost to $\partial^h \RR$}\label{ss:ExpBoost} Let us underline the following fact, see Figure~\ref{fig:fjords:submerg}. Suppose that we have a wide parabolic rectangle $\NN,$ $\Width(\NN)\asymp K\gg 1$ based on $T_\parab=[x,y]$. We note that $\NN$ is non-winding after removing $O(1)$-buffer. Assume that a parabolic rectangle $\RR$ with $\Width(\RR)\succeq 1$ is protected by $\NN$; i.e.~$\RR$ is based on the interval between $\partial^{h,0}\NN$ and $\partial^{h,1}\NN$. Then by Theorem~\ref{thm:par fjords}
\begin{equation}
\label{eq:subm rule} \log \frac{|\partial ^{h,0}\RR|}{\dist(v,\partial ^{h,0}\RR)} ,\sp \log \frac{|\partial ^{h,1}\RR|}{\dist(v,\partial ^{h,0}\RR)} \succeq K. 
\end{equation}

\begin{figure}
\begin{tikzpicture}[line width=0.4mm]

\draw (-0.2,0) -- (1,0)--(1,-6)--(2,-6)--(2,0)--(8.2,0);

\filldraw[blue] (0,0) circle (0.05 cm);
\filldraw[blue] (2,0) circle (0.05 cm);
\filldraw[blue] (1,0) circle (0.05 cm);
\filldraw[blue] (8,0) circle (0.05 cm);

\node[blue,above] at (0,0) {$v$};
\node[blue,above] at (1,0) {$x$};
\node[blue,above] at (2,0) {$y$};
\node[blue,above] at (8,0) {$w$};

\draw[red] (1,-0.5) --(2,-0.5);
\draw[red] (1,-3) --(2,-3);

\draw[opacity=0, fill=red, fill opacity=0.1] (1,-0.5) --(2,-0.5)--(2,-3)--(1,-3);

\node[right,red] at (2,-1.75) {$\NN,\sp \Width(\NN)\asymp K$};
\draw[red] (1,-0.5) --(2,-0.5);
\draw[red] (1,-3) --(2,-3);

\draw[blue] (1,-3.5) --(2,-3.5);
\draw[blue] (1,-4.5) --(2,-4.5);
\draw[opacity=0, fill=blue, fill opacity=0.1] (1,-3.5) --(2,-3.5)--(2,-4.5)--(1,-4.5);

\filldraw[blue] (1,-3.5) circle (0.05 cm);
\filldraw[blue] (1,-4.5) circle (0.05 cm);
\node[blue,left] at (1,-3.5) {$a$};
\node[blue,left] at (1,-4.5) {$b$};
\node[right,blue] at (2,-4) {$\RR$};
\node[right] at (2.3,-4.7) {if $\Width(\RR)\succeq 1$, then $\log \frac{|\partial^{h,0}\RR |}{\dist (v,a)}=\log \frac{\dist(a,b)}{\dist (v,a)}\succeq K$};

\draw[orange] (0.2,0) edge[bend left=50] (2.8,0);
\draw[orange] (1,0) edge[bend left=10] (2,0);
\node[orange,above] at (1.5,0.6) {$\XX$};


\end{tikzpicture}

\caption{Exponential Boost to $|\partial^h \RR|$: if a rectangle $\RR$ is protected by a wide buffer $\NN$, then $|\partial ^{h, i} \RR|$ is exponentially big,~\eqref{eq:subm rule}.}
\label{fig:fjords:submerg}
\end{figure}

\subsection{Central rectangles} We say that a parabolic rectangle $\RR$ based on $T'$ is \emph{central} if 
\[ 0.9< \frac{\dist_T(v, \lfloor \partial^h \RR '\rfloor)}{\dist_T( \lfloor \partial^h \RR '\rfloor,w)}<1.1;\]
i.e.~if the distances from $\partial^h \RR $ to $v$ and $w$ are essentially the same.

\begin{lem}[Central subrectangles]
\label{lem:central cond} Consider a parabolic non-winding rectangle $\RR$ based on $T'$ with $\Width(\RR)\gg_\lambda 1$. Then 
\begin{itemize}
\item either $\RR$ contains a parabolic non-winding central balanced geodesic subrectangle $\RR^\new$ with $\Width(\RR^\new) \ge \Width(\RR)/2$;  
\item or there is an interval 
\begin{equation}
\label{eq:lem:central cond}
I\subset T',\sp\sp   |I|> \length_{m+1} \sp\sp \text{such that}\sp\sp 
\log \Width^{+}_{\lambda,\div, m}(I)\succeq_\lambda \Width(\RR).
\end{equation}
 \end{itemize}
\end{lem}
\begin{proof}
Write $K\coloneqq \Width(\RR)\gg_\lambda 1.$ Let $\RR^\new$ be the rectangle obtained from $\RR$ by removing the outermost $K/3$ buffer $\NN$. By~\eqref{eq:subm rule}, we have
\begin{equation}
\label{eq:1:lem:central cond}
 \log\frac{|\partial^{h,0}\NN|}{\dist_{T'} (v,\partial^{h,0}\NN)}, \sp \log\frac{|\partial^{h,1}\NN|}
{\dist_{T'} (v,\partial^{h,0}\NN)} \succeq K.
\end{equation}
Since $\Width(\RR^\new)=2K/3$, using Item {\bf \ref{thm:par fjords:part2}} of Theorem~\ref{thm:par fjords}, we can select intervals $I\subset \partial^{h,0} \RR^\new$ and  $J\subset \partial^{h,1} \RR^\new$ such that the geodesic rectangle $\RR^\New\coloneqq \RR(I,J)\subset \wC\setminus  Z$ between $I,J$ is in $\RR^\new$ and satisfies:
\[|I|=|J|,\sp\sp \dist_{T'}(x, I)=\dist_{T'}(J,y), \sp\sp\text{ and }\sp\sp \Width (\RR^\New )\ge K/2.\]

Assume that $\RR^\New$ is not central. Since $ \dist_{T'}(y ,J) =\dist_{T'}(x ,I)$, we have
\[\dist_{T'}(y,w)= \dist_{T'}(J,w) - \dist_{T'}(x ,I) >0.1 \dist_{T'}(J,y)\succeq |\partial^{h,0} \NN|.\]
Using~\eqref{eq:1:lem:central cond}, we have:  
\[ \log \frac{\dist_{T'}(y,w)}{\dist(v,x)} \succeq \log \frac{|\partial^{h,0}\NN|}{\dist(v,x)}  \succeq K.\]
Corollary~\ref{cor:non-central disbalance} now implies the existence of a required interval $I$ with $\log \Width^{+}_{\lambda,\div, m}(I)\succeq_\lambda \Width(\RR)$.
\end{proof}

\part{Pseudo-Siegel disks and Snakes}
\label{part:psSiegDisk}

\section{Pseudo-Siegel disks} 
\label{s:wZ^m}

A pseudo-Siegel disk $\widehat Z^m$ is obtained from $\overline Z$ by filling-in  deep parts of parabolic fjords of levels $\ge m$. We will show in Theorem~\ref{mainthm:v3} that $\widehat Z^{-1}$ can be constructed to be a  uniform quasidisk. Consider a sufficiently small $\bdelta>0$.  

\begin{defn}
\label{dfn:PseudoSiegDisk}
A $\bdelta$-pseudo-Siegel disk $\widehat Z^m$ of level $m$ is a disk inductively constructed as follows:

\begin{itemize}
\item $\widehat Z^n=\overline Z$ for $n\gg 0$,
\item either $\widehat Z^m\coloneqq \widehat Z^{m+1}$, 
\item or $\widehat Z^m\coloneqq \widehat Z^{m+1}\cup Z^m$, where $Z^m$ is a $\bdelta/2$-near rotation domain (see~\S\ref{s:NearRotatSystem}), called the \emph{core} of $\widehat Z^m$, satisfying the compatibility conditions with $\widehat Z^{m+1}$ stated in \S\ref{ss:compat_wZ_Z}.
\end{itemize}
\end{defn}

If $\widehat Z^m\not= \widehat Z^{m+1}$, then we call $\widehat Z^m\coloneqq \widehat Z^{m+1}\cup Z^m$ a \emph{regularization} of $\wZ^{m+1}$ at \emph{level $m$.} Given $\wZ^m$, all its levels of regularization $m_i$ are enumerated as \[\dots>m_{i+1}>m_i>m_{i-1}>\dots \ge m.\] 
We say that $m_{i+1}$ is the level \emph{before} $m_i$ while $m_{i-1}$ is the level \emph{after} $m_i$. 


\subsubsection{Outline and Motivation}
\label{sss:outline:s:Trad W to W+} Pseudo-Siegel disks are inductively constructed as \[\wZ^m= \wZ^{m+1} \cup Z^m=\overline Z\cup \dots\cup Z^{m_i} \cup  Z^{m_{i-1}} \cup \dots\cup Z^{m_k} \cup Z^m, 
\] where $Z^{m}$ is obtained from $\wZ^{m+1}=\wZ^{m_k}$ by smoothing its boundary on level $m$. The boundary $\partial Z^m$ is a cyclic concatenation $\alpha_0\#\beta_0\#\alpha_1\#\beta_1\#\dots$, where $\alpha_i$ are ``channels'' through peninsulas of $\wZ^{m+1}$ and $\beta_i$ are ``dams'' in parabolic fjords of $\wZ^{m+1}$, see Figure~\ref{Fg:wZ_Z_m}. We require that there is a system of annuli around $\alpha_i, \beta_i$ making $Z^m$ a near-rotation domain~\S\ref{s:NearRotatSystem}.

We will require in Assumption~\ref{ass:wZ:bDelta} that dams are sufficiently deep in fjords so that the outer geometries of $Z$ and $\wZ^m$ are close: if the endpoints of intervals $I,J\subset \partial Z$ are in upper parts peninsulas, then $\Width_Z^+(I,J)=(1\pm\varepsilon)\Width^+_{\wZ^m}(I^m,J^m)$, where $I^m, J^m$ are the ``projections'' of $I,J$ onto $\wZ^m$, see details in \S\ref{ss:WellGrounded}. Here $\varepsilon$ is uniformly small independently of the number of regularizations.

We define the combinatorial distance on $\partial \wZ^m$ to be induced from $\partial Z$: 
\begin{equation}
\label{eq:dist on wZ^m}
\dist_{\partial \wZ^m} (x,y)\coloneqq \dist_{\partial Z} (x,y) \sp\sp\text{ for }\sp\sp x,y\in \partial \wZ^m \cap \partial Z.
\end{equation} With respect to this metric, the inner geometry of $\wZ^m$ has a description similar to $\overline Z$ -- see estimates in Theorem~\ref{thm:wZ:shallow scale}. The estimates depend on $\bdelta$; however on scale $\gg_\bdelta \length_{m}$, the estimates are uniform. Lemma~\ref{lmm:gound inter:penins} relates the inner geometry of peninsulas of $\wZ^m$ with the inner geometry of $\wZ^{m+1}$. As a consequence, Localization and Squeezing Lemmas~\ref{lem:trading  width to space}, \ref{lem:squeezing} hold for $\wZ^m$ -- compare with~\S\ref{sss:LocProp}, \S\ref{sss:SqueezProp}. The constant $\bdelta>0$ will be fixed in Section~\ref{s:Welding} so that the regularization $\wZ^{m+1}\leadsto \wZ^m$ can be iterated if $Z$ has deep parabolic fjords of level $m$.

\subsubsection{Regular intervals}
\label{sss:reg interv}
A \emph{regular} point of $\partial \widehat Z^m$ is a point in $\partial \widehat Z^m\cap \partial Z$. A \emph{regular} interval $I\subset \partial \widehat Z^m$ is an interval with regular endpoints. An interval $I\subset \partial Z$ is regular rel $\wZ^m$ if the endpoints of $I$ are in $\partial \wZ^m\cap \partial Z$.

The \emph{projection} of a regular interval $I\subset \partial \widehat Z^m$ onto $\partial Z$ is the interval $I^\bullet\subset \partial Z$ with the same endpoints and the same orientation as $I$.  All regular points of $I$ are in $I^\bullet$. We define the combinatorial length of $I$ by $|I|\coloneqq  |I^\bullet|$. Similarly is defined the \emph{projection} $I^k$ of a regular interval $I\subset\partial \wZ^m$ onto $\partial \widehat Z^k$ for $k>m$. 

For an interval $I\subset \partial Z$, the \emph{projection} $I^m$ onto $\partial Z^m$ is the shortest regular interval whose projection onto $\partial Z$ contains $I$. Similarly is defined the projection of an interval $I\subset \partial \widehat Z^m$ onto $\partial \widehat Z^n$ for $n<m$.  An interval $I\subset \partial Z$ is regular rel $\wZ^m$ if and only if $I= (I^m)^\bullet$

As for $\partial Z$, given $I, J \subset \partial  \wZ^m$, we set $\lfloor I, J \rfloor\coloneqq I\cup L\cup J$ , where $L\subset \partial \wZ^m$ is the complementary interval between $I$ and $J$ so that $I,L,J$ are clockwise oriented.


\subsection{Compatibility between $\widehat Z^{m+1}$ and $Z^m$}
\label{ss:compat_wZ_Z} 
In this subsection, we inductively define $\wZ^m=\wZ^{m+1}\cup Z^m$ completing Definition~\ref{dfn:PseudoSiegDisk}.

We say $\gamma$ is an \emph{external arc} of a closed topological disk $D$ if $\gamma$ is a simple arc in $\widehat \C\setminus \intr D$ such that $\gamma\cap \partial D$ consists of two endpoints of $\gamma$. Similarly, an \emph{internal arc} of $D$ is a simple arc $\ell\subset  D$ such that $\ell\cap \partial D$ consists of two endpoints of $\ell$.


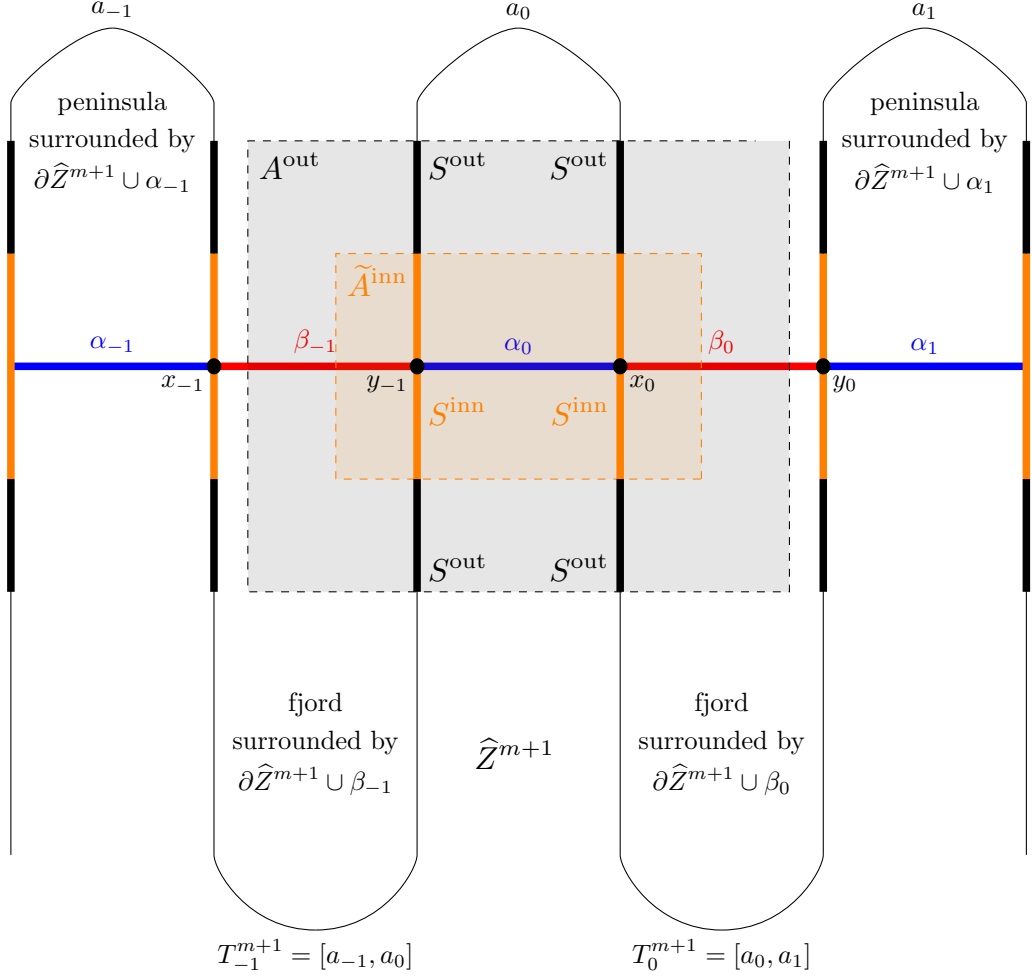
\begin{figure}[t!]
\[\begin{tikzpicture}

\begin{scope}[xscale =0.9]

\node[above] at  (0,12) {$a_0$};
\node[above] at  (6,12) {$a_1$};
\node[above] at  (-6,12) {$a_{-1}$};

\node[below] at  (3,0) {$T^{m+1}_0=[a_0,a_1]$};
\node[below] at  (-3,0) {$T^{m+1}_{-1}=[a_{-1},a_0]$};

\node[] at  (-3,3) {fjord};
\node[] at  (-3,2.5) {surrounded by};
\node[] at  (-3,2) {$\partial \wZ^{m+1}\cup \beta_{-1}$};

\node[] at  (3,3) {fjord};
\node[] at  (3,2.5) {surrounded by};
\node[] at  (3,2) {$\partial \wZ^{m+1}\cup \beta_{0}$};

\node[] at  (6,11) {peninsula};
\node[] at  (6,10.5) {surrounded by};
\node[] at  (6,10) { $\partial \wZ^{m+1}\cup \alpha_{1}$};

\node[] at  (-6,11) {peninsula};
\node[] at  (-6,10.5) {surrounded by};
\node[] at  (-6,10) { $\partial \wZ^{m+1}\cup \alpha_{-1}$};

\begin{scope}
\draw[shift={(0,1)}] (-3,-1) 
 .. controls (-2,-1) and (-1.5, -1+0.8) .. 
 (-1.5,0);
  
\draw (-1.5,1)  .. controls (-1.5,1) and (-1.5, 3) .. 
 (-1.5,11) 
  .. controls (-1.5,11.2) and (-0.5, 12) .. 
 (0,12);

\draw[xscale=-1,shift={(0,1)}] (-3,-1) 
 .. controls (-2,-1) and (-1.5, -1+0.8) .. 
 (-1.5,0);
  
\draw [xscale=-1](-1.5,1)
  .. controls (-1.5,1) and (-1.5, 3) .. 
 (-1.5,11) 
  .. controls (-1.5,11.2) and (-0.5, 12) .. 
 (0,12);
 \end{scope}
 
 \begin{scope}[shift={(6,0)}]
\draw[shift={(0,1)}] (-3,-1) 
 .. controls (-2,-1) and (-1.5, -1+0.8) .. 
 (-1.5,0);
  
\draw (-1.5,1)  .. controls (-1.5,1) and (-1.5, 3) .. 
 (-1.5,11) 
  .. controls (-1.5,11.2) and (-0.5, 12) .. 
 (0,12);

\draw [xscale=-1](-1.5,1)
  .. controls (-1.5,1) and (-1.5, 3) .. 
 (-1.5,11) 
  .. controls (-1.5,11.2) and (-0.5, 12) .. 
 (0,12);
 \end{scope}

\begin{scope}[shift={(-6,0)}]

\draw (-1.5,1)  .. controls (-1.5,1) and (-1.5, 3) .. 
 (-1.5,11) 
  .. controls (-1.5,11.2) and (-0.5, 12) .. 
 (0,12);

\draw[xscale=-1, shift={(0,1)}] (-3,-1) 
 .. controls (-2,-1) and (-1.5, -1+0.8) .. 
 (-1.5,0);
  
\draw [xscale=-1](-1.5,1)
  .. controls (-1.5,1) and (-1.5, 3) .. 
 (-1.5,11) 
  .. controls (-1.5,11.2) and (-0.5, 12) .. 
 (0,12);
 \end{scope}


 \draw[blue, line width=1mm ] (1.5,7.5) edge node[above]{$\alpha_0$} (-1.5,7.5);
 \draw[red, line width=1mm ] (1.5,7.5) edge node[above]{$\beta_0$} (4.5,7.5);
 \draw[blue, line width=1mm ] (7.5,7.5) edge node[above]{$\alpha_1$} (4.5,7.5);
  \draw[red, line width=1mm ] (-1.5,7.5) edge node[above]{$\beta_{-1}$} (-4.5,7.5);
 \draw[blue, line width=1mm ] (-7.5,7.5) edge node[above]{$\alpha_{-1}$} (-4.5,7.5);

\draw[dashed,fill=black, fill opacity=0.1 ] (4,10.5)-- (4,4.5)--(-4,4.5)--(-4,10.5)--(3.5,10.5);

\draw[orange, dashed,fill=orange, fill opacity=0.1 ] (2.7,6)-- (2.7,9)--(-2.7,9)--(-2.7,6)--(2.7,6);

 \node[below right  ,scale=1.2] at  (-4,10.5){$A^\out$}; 

 \node[below right, orange,scale=1.2] at (-2.7,9){$\widetilde A^\inn$};


\begin{scope}[shift={(-1.3,0)}]

  \node[above left  ,scale=1.2] at (2,2){$\wZ^{m+1}$}; 
\end{scope}

\begin{scope}[shift={(-9,0)}]
 \draw[orange, line width=1mm ] (1.5,9)--(1.5,6);

  \draw[ line width=1mm ] (1.5,9)--(1.5,10.5);
  \draw[ line width=1mm ] (1.5,6)--(1.5,4.5);

\end{scope}

\begin{scope}[shift={(-6,0)}]
 \draw[orange, line width=1mm ] (1.5,9)--(1.5,6);

  \draw[ line width=1mm ] (1.5,9)--(1.5,10.5);
  \draw[ line width=1mm ] (1.5,6)--(1.5,4.5);

\end{scope}

\begin{scope}[shift={(-3,0)}]
 \draw[orange, line width=1mm ] (1.5,9)--(1.5,6);
 \node[right, orange,scale=1.2] at (1.5,6.9){$S^\inn$}; 

  \draw[ line width=1mm ] (1.5,9)--(1.5,10.5);
  \draw[ line width=1mm ] (1.5,6)--(1.5,4.5);

 \node[above right  ,scale=1.2] at (1.5,4.5){$S^\out$}; 
 \node[below right  ,scale=1.2] at (1.5,10.5){$S^\out$}; 
\end{scope}
 
\begin{scope}[shift={(-0,0)}]
 \draw[orange, line width=1mm ] (1.5,9)--(1.5,6);
 \node[left, orange,scale=1.2] at (1.5,6.9){$S^\inn$}; 

  \draw[ line width=1mm ] (1.5,9)--(1.5,10.5);
  \draw[ line width=1mm ] (1.5,6)--(1.5,4.5);

 \node[above left  ,scale=1.2] at (1.5,4.5){$S^\out$}; 
\node[below left  ,scale=1.2] at (1.5,10.5){$S^\out$}; 
\end{scope}

\begin{scope}[shift={(3,0)}]
 \draw[orange, line width=1mm ] (1.5,9)--(1.5,6);

  \draw[ line width=1mm ] (1.5,9)--(1.5,10.5);
  \draw[ line width=1mm ] (1.5,6)--(1.5,4.5);

\end{scope}

\begin{scope}[shift={(6,0)}]
 \draw[orange, line width=1mm ] (1.5,9)--(1.5,6);

  \draw[ line width=1mm ] (1.5,9)--(1.5,10.5);
  \draw[ line width=1mm ] (1.5,6)--(1.5,4.5);


\end{scope}

\filldraw (-1.5,7.5)  circle (1mm) ;
 \node[ below left] at (-1.5,7.5) {$y_{-1}$};
\filldraw (-4.5,7.5)  circle (1mm) ;
 \node[ below left] at (-4.5,7.5) {$x_{-1}$};
\filldraw (1.5,7.5)  circle (1mm) ;
 \node[ below right] at (1.5,7.5) {$x_{0}$};
\filldraw (4.5,7.5)  circle (1mm) ;
 \node[ below right] at (4.5,7.5) {$y_{0}$};

  \end{scope}
\end{tikzpicture}\]
\caption{The channels $\alpha_i$, the dams $\beta_i$, the collars $A^{\inn}(\alpha_0),A^{\out}(\alpha_0)$, and buffers $S^\inn, S^\out$ of a pseudo-Siegel disk $\wZ^{m+1}$. (The inner boundary of $A^\inn$ is omitted.)}
\label{Fg:wZ_Z_m}
\end{figure}

\subsubsection{Channels and Dams}

Recall from~\S\ref{sss:diff tilings} that $\Dbb_m$ denotes the diffeo-tiling of level $m\ge -1$. Let us  enumerate intervals in $\Dbb_m$ clockwise as $T_i=[a_i, a_{i+1}]$; i.e. \[\partial Z= T_0\#T_1\#\dots\# T_{\qq_{m+1}-1}\] is the level $m$ tessellation of $\partial Z$ into diffeo-intervals. Then $f^{\qq_{m+1}}$ maps the $T_i$ almost into $T_{i-\pp_{m+1}}$. We also recall that $T'_i= T_i\cap f^{\qq_{m+1}}(T_i)$ (with a slight adjustment for $m=-1$). Let us denote by $T'^{m+1}_i, T^{m+1}_i$ the projections of $T'_i,T_i$ onto $\widehat Z^{m+1}$. By Assumptions~\ref{ass:wZ:2}, $T_i, T'_i$ are regular rel $\wZ^{n},\sp n\ge m $:  

\begin{assum}[Channels and dams]
\label{ass:wZ:2}
The clockwise tessellation~\eqref{eq:part U:unit intervals} of $\partial Z^m$ into unit intervals
\[\partial Z^m=L_0\#L_1\#\dots\# L_{\qq_m-1}\] satisfies $L_i=\alpha_i\#\beta_i$, where (see Figure~\ref{Fg:wZ_Z_m}): 
\begin{itemize}
\item $\alpha_i=\alpha_i^m$ is an internal arc of $\widehat Z^{m+1}$ connecting \[y_{i-1}\in T'^{m+1}_{i-1} \sp\sp \text{ and }\sp\sp x_{i}\in  T'^{m+1}_{i},\]  
\item  $\beta_i=\beta_i^m$ is an external arc of $\widehat Z^{m+1}$ connecting $x_{i},y_{i}\in T'^{m+1}_{i}$,
\item $x_i$ is on the left of $y_i$ in $T'^{m+1}_i$.
\end{itemize}
Moreover, $x_i,y_i \in \CP_{m+1}\setminus \CP_m$.

We say that $\alpha_i$ is a \emph{level $m$ channel} and  $\beta_i$ is a \emph{level $m$ dam}.\tqed
\end{assum}

We will require in Assumption~\ref{ass:wZ:Linking} that the $\alpha^m_i$, $\beta^m_j$ are pairwise disjoint except possibly at endpoints and that the $\alpha^m_i$ are disjoint from the $\alpha^n_j$ for all $n>m$. Components of $\wZ^m\setminus \wZ^{m+1}=Z^m\setminus\wZ^{m+1} $ and of $\wZ^{m}\setminus Z^m=\wZ^{m+1}\setminus Z^m$ will be called fjords and peninsulas, see~\S\ref{sss:Fjord Penin}.

\subsubsection{Collars and S-buffers} 
\begin{assum}[Collars]
\label{ass:wZ:collars}
There are closed collars around $\alpha_i$ and $\beta_i$  \[A(\alpha_i)=A^\inn (\alpha_i)\cup A^\out(\alpha_i)\hspace{0.5cm} \text{ and }\hspace{0.5cm} A(\beta_i)=A^\inn(\beta_i) \cup A^\out(\beta_i),\]
\[ \partial ^\inn A^\out (\alpha_i) = \partial ^\out A^\inn (\alpha_i), \sp\sp\sp\sp\sp\sp\sp \partial ^\inn A^\out (\beta_i) = \partial ^\out A^\inn (\beta_i) \]
with 
\[  \mod A^\inn (\alpha_i),\sp  \mod A^\out (\alpha_i),\sp \mod A^\inn (\beta_i),\sp  \mod A^\out (\beta_i)\sp \ge \bdelta \]
such that for all $i$ and all $k\in \{0,+1,\dots, \qq_{m+1}\}$ we have
\begin{itemize}
\item $A(\alpha_{i - k \pp_{m+1}})$ encloses $f^{k} (\alpha_i)$;
\item $A(\alpha_{i + k \pp_{m+1}})$ encloses the unique $f^k$-lift of $\alpha_i$ starting and ending at $\partial Z$;
\item $A(\beta_{i - k \pp_{m+1}})$ encloses $f^{k} (\beta_i)$;
\item $A(\beta_{i +k \pp_{m+1}})$ encloses the unique $f^k$-lift of $\beta_i$ starting and ending at $\partial Z$.
\end{itemize}
\end{assum}

In other words, $A(\alpha_i)$ and $A(\beta_i)$ control the difference between $Z^m$ and its image under $f^k$ for $|k|\le \qq_{m+1}$. The inner colors $A^\inn$ will be used later in this section to describe the inner geometry of $\wZ^m$. The outer colors $A^\out$ will be used to ``tame snakes,'' see Lemmas~\ref{lem:A|_L},~\ref{lem:I vs I^grnd}.

\begin{assum}[Intersection Pattern]
\label{ass:wZ:int patt}
For all $\alpha^m_i, \beta^m_i$, the simple closed curves
\[ \partial^\out A^\out(\alpha^m_i),\sp\sp \partial^\inn A^\out(\alpha^m_i)=\partial^{out}A^\inn(\alpha^m_i),\sp\sp  \partial^\inn A^\inn(\alpha^m_i), \]
\[ \partial^\out A^\out(\beta^m_i),\sp\sp \partial^\inn A^\out(\beta^m_i)=\partial^{out}A^\inn(\beta^m_i),\sp\sp  \partial^\inn A^\inn(\beta^m_i) \]
intersect $\partial \wZ^m$ at exactly $4$ points and these intersection points are in $\CP_{m+1}\setminus \CP_m$.

Moreover, all $12$ intersection points in \[ P(\beta^m_i)\coloneqq \partial \wZ^m \cap  \left(\partial^\out A^\out(\beta^m_i)\cup \partial^\inn A^\out(\beta^m_i) \cup  \partial^\inn A^\inn(\beta^m_i )\right) \]
are within $T_i$. The $6$ most left points of $P(\beta^m_i)\cap T_i$ are within 
\[ P(\alpha^m_{i})\coloneqq \partial \wZ^m \cap  \left(\partial^\out A^\out(\alpha^m_{i})\cup \partial^\inn A^\out(\alpha^m_{i}) \cup  \partial^\inn A^\inn(\alpha^m_{i} )\right) ,\]
the $6$ most right points of $P(\beta^m_i)\cap T_i$ are within $P(\alpha^m_{i+1})$.
\end{assum}

Let us denote by $\widetilde A(\alpha^m_i), \widetilde A^\inn(\alpha^m_i),\widetilde A(\beta^m_i), \widetilde A^\inn(\beta^m_i)$ the disks obtained by filling-in $A(\alpha^m_i),$ $A^\inn(\alpha^m_i), $ $A(\beta^m_i), $ $A^\inn(\beta^m_i)$. It follows from the Assumption~\ref{ass:wZ:int patt} that $\widetilde A, \widetilde A^\inn$ have the following intersection properties with $\partial \widehat Z^{m+1}$:
\[\widetilde A(\alpha_i)\cap \partial \widehat Z^{m+1} = S_{y_{i-1}} \cup S_{x_{i}},\sp\sp\sp\sp\sp\widetilde A(\beta_i)\cap \partial \widehat Z^{m+1} = S_{x_{i}} \cup S_{y_{i}},\]
\[\widetilde A^\inn(\alpha_i)\cap \partial \widehat Z^{m+1} = S^\inn_{y_{i-1}} \cup S^\inn_{x_{i}},\sp\sp\sp\widetilde A^\inn(\beta_i)\cap \partial \widehat Z^{m+1} = S^\inn_{x_{i}} \cup S^\inn_{y_{i}},\]
where
\begin{itemize}
\item $S^\inn_{x_i}\subset S_{x_i}$ are sub-intervals of $T'^{m+1}_i$ containing $x_i$,
\item $S^\inn_{y_i}\subset S_{y_i}$ are sub-intervals of $T'^{m+1}_i$ containing $y_i$,
\item all $S_{x_i},S_{y_i}$ are pairwise disjoint.\tqed
\end{itemize}

We say that $S_{x_i}$, $S_{y_i}$ are \emph{$S$-buffers of level $m$} and we say that $S^\inn_{x_i}$, $S^\inn_{y_i}$ are \emph{$S^\inn$-buffers of level $m$}. Note that $\partial \widehat Z^{m+1}$ may also contains many $S$- and $S^\inn$-buffers of deeper levels We also write:
\begin{equation}
\label{eq:dfn S^inn beta}
S^\inn(\beta_i^m)\coloneqq \widetilde A^\inn (\beta^m_i)\cap \partial \wZ^{m}=\left(S^\inn_{x_{i}}\cup \beta_i^m \cup  S^\inn_{y_{i}}\right)\setminus \intr(\wZ^m), 
\end{equation}
\begin{equation}
\label{eq:dfn S^inn alpha}
S^\inn(\alpha_i^m)\coloneqq \widetilde A^\inn (\alpha^m_i)\cap \partial \wZ^{m}=\left(S^\inn_{y_{i-1}}\cup   S^\inn_{x_{i}}\right)\setminus \intr(\wZ^m), 
\end{equation}
and similar with $S(\beta_i^m), S(\alpha_i^m)$.

\begin{lem}
The disk $Z^m$ (see Assumption~\ref{ass:wZ:2}) is a $\bdelta/2$-near rotation domain (see Section~\ref{s:NearRotatSystem}) with respect to \[A_i=A (L_i)\coloneqq A^\inn(\alpha_i) \square A^\inn(\beta_i) ;\]
see~\S\ref{sss:EnclAnn} for the definition of ``$\square$''.
\end{lem}
\begin{proof}
By Assumption~\ref{ass:wZ:collars}, the annulus $A_i$ controls the difference between $f^k(L_i)$ and $L_{i-k\pp_{m+1}}.$ It follows from Assumptions \ref{ass:wZ:int patt} and~\ref{ass:wZ:Linking} that $A_i$ intersects only $A_{i-1}$ and $A_{i+1}$. Since $\mod (A^\inn(\alpha_i)), \mod (A^\inn(\beta_i))\ge \bdelta$, we have $\mod(A_i)\ge \bdelta/2$ by Lemma~\ref{lem:square}.
\end{proof}

\begin{assum}[Combinatorial space]
\label{ass:wZ:CombSpace}
Each of the $15$ intervals in $T_i\setminus \big(P(\beta^n_i)\cup \{x_i, y_i\}\big)$ 
has length at least $200 \length_{m+1}$.

 Moreover, the subinterval $[x_i,y_i]\subset T^{m+1}_i$ has length at least $\frac 4 5  |T^{m+1}_i|$.
 \tqed
\end{assum}

In particular, most of $T^{m+1}_i$ is ``reclaimed'' during the regularization.

\subsubsection{Extra geometry constrains}
\label{sss:Geom in Fjor}

\begin{assum}[Conformal separation]
\label{ass:wZ:bDelta}
For all $i$ we have
\[ \Width(S_{x_i}, S_{y_i}),\sp \Width(S_{y_i}, S_{x_{i+1}})\le 1/\bdelta.\]\tqed
\end{assum}

Recall from~\eqref{eq:dfn:based on} that a rectangle $\RR$ is based on $T'^{m+1}_i$ if $\RR\subset \wC\setminus \intr \wZ^{m+1}$ and $\partial^h \RR\subset T'^{m+1}_i$. A rectangle $\RR$ based on $T'^{m+1}_i$ is
\begin{itemize}
\item \emph{parabolic} if $\dist_{T'^{m+1}_i}(\partial^{h,0} \RR,\partial^{h,1} \RR) \ge 6\min \{ | \partial^{h,0}\RR|,\sp | \partial^{h,1}\RR| \} +3\length_{m+1};$
\item balanced if $|\partial^{h,0}\RR|=|\partial^{h,1}\RR|$;
\item \emph{non-winding} if, in addition, every vertical curve in $\RR$ is homotopic in ${\C\setminus \intr \wZ^m}$ to a subcurve of $T'^{m+1}_i$ 
\end{itemize}

Consider a sufficiently big $\bDelta \gg 1$ -- it will be fixed in \S\ref{ss:WellGrounded}, see Remark~\ref{rem:fix bDelta}.

\begin{assum}[Extra outer protection]
\label{ass:wZ:EtraProt}
For every dam $\beta^m_i$ there is a balanced parabolic non-winding rectangle $\XX^m_i=\XX(\beta^m_i)$ based on $T'^{m+1}_i$ such that $\Width(\XX^m_i)\ge \bDelta$ and $\widetilde A(\beta^m_i)\setminus \wZ^{m+1}$ is in the bounded component of $\C\setminus( \wZ^{m+1}\cup \XX^m_i)$. \tqed
\end{assum}

In particular, $\beta_i$ is deep in the fjord associated with $T_i$, see~\ref{sss:fjords}. We assume that $\partial^{h,0} \XX^m_i < \partial^{h,1} \XX^n_i$ in $T_i$; i.e.~$\lfloor \XX^n_i \rfloor\subset T^{m+1}_i$. We denote by $\lfloor \XX^n_i \rfloor^m$ the projection of $\lfloor \XX^n_i \rfloor$ onto $\partial \wZ^m$.

\subsubsection{Minimal position of the collars} A collection $\Gamma=(\gamma_k)$ of external arcs of a topological disk $D$ is in \emph{minimal position rel $\partial D$} if every two arcs $\gamma_k, \gamma_t$ have the minimal intersection number up to homotopy in $(\wC\setminus D, \partial D)$. This means that $|\gamma_k\cap \gamma_t|\le 1$ and  $|\gamma_k\cap \gamma_t|= 1$ if and only if the endpoints of $\gamma_k$ are linked rel $\partial D$ to the endpoints of $\gamma_t$. Similarly, the minimal position for internal arcs is defined.

By Assumption~\ref{ass:wZ:int patt}, every simple closed curve
\[ \partial^\out A^\out(\alpha^m_i),\sp\sp \partial^\inn A^\out(\alpha^m_i),\sp\sp  \partial^\inn A^\inn(\alpha^m_i), \]
\[ \partial^\out A^\out(\beta^m_i),\sp\sp \partial^\inn A^\out(\beta^m_i),\sp\sp  \partial^\inn A^\inn(\beta^m_i) \]
is a cyclic concatenation of $2$ external and $2$ internal arcs of $\wZ^m$. We denote by $\Gamma^m_\ext$ and $\Gamma^m_\inn$ the set of external and internal arcs of the above boundaries of collars.

 \begin{assum}
\label{ass:wZ:Linking}
The set $\Gamma^m_\inn\cup \{\alpha^m_i\}_i $ is in minimal position rel $\partial \wZ^m$. The set
 \[  \bigcup_{n\ge m} \Gamma^n_\ext\cup \{\beta^n_i\}_i\cup \{ \partial^{v, 0}\XX^n_i, \partial^{v,1}\XX^n_i \}_i\] 
 is in minimal position rel $\partial Z$.
 
 Moreover, all landing points of curves in $\Gamma^m_\ext\cup \Gamma^m_\inn \cup \{\alpha^m_i, \beta^m_i\}_i\cup \{ \partial^{v, 0}\XX^m_i, \partial^{v,1}\XX^m_i \}_i$ are within $\CP_{m+1}\setminus \CP_m$.
 \tqed
\end{assum}

In particular, the $\XX^n_i$ geometrically separate dams of all levels.

\subsubsection{Fjords and Peninsulas}\label{sss:Fjord Penin} A connected component of $\wZ^m\setminus \wZ^{m+1}=Z^m\setminus \wZ^{m+1}$ is called a \emph{level $m$ fjord} while a connected component of $\wZ^{m}\setminus Z^m=\wZ^{m+1}\setminus Z^m$ is called a \emph{level $m$ peninsula}. Every peninsula $\Penin$ has a unique channel $\alpha_i^m$ on its boundary; we will often write $\Penin=\Penin(\alpha^m_i)$. The \emph{coast} of $\Penin$ is \[\partial^c \Penin\coloneqq \overline{ \partial \Penin\setminus \alpha^m_i}.\] The boundary $\partial \wZ^m$ is a concatenation of dams $\beta^m_i$ and coasts of peninsulas.

 Consider a peninsula $\Penin(\alpha^n_j)$ of $\wZ^n$, where $n>m$. By construction (Assumption~\ref{ass:wZ:2}), $\Penin$ contains a unique point $z\in \CP_n$. There are three possibilities:
\begin{itemize}
 \item $\partial^c \Penin$ (as well as $\Penin$) is in the interior of $\wZ^m$;
 \item $\partial^c \Penin\subset \partial \wZ^m$;
 \item one of the components of $\partial^c \Penin\setminus \{z\}$ is in $\intr\wZ^m$ while the other component is in $\partial \wZ^m$.
\end{itemize}
  In the last case, $z$ is an endpoint of a dam $\beta^k_s\subset \partial \wZ^m$ of generation $k<n$ (by Assumption~\ref{ass:wZ:Linking}).

\subsubsection{$S^\inn (\wZ^m) \subset S (\wZ^m)\subset S^\well (\wZ^m)$} 
\label{sss:S^well}
 Let us write
\[ S^\inn(\wZ^m)\coloneqq \bigcup_{n,j} S^\inn (\beta^n_j), \sp S(\wZ^m)\coloneqq \bigcup_{n,j} S (\beta^n_j), \sp S^\well(\wZ^m)\coloneqq \bigcup_{n,j} \lfloor \XX^n_j\rfloor ,\]
where the unions are taken over all $n\ge m$ and $j$. 

By construction: $\CP_m \subset\partial \wZ^m\setminus  S^\well (\wZ^m)$ and, moreover:

\begin{lem}
\label{lem:S are disj}
Every connected component of $S^\well (\wZ^m)$ is $\lfloor \XX^n_j\rfloor $ for some $n\ge m, j$.

Similarly, every connected component of $S (\wZ^m)$ is $S(\beta^n_i) $ for some $\beta^n_i$.

Similarly, every connected component of $S^\inn (\wZ^m)$ is $S^\inn(\beta^n_i) $ for some $\beta^n_i$.\qed
\end{lem}

We say that an interval $I\subset \partial \wZ^m$ is \emph{well-grounded} if its endpoints are in ${\partial \wZ^m\setminus S^\well (\wZ^m)}$. An interval $I\subset \partial Z$ is \emph{well-grounded} rel $\wZ^m$ if $I$ is regular and its projection $I^m\subset \partial \wZ^m$ is well-grounded.

\subsubsection{$\UU\big(\wZ^m\big)\subset \XX\big(\wZ^m\big)$} \label{sss:UU and ZZ}Let us denote by $U(\alpha^n_i)\coloneqq \widetilde A(\alpha^n_i)\setminus A(\alpha^n_i)$ and $U(\beta^n_i)\coloneqq \widetilde A(\beta^n_i)\setminus A(\beta^n_i)$ the topological disks surrounded by $A(\alpha^n_i)$ and $A(\beta^n_i)$. Let us write
\[\UU\big(\wZ^m\big)\coloneqq \wZ^m\cup \bigcup_{n,i} \big( U(\alpha^n_i)\cup U(\beta^n_i)\big) =\wZ^m\cup \bigcup_{n,i}  U(\beta^n_i),\]
where the equality follows from Assumption~\ref{ass:wZ:int patt}. Since $A(\alpha^n_i),A(\beta^n_i)$ control the difference between $\wZ^m$ and the induced image under $f^k$, $|k|\le \qq_{m+1}$ (Assumption~\ref{ass:wZ:collars}), the map $(f\mid \overline Z)^k\colon \overline Z\selfmap$ extends uniquely to 
\begin{equation}
\label{eq:dfn UU(wZ)}
f^k\colon \ \wZ^m\ \overset{1:1}{\longrightarrow}\ (f^k)_* \big(\wZ^m\big) \ \subset \ \UU\big(\wZ^m\big),\sp\sp\sp\text{ where }\sp |k|\le \qq_{m+1}.
\end{equation}

Let us also set
\begin{equation}
\label{eq:dfn XX(wZ)}
\XX\big(\wZ^m\big)\coloneqq \operatorname{Filling-in~of }\left(\bigcup_{n,j}\XX^n_j \cup \wZ^m \right).
\end{equation} 

By construction, $\XX\big(\wZ^m\big)$ contains all $A(\alpha^n_i)$ and $A^(\beta^n_i)$.

\subsubsection{Pullbacks of $\wZ^m$} 
\label{sss:Pullbacks of wZ^m}Let $\stab(\wZ^m)\in \N_{\ge 0}$ be the smallest number such that the distance between $\partial^h \XX^{m_s}_j$ and the endpoints of $T'^{m_s}_j$ is at least $(\stab(\wZ^m)+1)\length_{m_s+1}$ for every level of regularization $m_s\ge m$ and every $j$. Then pullbacks of $\wZ^m$ are well defined up to $\stab(\wZ^m)\qq_{m+1}$ iterates:

\begin{lem}
\label{lem:Pullback of wZ^m}
For every $t\le \stab(\wZ^m) \qq_{m+1}$ all the $\alpha^n_i,\beta^n_i,A(\alpha^n_i),A(\beta^n_i), \XX(\beta^n_i)$ have univalent lifts along $f^t\colon \overline Z\to \overline Z$; the resulting lifts $\alpha^n_{i,-t},$ $\beta^n_{i,-t},$ $A(\alpha^n_{i,-t}),$ $A(\beta^n_{i,-t}), \XX(\beta^n_{i,-t})$ form a $(\bdelta,\bDelta)$ pseudo-Siegel disk $\wZ^m_{-t}$ such that $f^t\colon \wZ^m_{-t}\to \wZ^m$ is conformal.
\end{lem}
\begin{proof}
Since the $\XX^n_i$ are non-winding, they have univalent lifts for all $t\le  \stab(\wZ^m)  \qq_{n+1}$ -- compare with Lemma~\ref{lem:Rect:pullback}.  Since the $\XX^n_i$ separate $A(\beta^n_i)\setminus \wZ^{n+1}$ from $\CP_n$ (Assumption~\ref{ass:wZ:EtraProt}) and since $A(\alpha^n_i)\setminus \wZ^{n+1}\subset A(\beta^n_{i-1})\cup A(\beta^n_{i})$,  the filled-in collars $\widetilde A(\alpha^n_i),\widetilde A(\beta^n_i)$ also have univalent lifts and the statement follows.
\end{proof}

\subsubsection{Geodesic pseudo-Siegel disks}\label{sss:regul within rect}  
We say $\wZ^m$ is a \emph{geodesic pseudo-Siegel disk} if 
\begin{itemize}
\item $\bigcup_{n\ge m} \Gamma^n_\ext\cup \{\beta^n_i\}_i\cup \{ \partial^{v, 0}\XX^n_i, \partial^{v,1}\XX^n_i \}_i$ consists of hyperbolic geodesics of $\wC\setminus  \overline Z$;
\item $\Gamma^n_\inn\cup \{\alpha^m_i\}_i $ consists of hyperbolic geodesics of $\intr \wZ^{n+1}$ for every regularization level $n\ge m$; and
\item $\stab(\wZ^m)\ge 10$. 
\end{itemize}

 Consider a parabolic non-winding rectangle $\RR$ based on $T'_i\subset \partial Z$.  By Lemma~\ref{lem:Rect:pullback}, $\RR$ has a univalent pullback along $f^{\qq_{m+1}}\colon {T'\boxminus \theta_{m+1}}\to T$ for ${j\in \{0,1,\dots, \qq_{m+1}-1\}}$. Let $\RR_{-j}$ be the lift of $\RR$ along $f^j\colon \overline Z\selfmap$. We set 
 \[\orb_{-\qq_{m+1}+1} \RR \coloneqq  \bigcup _{j=0}^{\qq_{m+1}-1}  \RR_{-j},\]
i.e., $\orb_{-\qq_{m+1}+1} \RR$ is the set of rectangles obtained by spreading around $\RR$ using pullbacks. 
 
 We say that a regularization $\wZ^{m}=Z^{m+1}\cup \wZ^{m+1}$ is \emph{within $\orb_{-\qq_{m+1}+1} \RR$} if 
\[   \Gamma^m_\ext\cup \{\beta^m_i\}_i\cup \{ \partial^{v, 0}\XX^m_i, \partial^{v,1}\XX^m_i \}_i\subset \orb_{-\qq_{m+1}+1} \RR.\]

\begin{rem}
\label{rem:easy cond:well ground} If $\wZ^m$ is geodesic, then $\stab(\wZ^m)\ge 10$ implies that
\begin{equation}
\label{eq:rem:easy cond:well ground}
\left(\bigcup_{|i|\le 2\qq_{m+1}} [f\mid \partial Z]^i(\CP_m)\right)  \cap S^\well (\wZ^m)=\emptyset.
\end{equation} 
Therefore, if the endpoints of an interval $J\subset \partial Z$ are in $\CP_m$, then $f^i(J)$ is well-grounded rel $\wZ^m$ for all $i\le 2\qq_{m+1}$.

For instance, let $I\subset \partial Z$ be a combinatorial level-$m$ interval such that one of the endpoints of $I$ is in $\CP_m$. Let $I_s, \sp s<\qq_{m+1}$ be the intervals obtained by spreading around $I=I_0$, see~\S\ref{sss:spread around}. Then all $I_s$ are well-grounded rel $\wZ^m$.
\end{rem}

\subsection{Outer geometry of $\wZ^m$}
\label{ss:WellGrounded}
In this subsection we will show that $\overline Z$ and $\wZ^m$ have comparable outer geometries with respect to grounded intervals. 
\subsubsection{Removing small components from a rectangle}  Fix a big $\Delta\gg 1$ and some $\tau>1$. Consider a rectangle $\RR\subset \wC$. Let $D_i, i\in I$ be a finite set of open Jordan disks in $\RR$ with pairwise disjoint closures satisfying the following properties:
\begin{itemize}
\item[(A)] every $D_i$ is a attached to a unique side $T_i$ of $\RR$: the intersection $\partial D_i \cap \partial \RR$ is a closed arc within $T_i$;
\item[(B)] there is a rectangle \[\YY_i\subset \RR,\sp\sp \partial^h \YY_i =\partial \YY_i\cap \partial \RR \subset T_i,\sp\sp \Width(\YY_i)\ge \Delta \] protecting $D_i$ in the following sense: $\partial \RR\setminus T_i$ and $D_i$ are in different components of $\RR\setminus \YY_i$.
\end{itemize}

Let us denote by $T^\opp_i\subset \partial \RR$ the opposite side to $T_i$; i.e.~$T_i\cup T^\opp_i$ is either $\partial ^v\RR$ or $\partial ^h\RR$. We will also consider the following weakenings of (B):
\begin{itemize}
\item[(B${}_{v}$)] if $T_i$ is a vertical side of $\RR$, then there is a rectangle 
\[\YY_i\subset \RR,\sp\sp \partial^h \YY_i =\partial \YY_i\cap \partial \RR \subset T_i,\sp\sp \Width(\YY_i)\ge \tau \] 
protecting $D_i$ in the following sense: $ T^\opp_i$ and $D_i$ are in different components of $\RR\setminus \YY_i$.
\item[(B${}_{h}$)] if $T_i$ is a horizontal side of $\RR$, then there is a rectangle 
\[\YY_i\subset \RR,\sp\sp \partial^h \YY_i =\partial \YY_i\cap \partial \RR \subset \partial \RR \setminus T^\opp_i,\sp\sp \Width(\YY_i)\ge \Delta \] 
protecting $D_i$: $ T^\opp_i$ and $D_i$ are in different components of $\RR\setminus \YY_i$.
\item[(B$'{}_{h}$)] if $T_i$ is a horizontal side of $\RR$, then there is a rectangle 
\[\YY_i\subset \RR,\sp\sp \partial^h \YY_i =\partial \YY_i\cap \partial \RR \subset \partial \RR \setminus T^\opp_i,\sp\sp \Width(\YY_i)\ge \Delta_i \] 
protecting $D_i$: $ T^\opp_i$ and $D_i$ are in different components of $\RR\setminus \YY_i$, where:
\[\Delta_i\ge \tau \sp \text{ for } \sp i \in S\sp\sp\text{ and }\sp\sp \Delta_i\ge \Delta \sp \text{ for }\sp i \not\in S.\]
(Here $S$ is an index set.)
\end{itemize}

 Set \[ \RR'\coloneqq \overline{ \intr \RR\setminus  \bigcup_{i\in I} D_i }\]
and view $\RR'_i$ as a rectangle with the same vertices as $\RR$ and with the same labeling of sides; i.e.~$\partial^h \RR$ and $\partial ^h \RR'$ have an infinite intersection. In other words, $\RR'$ is obtained from $\RR$ by slightly moving its boundaries towards interior; the motion is geometrically controlled by rectangles $\YY_i$. If we view $\RR$ as an outer rectangle in $\wC$ (i.e.,~$\infty\in \intr \RR$), then $\RR'$ is obtained from $\RR$ by filling in fjords, see Figure~\ref{fig:R and R'}.

\begin{figure}
\begin{tikzpicture}[scale=0.7]
\begin{scope}[shift={(-0.5,-6.5)}]
 
 \node at (4.65,2.5) {$\wC \setminus \RR$};

\draw[opacity=0, fill, fill opacity=0.1] 
(0,0) --
(2,0) -- (2,2)--(2.3,2) --(2.3,0)-- 
(3,0) -- (3,2)--(3.3,2) --(3.3,0)-- 
(4,0) -- (4,2)--(4.3,2) --(4.3,0)--
(5,0) -- (5,2)--(5.3,2) --(5.3,0)--
(6,0) -- (6,2)--(6.3,2) --(6.3,0)--
(7,0) -- (7,2)--(7.3,2) --(7.3,0)--  
(9.3,0) -- 
(9.3,0.85)-- (8.3,0.85)--(8.3,1.15) -- (9.3,1.15)--
(9.3,1.85)-- (8.3,1.85)--(8.3,2.15) -- (9.3,2.15)--
(9.3,2.85)-- (8.3,2.85)--(8.3,3.15) -- (9.3,3.15)--
(9.3,3.85)-- (8.3,3.85)--(8.3,4.15) -- (9.3,4.15)--
(9.3,5) --
(7.3,5) -- (7.3,3)--(7,3) --(7,5)--
(6.3,5) -- (6.3,3)--(6,3) --(6,5)--
(5.3,5) -- (5.3,3)--(5,3) --(5,5)--
(4.3,5) -- (4.3,3)--(4,3) --(4,5)--
(3.3,5) -- (3.3,3)--(3,3) --(3,5)--
(2.3,5) -- (2.3,3)--(2,3) --(2,5)--
(0,5) --
(0,4.15)-- (1,4.15)--(1,3.85) -- (0,3.85)--
(0,3.15)-- (1,3.15)--(1,2.85) -- (0,2.85)--
(0,2.15)-- (1,2.15)--(1,1.85) -- (0,1.85)--
(0,1.15)-- (1,1.15)--(1,0.85) -- (0,0.85)--
(0,0);


\draw[line width =0.05 cm,orange] 
(2,0) -- (2.3,0)
(0.5,0) ..controls (1.8,-0.7) and (2.5,-0.7).. (3.8,0)
(0,3.15)--(0,2.85)
(0,4.4)..controls  (-0.5,3.5) and (-0.5,2.5) ..(0,1.6);
\node[left,orange] at (-0.6,3) {$\YY_i$};

\draw[opacity=0, fill=orange, fill opacity=0.2] 
(2.3,0) --(3,0) -- (3,2)--(3.3,2) --(3.3,0)--
(3.3,0) 
(0.5,0) ..controls (1.8,-0.7) and (2.5,-0.7).. (3.8,0)
(0,4.15)-- (1,4.15)--(1,3.85) -- (0,3.85)
(0,2.15)-- (1,2.15)--(1,1.85) -- (0,1.85)
(0,4.4)..controls  (-0.5,3.5) and (-0.5,2.5) ..(0,1.6);

\node[below,orange] at(2.15,-0.5) {$\YY_j$};

\draw[line width =0.08 cm,red] 
(0,0) --
(2,0) -- (2,2)--(2.3,2) --(2.3,0)-- 
(3,0) -- (3,2)--(3.3,2) --(3.3,0)-- 
(4,0) -- (4,2)--(4.3,2) --(4.3,0)--
(5,0) -- (5,2)--(5.3,2) --(5.3,0)--
(6,0) -- (6,2)--(6.3,2) --(6.3,0)--
(7,0) -- (7,2)--(7.3,2) --(7.3,0)--  
(9.3,0) 
(9.3,5) --
(7.3,5) -- (7.3,3)--(7,3) --(7,5)--
(6.3,5) -- (6.3,3)--(6,3) --(6,5)--
(5.3,5) -- (5.3,3)--(5,3) --(5,5)--
(4.3,5) -- (4.3,3)--(4,3) --(4,5)--
(3.3,5) -- (3.3,3)--(3,3) --(3,5)--
(2.3,5) -- (2.3,3)--(2,3) --(2,5)--
(0,5);

\draw[line width =0.08 cm,blue] 
(9.3,0) -- 
(9.3,0.85)-- (8.3,0.85)--(8.3,1.15) -- (9.3,1.15)--
(9.3,1.85)-- (8.3,1.85)--(8.3,2.15) -- (9.3,2.15)--
(9.3,2.85)-- (8.3,2.85)--(8.3,3.15) -- (9.3,3.15)--
(9.3,3.85)-- (8.3,3.85)--(8.3,4.15) -- (9.3,4.15)--
(9.3,5) 
(0,5) --
(0,4.15)-- (1,4.15)--(1,3.85) -- (0,3.85)--
(0,3.15)-- (1,3.15)--(1,2.85) -- (0,2.85)--
(0,2.15)-- (1,2.15)--(1,1.85) -- (0,1.85)--
(0,1.15)-- (1,1.15)--(1,0.85) -- (0,0.85)--
(0,0);

\filldraw (0,0) circle (0.12cm)
(9.3,0) circle (0.12cm)
(9.3,5) circle (0.12cm)
(0,5) circle (0.12cm);

\node[red, above] at (4.65,5) {$\partial^{h,1} \RR$};
\node[red, below] at (4.65,0) {$\partial^{h,0} \RR$};

 \end{scope}

 \begin{scope}[shift={(10,-6.5)}]

 \node at (4.65,2.5) {$\wC\setminus \RR$};

\draw[opacity=0, fill, fill opacity=0.1] 
(0,0) --
(9.3,0) -- 
(9.3,5) --
(0,5) --
(0,0);

\draw[line width =0.08 cm,red] 
(0,0) --
(9.3,0)  
(0,5) --(9.3,5) ;
\draw[line width =0.08 cm,blue] 
(9.3,0) -- 
(9.3,5)
(0,5) --
(0,0);

\filldraw (0,0) circle (0.12cm)
(9.3,0) circle (0.12cm)
(9.3,5) circle (0.12cm)
(0,5) circle (0.12cm);

\draw[line width =0.08 cm,dashed, red] 
(2,0) -- (2,2)--(2.3,2) --(2.3,0)  
(3,0) -- (3,2)--(3.3,2) --(3.3,0)  
(4,0) -- (4,2)--(4.3,2) --(4.3,0) 
(5,0) -- (5,2)--(5.3,2) --(5.3,0) 
(6,0) -- (6,2)--(6.3,2) --(6.3,0) 
(7,0) -- (7,2)--(7.3,2) --(7.3,0)   
(7.3,5) -- (7.3,3)--(7,3) --(7,5) 
(6.3,5) -- (6.3,3)--(6,3) --(6,5) 
(5.3,5) -- (5.3,3)--(5,3) --(5,5) 
(4.3,5) -- (4.3,3)--(4,3) --(4,5) 
(3.3,5) -- (3.3,3)--(3,3) --(3,5) 
(2.3,5) -- (2.3,3)--(2,3) --(2,5) 
;

\draw[line width =0.08 cm,dashed,blue] 
(9.3,0.85)-- (8.3,0.85)--(8.3,1.15) -- (9.3,1.15)
(9.3,1.85)-- (8.3,1.85)--(8.3,2.15) -- (9.3,2.15)
(9.3,2.85)-- (8.3,2.85)--(8.3,3.15) -- (9.3,3.15)
(9.3,3.85)-- (8.3,3.85)--(8.3,4.15) -- (9.3,4.15)
(0,4.15)-- (1,4.15)--(1,3.85) -- (0,3.85)
(0,3.15)-- (1,3.15)--(1,2.85) -- (0,2.85)
(0,2.15)-- (1,2.15)--(1,1.85) -- (0,1.85)
(0,1.15)-- (1,1.15)--(1,0.85) -- (0,0.85)
;

\node[red, above] at (4.65,5) {$\partial^{h,1} \RR'$};
\node[red, below] at (4.65,0) {$\partial^{h,0} \RR'$};

\end{scope}
\end{tikzpicture}
\caption{The (outer) rectangle $\RR'$ is obtained from $\RR$ by filling fjords. If all the fjords are protected by wide rectangles $\YY_j$ (orange),~i.e., fjords have narrow entrances, then $\Width(\RR)$ is close to $\Width( \RR')$.}
\label{fig:R and R'}
\end{figure}
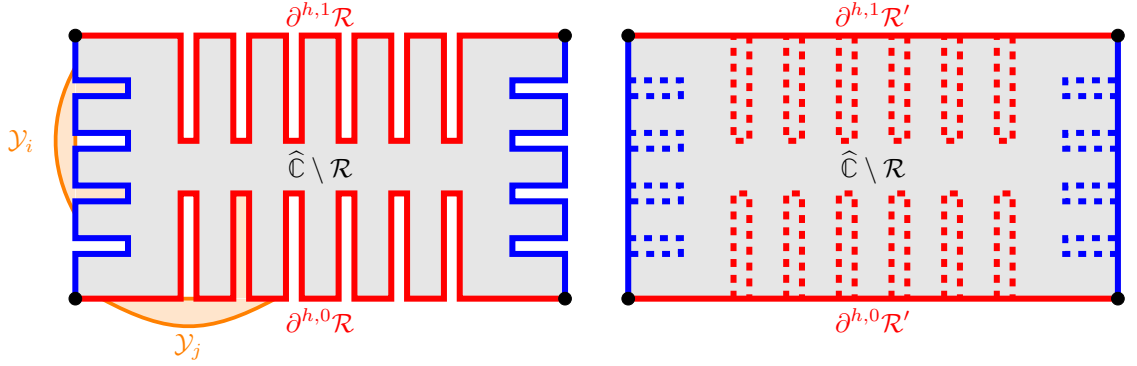

\begin{lem}
\label{lem:R and R'}
For every $\varepsilon>0$ and $\Delta \gg_\varepsilon 1$ the following holds. For $\RR$ and $\RR'$ satisfying the above (A) and (B) we have:
\begin{equation}
\label{eq:lem:R and R'}
 1-\varepsilon < \frac{\Width(\RR')}{\Width(\RR)}<1+\varepsilon.
 \end{equation}
\end{lem}
\begin{proof}
Let us conformally replace $\RR$ with a Euclidean rectangle, see~\S\ref{sss:Rectan}; i.e. we assume that $\RR=[0,x]\times [0,1]$. Since every $D_i$ is separated from three sides of $\RR$ by a wide rectangle $\YY_i$, we obtain that $D_i$ has a Euclidean diameter less than $\delta\min\{x,1\}$, where $\delta\to 0$ as $\Delta\to \infty$. Therefore, the width of $\RR'$ is estimated from below and above by the width of $[\delta, x-\delta]\times [0,1]$ and of $[0, x]\times [\delta,1-\delta]$. This implies~\eqref{eq:lem:R and R'}.   
\end{proof}

\begin{lem}
\label{lem:2:R and R'}
For every $\varepsilon>0$ and $\Delta \gg_\varepsilon 1$ the following holds. 
\begin{itemize}
\item if $\RR$ and $\RR'$ satisfy (A) and (B${}_{v}$), then:
 \begin{equation}
\label{eq:lem:2:R and R'}
 \Width(\RR) - O_\tau(1) < \Width(\RR')
 \end{equation}
\item if $\RR$ and $\RR'$ satisfy (A) and (B${}_{h}$), then:
 \begin{equation}
\label{eq:lem:2:R and R':2}
  \Width(\RR')< (1+\varepsilon)\Width(\RR) 
 \end{equation}
 \item if $\RR$ and $\RR'$ satisfy (A) and (B$'{}_{h}$), then:
 \begin{equation}
\label{eq:lem:2:R and R':3}
  \Width(\RR')< (1+\varepsilon)\Width(\RR) +O_\tau(|S|).
 \end{equation}
\end{itemize}
\end{lem}
\begin{proof}
As in the proof of Lemma~\ref{lem:R and R'}, we assume that $\RR=[0,x]\times [0,1]$ is a Euclidean rectangle. We also assume that $x=\Width(\RR)>1$ -- the only relevant case.

Assume $\RR,\RR'$ satisfy (A) and (B${}_{v}$). If $T_i$ is vertical, then $\diam D_i <C_\tau$ for some $C_\tau>0$. Therefore, $\Width(\RR')$ is estimated from below by the width of $[C_\tau, x-C_\tau]\times[0,1]$ and~\eqref{eq:lem:2:R and R'} follows.

Assume $\RR,\RR'$ satisfy (A) and (B${}_{h}$). If $T_i$ is horizontal, then $\diam D_i <c_\Delta$, where $c_\Delta\to 0$ as $\Delta \to \infty$. Therefore, $\Width(\RR')$ is estimated from above by the width of \[\RR^\new\coloneqq [0,x]\times[c_\Delta, 1-c_\Delta]\] and~\eqref{eq:lem:2:R and R':2} follows. 

Assume $\RR,\RR'$ satisfy (A) and (B$'{}_{h}$). Consider the vertical foliation $\Fam'$ of $\RR'$. The width of curves in $\Fam'$ landing at $\bigcup_{i\in S}D_i$ is $O_\tau(|S|)$. The width of the remaining curves in $\Fam'$ is bounded by $\Width(\RR^\new)$ because every remaining curve contains a subcurve in $\RR^\new$  connecting $\partial ^h\RR^\new$. This implies~\eqref{eq:lem:2:R and R':3}.
\end{proof}

\subsubsection{Well grounded intervals} Recall from~\S\ref{sss:S^well} that  a regular interval $I^m\subset \partial \wZ^m$ is well-grounded if its endpoints are not in $S^\well (\wZ^m)$.  Recall also that $\bDelta$ is a parameter from Assumption~\ref{ass:wZ:EtraProt}.

\begin{lem}
\label{lem:W^+:well grnd int}
For every $\varepsilon>0$ and $\bDelta\gg_\varepsilon 1 $ the following holds. For  every pseudo-Siegel disk $\wZ^m$ and every pair of well-grounded intervals $I^m,J^m\subset \partial \wZ^m$, we have:
\[ 1-\varepsilon \le\frac{\Width^+_{\wZ^m} (I^m,J^m)}{\Width^+_{Z} (I,J)} \le 1+\varepsilon,\]
where $I,J$ are the projections of $I^m,J^m$ onto $\partial Z$. Similarly, if $\mu^+_{\overline Z}, \mu^+_{\wZ^m}$ are the outer harmonic measures of $\overline Z,\wZ^m$ rel $\infty$, then 
\[ 1-\varepsilon \le \frac{\mu^+_{\wZ^m} (I^m)}{\mu^+_{Z} (I)} \le 1+\varepsilon.\]
\end{lem}

\begin{proof}
Let us enumerate all components of $\intr \wZ^m\setminus \overline Z$ as $D_i,\sp i\in I^m$. Every component $D_i$ is bounded by $\overline Z$ and a level $n_i\ge m$ dam $\beta^{n_i}_i$. By  Assumption~\ref{ass:wZ:EtraProt}, there is a wide rectangle 
\[\XX^{n_i}_{i}\subset \C\setminus \intr(\wZ^{n_i+1}), \sp\sp \partial^h\XX^{n_i}_{i}\subset \partial \wZ^{n_i+1} \sp\sp\text{ with }\sp\sp \Width(\XX_{i})\ge \bDelta\]
separating the endpoints of $\beta^{n_i}_i$ from the endpoints of $I,J$.

We denote by $\upstar \XX^{n_i}_i$ the rectangle obtained from $\XX^{n_i}_i$ by adding all bounded components of $\C\setminus (\overline Z\cup \partial ^h \XX^{n_i}_i )$; i.e. we adjust the horizontal boundary of $\XX^{n_i}_i$ by adding fjords so that $\partial ^h \upstar \XX^{n_i}_i\subset \partial Z$.

 \begin{lem}
 \label{lem:CompInter:prf}
   If $\bDelta\gg_\bdelta 1$, then $\Width(\upstar\XX^{n_i}_i)\ge \bDelta/2$ (independently of the number of regularizations). 
 \end{lem}
\begin{proof}
We proceed by induction on $n$: we assume that the statement is verified for all $\upstar \XX^{n_i}_i$ with $n_i\ge n+1$, and we will prove it for all $\upstar \XX^{n_i}_i$ with $n_i=n$.

Consider a rectangle $\XX^n_i$ and recall that $\upstar\XX^n_i$ is obtained from $\XX^n_i$ by adding fjords of level $>n$. Every such fjord is separated from $\intr \XX^n_i$ by a dam $\beta_j^{n_j}$ with $n_j>n$. Consider the protection $\XX_j^{n_j}$ from Assumption~\ref{ass:wZ:EtraProt}. By construction, $\XX_j^{n_j}\subset \upstar\XX_i^n$ and hence $\upstar \XX_j^{n_j}\subset \upstar \XX_i^n$. By the induction assumption, $\Width(\upstar \XX_j^{n_j}) \ge \bDelta/2\gg_\varepsilon 1$. Applying Lemma~\ref{lem:R and R'} for the rectangles $\RR=\upstar \XX^n_i$, $\RR'=\XX^n_i$ and protections $\{\upstar \XX_j^{n_j}\}$, we obtain that their widths are close; in particular, $\Width(\upstar \XX^n_i)> \Width (\XX^n_i)/2\ge \bDelta/2$.
\end{proof}

 Lemma~\ref{lem:W^+:well grnd int} now follows from Lemmas~\ref{lem:R and R'} and~\ref{lem:CompInter:prf} by viewing $\wC\setminus Z$ as a rectangle with horizontal sides $I,J$.
\end{proof}

\subsubsection{Grounded intervals}
\label{sss:ground inter}
  A regular interval $I^m\subset \partial \wZ^m$ is \emph{grounded} if its endpoints are not in $S^\inn (\wZ^m)$. An interval $I\subset \partial Z$ is \emph{grounded rel $\wZ^m$} if its projection $I^m\subset \partial \wZ^m$ is grounded.

\begin{lem}
\label{lem:W+:ground inter}
For every $\varepsilon>0$ and $\bDelta\gg_{\bdelta,\varepsilon} 1 $ the following holds.
If \[I, J\subset \partial Z\sp\sp\text{ with }\sp \dist(I,J)\ge 3 \min\{|I|,|J|\}\] is a pair of grounded rel $\wZ^m$ intervals,  then
\begin{equation}
\label{eq:lem:W+:ground inter} \Width^{+}_{Z} (I,J) - O_{\delta} (1)<  \Width^{+}_{\wZ^m} (I^m,J^m) <(1+\varepsilon) \Width^{+}_{Z} (I,J)+ O_{\delta} (1) ,
\end{equation}

where $I^m,J^m\subset \partial \wZ^m$ are the projections of $I,J$ onto $\wZ^m$.
\end{lem}

\begin{proof}We view $\wC\setminus Z$ as a rectangle $\RR$ with horizontal sides $I,J$ and we view $\wC\setminus Z^m$ as a rectangle $\RR'$ with horizontal sides $I',J'$.

Consider a dam $\beta^n_i$. By Assumption~\ref{ass:wZ:int patt}, $A^\inn(\beta^n_i)\setminus \wZ^{n+1}$ consists of $2$ rectangles; let the rectangle \[\AA^n_i\subset \wC\setminus \intr \wZ^{n+1},\sp\sp \partial^h \AA^n_i=\AA^n_i\cap \partial \wZ^{n+1}\]
be the closure of the connected component of $A^\inn(\beta^n_i)\setminus \wZ^{n+1}$ separating (i.e., protecting) $\beta^n_i$ from $\CP_n$. Since $\mod (A(\beta^n_i))\ge \bdelta$ we have $\Width(\AA^n_i)\ge \bdelta$.

As in the proof of Lemma~\ref{lem:W^+:well grnd int}, we define $\upstar \AA^n_i$ to be the rectangle obtained from  $\AA^n_i$ by adding all fjords between $\overline Z$ and $\partial ^h \AA^n_i$. Since $\partial^h \AA^n_i$ consists of a pair of well-grounded intervals of $\partial \wZ^n$, the argument in Lemma~\ref{lem:W^+:well grnd int} is applicable to $\AA^n_i$ and shows that $\Width\left(\upstar \AA^n_i\right) \ge \Width (\AA^n_i)/(1+\epsilon) > \bdelta/2\eqqcolon \tau$ because  $\bDelta\gg_{\bdelta} 1 $.

Since $\upstar \AA^n_i$ satisfy (B${}_{v}$) for $\RR,\RR'$, the first inequality in~\eqref{eq:lem:W+:ground inter} follows from~\eqref{eq:lem:2:R and R'}. Below we will remove $O(1)$ buffers from $\RR,\RR'$ so that the new rectangle $\RR^\new, \RR'^\new$ satisfies~\eqref{eq:lem:2:R and R':3} with $|S|\le 3$.

Assume that $|I|\le |J|$. Let $k$ be the level of $I$: the unique number satisfying $\frac 4 5 \length_{k} > |I|\ge \frac 4 5\length_{k+1}$, where $\frac 4 5 $ is from Assumption~\ref{ass:wZ:CombSpace}. Consider dams $\beta^n_i\subset I^n\cup J^n$. We distinguish the following three cases.

Assume $n>k$. Since 
\[\dist(I^n,J^n)\ge 3 \min\{|I^n|,|J^n|\}=3|I^n|> 2\length_{k+1}>|\lfloor \partial ^h \XX^n_i\rfloor |,\]
$\lfloor  \partial ^h \XX^n_i\rfloor$ is disjoint from either $I^n$ or $ J^n$; i.e.~$\upstar \XX^n_i$ satisfy (B${}_{h}$) for $\RR,\RR'$.

 Assume $n=k$. Then $\beta^n_i\subset J^m$ because $|\beta^n_i|>|I|$ by Assumption~\ref{ass:wZ:CombSpace}. There are at most two level $n=k$ dams $\beta^n_i$ such that $\upstar \XX^n_i$ intersects $I^n$. Such exceptional dams are protected $\upstar \AA^n_i$ and we add these dams into $S$. The remaining dams are protected by $\upstar \XX^n_i$.
 
 Assume $n<k$. Again $\beta^n_i\subset J^n$. For every $\beta^n_i$, we recognize disjoint genuine subrectangles $\XX^{n,\inn}_i,\XX^{n,\out}_i\subset \XX^n_i$ such that $\dist(\partial^h \XX^{n,\out}_i,\sp \partial^h \XX^{n,\inn}_i)>|I|$ and
 \[  \Width(\XX^{n,\inn}_i), \Width(\XX^{n,\out}_i)\ge \bDelta/3,\sp\sp\sp \lfloor \partial^h\XX^{n,\inn}_i \rfloor\subset \lfloor\partial^h \XX^{n,\out}_i \rfloor.\]
If $I^n$ is disjoint from all $\partial ^h\XX^{n,\inn}_i $ for $n<k$, then $\upstar\XX^{n,\inn}_i $ satisfy (B${}_{h}$) and the lemma follows from~\eqref{eq:lem:2:R and R':3}. Suppose that $n$ is minimal so that for a dam $\beta^n_i$ the interval $I^n$ intersects $\partial ^h\XX^{n,\inn}_i$; note that the dam $\beta^n_i$ is unique. Then \[\partial^{h,0}\XX^{n,\out}_i<I^n< \XX^{n,\out}_i\] and for every $t\ge n, \sp \beta^{t}_j\not=\beta^n_i$, the rectangle $\XX^{n,\out}_i$ separates $I^n$ from $ \beta^{t}_j$. By removing a $O_\delta(1)$-buffer from $\RR,\RR'$ we can assume that the new rectangles $\RR^\new, \RR'^\new$ do not cross $\XX^{n,\out}_i$; the new rectangles are disjoint from all $\beta^t_j\not=\beta^n_i$ with $t\ge k$. We add $\beta^n_i$ to $S$ and apply~\eqref{eq:lem:2:R and R':3}.
\end{proof}

\subsubsection{Restrictions of rectangles}\label{sss:RestrOfFamil} Consider a rectangles $\RR\subset \wC\setminus Z$ and assume that there is a connected graph $G\subset \RR$ containing all vertices of $\RR$ such that $G\subset \wC\setminus \intr \XX(\wZ^m)$, see~\eqref{eq:dfn XX(wZ)}. Then the \emph{restriction} $\RR^m$ of $\RR$ to the complement of $\wZ^m$ is the connected component of $\RR\setminus \intr \wZ^m$ containing $G$. We view $\RR^m$ as a rectangle with the same vertex set and the same orientation of sides as $\RR$. Observe that every connected component of $\RR\setminus \RR^m$ is separated from $G$ by a restriction of some $\Fam(\upstar\XX^{n_i}_i)$. Therefore, Lemmas~\ref{lem:CompInter:prf} and~\ref{lem:R and R'} imply that
\begin{equation}
\label{eq:lem:RR vs RR^m}
 1-\varepsilon < \frac{\Width(\RR)}{\Width(\RR^m)}<1+\varepsilon, 
 \end{equation}
 for a small $\varepsilon>0$, where $\bDelta\gg_\varepsilon 1$.

\begin{rem}[Fixing $\bDelta$]
\label{rem:fix bDelta}
From now on we assume that $\bDelta\gg 1$ is fixed so that $\varepsilon$ in Lemmas~\ref{lem:W^+:well grnd int} and~\ref{lem:W+:ground inter}  and in~\eqref{eq:lem:RR vs RR^m} is small.
\end{rem}

\subsection{The geometry on shallow scales} On scale $\ge \length_m$, the geometry of $\wZ^m$  is controlled by its core $Z^m$:

\begin{thm}[Log-Rule for $\wZ^m$]
\label{thm:wZ:shallow scale}
Consider two regular intervals \[I,J\subset \partial \wZ^m\sp\sp\sp\text{ such that }\sp\sp |I|,\ |J|,\ \dist(I,J) \ge 1/\qq_{m+1} .\] If $\dist (I,J)\le \min\{|I|,|J|\}$, then
\begin{equation}
\label{eq:1:prop:wZ:shallow scale}
\Width^-(I,J) \asymp_\bdelta \log \frac{\min\{|I|,|J|\}}{\dist (I,J)} +1;
\end{equation}
otherwise 
\begin{equation}
\label{eq:2:prop:wZ:shallow scale}
\Width^-(I,J) \asymp_\bdelta \left( \log \frac{\dist (I,J)}{\min\{|I|,|J|\}} +1\right)^{-1}.
\end{equation} 

Moreover, there is a constant $T_\bdelta>1$ such that if $|I|,\ |J|,\ \dist(I,J) \ge T_\bdelta/\qq_{m+1}$, then ~\eqref{eq:1:prop:wZ:shallow scale} and~\eqref{eq:2:prop:wZ:shallow scale} are independent of $\bdelta$: if $\dist (I,J)\le \min\{|I|,|J|\}$, then
\begin{equation}
\label{eq:1:prop:wZ:shallow scale:impr}
\Width^-(I,J) \asymp \log \frac{\min\{|I|,|J|\}}{\dist (I,J)} +1;
\end{equation}
otherwise 
\begin{equation}
\label{eq:2:prop:wZ:shallow scale:impr}
\Width^-(I,J) \asymp \left( \log \frac{\dist (I,J)}{\min\{|I|,|J|\}} +1\right)^{-1}.
\end{equation} 
 \end{thm}

We refer to~\eqref{eq:1:prop:wZ:shallow scale:impr} and~\eqref{eq:2:prop:wZ:shallow scale:impr} as \emph{beau coarse-bounds}, compare with Theorem~\ref{thm:beau:part U: c quasi line}.  
 
\begin{proof} For a channel $\alpha_i\subset \partial Z^m$, let us denote by $\Width^-_{Z^m, 3}(\alpha_i)$ the width of curves starting at $\alpha_i$ and ending at $\partial Z^m\setminus (\beta_{i-1}\# \alpha_i\# \beta_{i})$. Similarly, $\Width^-_{Z^m, 3}(\beta_i)$ is the width of curves starting at $\beta_i$ and ending at $\partial Z^m\setminus (\alpha_{i}\# \beta_i\# \alpha_{i+1})$.

\begin{claim7}
\label{cl7:small case} For all channels and dams of $\partial Z^m$, we have:
\[ \Width^-_{Z^m, 3}(\alpha_i)\asymp_\bdelta 1,\sp\sp\sp\sp\Width^-_{Z^m, 3}(\beta_i)\asymp_\bdelta 1 .\] 
\end{claim7}
\begin{proof}
We have $\Width^-_{Z^m, 3}(\alpha_i)\preceq_\bdelta 1$ because $A^\inn(\alpha_i)$ separates $\alpha_i$ from $\partial Z^m\setminus (\beta_{i-1}\# \alpha_i\# \beta_{i})$. Similarly, $\Width^-_{Z^m, 3}(\beta_i)\preceq_\bdelta 1$. Since 
\[\left(\Width^-_{Z^m, 3}(\alpha_i) \right)^{-1} = \Width^-_{Z^m}(\beta_{i-1},\beta_i)\le \Width^-_{Z^m,3}(\beta_{i-1})\preceq_\bdelta 1,\]
where $\Width^-_{Z^m}(\beta_{i-1},\beta_i)$ is the width of curves connecting $\beta_{i-1}$ and $\beta_i$, we obtain $\Width^-_{Z^m, 3}(\alpha_i)\asymp_\bdelta 1$. Similarly, $\Width^-_{Z^m, 3}(\beta_i)\asymp_\bdelta 1$.
\end{proof}

\begin{claim7}
\label{cl7:base case}
If $\min\{|I|,|J|\}\asymp \dist(I,J)$, then $\Width^-_{\wZ^m}(I,J)\asymp_\bdelta 1$.
\end{claim7}
\begin{proof}  Let us define the \emph{projection} $I^\circ$ of $I$ onto $\partial Z^m$ to be the minimal concatenation 
\[\alpha_i\#\beta_{i+1}\#\beta_{i+1}\#\dots \#\beta_k\#\alpha_{k+1}\]
(starting and ending with a channel) such that every component of $I\setminus I^\circ$ is in the peninsula whose channel is in $I^\circ$. Similarly, $J^\circ $ is defined. Since $\dist(I,J)\ge 1/\qq_{m+1}$, the intervals $I^\circ$ and $J^\circ$ are separated by at least one dam (follows from Assumption~\ref{ass:wZ:CombSpace}) and we still have $\min\{|I^\circ|, |J^\circ|\} \asymp \dist(I^\circ ,J^\circ)$ with respect to the distance of $\partial Z^m$.  

The case $\min\{|I|,\ |J|\},\ \dist(I,J) \asymp 1/\qq_{m+1}$ follows from Claim~\ref{cl7:small case} by spiting intervals into finitely many channels and dams. Let us assume that $|I|,\ |J|,\ \dist(I,J) \ge 50/\qq_{m+1}$.    

By Lemma~\ref{lem:buffer:R O} that the set of curves in $\Fam\coloneqq \Fam_{\wZ^m}^-(I,J)$ entering a peninsula with a channel in $\partial Z^m\setminus (I^\circ \cup J^\circ)$ is a buffer of $\Fam_{\wZ^m}^-(I,J)$. By removing two $O_\bdelta (1)$-buffers from $\Fam_{\wZ^m}^-(I,J)$, we obtain that curves in the new family $\Fam^\new$ do not enter any peninsula whose channel is in $\partial Z^m\setminus (I^\circ \cup J^\circ)$ -- such channels are separated from $I^\circ\cup J^\circ$ by a definite  annulus. Let $\Fam^\New$ be the restriction of $\Fam^\new$ to the family $\Fam^-_{Z^m}(I^\circ, J^\circ)$ (see~\S\ref{sss:short subcurves}); i.e., $\Fam^\New$ consists of the first shortest subcurves $\gamma'$ of $\gamma\in \Fam^\new$ such that $\gamma'$ connects $I^\circ$ and $J^\circ$. We have:
\[\Width^-_{\wZ^m}(I,J)-O_\bdelta(1)\le \Width(\Fam^\new) \le \Width(\Fam^\New)\le \Width^-_{Z^m}(I^\circ,J^\circ ) \asymp 1   \]
by Proposition~\ref{prop:part U: c quasi line}; i.e.~$\Width^-_{\wZ^m}(I,J)\preceq_\bdelta 1$. If $A,B$ are the components of $\partial \wZ^m\setminus (I,J)$, then the above argument shows $\Width^-_{\wZ^m}(A,B)\preceq_\bdelta 1$. Therefore, $\Width^-_{\wZ^m}(I,J)\asymp_\bdelta 1$.
\end{proof}

The case $\dist (I,J)\le \min\{|I|,|J|\}$ follows from Claim~\ref{cl7:base case} by applying the Splitting Argument (Remark~\ref{rem:SplitArg}). By the same reason, the second part of the theorem follows from the following claim:

\begin{claim7}
\label{cl7:LS case}
There is a constant $T_\bdelta>50$ such that
if \[|I|,\ |J|,\ \dist(I,J) \ge T_\bdelta/\qq_{m+1} \sp\sp\text{ and }\sp\sp\min\{|I|,|J|\}\asymp_\bdelta \dist(I,J),\] then $\Width^-(I,J)\asymp 1$.
\end{claim7}
\begin{proof}
Assume $T_\bdelta\gg 50$ is sufficiently big and let us consider $I^\circ$ and $J^\circ$ as in the proof of Claim~\ref{cl7:base case}. Let us enlarge $I^\circ,J^\circ$ by adding $\sqrt{T_\bdelta}$ unit intervals on both sides of $I^\circ, J^\circ$; the new intervals $\widetilde I^\circ, \widetilde J^\circ$ still satisfy $\min\{|\widetilde I^\circ|, |\widetilde J^\circ| \}\asymp \dist(\widetilde I^\circ,\widetilde J^\circ)$.

We claim that by removing two $1$-buffers from $\Fam^-_{\wZ^m}(I,J)$ we obtain the family $\Fam^\new$ so that curves in $\Fam^\new$ do not enter any peninsula with a channel in ${\partial Z^m\setminus (\widetilde I^\circ\cup \widetilde J^\circ)}$. Indeed, consider a channel $\alpha_k\in \partial Z^m\setminus (\widetilde I^\circ \cup \widetilde J^\circ)$, and let $A,B\subset \partial \wZ^m$ be intervals of length $\sqrt{T_\bdelta}$ attached to $\alpha_k$. Since we already established~\eqref{eq:1:prop:wZ:shallow scale}, $\Width^-_{\wZ^k}(A,B)\asymp_\bdelta \log \sqrt{T_\bdelta}\gg 1$; removing a $O_\bdelta(1)$-buffer from $\Fam^-_{\wZ^k}(A,B)$, we obtain a family of curves $\RR_k$ such that $\Width(\RR_k)>1$ and $\RR_k$ separates $\alpha_k$ from $I^\circ\cup J^\circ$. The set of  vertical curves in $\Fam_{\wZ^m}^-(I,J)$ that intersect $\alpha$ form a buffer of $\RR$ by Lemma~\ref{lem:buffer:R O}; this buffer has width less than $1$ by $\Width(\RR_k)>1$.

 As in the proof of Claim~\ref{cl7:base case}, we now define $\Fam^\New$ to be the restriction of $\Fam^\new$ to the family $\Fam^-_{Z^m}(\widetilde I^\circ,\widetilde J^\circ)$ (see~\S\ref{sss:short subcurves}); i.e., $\Fam^\New$ consists of the first shortest subcurves $\gamma'$ of $\gamma\in \Fam^\new$ such that $\gamma'$ connects $\widetilde I^\circ$ and $\widetilde J^\circ$. We have:
\[\Width^-_{\wZ^m}(I,J)-O(1)\le \Width(\Fam^\new) \le \Width(\Fam^\New)\le \Width^-_{Z^m}(I^\circ,J^\circ ) \asymp 1 \]
 by Theorem~\ref{thm:beau:part U: c quasi line}; i.e.~$\Width^-_{\wZ^m}(I,J)\preceq 1$.  And similarly, $\left(\Width^-_{\wZ^m}(I,J)\right)^{-1}= \Width^-_{\wZ^m}(A,B)\preceq 1$, where $A,B$ are the components of $\partial \wZ^m\setminus (I,J)$.
\end{proof}
\end{proof}

\subsubsection{Hyperbolic geodesics in $\wZ^m$} Let us extend continuously the distance function $\dist_{\partial \wZ^m}(\ ,\ )$ specified in~\eqref{eq:dist on wZ^m} to all points of $\partial \wZ^m$. Given two sets $X,Y$ intersecting $\partial \wZ^m$, we define the \emph{$\partial \wZ^m$-distance between $X,Y$} to be $\dist_{\partial \wZ^m}\left(X\cap \partial \wZ^m,\ Y\cap \partial \wZ^m \right)$. 

\begin{lem}
\label{lem:HypGeod in wZ^m}
There is an $M=M(\bdelta)\ge 1$ such that the following properties hold for every hyperbolic geodesic $\gamma\subset \wZ^m$ connecting $x,y\in \partial \wZ^m$.
\begin{enumerate}[label=\text{(\Roman*)},font=\normalfont,leftmargin=*]
\item \label{Case1:lem:HypGeod in wZ^m} If the $\partial \wZ^m$-distance from $\{x,y\}$ to $\widetilde A^\inn(\beta^n_i)$ is at least $M\length_{m}$, then $\gamma$ is disjoint $\widetilde A^\inn(\beta^n_i)$.

\item\label{Case2:lem:HypGeod in wZ^m} Similarly, if the $\partial \wZ^m$-distance from $\{x,y\}$ to $\widetilde A^\inn(\alpha^n_i)$ is at least $M\length_{m}$, then $\gamma$ is disjoint $\widetilde A^\inn(\alpha^n_i)$.
\end{enumerate}

\noindent Let $\gamma_k=(f^k)_*(\gamma)\subset \UU\big(\wZ^m\big)$ be the image of $\gamma$ under~\eqref{eq:dfn UU(wZ)}, and assume that the endpoints of $\gamma_k$ are in $U_{x,k}, U_{y,k}$ -- the components of 
\begin{equation}
\label{eq:comp U ni}
\{\widetilde A(\alpha^n_i)\setminus A(\alpha^n_i), \ \widetilde A(\beta^n_i)\setminus A(\beta^n_i)\mid n,i\},\sp\sp\sp\text{ see~\S\ref{sss:UU and ZZ}} .
\end{equation}

\begin{enumerate}[label=\text{(\Roman*)},start=3,font=\normalfont,leftmargin=*]
\item \label{Case3:lem:HypGeod in wZ^m}  If for $a,b\in \partial \wZ^m$ the  $\partial \wZ^m$-distance from $U_{x,k}\cup U_{y,k}$ to $\{a,b\}$ is at least $M\length_{m}$, then $\gamma_k$ is disjoint from the hyperbolic geodesic of $\wZ^m$ between $a,b$.
\end{enumerate}
\end{lem}
\begin{proof}
Choose a sufficiently big $M=M(\bdelta)$ and assume that $\widetilde A^\inn (\beta^n_i)$ has distance at least $M\length_{m+1}$ to $\{x,y\}$. By Theorem~\ref{thm:wZ:shallow scale}, there is a sufficiently wide geodesic rectangle $\RR\subset\wZ^m$ with $\partial^h \RR\subset \partial \wZ^m$ such that $\partial^h \RR$ separates $\widetilde A (\beta^n_i)\cap \partial \wZ^m$ from $\{x,y\}$ in $\partial \wZ^m$. By removing $1/\bdelta$ buffers from $\RR$, we obtain a rectangle $\RR^\new$ that separates $\widetilde A^\inn (\beta^n_i)$ from $\gamma$. This proves~\ref{Case1:lem:HypGeod in wZ^m} and~\ref{Case2:lem:HypGeod in wZ^m} is similar.

Properties~\ref{Case1:lem:HypGeod in wZ^m},~\ref{Case2:lem:HypGeod in wZ^m} and Assumption~\ref{ass:wZ:collars}, imply that $\gamma_k$ is disjoint from every component $U'$ of~\eqref{eq:comp U ni} that has $\partial \wZ^m$-distance at least $M\length_{m}$ to $U_{n,k}\cup U_{n,k}$. This implies \ref{Case3:lem:HypGeod in wZ^m}.
\end{proof}

\subsection{The inner geometry of peninsulas}\label{ss:peninsulas:geometry} For a level $n\ge m$ peninsula $\Penin$ of $\wZ^n$ with $\partial ^c\Penin\subset \partial \wZ^m$ and intervals $I,J\subset \partial ^c\Penin$, we denote by $\Fam_\Penin^-(I,J)$ the family of curves in $\Penin$ connecting $I$ and $J$. We write $\Width^-_\Penin(I,J)\coloneqq \Width(\Fam^-_\Penin(I,J))$.

By a \emph{grounded pair} of intervals $I,J\subset\partial \wZ^m$, we mean a pair of disjoint grounded intervals~\S\ref{sss:ground inter}.

\begin{lem}
\label{lmm:GP:rest}
Let $I,J\subset \partial \wZ^m$ be a grounded pair of intervals, and assume that $I$ is within a level $k\ge m$ peninsula $\Penin$ of $\wZ^m$. Set $J^\new\coloneqq J\cap \partial \Penin$. Then $\partial ^c\Penin\subset \partial \wZ^m$ and
\begin{equation}
\label{eq:lmm:GP:rest}
\Width^-_\Penin(I,J^\new)=\Width^-_{\wZ^{k+1}}(I,J^\new)-O_\bdelta(1)= \Width_{\wZ^m}^-(I,J)-O_\bdelta(1).
\end{equation}
If $J^\new \not=\emptyset$, then $I,J^\new$ is a grounded pair of $\wZ^{k+1}$.
\end{lem}
\begin{proof}

Since the endpoints of $I$ are grounded, $\partial^c \Penin$ is not in any $S^\inn$ buffer of level $<k$; this implies that  $\partial^c \Penin\subset \partial \wZ^m$. Let $\alpha_k$ be the channel of $\Penin$. If $J^\new=\emptyset$, then every curve in $\Fam^-_{\wZ^m}(I,J)$ crosses $A^\inn(\alpha_k)$ before intersecting $\alpha_k$; i.e.~all parts of~\eqref{eq:lmm:GP:rest} are equal to $O_{\bdelta}(1)$ and the lemma follows.

If $J^\new\not=\emptyset$, then  $J^\new$ is a grounded interval of $\wZ^{k+1}$ because the endpoints of $\alpha$ are not in any $S^\inn$-buffers of level $\ge k+1$. Since $I\subset \partial ^c\Penin\setminus S^\inn (\alpha)$, the width of curves in $\Fam^-_{\wZ^m}(I,J)$ that are not in $\Penin$ is $O_{\bdelta}(1)$ because every such curve crosses $A^\inn(\alpha_k)$. This implies the lemma. 
\end{proof}

The following lemma combined with Theorem~\ref{thm:wZ:shallow scale} allows us to inductively calculate the width between grounded intervals.

\begin{lem}
\label{lmm:gound inter:penins}
Let $I,J\subset \partial \wZ^m$ be a grounded pair with $|\lfloor I,J\rfloor \le 1/2$.
 Assume that both $I,J$ intersect a level $m$ peninsula $\Penin=\Penin(\alpha^m_i)$. Set \[I^\new\coloneqq I\cap \partial^c \Penin , \sp\sp I'\coloneqq I\setminus(\beta^m_{i-1}\cup I^\new) \sp\sp J^\new \coloneqq J\cap \partial^c\Penin,\sp\sp J'\coloneqq J\setminus (J^\new\cup \beta^m_i).\] Then: 
\[\begin{array}{cc}
\Width^-_{\wZ^m}(I,J)&=\Width^-_{\wZ^m}(I',J') + \Width_\Penin^-(I^\new, J^\new)+O_\bdelta(1)\\
&\sp=\Width^-_{\wZ^m}(I',J') + \Width_{\wZ^{m+1}}^-(I^\new, J^\new)+O_\bdelta(1)
\end{array}.\]
\end{lem}
\begin{proof}
 The intervals $I^\new$ and $J^\new$ are grounded intervals of $\wZ^{m+1}$ because the endpoints $y_{i-1},x_i$ of $\alpha^m_i$ are not in any $S^\inn$-buffer of level $\ge m+1$. We need to show that the width of the set $\Fam'$ of curves in $\Fam^-_{\wZ^m}(I,J)$ intersecting $\alpha^m_i$ is $O_\bdelta(1)$. Let $\AA$ be the set of curves in $\Fam'$ that are in $\widetilde A^\inn(\alpha^m_i)$. Note that $\Width(\Fam' \setminus \AA)= O_\bdelta(1)$ because every curve in $\Fam'\setminus \AA$ crosses $A^\inn(\alpha^m_i)$.

Assume $\AA\not=\emptyset$. This is only possible if
\[\beta^m_{i-1} \cup (S^\inn(y_{i-1})\cap \partial \wZ^m) \subset I\sp\sp\text{ and }\sp\sp (S^\inn(x_{i})\cap \partial \wZ^m) \cup \beta^m_{i}  \subset J\]
because $I,J$ are grounded. Therefore, every curve in $\AA$ has a subcurve connecting $S^\inn _{y_{i-1}}$ and $S^\inn _{x_{i}}$. By Assumption~\ref{ass:wZ:bDelta}, we obtain $\Width(\AA)= O_\bdelta(1)$.
\end{proof}

\subsection{Localization and Squeezing lemmas} 
\label{ss:wZ:local lmm}
We are now ready to establish Localization and Squeezing Properties for $\partial \wZ^m$, compare with \S\ref{sss:LocProp} and~\S\ref{sss:SqueezProp}. See~\S\ref{sss:LocProp} for the notion of an innermost subpair.

\begin{loclmm}
\label{lem:trading  width to space}
For every $\lambda>1$ the following holds. If $I, J\subset \partial\wZ^m$ is a grounded pair  with $|\lfloor I, J\rfloor |\le 1 -\frac 1 \lambda \min\{|I|,|J|\}$, then there is an innermost subpair \[ I^\new,J^\new\sp\sp\text{ with }\sp\sp I^\new\subset I,\sp J^\new\subset J \] satisfying \[|I^\new|,|J^\new|\le \frac 1 \lambda \min\{|I|, |J|\}\] such that, up to $O_{\bdelta}(\log  \lambda)$, the width of $\Fam ^-(I,J)$ is contained in $\Fam^-(I^\new,J^\new)$:
\begin{equation}
\Width^- (I\setminus I^\new,J)+ \Width^-(I,J\setminus J^\new) =O_{\bdelta}(\log  \lambda).
\end{equation}

Moreover, we can assume that $\max\{|I^\new|, |J^\new|\} < 2 \min\{ |I^\new|, |J^\new|\}$. The subpair $ I^\new,J^\new$ is grounded rel $\wZ^n$, where $n\ge m$ is the deepest level such that $I^\new,J^\new\subset \partial \wZ^n$.
\end{loclmm}

\begin{squeelmm}
\label{lem:squeezing}
There is a constant $C_\bdelta$ such that the following holds. If $I,J\subset \partial \wZ^m$ is a grounded pair of intervals with $|\lfloor I, J\rfloor |\le 1/2$ such that \[\Width^-(I,J)\ge C_\bdelta \log \lambda, \sp\sp\lambda>2,\]
then $\dist (I,J) \le\frac{1}{\lambda} \min\{|I|,|J|\}$.
\end{squeelmm}

\begin{proof}[Proof of Lemmas \ref{lem:trading  width to space} and~\ref{lem:squeezing}] By splitting $I,J$ into shorter grounded intervals, we can assume that $|\lfloor I, J\rfloor |\le 1/2$. 

Let $L\subset \lfloor I,J \rfloor$ be the shortest complementary interval between $I$ and $J$. The case $\dist(I,J) =|L|\ge \length_m$ follows from Theorem~\ref{thm:wZ:shallow scale}: we can find points $a,b\in \CP_{m}$ so that the right interval $I^\new$ of $I\setminus \{a\}$ and the left interval $J^\new$ of $J\setminus \{b\}$ satisfy the conclusion of Lemma \ref{lem:trading  width to space}. (The intervals $I^\new, J^\new$ are grounded because the set $\CP_m$ is away from $S^\inn$-buffers.)

Let $m_j\ge m$ be the smallest regularization level such that $|I|,|J|\ge \length_{m_j}$. If $m< m_j$, then by Lemma~\ref{lmm:GP:rest}, up to $O_\bdelta(1)$ the width of $\Fam^-_{\wZ^m}(I,J)$ is in $\Fam^-_{\wZ^{m_j}}(I^\new,J^\new)$; set $I_1=I^\new,J_1=J^\new\subset \partial \wZ^m\cap \partial\wZ^{m_j}$. If $m=m_j$, set $I_1\coloneqq I, \sp J_1\coloneqq J$. Note that $L$ is still the shortest complementary interval between $I_1, J_1$.

The case $\dist(I_1,J_1) =|L|\ge \length_{m_j}$ follows again from Theorem~\ref{thm:wZ:shallow scale}. Otherwise,  as in Lemma~\ref{lmm:gound inter:penins}, we recognize subintervals $I'_1,  I^\new_1\subset I_1$ and $J^\new_1, J'_1\subset J_1$. Up to $O_\bdelta(1)$, the family $\Fam^-_{\wZ^{m_j}} (I_1,J_1)$ is $\Fam^-_{ \wZ^{m_j} }(I'_1,J'_1)\sqcup \Fam^-_{ \wZ^{m_j+1} }(I^\new_1,J^\new_1)$. Applying Theorem~\ref{thm:wZ:shallow scale} to $\Fam^-_{ \wZ^{m_j} }(I'_1,J'_1)$, we either conclude the lemma if $\min\{|I'_1|,|J'_1|\}\ge \lambda  |L|$ or we reduce the problem to the pair $I_2\coloneqq I_1^\new, J_2\coloneqq J^\new_1$ with a new $\lambda_2< \lambda/2$ because $|I_2|,|J_2|\le 2\length_{m_k}/5$ by Assumption~\ref{ass:wZ:CombSpace}. Proceeding by induction, we conclude the lemma with at most $\log_2 \lambda$ steps. 
\end{proof}

\subsection{Grounded subintervals}
\label{ss:Bett Grnd Inter}
As before, we extend continuously the distance function $\dist_{\wZ^m}$ specified in~\eqref{eq:dist on wZ^m} to all points of $\partial \wZ^m$. Then $\Fam^-_\lambda(I)$ and $\Width^-_\lambda(I)$ are well defined for all intervals $I\subset \partial \wZ^m$. 

Given an interval $J\subset \partial Z$, we define:
\begin{itemize}
\item $J^\GRND\subset J$ be the biggest grounded interval in $J$; and
\item $J^\grnd\supset J$ be the smallest grounded interval containing $J$.
\end{itemize} 
We allow $J^\grnd =\emptyset$ or $J^\GRND = \partial \wZ^m$.

\begin{lem}[$X$ vs $X^\grnd$]
\label{lem:I vs I^grnd:with separ} Consider a family $\Fam(I,J)$, where $I,J\subset \partial \wZ^m$. Let $A,B\subset\partial \wZ^m$ be two complementary intervals to $I,J$. 
\begin{enumerate}[label=\text{(\roman*)},font=\normalfont,leftmargin=*]
\item \label{Case1:lem:I vs I^grnd:with separ}If $\dist(I,J)>\length_{m}$, then for every interval $X\subset \partial \wZ^m$, there are at most $O_\bdelta(1)$ curves in $\Fam(I,J)$ intersecting ${X\setminus X^\grnd}$. 
\item \label{Case2:lem:I vs I^grnd:with separ}If $|I|, |J|, |A|\ge \length_m$, then for every interval $X\subset A$, there are at most $O_\bdelta(1)$ curves in $\Fam(I,J)$ intersecting ${X\setminus X^\grnd}$. 
\end{enumerate}
\end{lem}
\begin{proof}
Consider Case~\ref{Case1:lem:I vs I^grnd:with separ}. Every component (out of at most $2$) of ${X\setminus X^\grnd}$ is within $ S^\inn(\beta)$ for a dam $\beta\subset L$. Since $|A|, |B|\ge \length_{m}$, the color $A^\out (\beta)$ is disjoint (and hence separate $S^\inn(\beta)$) from  either $I$ or $J$. The lemma now follows from $\mod A^\out (\beta)\ge \bdelta$.

Case~\ref{Case2:lem:I vs I^grnd:with separ} follows from a similar argument. Every component of ${X\setminus X^\grnd}$ is within $ S^\inn(\beta)$. Since $|I|, |J|, |A|\ge \length_m$, the annulus $A^\out (\beta)$ separates $S^\inn(\beta)$ from either $I$ or $J$.
\end{proof}

\subsection{Pseudo-bubbles}\label{ss:Ps-bubles}

 A \emph{bubble} $Z_\ell$ is the closure of a connected component of $f^{-k}(Z)\setminus Z$. The \emph{generation} of $Z_\ell$ is the minimal $k$ such that $f^k(Z_\ell)=\overline Z$; i.e. $f^k\colon Z_\ell\to \overline Z$ is the first landing. Given a pseudo-Siegel disk $\widehat Z^m$, the \emph{pseudo-bubble} $\widehat Z_\ell$ is the closure of the connected component of $f^{-k}\big( \intr \widehat Z^m\big)$ containing $\intr Z_\ell$. In other words, $\widehat Z_\ell$ is obtained from $Z_\ell$ by adding the lifts of all reclaimed fjords (components of $\widehat Z^m\setminus \overline Z$) along $f^k\colon Z_\ell\to \overline Z$.

Channels, dams, collars $A^\inn, A^\out$, extra protections $\XX^n_i$ are defined for $\wZ_\ell$ as pullbacks of the corresponding objects along $f^k\colon \wZ_\ell\to  \wZ^m$. For instance, $\XX(\wZ_\ell)$ is the pullback of $\XX(\wZ^m)$, see~\eqref{eq:dfn XX(wZ)}, under $f^k$. The length of an interval $I\subset \partial \wZ_\ell$ is the length of its image $f^k(I)\subset \partial \wZ^m$. All results of this section are valid for pseudo-bubbles. In particular, the results concerning the inner geometry of $\wZ^m$ (such as Theorem~\ref{thm:wZ:shallow scale}, Lemmas~\ref{lem:trading  width to space} and~\ref{lem:squeezing}) are obtained by identifying $\wZ_\ell$ with $\wZ^m$ via $f^k$. The results concerning the outer geometry of $\wZ_\ell$ (see~\S\ref{ss:WellGrounded}) are obtained by repeating the arguments.

\section{Snakes}\label{s:snakes} Consider the family $\Fam^\circ _L(I,J)=\Fam^\circ_{L,\wZ^m}(I,J)$ as in~\S\ref{sss:SnakeLmm:simpl toplog}. Our principal result of the section is the following generalization of Lemma~\ref{simplmm:SnLmm:Z}:

\begin{snakelmm}
\label{simplmm:SnLmm:wZ}
Let $I,J\subset \partial \wZ^m$ be a pair of grounded intervals and let $L$ be a complementary interval between $I,J$. Normalize $I<L<J$ and set
\[K\coloneqq \Width^\circ_L(I,J) - \Width^+(I,J).\]
Assume that $|N|\ge \length_{m}$, where $N=\partial \wZ^m\setminus (I\cup L\cup J)$ is the second complementary interval between $I$ and $J$.

 If $K\gg_\bdelta \log \lambda$ with $\lambda>2$, then there are grounded intervals 
 \begin{equation}
\label{eq:SnakeLmm:intervals}
 J_1,I_1\subset L,\sp\sp |J_1|< \frac{\dist(I, J_1)}{\lambda},\sp |I_1|<\frac{\dist( I_1,J)}{\lambda},\sp\sp I< J_1<I_1<J
 \end{equation} such that 
\begin{equation}
\label{eq:SnakeLmm}  \Width^\circ_{L_a} (I,J_1) \oplus\Width^\circ_{L_b} (I_1,J)\ge K - O_\bdelta(\log \lambda),   
\end{equation}
where $L_a,L_b\subset L$ are the intervals between $I,J_1$ and $I_1,J$ respectively:
\[\begin{tikzpicture}
\draw (-6,0) -- (6,0);
\draw[line width =0.8mm,red ] (-5.5,0)--(-3,0);
\node[below,red ] at (-4.2,0){$I$};
\draw[line width =0.8mm,blue ] (5.5,0)--(3,0);
\node[below,blue ] at (4.2,0){$J$};

\draw[line width =0.8mm,red ] (-1,0)--(-0.3,0);
\node[below,red ] at (-0.7,0){$J_1$};

\draw[line width =0.8mm,blue ] (1,0)--(0.3,0);
\node[below,blue ] at (0.7,0){$I_1$};

\draw [line width =1.2mm,red ] (-4,0)
 .. controls (-3.5, 0.5) and (-2.5,0.5) ..
 (-2,0)
 .. controls (-1.8, -0.5) and (-1.6,-0.5) ..
 (-1.4,0)
  .. controls (-1.2, 0.5) and (-0.8,0.5) ..
 (-0.6,0);
 \node[above,red] at (-1.7, 0.3) {$\Fam^\circ_{L_a} (I,J_1)$};

\draw [line width =1.2mm,blue ] (4,0)
 .. controls (3.5, 0.5) and (2.5,0.5) ..
 (2,0)
 .. controls (1.8, -0.5) and (1.6,-0.5) ..
 (1.4,0)
  .. controls (1.2, 0.5) and (0.8,0.5) ..
 (0.6,0);
 \node[above,blue] at (1.8, 0.3) {$\Fam^\circ_{L_b} (I_1,J)$};

\end{tikzpicture}
\]
\end{snakelmm}

\noindent Note that~\eqref{eq:SnakeLmm} implies  
\begin{equation}
\label{eq:SnakeLmm:2} \max\{\Width^\circ_{L_a} (I,J_1), \Width^\circ_{L_b} (I_1,J)\}\ge 2K -O_\bdelta(\log \lambda).
\end{equation}

\begin{cor}
\label{cor:SnakeLmm}
Under the assumption of Lemma~\ref{simplmm:SnLmm:wZ}, there is an interval $I^\new\subset L^\bullet\subset \partial Z$ grounded rel $\wZ^m$ such that $\Width^+_\lambda(I^\new)\succeq K$, where $L^\bullet$ is the projection of $L$ onto $\partial Z$.
\end{cor}

\subsubsection{Outline and motivation of the section} \label{sss:outline of snakes}  The proof of Lemma~\ref{simplmm:SnLmm:wZ} repeats the argument of Lemma~\ref{simplmm:SnLmm:Z} (the Snake Lemma for $Z$) with an additional input from Lemma~\ref{lem:I vs I^grnd}. More precisely, the Series Decomposition~\S\ref{sss:SerDecomp for F circ} yields families $\Fam^\circ_{L_a} (I,J_1)$ and $\Fam^\circ_{L_b} (I_1,J)$ shown on the figure in Lemma~\ref{simplmm:SnLmm:wZ}. By Localization Lemma~\ref{lem:trading  width to space}, we can assume that $|J_1|, |I_1|$ are small compared with $\dist(I,J_1),\dist(I_1,J)$ respectively. And by Lemma~\ref{lem:I vs I^grnd}, we can assume that $J_1,I_1$ are grounded rel $\wZ^m$; i.e., Snake Lemma~\ref{simplmm:SnLmm:wZ} can be iterated.

This allows to trade families entering $\intr \wZ^m$ into outer families (Corollary~\ref{cor:SnakeLmm}). Indeed, assuming that $\Width^\circ_{L_a} (I,J_1)\ge 2K -O_\bdelta(\log \lambda)$, either $\Width^+_{L_a} (I,J_1)\ge K/3$ or repeating the Snake Lemma we find $I_2,J_2$ with $\Width^\circ_{L_2} (I_2,J_2)\ge \frac 3 2  K $ enlarging the family. Since $\wZ^m$ is a non-uniformly qc disk, the process eventually stops.

For applications, we will need several variations of the Snake Lemma. In~\S\ref{ss:SL:barriers} we state the Snake Lemma ``with toll barriers'': if $\Fam^\circ_L(I,J)$ contains a lamination $\RR$ submerging into $L$ at least $n$ times, then $2K$ in~\eqref{eq:SnakeLmm:2} can be replaced by $nK$. In~\S\ref{lem:sneaking} we state the Sneaking Lemma when $\Fam^\circ(I,J)$ ``sneaks'' through a wide outer rectangle $\RR$, see Figure~\ref{Fg:snake thorugh WR} for illustration. Both versions will be used in Snake-Lair Lemma~\ref{lem:Hive Lemma} to amplify the width of degeneration.


\subsection{Proof of Snake Lemma~\ref{simplmm:SnLmm:wZ}} We will prove the Snake Lemma in \S\ref{sss:Prf SnaekLmm} after introducing an auxiliary subfamily $\Fam^\str_L(I,J)\subset \Fam^\circ_L(I,J)$. The family $\Fam^\str_L(I,J)$ consists of curves omitting channels (and some space around them) that have an endpoint in $L^c$. We will show that $\Width^\circ_L(I,J)- \Width^\str_L(I,J)=O_\bdelta(1)$ and that at most $O_\bdelta(1)$-curves in $\Fam^\str_L(I,J)$ intersects $X\setminus X^\grnd$ for every interval $X\subset L$.

\subsubsection{$\Fam^\str$-family}
\label{ss:Fam^*} Let us fix an interval $L\subset\partial \wZ^m$. For a channel $\alpha=\alpha^n_i\subset \wZ^m$ with $|\alpha \cap L^c|\ge 1$ (i.e., at least one of the endpoints of $\alpha$ is in $L^c$), define
\begin{itemize}
\item $\alpha^{\str L^c}\coloneqq \alpha$ if $|\alpha\cap L^c|=2$;
\item $\alpha^{\str L^c}$ to be the connected component of $A^\inn(\alpha)\setminus L^c$ intersecting $\intr \wZ^m$ if $|\alpha\cap L^c|=1$ (i.e., this connected component is attached to $(L^c)^-$). 
\end{itemize}
We define \[\big(L^c\big)^{\str}\coloneqq \text{ filling in of } \left[ L^c\cup \bigcup _{|\alpha\cap L^c|\ge 1} \alpha^{\str L^c}\right].\]
We remark that if $\length_{m_{i+1}}\le |L|< \length_{m_i}$, where the $m_i$ are level of regularizations, then $\wZ^m\setminus \big(L^c\big)^{\str}$ is within the level $m_i$ peninsula containing $L$ on its boundary.

Finally, we define $\Fam^\str_L(I,J)$ to be the set of curves in $\Fam^\circ_L(I,J)$ that are in $\wC\setminus \intr (L^c)^\str$.

\begin{lem}[Trading $\Fam^\circ $ into $ \Fam^\str$]
\label{lem:A|_L} Under the assumption of Lemma~\ref{simplmm:SnLmm:wZ}, the family $\Fam^\str_L(I,J)\setminus \Fam^+(I,J)$ contains a rectangle $\RR$ such that
\[ \Width(\RR)= \Width^\circ_L(I,J)-\Width^+(I,J)-O_\bdelta(1).\]

\end{lem}

\begin{proof} Write $K\coloneqq\Width^\circ_L(I,J)-\Width^+(I,J)$. By Lemma~\ref{lem:Fam^circ:R}, $\Fam^\circ_L(I,J)$ contains a rectangle $\RR$ submerging into $\wZ^m$ with $\Width(\RR)=K-O(1)$. Let $\RR^\new$ be the rectangle obtained by removing the $1/\bdelta$ innermost buffer (attached to $(L^c)^-$) from $\RR$. We claim that $\RR\subset \wC \setminus \big(L^c\big)^{\str}$.  

Assume first that $|\alpha\cap L^c|=2$. Since $\alpha=\alpha^{\str L^c}$ is attached to $(L^c)^-$, vertical curves in $\RR$ intersecting $\alpha$ form a buffer of $\RR$, see Lemma~\ref{lem:buffer:R O}. Since curves in this buffer cross $A^\inn$ after entering $\intr \wZ^m$ through $L$, the width of the buffer is $\le 1/\delta$.  

Assume now that $|\alpha\cap L^c|=1$. We {\bf claim} that $A^\out(\alpha)$ is disjoint from either $I$ or $J$. The claim will imply that the width of the buffer formed by curves intersecting $\alpha^{\str L^c}$ (these curves form a buffer of $\RR$ by Lemma~\ref{lem:buffer:R O}) is $\le 1/\bdelta$ because $\alpha^{\str L^c}$ is separated from either $I$ or $J$ by  $A^\out(\alpha)$.

\begin{proof}[Proof of the claim] Write $\alpha=\alpha^n_i=[y_{n-1},x_n]$ as in Figure~\ref{Fg:wZ_Z_m} and assume that $y_{n-1}\in L^c$; the  opposite case is analogous.

 Suppose first that $x_n\in \intr \wZ^n$. Then the unique point $z$ in $\partial^c\Penin (\alpha^n_i)\cap \CP_n$ is an endpoint of a channel $\beta^k_s\subset \partial \wZ^m$ for $k<n$, see~\S\ref{sss:Fjord Penin}. Therefore, $A^\out(\alpha^n_i)\cap \partial \wZ^m\subset S^\inn_z$. Since $I, J$ are grounded, $ S^\inn_z$ (and hence $A^\out(\alpha^n_i)$) is disjoint from either $I$ or $J$.
 
 Suppose now that $x_n\in L$. Since $L$ is grounded, it also contains $\beta^n_{i}$ -- the dam after $\alpha^n_i$. We obtain that $\beta^n_{i}$ separates $S_{x_n}$ from $J$ while $N$ separates $S_{y_{n-1}}$ from $J$ because $|N|\ge \length_m$. Therefore, $A^\out(\alpha^n_i)$ separates $\widetilde A^\inn(\alpha^n_i)$ from $J$.
\end{proof}
\end{proof}

\begin{lem}[$X$ vs $X^\grnd$]
\label{lem:I vs I^grnd} Consider a family $\Fam^\str_L(I,J)$ from Lemma~\ref{lem:A|_L}. Then for every interval $X\subset L$, there are at most $O_\bdelta(1)$ curves in $\Fam^\str_L(I,J)$ intersecting ${X\setminus X^\grnd}$. 
\end{lem}
\noindent Combined with Lemma~\ref{lem:A|_L}, there are at most $O_\bdelta(1)$ curves in $\Fam^\circ (I,J)$ intersecting ${X\setminus X^\grnd}$. 
\begin{proof}
Let us start the proof with the following two properties.
\begin{claim6} 
\label{cl6:1}
Let $\alpha$ be a channel such that either both endpoints of $\alpha$ are in $L$ or one of the endpoints of $\alpha$ is in $L$ and the second endpoint is in $\intr \wZ^m$. Then at most $O_\bdelta(1)$ curves in $\Fam^\str_L(I,J)$ intersect $\widetilde A^\inn(\alpha)$. 
\end{claim6}
\begin{proof}[Proof of the claim]
Assume first that both endpoints of $\alpha$ are in $L$. Then we have $\partial^c\Penin(\alpha)\subset L$. Since $L$ is grounded, it also contains two dams attached to $\alpha$. Therefore, $\widetilde A^\out(\alpha)$ is disjoint from $I\cup J$. The claim now follows from $\mod  A^\out(\alpha)\ge \bdelta$.

Assume that one of the endpoints of $\alpha$ is in $\intr Z$. Then the unique point $z$ in $\partial^c\Penin (\alpha^n_i)\cap \CP_n$ is an endpoint of a dam $\beta^k_s\subset \partial \wZ^m$ for $k<n$, see~\S\ref{sss:Fjord Penin}. We have $A^\out(\alpha^n_i)\cap \partial \wZ^m\subset S^\inn_z$ and $A^\out(\alpha^n_i)$ separates $\widetilde A^\inn(\alpha)$ from $I,J$.
\end{proof}

\begin{claim6} 
\label{cl6:2}
Let $\beta\subset L$ be a dam. Then at most $O_\bdelta(1)$ curves in $\Fam^\str_L(I,J)$ intersect $ S^\inn(\beta)$. 
\end{claim6}
\begin{proof}[Proof of the claim]
Since $L$ is grounded, $A^\inn(\beta)$ separates $\beta$ from $I\cup J$; thus at most $O_\bdelta(1)$ curves in $\Fam^\str_L(I,J)$ intersect $\beta$. Write $\beta=\beta^n_i$; then $S^\inn(\beta)\subset S^{\inn}(\alpha^n_{i})\cup S^{\inn}(\alpha^n_{i+1})$. By Lemma~\ref{lem:A|_L} and Claim~\ref{cl6:2}, at most $O_\bdelta(1)$ curves in $\Fam^\str_L(I,J)$ intersect $S^{\inn}(\alpha^n_{i})\cup  S^{\inn}(\alpha^n_{i+1})$.
\end{proof}

The lemma now follows from Claim~\ref{cl6:2} because every component (out of at most $2$) of ${X\setminus X^\grnd}$ is within $ S^\inn(\beta)$ for a dam $\beta\subset L$.
 \end{proof}

\subsubsection{Proof of Snake Lemma~\ref{simplmm:SnLmm:wZ}}
\label{sss:Prf SnaekLmm} Let $\RR\subset\Fam_L^\str(I,J)\setminus \Fam^\circ= (I,J)$ with $\Width(\RR)=K-O_\bdelta(1)$ be a rectangle (a snake) from Lemma~\ref{lem:A|_L} realizing $K$.  Applying Series Decomposition~\S\ref{sss:SerDecomp for F circ} to $\RR$, we obtain that $\Fam(\RR)$ consequently overflows the laminations 
 \begin{equation}
 \label{eq:SL:prf:SerDecomp}
 \Fam_a\subset \Fam^\circ (I,J_a),\sp \sp \Gamma\subset \Fam^-(J_a,I_b),\sp \sp \Fam_b\subset \Fam^\circ (I_b, J),
 \end{equation}
where $J_a, I_b\subset L$. Let $J_a^\grnd, I_b^\grnd$ be the biggest grounded intervals in $J_a, I_b$, see~\S\ref{ss:Bett Grnd Inter}. By Lemma~\ref{lem:I vs I^grnd},  the width of vertical curves in $\RR$ intersecting $\big(J_a\setminus J^\grnd_a\big)\cup \big(I_b\setminus I^\grnd_b\big)$ is $O_\bdelta(1)$; removing these curves from $\RR$ and their restrictions from the laminations in~\eqref{eq:SL:prf:SerDecomp},  we obtain that the new rectangle $\RR^\new, \Width(\RR^\new)\ge K-O_\bdelta(1)$ such that  $\Fam(\RR^\new)$ consequently overflows
 \begin{equation}
 \label{eq:SL:prf:SerDecomp:2}
 \Fam^\new_a\subset \Fam^\circ (I,J^\grnd_a),\sp \sp \Gamma^\new\subset \Fam^-(J^\grnd_a,I^\grnd_b),\sp \sp \Fam^\new_b\subset \Fam^\circ (I^\grnd_b, J).
 \end{equation}  
 
 By Localization and Squeezing Lemmas~\ref{lem:trading  width to space},~\ref{lem:squeezing}, $J^\grnd_a, I^\grnd_b$ contains an innermost pair $J_1,I_1$ such that \[ |\lfloor J_1, I_1\rfloor |\le \frac 1 {5\lambda} \{|I^\grnd_a|,\ |J^\grnd_b|\}\]
and up to $O_\bdelta(\log \lambda)$-width the family $\Fam^-(J_a,I_b)$ is in  $\Fam^-(J_1,I_1)$:
\begin{equation}
\label{eq:SL:prf:Small Interv}
\Width^-(J_a\setminus J_1,\sp I_b) +\Width^-(J_a,\sp I_b\setminus I_1)=O_\bdelta(\log \lambda).
\end{equation}
 
Let $\RR^\New$ be the lamination obtained from $\RR^\new$ by removing all $\gamma\in \Fam(\RR)$ with $\gamma_a^d\not \in \Fam^-(J_1,I_1)$ or  $\gamma_b^d\not \in \Fam^-(J_1,I_1)$.  Then $\Width(\RR^\New)=K-O_\bdelta(\log \lambda)$. 

Applying the Series Decomposition~\S\ref{sss:SerDecomp for F circ}  
to $\RR^\New$, we obtain that $\RR^\New$ consequently overflows 
 \begin{equation}
 \label{eq:SL:prf:SerDecomp:3}
\Fam^\New_a\subset \Fam^\circ (I,J^\New_a),\sp \sp \Gamma^\New\subset \Fam^-(J^\New_a,I^\New_b),\sp \sp \Fam^\New_b\subset \Fam^\circ (I^\New_b, J),
\end{equation}
where $J^\New_a\subset J_1$, $I^\New_b\subset I_1$ and $\lfloor J_a^\New, I^\New_b\rfloor\subset \lfloor J_1, I_1\rfloor$. Set $J_2\coloneqq \big(J^\New_a\big)^\grnd$ and $I_2\coloneqq \big(I^\New_b\big)^\grnd$. By Lemma~\ref{lem:I vs I^grnd}, the width of curves in $\RR^\New$ intersecting $J^\New_a\setminus J_2$ or $I^\New_b\setminus I_2$ is at most $O_\bdelta(1)$. Removing these curves from $\RR^\New$ and their restrictions from the laminations in~\eqref{eq:SL:prf:SerDecomp:3}, we obtain that the new $\RR^\NEW$ consequently overflows
\begin{equation}
 \label{eq:SL:prf:SerDecomp:4}
\Fam^\NEW_a\subset \Fam^\circ (I,J_2),\sp \sp \Gamma^\NEW\subset \Fam^-(J_2,I_2),\sp \sp \Fam^\NEW_b\subset \Fam^\circ (I_2, J),
\end{equation}
where $\lfloor J_2,I_2\rfloor \subset \lfloor J_1, I_1\rfloor.$ Therefore,~\eqref{eq:SL:prf:Small Interv} and~\eqref{eq:SL:prf:SerDecomp:4} imply the lemma.

\begin{rem}
\label{rem:SL:steps}
Let us summarize the steps in the proof of Lemma~\ref{simplmm:SnLmm:wZ}: 
\begin{enumerate}[label=\text{(\alph*)},font=\normalfont,leftmargin=*]
\item\label{rem:SL:steps:a} first we apply Series Decomposition~\S\ref{sss:SerDecomp for F circ} to $\Fam(\RR)$, see~\eqref{eq:SL:prf:SerDecomp};
\item\label{rem:SL:steps:b} then we apply Lemma~\ref{lem:I vs I^grnd} to obtain grounded intervals $J_a^\grnd, I_b^\grnd$, see~\eqref{eq:SL:prf:SerDecomp:2};
\item\label{rem:SL:steps:c} then we apply  Localization Lemma~\ref{lem:trading  width to space} to $J_a^\grnd, I_b^\grnd$;
\item\label{rem:SL:steps:d} then we reapply the Series Decomposition~\S\ref{sss:SerDecomp for F circ}, see~\eqref{eq:SL:prf:SerDecomp:3}; 
\item \label{rem:SL:steps:e}finally,  we reapply Lemma~\ref{lem:I vs I^grnd} to obtain required~\eqref{eq:SL:prf:SerDecomp:4}. 
\end{enumerate}
In Steps~\ref{rem:SL:steps:b},~\ref{rem:SL:steps:c},~\ref{rem:SL:steps:e} we remove at most $O_\bdelta(\log \lambda)$ curves from $\RR$.
\end{rem}

\subsubsection{Submerging laminations} We can refine the Snake Lemma as follows:

\begin{lem}
\label{lem:SL:subm rect}
Under the assumptions of Lemma~\ref{simplmm:SnLmm:wZ}, consider a lamination $\RR\subset \Fam^\circ_L (I,J)\setminus \Fam^+(I,J)$ with $\Width(\RR)\gg_\bdelta \log \lambda$. Then there are intervals $J_1,I_1$ satisfying \eqref{eq:SnakeLmm:intervals} and there are laminations $\RR_a\subset \Fam_{L_a}(I,J_1),$ $\RR_b\subset \Fam^\circ_{L_b} (I_1,J)$ such that \[\Width(\RR_a)\oplus \Width(\RR_b)\ge \Width(\RR)-O_\bdelta(\log \lambda),\]
and such that $\RR_a,\RR_b$ are restrictions of sublaminations of $\RR$.
\end{lem}
\begin{proof}
By Lemma~\ref{lem:A|_L}, we can remove $O_\bdelta(1)$ buffers from $\RR$ so that the new lamination $\RR^\new$ is in $\Fam^\str_L(I,J)$. We apply the argument of~\S\ref{sss:Prf SnaekLmm} to $\RR^\new$ until~\eqref{eq:SL:prf:SerDecomp:4}, and then we set $\RR_a\coloneqq \Fam^\NEW_a$ and $\RR_b\coloneqq \Fam^\NEW_b$.
\end{proof}

\begin{rem} \label{rem:SL:ShallowScale} 
The condition $|N|\ge \length_{m}$ in Snake Lemmas~\ref{simplmm:SnLmm:wZ}~\ref{lem:SL:subm rect} can be omitted if the lamination $\RR$ in Lemma~\ref{lem:SL:subm rect} has the following property:
\begin{enumerate}[label=\text{(\Alph*)},start=24,font=\normalfont,leftmargin=*]
\item\label{property:X} for every interval $X\subset \partial \wZ_\ell$, there are at most $O_\bdelta(1)$ vertical curves in $\RR$ intersecting $X\setminus X^\grnd$.
\end{enumerate}
In the proof, Property~\ref{property:X} substitutes Lemma~\ref{lem:I vs I^grnd}.
\end{rem}

\subsection{Trading $\Width$ into $\Width^+$} In this subsection we will prove Corollary~\ref{cor:SnakeLmm} as well as several of its variations.
\begin{proof}[Proof of Corollary~\ref{cor:SnakeLmm}]
Snake Lemma~\ref{simplmm:SnLmm:wZ} implies \eqref{eq:SnakeLmm:2}. Assume that $\Width^\circ_{L_a}(I,J_1)\ge 2K - O_\bdelta(\log \lambda)$; the second case is analogous.

If $\Width^+(I,J_1)\ge \frac 13 K$, then  we set $I^\new=J_1^\bullet$ to be the projections of $J_1$ onto $\partial Z$. By Lemma~\ref{lem:W+:ground inter} we have 
\[  \Width^+_{\lambda,  Z}(I^\new) \ge \Width_{ Z}^+(I^\bullet, J^\bullet_1) \ge \Width^+_{ \wZ^m}(I,J_1)-O_\bdelta(1)\succeq K.\]

If $\Width^+_{L_a}(I,J_1)< \frac 5 3 K$, then applying Snake Lemma~\ref{simplmm:SnLmm:wZ}  again, we find $I_2,J_2\subset L$ with \[\Width^\circ_{L_2} (I_2,J_2)\ge \frac 3 2  K \sp\sp \text{ and }\sp\sp  \min\{|I_2|,|J_2|\} < \frac 1 \lambda \dist(I_2,J_2).\] 

The case $\Width^+(I_2,J_2)\ge \frac 13 K$ is treated as above. If  $\Width^+(I_2,J_2)< \frac 13 K$, then applying Snake Lemma~\ref{simplmm:SnLmm:wZ} again, we find $I_3,J_3\subset L$ with \[\Width^\circ_{L_3} (I_3,J_3)\ge \left(\frac 32\right)^2 K \sp\sp \text{ and }\sp\sp  \min\{|I_3|,|J_3|\} < \frac 1 \lambda \dist(I_3,J_3),\]
i.e., $\Width^\circ_{L_n}(I_n,J_n)\ge \big(3/2\big)^n K$ growth exponentially fast. Since $\wZ^m$ is a non-uniformly qc disk, the process eventually stops: we obtain $\Width^+(I_n, J_n)\ge \frac 1 3K $ for some $n$ and grounded intervals $I_n,J_n$ with $\min\{|I_n|,|J_n|\}<\frac 1 \lambda\dist(I_n,J_n)$. Lemma~\ref{lem:W+:ground inter} allows to replace $I_n,J_n$ with their projections onto $\partial Z$.
\end{proof}

\subsubsection{Scale $\ge \length_m$} Trading $\Width$ into $\Width^+$ is more straightforward if intervals have length $\ge \length_m$ thanks to Lemma~\ref{lem:I vs I^grnd:with separ}.

\begin{lem}
\label{lem:trad width to width+}
Let $I,  L\subset \partial \wZ$ be disjoint intervals such that \[|I|\ge \length_m, \sp \sp \dist( I, L)\ge \length_{m+1}, \sp\sp |L|\ge \frac 12, \sp\sp |L^c|\le \lambda \length_{m+1},\]
\[K\coloneqq \Width(I, L) -\Width^+(I,L) \gg_{\bdelta,\lambda}1,\sp\sp \sp\text{ where  }\sp\sp \lambda \ge 3. \]  
Then there is an interval $I^\new\subset L^c\setminus I\subset \partial Z$ grounded rel $\wZ^m$ such that
\[|I^\new|<\frac 1 \lambda |I|\sp\sp \text{ and }\sp\sp \Width_\lambda^+(I^\new)\succeq K.\]
\end{lem}
\begin{proof}

Let $\RR$ be the vertical family of $\Fam(I,L)$, see~\S\ref{sss:can rect}. By Lemma~\ref{lem:I vs I^grnd:with separ}, Property~\ref{property:X} (see Remark~\ref{rem:SL:ShallowScale}) holds for $\RR$.

Let $\widetilde I, \widetilde L$ be slight enlargements of $I^\grnd,L^\grnd$ such that every interval in $\widetilde I\setminus I^\grnd, \widetilde L\setminus L^\grnd$ has length $\frac{1}{2\lambda }|I|$. Applying Lemma~\ref{lmm:W into W(I+,J+)}, we construct intervals $\widehat I, \widehat L$ with $I\subset \widehat I \subset \widetilde I$ and $L\subset \widehat L \subset \widetilde L$ such that there is a restriction $\FamG\subset \Fam(\widehat I^+, \widehat L^{+})$ of a sublamination of $\RR$ satisfying
\begin{itemize}
\item $\Width(\RR|\FamG)= \Width(I,L)- O(C)= K - O_\bdelta(\log \lambda)$;
\item $\FamG$ is disjoint from the central arc in $\Fam^-(I,L)$.
\end{itemize}
Here $C=O_\bdelta(\log \lambda)$ by Lemmas~\ref{lem:trading  width to space} and~\ref{lem:I vs I^grnd:with separ}.

Let $A,B$ be the complementary intervals to $\widehat I, \widehat L$, and let  $I_a,I_b$ be two intervals (possibly empty) in $\widehat I\setminus I$. Since $\Width(I,L)-\Width^+(I,L)=K$, there are two possibilities:
\begin{enumerate}[label=\text{(\Roman*)},font=\normalfont,leftmargin=*]
\item \label{cond:prf:lem:W to W^+:1} either $\FamG$ contains a lamination $\FamH$ in $\Fam^+\big(I_a, \widehat L\big)\cup \Fam^+\big(I_b, \widehat L\big)$ with $\Width(\RR|\FamH)\ge K/3$, see~\S\ref{sss:Restr of Lam}; 
\item \label{cond:prf:lem:W to W^+:2} or $\FamG$ contains a lamination $\FamH$ intersecting $A\cup B$ with $\Width(\RR|\FamH)\ge K/3$, where $A,B$ are two intervals of $\partial \wZ^m\setminus \big( \widehat I\cup \widehat L \big)$.
\end{enumerate}

Case~\ref{cond:prf:lem:W to W^+:1} follows from Lemma~\ref{lem:W+:ground inter} and Property~\ref{property:X} by defining $I^\new$ to be the projection of either $I^\grnd_a$ or $I^\grnd_b$ onto $\partial Z$.

Consider Case~\ref{cond:prf:lem:W to W^+:2}. Let $\FamH_A$ be the lamination of curves in $\FamH$ intersecting $A$ before intersecting $B$. Similar, $\FamH_B\subset \FamH$ consists of curves intersecting $B$ before $A$. We have $\FamH=\FamH_A\sqcup \FamH_B$. Below we assume that $\Width(\FamH_A)\ge K/6$; the case $\Width(\FamH_B)\ge K/6$ is analogous.

Since $\FamH_A$ is disjoint from the central arc in $\Fam^-(I,J)$, we can restrict $\FamH_A$ to the lamination $\FamP$ in $\Fam^\circ_A \big(\widehat I , \widehat L \cup B \big)$; i.e.,~$\FamP$ consists of the first shortest subcurve $\gamma'$ of $\gamma\in \FamH_A$ such that $\gamma'$ connects $\widehat I^+$ to $\big(\widehat J\cup B\big)^+$.

Because of Property~\ref{property:X}, Snake Lemma~\ref{simplmm:SnLmm:wZ} is applicable to $\PP$, see Remark~\ref{rem:SL:ShallowScale}. Therefore, there are grounded intervals $J_1\subset A,L_a$ such that $\Width^\circ_{L_a} (\widehat I^\grnd,J_1)\succeq K$ for and $|J_1|<\lambda \dist(J_1,\widehat I^\grnd)$. Corollary~\ref{cor:SnakeLmm} applied to $\Fam^\circ_{L_a} (I,J_1)\succeq K$ finishes the proof.
\end{proof}

\subsection{Rectangles crossing pseudo-bubbles} \label{sss:RR through D} Consider a rectangle $\RR$ and a closed topological disk $D$ such that $\partial^h\RR\subset \wC\setminus  D$. Assume that all vertical curves in $\RR$ intersect $D$. We denote by $x,y\subset \partial D$ the first intersections of $\partial ^{v, \ell}\RR,\partial ^{v, \rho}\RR$ with $D$; and let $I=[x,y]\subset \partial D$ be an interval with endpoints $x,y$. We say that $\RR$ \emph{crosses $D$ through $I$} if for every $\gamma\in \Fam(\RR)$
\begin{itemize}
\item the first intersection of $\gamma$ with $D$ is in $I$; and
\item the last intersection of $\gamma$ with $D$ is in $I^c=\partial D\setminus I$.
\end{itemize}

\begin{lem}
\label{lem:rect rhrough PB}
Assume that a rectangle $\RR, \ \Width(\RR)=K$ crosses a pseudo-bubble $\wZ_\ell$ (see~\S\ref{ss:Ps-bubles}) through $I\subset \partial \wZ_\ell$. Assume also that either $\partial^{h,0}\RR$ or $\partial^{h,1} \RR$ is disjoint from $\XX(\wZ_\ell)$, see~\eqref{eq:dfn XX(wZ)}. Then one of the following holds for every $\lambda >2$.
\begin{enumerate}[label=\text{(\Roman*)},font=\normalfont,leftmargin=*]
\item\label{c1:lem:rect rhrough PB} There is a grounded interval $B\subset \left[\big(1+ \lambda^{-2}\big)I\right]\setminus I\subset \partial \wZ_\ell$ and there is a sublamination $\widetilde \FamQ\subset \Fam(\RR)$ with $\Width(\widetilde \FamQ)\succeq K-O_\bdelta(\log \lambda)$ such that the restriction (see~\S\ref{sss:short subcurves}) $\FamQ$ of $\widetilde \FamQ$ to the family from $\partial D$ to $\partial^{h,1}\RR$ starts in $B$.
\item\label{c2:lem:rect rhrough PB} There is a grounded interval $B\subset \left[\big(1+ \lambda^{-2}\big)I\right]\setminus I\subset \partial \wZ_\ell$ and there is a lamination $\FamQ\subset \Fam^+\left(B, \big[(\lambda B)^c \big]^\grnd\right)$ such that $\Width(\FamQ)\succeq K-O_\bdelta(\log \lambda)$ and such that $\FamQ$ is a restriction (see~\S\ref{sss:short subcurves}) of a sublamination of $\RR$.
\end{enumerate}
\end{lem}
\begin{proof}
Since either $\partial^{h,0}\RR$ or $\partial^{h,1} \RR$ is disjoint from $\XX(\wZ_\ell)$, we have
\begin{enumerate}[label=\text{(\Alph*)},start=24,font=\normalfont,leftmargin=*]
\item \label{prop:X vs X^grnd:prf}for every interval $X\subset \partial \wZ_\ell$, there are at most $O_\bdelta(1)$ vertical curves in $\RR$ intersecting $X\setminus X^\grnd$.
\end{enumerate}
(This is Property~\ref{property:X} of Remark~\ref{rem:SL:ShallowScale}.) Write \[\widetilde I \coloneqq \left[\big(1+ \lambda^{-5}\big)I\right]^\grnd\sp\sp \text{and }\sp\sp N\coloneqq \left[\partial \wZ_\ell\setminus \big(1+ \lambda^{-5}\big)I\right]^\grnd.\] Let $\Fam^\new$ be the lamination obtained from $\Fam(\RR)$ by removing all $\gamma\in \Fam(\RR)$ that have a subarc in $\wZ_\ell$ connecting $I$ and $\widetilde I^c$. By Property~\ref{prop:X vs X^grnd:prf}, the width of removed curves is bounded by $\Fam^-(I^\grnd, N)$, and by Squeezing Lemma~\ref{lem:squeezing}, $\Fam^-(I^\grnd, N)=O_\bdelta(\log \lambda).$ Therefore, $\Width(\Fam^\new)\ge K-O_\bdelta(\log \lambda)$.

Following notations of~\S\ref{sss:SerDecomp for F circ}, for every $\gamma\in \Fam^\new$, let  
\begin{itemize}
\item $\gamma_b^d\subset \wZ_\ell$ be the last subarc  of $\gamma$ connecting $I$ to $\widetilde I\setminus I$; and
\item $\gamma_b$ be the subarc of $\gamma$ after $\gamma^d_b$.
\end{itemize}

We set \[\Gamma_b\coloneqq\{\gamma_b^d\mid \gamma\in \Fam^\new\}\sp\sp\text{ and }\sp\sp \widetilde \Fam_b\coloneqq \{\gamma_b\mid \gamma\in \Fam^\new\},\]
where $\Width\big(\RR|\widetilde \Fam_b\big)\succeq K-O_\bdelta(\log \lambda)$ following conventions of \S\ref{sss:Restr of Lam}. Let $\widetilde I^\new\subset \widetilde I$ be the subinterval bounded by the leftmost and rightmost endpoints of curves in $\Gamma_b$. And let $\Fam_b$ be the restriction (see~\S\ref{sss:short subcurves}) of $\widetilde \Fam_b$ to the family of curves from $\widetilde I^\new$ to $\partial^{h,1} \RR$. Since $\Gamma_b$ is a lamination, every curve in $\Fam_b$ starts in $\left[\widetilde I^\new\setminus I \right]^+$.

If a sufficient part of $\Fam_b$ is outside of $\wZ_\ell$, then we obtain Case~\ref{c1:lem:rect rhrough PB} of the lemma (we apply Property~\ref{prop:X vs X^grnd:prf}  to construct  a grounded interval $B\subset \widetilde I^\new\setminus I$). 

Assume converse. Then using Property~\ref{prop:X vs X^grnd:prf} we find a grounded interval $J\subset \widetilde I^\new\setminus I$  and a sublamination $\FamH$ of $\Fam_b$ with $\Width(\RR|\FamH)\succeq K-O_\bdelta(\log \lambda)$ such that curves in $\FamH$ start in $J^+$ and every $\gamma\in \FamH$ intersects $\partial \wZ_\ell\setminus J$. 

Let $\FamH_1$ be the sublamination of $\FamH$ consisting of $\gamma\in \FamH$ such that the first intersection of $\gamma$ with $\partial \wZ_\ell\setminus J$ is in $\left[(\lambda J)^c\right]^\grnd$. And we set $\FamH_2\coloneqq \FamH\setminus \FamH_1$.

If $\Width(\RR|\FamH_1)\ge \Width(\RR|\FamH_2)$, then the Case~\ref{c2:lem:rect rhrough PB} of the lemma is obtained by restricting $\FamH_1$ to the family $\Fam^+\left(J, \big[(\lambda J)^c \big]^\grnd\right)$.

Assume that $\Width(\RR|\FamH_1)\le \Width(\RR|\FamH_2)$. Let $J_a,J_b$ be the connected components of $(\lambda J)^\grnd\setminus J^\grnd$ with $J_a<J<J_b$ in $\lambda J$. Let 
\begin{itemize}
\item $\FamH_a$ be the set of curves in $\FamH_2$ intersecting $J_a$ before intersecting $(J_a)^c$; and 
\item $\FamH_b$ be the set of curves in $\FamH_2$ intersecting $J_b$ before intersecting $(J_b)^c$.
\end{itemize} By Property~\ref{prop:X vs X^grnd:prf}, at most $O_\bdelta(1)$ curves in $\RR$ intersect $(\lambda J)\setminus (\lambda J)^\grnd$ before intersecting $(\lambda J)^\grnd\cup (\lambda J)^c$. Therefore, either $\Width(\RR|\FamH_a)\succeq K -O_\bdelta(\log \lambda)$ or  $\Width(\RR|\FamH_b)\succeq K -O_\bdelta(\log \lambda)$. Below we will assume that $\Width(\RR|\FamH_b)\succeq K -O_\bdelta(\log \lambda)$; the case of $\FamH_a$ is analogous.

 Let $\rho\in \FamH_b$ be the curve with the rightmost starting point in $J$, let $z$ be the first intersection of $\rho$ with $J_b$, and let $J'\subset J_b$ be the subinterval of $J_b$ between $J$ and $z$. Following notations of~\S\ref{sss:SerDecomp for F circ}, for every $\gamma\in \FamH_b$, let  
\begin{itemize}
\item $\gamma_a^d\subset \wZ_\ell$ be the first subarc  of $\gamma$ connecting $J'$ to $(J\cup J')^c$; and
\item $\gamma_a$ be the subarc of $\gamma$ before $\gamma^d_a$.
\end{itemize}
We set \[\Gamma_a\coloneqq\{\gamma_a^d\mid \gamma\in \FamH_b\}\sp\sp\text{ and }\sp\sp \widetilde \PP\coloneqq \{\gamma_a\mid \gamma\in \FamH_b\},\]
where $\Width\big(\RR|\widetilde \PP\big)\succeq K-O_\bdelta(\log \lambda).$ Let $N\subset J'$ be the subinterval bounded by the leftmost starting point of curves in $\Gamma_a$ and $z$. We now apply Steps~\ref{rem:SL:steps:b},~\ref{rem:SL:steps:c},~\ref{rem:SL:steps:d},~\ref{rem:SL:steps:e} (see Remark~\ref{rem:SL:steps}) of~\S\ref{sss:Prf SnaekLmm}, where Property~\ref{prop:X vs X^grnd:prf} substitutes Lemma~\ref{lem:I vs I^grnd}, to localize $N$ and to contract a lamination \[\FamQ\subset \Fam^\circ(J,N),\sp\sp |N|<\frac{1}{2\lambda} \dist(J,N),\sp\sp \Width(\RR| \FamQ)\succeq K-O_\bdelta(\log \lambda).\] 
Applying iteratively Lemma~\ref{lem:SL:subm rect} and repeating the argument of Corollary~\ref{cor:SnakeLmm}, we can replace $\FamQ$ with a required $\FamQ^\new$ in the complement of $\intr \wZ_\ell$.
\end{proof}
\subsection{Snakes with barriers}
\label{ss:SL:barriers}
 Consider a grounded pair $I,J\subset \wZ^m$ with $|J|> 1/2$. Let $A,B$ be two complementary intervals between $I$ and $J$.  We assume that the intervals are clockwise oriented as $A<I<B<J$.

Let $\ell_1,\ell_2,\dots, \ell_n$ be pairwise disjoint simple arcs in  $\C\setminus \wZ^m$ such that every $\ell_i$ connects $a_i\in A$ and $b_i\in B$ with the orientation $a_i<a_{i-1}<b_{i-1}<b_i$ for every $i>1$. We say that $\ell_1,\dots, \ell_n$ are \emph{barriers} for a lamination $\RR\subset \Fam(I,J)$ if no curves in $\RR$ intersect $\bigcup _{i=1}^n\ell_i$.

We say that a curve $\gamma\in \RR$ 
\begin{itemize}
\item \emph{skips under $[a_i, a_{i-1}]$} if $\gamma\cap \wZ^m$ contains a subcurve connecting two different components of $A\setminus [a_i, a_{i-1}]$;
\item \emph{skips under $[b_{i-1}, b_{i}]$} if $\gamma\cap \wZ^m$ contains a subcurve connecting two different components of $B\setminus [b_{i-1}, b_{i}]$;
\item  \emph{skips under $[a_i, a_{i-1}]\cup [b_{i-1}, b_{i}]$} if $\gamma$ skips under $[a_i, a_{i-1}]$ or under $[b_{i-1}, b_{i}]$.
\end{itemize}

\begin{defn}[Toll barriers] 
\label{defn:barriers}
Let $\RR \subset \Fam(I,J)$  be a lamination with barriers $\ell_1,\dots, \ell_n$ as above,  where $I,J\subset \partial \wZ^m$ is a grounded pair with $|J| >1/2$. Then $\ell_1,\dots, \ell_n$ are \emph{toll barriers for $\RR$} if for all $i$, no curves in $\RR$ skip under $[a_i, a_{i-1}]\cup [b_{i-1}, b_{i}]$. 
\end{defn}

\begin{lem}
\label{lem:barr vs toll barriers}
For every $\lambda,\chi>1$ the following holds.  Assume that $\ell_1,\dots, \ell_n$ are barriers for a lamination $\RR \subset \Fam(I,J)$, where $I,J\subset \partial \wZ$ is a grounded pair with $|J| >1/2$. Assume moreover, that 
\[ 1/\chi <\frac{|A|}{|I|}, \frac{|B|}{|I|} < \chi\sp\sp\text{ and }\sp\sp  |A|,|B|\ge 2\length_{m},\]
where $A,B$ are the complementary intervals between $I$ and $J$ as above. 

Then either 
\begin{itemize}
\item $\dist_{\partial \wZ^m}(a_i,a_{i-1}) <|I|/\lambda$ or $\dist_{\partial \wZ^m}(b_i,b_{i-1}) <|I|/\lambda$ for some $i$ with respect to any extension of $\dist_{\partial \wZ^m}(\sp,\sp ) $ from regular to all points;
\item  or after removing $O_{\chi, n,\bdelta}(\log \lambda)$-curves from $\RR$, we obtain a lamination $\RR^\new$ for which  $\ell_1,\dots, \ell_n$ are toll barriers.  
\end{itemize}
\end{lem}

Let us say that a curve $\gamma\in \RR$ \emph{skip under $[a_{i+1},a_i]\cup [b_i,b_{i+1}]$} if $\gamma$ intersects the interval $\{z\mid b_{i+1} <z< a_{i+1}\}\supset J$ before intersecting $[a_{i+1},a_i]\cup [b_i,b_{i+1}]$.

\begin{proof}
If the width of curves skipping under $[a_{i+1},a_i]$ or under $[b_i,b_{i+1}]$ is at least $C\gg _{\bdelta,\chi}\log \lambda$, then Squeezing Lemma~\ref{lem:squeezing} applied to $A\setminus [a_{i+1},a_i]^\GRND$ or to $B\setminus [b_i,b_{i+1}]^\GRND$ implies that either $\dist_{\partial \wZ^m}(a_i,a_{i-1}) <|I|/\lambda$ or $\dist_{\partial \wZ^m}(b_i,b_{i-1}) <|I|/\lambda$ holds.
\end{proof}

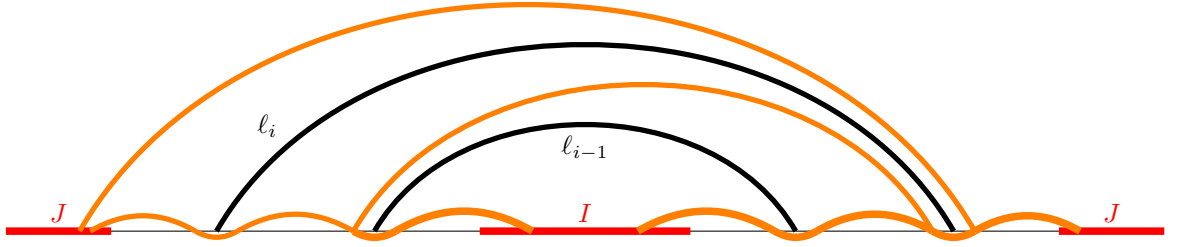
\begin{figure}[t!]
\[\begin{tikzpicture}[scale=1.4]

\node[red,above] at(5,0){$J$}; 
\node[red,above] at(-5,0){$J$};

\draw (-5.2,0) -- (5.2,0);

\draw[red,line width=1mm] (5.5,0)--(4.5,0)
(-5.5,0)--(-4.5,0);

\draw[red,line width=1mm] (1,0)--(-1,0);

\node[above,red] at (0,0)  {$I$}; 


\draw[line width=0.7mm] (2,0)edge[bend right=60](-2,0);
\draw[line width=0.7mm] (3.5,0)edge[bend right=60](-3.5,0);

\node[below] at (0,1) {$\ell_{i-1}$};
\node[right] at (-3.2,1) {$\ell_i$};

\draw[orange,line width=1mm]  (-0.5,0) edge[bend right] (-1.8,0)
 (-1.8,0)edge[bend left] (-2.2,0);
 
\draw[orange,line width=1mm]  (0.5,0) edge[bend left] (1.8,0)
 (1.8,0)edge[bend right] (2.2,0)
 (2.2,0) edge[bend left] (3.3,0)
(3.3,0) edge[bend right] (3.7,0)
(3.7,0) edge[bend left] (4.7,0) ; 

\draw[orange,line width=0.7mm]  (-2.2,0) edge[bend right] (-3.3,0)
(-3.3,0) edge[bend left] (-3.7,0)
 (-2.2,0) edge[bend left=60] (3.3,0)
 (3.7,0) edge[bend right=60] (-4.8,0) 
(-3.7,0) edge[bend right] (-4.7,0)  ;

\end{tikzpicture}\]
\caption{An example of a lamination (orange) with two toll barriers (black).}
\label{Fig:long snake}
\end{figure}

\begin{snakelmm}[with toll barriers]
\label{lem:SnakeWithCheckpts}
 Suppose that a lamination $\RR\subset\Fam(I,J)$ has $n\ge 3$ toll barriers, where $I,J\subset \wZ^m$ is a grounded pair with $|J|>1/2$. Assume also that $|A|,|B|\ge 2\length_m$, where $A,B$ are the complementary intervals between $I,J$.

 Then for every $\lambda>1$ there is an interval $T\subset \partial Z$ grounded rel $\wZ^m$ such that 
 \[\Width^+_\lambda(T)\succeq n \Width(\RR) - O_{\bdelta,n}(\log \lambda). \]
\end{snakelmm}
We will give a proof of Lemma~\ref{lem:SnakeWithCheckpts} after introducing the Series Decomposition for $\RR$. 

\subsubsection{Series Decomposition for laminations with toll barriers} \label{sss:SetDecomp:TallBarr}The construction  below is an adaptation of Series Decomposition~\S\ref{sss:SerDecomp for F circ}.  We assume that $I=[a_0,b_0]$ and $J=[b_{n+1},a_{n+1}]$ so that $a_{n+1}<a_n<\dots <a_0<b_0<\dots <b_{n+1}$. Consider a curve $\gamma\colon[0,1]\to \C$ in $ \RR$. For $i\in \{2,3\dots, n\}$, we will define below:
\begin{itemize}
\item $\gamma_i^{df}$ the first passage of $\gamma$ under $\{a_i, b_i\}$;
\item $\gamma_{i-1}^{dl}$ the last passage of $\gamma$ under $\{a_{i-1}, b_{i-1}\}$ before $\gamma_i^{df}$;
\item $\gamma_i$ the subcurve of $\gamma$ between $\gamma_{i-1}^{dl}$ and $\gamma_i^{df}$.
\end{itemize}
Then we will specify the laminations $\Fam_i$ and $\Gamma_{i}$.

\noindent {\bf Definition of} $\gamma_i^{df},\gamma_{i-1}^{dl}, \gamma_i$. Set:
\begin{itemize}
\item $\gamma(\tilde t_i)\in \partial \wZ^m$ to be the first intersection of $\gamma$ with 
$[a_{i+1}, a_{i}]\cup [b_{i},b_{i+1}]$;
\item $\gamma(t_i)\in \partial \wZ^m$ to be the last before $\tilde t_i$ intersection of $\gamma$ with $[a_i, a_{i-1}]\cup [b_{i-1},b_i]$;
\item $\gamma^{df}_i$ to be the subcurve $\gamma\mid [t_i,\tilde t_i]$;
\item $\gamma( \tau_{i-1})\in \partial \wZ^m$ to be the last intersection of $\gamma$ with 
$[a_{i-1}, a_{i-2}]\cup [b_{i-2},b_{i-1}]$ before $t_i$;
\item $\gamma(\tilde \tau_{i-1})\in \partial \wZ^m$ to be the first after $\tau_{i-1}$ intersection of $\gamma$ with $[a_i, a_{i-1}]\cup [b_{i-1},b_i]$;
\item $\gamma^{dl}_{i-1}$ to be the subcurve $\gamma\mid [\tau_{i-1},\tilde \tau_{i-1}]$;
\item $\gamma_i$ is the subcurve of $\gamma$ between $\gamma_{i-1}^{dl}$ and $\gamma_i^{df}$.
\end{itemize}

By construction, $\gamma_i$ 
\begin{itemize}
\item is disjoint from $\ell_{-1}\cup \ell_i$;
\item can only submerge into $\intr \wZ^m$ thorough $[a_i, a_{i-1}]\cup [b_{i-1},b_i]$; and
\item can not travel between $[a_i, a_{i-1}]$ and $ [b_{i-1},b_i]$ within $\wZ^m$.
\end{itemize}

\noindent {\bf Laminations} $\Gamma_i=\Gamma_{a_i}\cup \Gamma_{b_i}$ and $ \Fam_i=\Fam_{a_i}\cup \Fam_{b_i}$. Set 
\[ \widetilde \Fam_i \coloneqq\{\gamma_i\mid \gamma \in \RR\} \sp\sp \text{ and }\sp\sp \Gamma_i \coloneqq \left\{\gamma^{df}_i, \gamma^{dl}_i\mid \gamma\RR\right\} .\]

Every curve $\beta\in \Gamma_i$ is either under $a_i$ or under $b_i$ depending on whether $\beta$ connects $[a_i, a_{i-1}]$ to $[a_{i+1},a_{i}]$ or $[b_{i-1}, b_i]$ to $[b_i,b_{i+1}]$. Let $\Gamma_{a_i}, \Gamma_{b_i}$ be the sublaminations of $\Gamma_i$ consisting of curves that are below $a_i$ and $b_i$ respectively. One of the $\Gamma_{a_i}, \Gamma_{b_i}$ can be empty.

Similarly, we decompose $\widetilde \Fam_i=\widetilde \Fam_{a_i}\cup \widetilde \Fam_{a_i}$ as follows: 
\begin{itemize}
\item $\widetilde \Fam_{a_i}$ consists of curves $\gamma_i\in \widetilde \Fam_i$ such that $\gamma^{df}_i\in \Gamma_{a_i}$; and 
\item $\widetilde \Fam_{b_i}$ consists of curves $\gamma_i\in \widetilde \Fam_i$ such that $\gamma^{df}_i\in \Gamma_{b_i}$.
\end{itemize}

For every $a_i$ and $b_i$ set $\beta_{a_i}$ and $\beta_{b_i}$ to be the lowest curves in $\Gamma_{a_i}$ and $\Gamma_{b_i}$ and specify the following intervals: 
\begin{itemize}
\item $J_{a_i}$ to be  between the right endpoint of $\beta_{a_i}$ and $a_i$,
\item $I_{a_i}$ to be  between the left endpoint of $\beta_{a_i}$ and $a_i$,
\item $J_{b_i}$ to be  between the left endpoint of $\beta_{b_i}$ and $b_i$,
\item $I_{b_i}$ to be  between the right endpoint of $\beta_{b_i}$ and $b_i$.
\end{itemize} 
 If $\Gamma_{a_i}$ (resp $\Gamma_{b_i}$) is empty, then $J_{a_i}, I_{a_i}$ (resp $J_{b_i}, I_{b_i}$) are trivial. 
 
 By construction every curve in $\widetilde \Fam_i=\widetilde \Fam_{a_i}\cup \widetilde \Fam_{b_i}$ connects $I_{a_{i-1}}\cup I_{b_{i-1}}$ to $J_{a_i}\cup J_{b_i}$ and is disjoint from barriers $\ell_{i-1}, \ell_i$ and arcs $\beta_{a_{i-1}},\beta_{b_{i-1}},\beta_{a_i},\beta_{b_{i}}$.
 
We define 
\begin{equation}
\label{eq:SerDecomp:barriers:Fam}
\Fam_{a_i} \subset \Fam^\circ(J_{a_i},\ \lfloor I_{a_{i-1}}, J_{b_i}\rfloor )\sp\sp\text{ and }\sp\sp \Fam_{b_i} \subset \Fam^\circ(\lfloor J_{a_i},I_{b_{i-1}} \rfloor , J_{b_i})
\end{equation} 
to be the restrictions of $\widetilde \Fam_{a_i}, \widetilde \Fam_{b_i}$ to the associated families; i.e.:
\begin{itemize}
\item $\Fam_{a_i} $ consists of the first shortest subcurves in $\widetilde \Fam_{a_i}$ connecting $\lfloor I_{a_{i-1}}, J_{b_i}\rfloor$ and $J_{a_i}$;  
\item $\Fam_{b_i} $ consists of the first shortest subcurves in $\widetilde \Fam_{b_i}$ connecting $\lfloor J_{a_i},I_{b_{i-1}} \rfloor$ and $J_{b_i}$.
\end{itemize}

{\bf Series Decomposition.} We obtain that $\RR$ consistently overflows 
\begin{equation}
\label{eq:dfn:SD:barr}
 \Gamma_{1}=\Gamma_{a_1}\cup \Gamma_{b_1},\sp \Fam_2=\Fam_{a_2}\cup \Fam_{b_2}, \  \dots,\  \Fam_{n}=\Fam_{a_n}\cup \Fam_{b_n}, \sp\Gamma_{n}=\Gamma_{a_n}\cup \Gamma_{b_n}
\end{equation}

\subsubsection{Proof of Lemma~\ref{lem:SnakeWithCheckpts}} Follows by repeating the steps in the proof of Lemma~\ref{simplmm:SnLmm:wZ}, see~\S\ref{sss:Prf SnaekLmm} and Remark~\ref{rem:SL:steps}. 

Consider Series Decomposition~\ref{eq:dfn:SD:barr}. We recall that every $\Fam_i$ satisfies~\eqref{eq:SerDecomp:barriers:Fam}.

By Lemma~\ref{lem:I vs I^grnd:with separ}, at most $O_\bdelta(n)$ curves in $\RR$ intersect \[\bigcup_{i=2}^n \left(\left(I_{a_i}\setminus I^\grnd_{a_i}\right)\cup \left(I_{b_i}\setminus I^\grnd_{b_i}\right)\cup\left(J_{a_i}\setminus J^\grnd_{a_i}\right)\cup \left(J_{b_i}\setminus J^\grnd_{b_i}\right) \right);\]
removing all such curves from $\RR$, we can assume that $I_{a_i}, I_{b_i}, J_{a_i}, J_{b_i}$ are grounded (by replacing them with $I^\grnd_{a_i}, I^\grnd_{b_i}, J^\grnd_{a_i}, J^\grnd_{b_i}$).

By Localization and Squeezing Lemmas~\ref{lem:trading  width to space},~\ref{lem:squeezing} applied to \[\Fam^-(I_{a_i}, J_{a_i})\supset \Gamma_{a_i}\sp\sp\text{ and }\sp\sp \Fam^-(I_{b_i}, J_{b_i})\supset \Gamma_{b_i},\] 
$I_{a_i}, J_{a_i}$ and $I_{b_i}, J_{b_i}$ have innermost subpairs $I^\new_{a_i}, J^\new_{a_i}$ and $I^\new_{b_i}, J^\new_{b_i}$ such that, up to $O_\bdelta(\log \lambda)$, the width of $\Fam^-(I_{a_i}, J_{a_i}), \Fam^-(I_{b_i}, J_{b_i})$ is contained in $\Fam^-(I^\new_{a_i}, J^\new_{a_i}),$ $\Fam^-(I^\new_{b_i}, J^\new_{b_i})$ and such that $\lfloor I^\new_{a_i}, J^\new_{a_i}\rfloor, \lfloor J^\new_{b_i}, I^\new_{b_i}\rfloor$ are at least $5\lambda$ times smaller than $I_{a_i}, J_{a_i}, I_{b_i}, J_{b_i}$ respectively. Removing all curves in $\RR$ intersecting
\[ \big(I_{a_i}\setminus I^\new_{a_i}\big)\cup  \big(J_{a_i}\setminus J^\new_{a_i}\big)\cup  \big(I_{b_i}\setminus I^\new_{b_i}\big)\cup \big(J_{b_i}\setminus J^\new_{b_i}\big),\]
and then reapplying Series Decomposition~\ref{eq:dfn:SD:barr}, we obtain that the new $J_{a_i}$ and $J_{b_i}$ have small length compared to their distances to $I_{a_i}$ and $I_{b_i}$ respectively. 

Since $\RR$ consequently overflows the $\Fam_i$, there is an $i$ such that \[\Width(\Fam_i)= \Width(\Fam_{a_i})+\Width(\Fam_{b_i})\succeq n\Width(\RR) - O_{\bdelta, n} (\log \lambda).\]
The lemma now follows by applying Lemma~\ref{lem:trad width to width+} to either $\Fam_{a_i}$ or $\Fam_{b_i}$ -- they satisfy~\eqref{eq:SerDecomp:barriers:Fam} and the $\lambda$-separation.
\qed

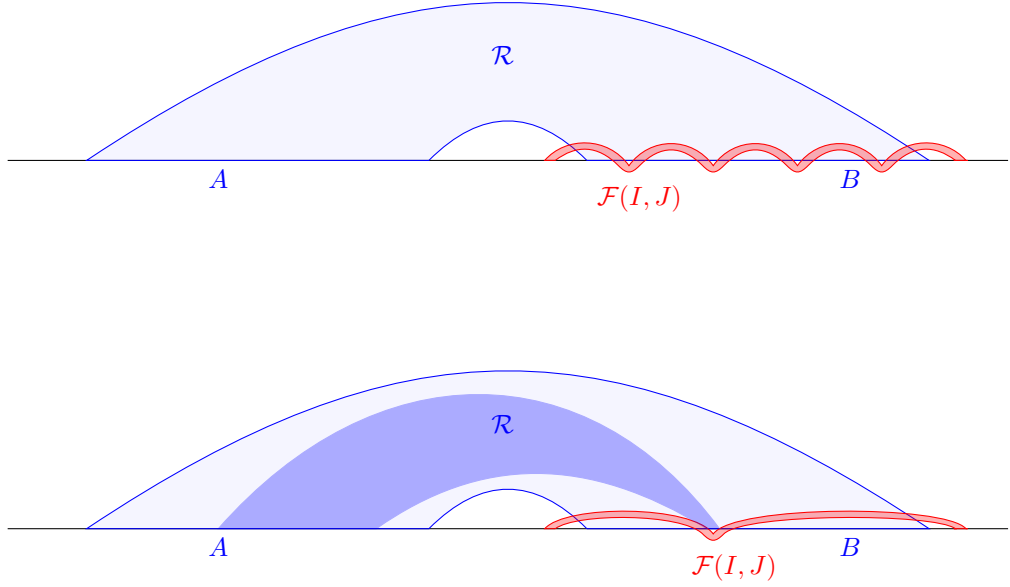
\begin{figure}[t!]
\[\begin{tikzpicture}[scale=1.4]

\draw (-7,0) -- (2.5,0);

\node[blue] at (-2.3,1) {$\RR$};

\node[red,below] at (-1,-0.15) {$\Fam(I,J)$};
\node[blue,below] at (1,-0.0) {$B$};

\draw[blue, fill=blue, fill opacity=0.04] (-1.5,0) 
.. controls (-2, 0.5) and (-2.5,0.5) .. 
 (-3,0)
 .. controls (-3, 0) and (-6,0) .. 
 (-6.25,0)
.. controls (-3.25, 2) and (-1.25,2) .. 
(1.75,0)
.. controls (1.75,0) and  (-1.5,0)  .. 
 (-1.5,0) ;

 \node[blue, below] at (-5,0) {$A$};

\begin{scope}[shift={(0,-0.05)} ,scale=0.4]

\draw[red, fill=red, fill opacity=0.3]  (-4.5,0.125) 
.. controls (-4, 0.5) and (-3.5,0.5) .. 
(-3,0)
.. controls (-2.8, -0.2) and (-2.7,-0.2) .. 
(-2.5,0)
.. controls (2-4, 0.5) and (2-3.5,0.5) .. 
(2-3,0)
.. controls (2-2.8, -0.2) and (2-2.7,-0.2) .. 
(2-2.5,0)
.. controls (4-4, 0.5) and (4-3.5,0.5) .. 
(4-3,0)
.. controls (4-2.8, -0.2) and (4-2.7,-0.2) .. 
(4-2.5,0)
.. controls (6-4, 0.5) and (6-3.5,0.5) .. 
(6-3,0)
.. controls (6-2.8, -0.2) and (6-2.7,-0.2) .. 
(6-2.5,0)
.. controls (8-4, 0.5) and (8-3.5,0.5) .. 
(8-3,0.125)
.. controls (7.8-3,0.125) and (6.2-0.75,0.125 ) .. 
(6-0.75,0.125 )
.. controls   (6-1.25,0.7) and (6-2.25, 0.7) .. 
(4-0.75,0)
.. controls   (4-1.25,0.7) and (4-2.25, 0.7) .. 
(2-0.75,0)
.. controls (2-1.25,0.7) and (2-2.25, 0.7) .. 
(-0.75,0)
.. controls (-1.25,0.7) and (-2.25, 0.7).. 
(-2.75,0)
.. controls (-3.25,0.7) and  (-4.25, 0.7) .. 
(-4.75,0.125) 
.. controls (-4.75,0.125)  and  (-4.5,0.125) .. 
(-4.5,0.125); 

\end{scope}

\begin{scope}[shift={(0,-3.5)}]

\draw (-7,0) -- (2.5,0);

\node[blue] at (-2.3,1) {$\RR$};

\node[red,below] at (-0.1,-0.15) {$\Fam(I,J)$};
\node[blue,below] at (1,-0.0) {$B$};

\draw[blue, fill=blue, fill opacity=0.04] (-1.5,0) 
.. controls (-2, 0.5) and (-2.5,0.5) .. 
 (-3,0)
 .. controls (-3, 0) and (-6,0) .. 
 (-6.25,0)
.. controls (-3.25, 2) and (-1.25,2) .. 
(1.75,0)
.. controls (1.75,0) and  (-1.5,0)  .. 
 (-1.5,0) ;

\begin{scope}[shift={(0,-0.05)} ,scale=0.4]

\draw[red, fill=red, fill opacity=0.3]  (-4.5,0.125) 
.. controls (-4, 0.5) and (2-3.5,0.5) .. 
(2-3,0)
.. controls (2-2.8, -0.2) and (2-2.7,-0.2) .. 
(2-2.5,0)
.. controls (4-4, 0.5) and (8-3.5,0.5) .. 
(8-3,0.125)
.. controls (7.8-3,0.125) and (6.2-0.75,0.125 ) .. 
(6-0.75,0.125 )
.. controls   (6-1.25,0.7)  and (2-2.25, 0.7) .. 
(-0.75,0)
.. controls (-1.25,0.7) and  (-4.25, 0.7) .. 
(-4.75,0.125) 
.. controls (-4.75,0.125)  and  (-4.5,0.125) .. 
(-4.5,0.125); 

\coordinate (l1) at  (-0.6,0.125);

\coordinate (l2) at  (-0.9,0.125);

\end{scope}

\draw[blue, fill=blue, opacity=0.3] (l2) 
.. controls (-1.5,0.7)  and  (-2.5,0.7) .. 
(-3.5,0)
.. controls (-3.5,0)  and  (-3.5,0) .. 
(-5,0)
.. controls (-3.5,1.7)  and  (-1.5,1.7) .. 
(l1)
.. controls (l1)  and  (l2) .. 
(l2);
 \node[blue, below] at (-5,0) {$A$};

\end{scope}

\end{tikzpicture}\]
\caption{Two patterns for $\Fam(I,J)$ to sneak through $\RR$.}
\label{Fg:snake thorugh WR}
\end{figure}

\subsection{Sneaking Lemma}
\begin{sneaklmm}
\label{lem:sneaking}
Let $\wZ^m$ be a $\bdelta$ pseudo-Siegel disk. For all $t, \chi, \lambda>2$ there is a $\kappa (t)>2$ such that the following holds.

Suppose \[I,J\subset \partial \wZ^m,\sp\sp |I|\ge \length_m, \sp  \sp |A|,|B| \ge  2\length_m,\sp\sp  |J|> 1/2\] is a grounded pair and denote by $A,B$ two complementary intervals between $I$ and $J$; i.e., $A\cup B =\overline{\partial \wZ^m\setminus (I\cup J)}$. Suppose also that
\[  1/\chi \le \frac{|A|}{|I|}, \frac{|B|}{|I|}\le \chi.\]
If
\begin{itemize}
\item  $\Width(I,J)\eqqcolon K\gg_{\chi,\lambda, \bdelta} 1$; and
\item $\Width^+(A,B) \ge \kappa(t) K$,
\end{itemize}
then there is a $[tK,\lambda]^+$-wide interval $T\subset \partial Z$ with $|T| < |I|$ such that $T$ is grounded rel $\wZ^m$.
\end{sneaklmm}
\begin{proof}
Let us denote by $\RR$ the canonical rectangle of $\Fam^+(A,B)$. The idea of the proof is illustrated on Figure~\ref{Fg:snake thorugh WR}: either the family $\Fam_(I,J)$ submerges many times in $A$ or $B$, or a substantial part of $\Fam^+(A,B)$ is focused; i.e. it starts or terminates in a sufficiently small interval.

Let us select in $\RR$ a disjoint union of rectangles \[\RR_1\sqcup \RR_2\sqcup \dots\sqcup \RR_m,\sp\sp\text{ with }\sp \Width(\RR_i) = (t+1) K,\sp\sp \text{ and }\sp m \approx  \kappa(t)/(t+1).\] We assume that $A_i$ and $B_i$ is the base and the roof of $\RR_i$ respectively, and that they have the following orientation:
\[ A_{m}<A_{m-1}<\dots <A_1<B_1<\dots <B_m;\]
in particular, $\RR_{i+1}$ is above $\RR_i$.

By Lemma~\ref{lem:vet boundar}, we can forget $O_m(1)$- curves in $\Fam(I,J)$ and we can choose a vertical curve $\beta_i$ in the inner $O_m(1)$-buffer of every $\RR_i$ such that the remaining part $\LL$ of $\Fam(I,J)$ is disjoint from every $\beta_i$. We assume that $\beta_i$ connects $a_i\in A_i$ with $b_i\in B_i$. We denote by $\RR^\new_i$ the rectangle obtained from $\RR_i$ by removing an inner $O(1)$-buffer so that the horizontal sides of $\RR^\new_i$ are within $[a_{i+1},a_i]\sqcup [b_i,b_{i+1}]$.

We obtain that $\beta_i$ are barriers for $\LL$ as in~\S\ref{ss:SL:barriers} and that $\RR^\new_i$ is between $\beta_{i+1}$ and $\beta_i$. By Lemma~\ref{lem:barr vs toll barriers}, we have two possibilities (depending whether a sufficiently wide part of $\LL$ skips under $[a_{i+1},a_i]\cup [b_i,b_{i+1}]$):

{\bf Case \RN{1}:} $[a_{i+1},a_i]$ or $[b_i,b_{i+1}]$ is smaller than $|I|/\lambda$. Then either $\Fam_{\lambda} [a_{i+1},a_i]$ or $\Fam_{\lambda} [b_{i+1},b_i]$ contains $\RR^\new_i$. The statement follows from Lemmas~\ref{lem:W+:ground inter} and~\ref{lem:I vs I^grnd:with separ} by setting $T$ to be the projection of either $[a_{i+1},a_i]^\grnd$ or $[b_i,b_{i+1}]^\grnd$ on $\partial Z$. 

{\bf Case \RN{2}:} we can remove $O_{\lambda,\bdelta,\chi}(1)$-part from $\LL$ so that $\beta_{1},\dots, \beta_{m}$ are toll barriers for the remaining $\LL^\new$.  The statement now follows from Snake Lemma~\ref{lem:SnakeWithCheckpts} (with toll barriers).
\end{proof}

\subsection{Families that block each other}

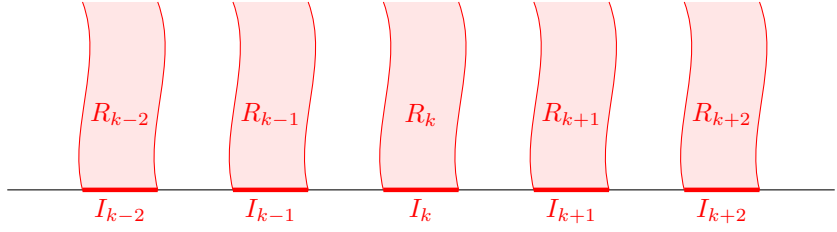
\begin{figure}[t!]

\[
\begin{tikzpicture}

\draw (-5,0)--(6,0);

\draw[red, line width=0.6mm] (0,0) --(1,0);

\draw[red, fill=red, fill opacity=0.1] 
(1,2.5) 
 .. controls  (1.3,1.7) and (1-0.2, 0.8)  .. 
(1,0)
.. controls  (0.8,0) and (0.2, 0)  ..
 (0,0) 
 .. controls (-0.2, 0.8) and (0.3,1.7) .. 
(0,2.5);

\node[red] at (0.5,1) {$R_k$};
\node[below,red] at (0.5,0){$I_k$};

\begin{scope}[shift ={(-2,0)}]

\draw[red, line width=0.6mm] (0,0) --(1,0);

\draw[red, fill=red, fill opacity=0.1] 
(1,2.5) 
 .. controls  (1.3,1.7) and (1-0.2, 0.8)  .. 
(1,0)
.. controls  (0.8,0) and (0.2, 0)  ..
 (0,0) 
 .. controls (-0.2, 0.8) and (0.3,1.7) .. 
(0,2.5);

\node[red] at (0.5,1) {$R_{k-1}$};
\node[below,red] at (0.5,0){$I_{k-1}$};

\end{scope}

\begin{scope}[shift ={(-4,0)}]

\draw[red, line width=0.6mm] (0,0) --(1,0);

\draw[red, fill=red, fill opacity=0.1] 
(1,2.5) 
 .. controls  (1.3,1.7) and (1-0.2, 0.8)  .. 
(1,0)
.. controls  (0.8,0) and (0.2, 0)  ..
 (0,0) 
 .. controls (-0.2, 0.8) and (0.3,1.7) .. 
(0,2.5);

\node[red] at (0.5,1) {$R_{k-2}$};
\node[below,red] at (0.5,0){$I_{k-2}$};

\end{scope}

\begin{scope}[shift ={(2,0)}]

\draw[red, line width=0.6mm] (0,0) --(1,0);

\draw[red, fill=red, fill opacity=0.1] 
(1,2.5) 
 .. controls  (1.3,1.7) and (1-0.2, 0.8)  .. 
(1,0)
.. controls  (0.8,0) and (0.2, 0)  ..
 (0,0) 
 .. controls (-0.2, 0.8) and (0.3,1.7) .. 
(0,2.5);

\node[red] at (0.5,1) {$R_{k+1}$};
\node[below,red] at (0.5,0){$I_{k+1}$};

\end{scope}

\begin{scope}[shift ={(4,0)}]

\draw[red, line width=0.6mm] (0,0) --(1,0);

\draw[red, fill=red, fill opacity=0.1] 
(1,2.5) 
 .. controls  (1.3,1.7) and (1-0.2, 0.8)  .. 
(1,0)
.. controls  (0.8,0) and (0.2, 0)  ..
 (0,0) 
 .. controls (-0.2, 0.8) and (0.3,1.7) .. 
(0,2.5);

\node[red] at (0.5,1) {$R_{k+2}$};
\node[below,red] at (0.5,0){$I_{k+2}$};

\end{scope}

\end{tikzpicture}
\]

\caption{Illustration to Lemma~\ref{lem:no inf quasi additivity}: rectangles $R_k$ do not exist because, otherwise, they would block each other.}
\label{fig:F_n don't go far}
\end{figure}

We will need the following simple fact.


\begin{lem}
\label{lem:no inf quasi additivity}
Let $\wZ^m$ be a pseudo-Siegel disk. There is no sequence of disjoint regular intervals
\[ I_n=I_0,I_1,\dots , I_{n-1},\sp\sp\sp |I_k| \ge \length_m,\sp\sp \Width^+( I_k, L^c_k) \ge 3 \]
enumerated either counterclockwise or clockwise such that $I_{k-1}\cup I_k\cup I_{k+1}\subset  L_k$.
\end{lem}
\begin{proof}
Suppose converse. Let $\Fam_k$ be the canonical rectangle of $\Fam^+( I_k, L^c_k)$ in $\wC\setminus \intr(\wZ^m)$. By removing $1$-buffers on each side of $\Fam_k$, we obtain closed rectangles $R_k\subset \Fam^+( I_k, L^c_k)$ such that the $R_k$ are disjoint and have width at least $1$. This is impossible because $R_k$ block each other, see Figure~\ref{fig:F_n don't go far}. Indeed, let us choose \[\big[\gamma\colon [0,1]\to \wC\setminus \intr(\wZ^m)\big]\in \bigsqcup_{i=1}^{n} R_k\] so that $\gamma(1)\boxminus \gamma(0)$ is the minimal possible. Assuming $\gamma\in R_k$ and using $I_{k-1}\sqcup I_{k+1}\subset L_k$, we obtain that some $\ell \in R_{k-1}\sqcup R_k$ would have smaller difference $\ell(1)\boxminus \ell(0)$.
\end{proof}

\section{Welding of $\wZ^{m+1}$ and parabolic fjords}
\label{s:Welding}

Let $\wZ^{m+1}$ be a $\bdelta$-pseudo-Siegel disk. For an interval $J^{m+1}=[x,y]\subset \partial \wZ^{m+1}$, let $[x,y]_{\wZ^{m+1}}$ and $[x,y]_{\wC\setminus \wZ^{m+1}}$ be the hyperbolic geodesics of $\intr \wZ^{m+1}$ and of $\wC\setminus \wZ^{m+1}$ connecting $x$ and $y$. Define $O_{J^{m+1}}\supset J^{m+1}$ to be the closed topological disk bounded by $[x,y]_{\wZ^{m+1}}\cup [x,y]_{\wC\setminus \wZ^{m+1}}\eqqcolon\partial O_{J^{m+1}}$, see Figure~\ref{fig:O_J}.

Consider $T\in \Dbb_m,\sp m\ge -1$ and let $T'$ be as in~\S\ref{sss:diff tilings}; i.e.~$T'\coloneqq T\cap f^{\qq_{m+1} }(T)$ for $m\ge 0$ with an appropriate adjustment for $m=-1.$ Assume that there is a sufficiently wide non-winding parabolic rectangle based on $T'$. By Theorem~\ref{thm:par fjords}, all such wide rectangles are essentially based  on $T_\parab\subset T'$. Theorem~\ref{thm:par fjords} also describes the outer geometry of $\overline Z$ above $T_\parab$ on scale $\ge \length_{m+1}$.

\begin{weldlmm}
\label{lmm:welding}
Consider a concatenation of intervals $J=N\# I\# M\subset T_\parab$ with $|J|<\frac 1 2 |T_\parab|$, where $T_\parab$ is from Theorem~\ref{thm:par fjords}, such that the endpoints of $N, I,M$ are within $\CP_{m+1}$. Assume that
\begin{equation}
\label{lmm:welding:cond}
|N|\asymp |I| \asymp |M|.
\end{equation}
Then there is a constant $\bvarepsilon>0$ depending only on ``$\asymp$'' in~\eqref{lmm:welding:cond} such that the following holds for all $\lambda>2$. If $\wZ^{m+1}$ is a $\bdelta$-pseudo-Siegel disk and if ${\nu\coloneqq |I|/\length_{m+1}\gg_{\bdelta,\lambda} 1}$ (where ``~$\gg_{\bdelta,\lambda}$'' also depends on ``$\asymp$'' in~\eqref{lmm:welding:cond}), then either 
\begin{equation}
\label{eq:1:lmm:welding}
\mod \left(O_{J^{m+1}}\setminus I^{m+1} \right)\ge \bvarepsilon
\end{equation}
holds, where $I^{m+1},J^{m+1}$ are the projections of $I,J$ onto $\partial \wZ^{m+1}$, or there is an interval $S\subset \partial Z$ with $|S|<\length_{m+1}$ such that $S$ is grounded rel $\wZ^{m+1}$ and
\begin{equation}
\label{eq:2:lmm:welding}
\log \Width^+_{\lambda}(S) \succeq  \log \nu.
\end{equation}
\end{weldlmm}

\begin{rem}
\label{rem:error does not accum}
We emphasize that ``$\bvarepsilon$ and $\succeq$'' in~\eqref{eq:1:lmm:welding} and~\eqref{eq:2:lmm:welding} are independent of $\bdelta$. Only the scale on which the Welding Lemma works  (i.e.,~how big is $\nu$) depends on $\bdelta$. This independence of $\bdelta$ follows from beau coarse-bounds for $\wZ^{m+1}$ (Theorem~\ref{thm:wZ:shallow scale}) that are based on beau coarse-bounds for near rotation domains (Theorem~\ref{thm:beau:part U: c quasi line}). The independence of $\bdelta$ implies that the error does not increase during the regularization $\dots \wZ^{m+1}\leadsto \wZ^{m} \leadsto\wZ^{m-1}\leadsto\dots$-- see Corollary~\ref{cor:regul}.

In the application, the comparison ``$\asymp$'' for~\eqref{lmm:welding:cond} will be selected in~\eqref{eq:mod O:bounds}; see also Remark~\ref{rem:selections of:lmm:welding:cond}. For the outline of the proof, see~\S\ref{sss:outline of Welding}. 
\end{rem}

Recall from~\S\ref{sss:regul within rect} that a regularization $\wZ^{m}=Z^{m+1}\cup \wZ^{m+1}$ is within $\orb_{-\qq_{m+1}+1} \RR$ if all relevant objects are within the backward orbit of a rectangle $\RR$. The following result is our primary tool of constructing pseudo-Siegel disks.

\begin{cor}
\label{cor:regul}
There is a sufficiently small $\bdelta>0$ with the following properties. Consider $T\in \Dbb_m$ and let $T'$ be as above. Let $\RR$ be a non-winding parabolic rectangle based on $T'$ with $\Width(\RR)\gg_{\bdelta} 1$. Let $\wZ^{m+1}$ be a geodesic $\bdelta$-pseudo-Siegel disk, see~\S\ref{sss:regul within rect}. Then:
\begin{enumerate}
\item either there is a geodesic $\bdelta$-pseudo-Siegel disk $\wZ^m=Z^m\cup \wZ^{m+1}$ with its level-$m$ regularization within $\orb_{-\qq_{m+1}+1} \RR$;
\label{case:1:thm:regul}
\item or there is an interval \[I\subset T',\sp\sp   |I|> \length_{m+1} \sp\sp \text{such that}\sp\sp 
\log \Width^{+}_{\lambda,\div, m}(I)\succeq \Width(\RR);
\]\label{case:2:thm:regul}
\item or there is a grounded rel $\wZ^{m+1}$ interval \[I\subset \partial Z\sp\sp \text{ with }\sp |I|\le \length_{m+1} \sp \text{ such that }\sp\log \Width^+_\lambda(I)\succeq  \Width(\RR).\]
\label{case:3:thm:regul}
\end{enumerate}  
\end{cor}
 
 We refer to Cases~\eqref{case:2:thm:regul} and~\eqref{case:3:thm:regul} as \emph{exponential boosts}.

\begin{rem}
\label{rem:thm:regul} Calibration Lemma~\ref{lmm:CalibrLmm} will reduce Case~\eqref{case:2:thm:regul} to Case~\eqref{case:3:thm:regul}.
\end{rem}

\begin{rem}
\label{rem:fixing deltas} Starting with (next) Section~\ref{part:CovCalibr}, we fix a sufficiently small $\bdelta>0$ so that Corollary~\ref{cor:regul} is applicable. 
\end{rem}

\begin{figure}[t!]

\[
\begin{tikzpicture}

\draw[line width=0.4mm] (-5,0)--(7,0);

\node[above] at (6,0) {$\partial \wZ^{m+1}$};


\filldraw[blue] (-4,0) circle (0.04 cm);
\node[above left,blue] at (-4,0) {$x$};

\filldraw[blue] (4,0) circle (0.04 cm);
\node[above right,blue] at (4,0) {$y$};

\filldraw[red] (-1,0) circle (0.04 cm);
\node[above ,red] at (-1,0) {$a$};
\filldraw[red] (1,0) circle (0.04 cm);
\node[above ,red] at (1,0) {$b$};

\draw[orange,line width=1.2mm]
(-0.6,0)
.. controls (-0.8,0.6) and (-1.2, 0.6) .. 
(-1.4,0)
.. controls (-1.6,-0.6) and (-2, -0.6) .. 
(-2.2,0)
.. controls (-2.4,0.6) and (-2.8, 0.6) .. 
(-3,0)
.. controls (-3.2,-0.6) and (-3.6, -0.6) .. 
(-3.8,0)
.. controls (-3.9,0.1) and (-3.95, 0) .. 
(-4,0);

\node[orange,above] at (-2.4,0.5)  {$\RR$};

\draw[blue, line width=0.4mm] (-4,0) edge[bend left=60](4,0);
\draw[blue, line width=0.4mm] (-4,0) edge[bend right=60](4,0);
\node[above,blue] at (0,2) {$[x,y]_{\wC\setminus \wZ^{m+1}}$};
\node[below,blue] at (0,-2) {$[x,y]_{\wZ^{m+1}}$};

\node[above,blue] at (1,1) {$O_{J^{m+1}}$};

\draw[red, line width=0.6mm] (-1,0) --(1,0);
\node[below,red] at (0,0) {$I^{m+1}$};

\end{tikzpicture}
\]

\caption{The open disk $O_{J^{m+1}}$ is bounded by hyperbolic geodesics in the interior and exterior of $\wZ^{m+1}$. If ${\mod (O_{J^{m+1}}\setminus I^{m+1})}$ is small, then there is a wide lamination $\RR$ submerging many times into $\wZ^{m+1}$.}
\label{fig:O_J}
\end{figure}
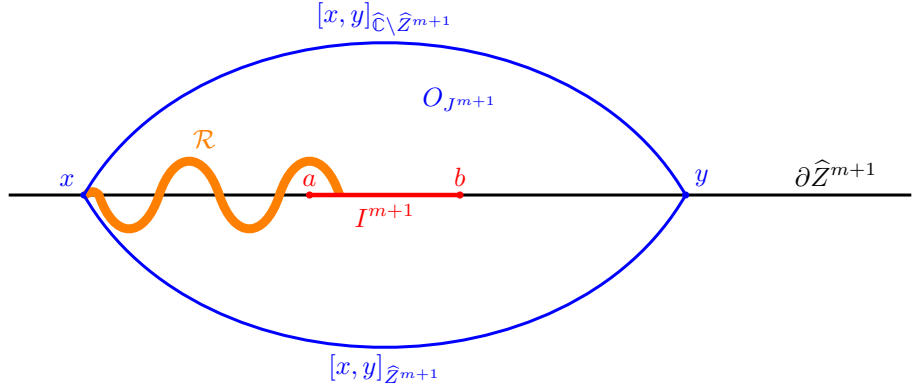

\subsubsection{Outline and Motivation} \label{sss:outline of Welding} Note that we already have a control of
\begin{itemize}
\item the outer geometry of $\overline Z$ on scale $\ge \length_{m+1}$ above $T_\parab$ -- Theorem~\ref{thm:par fjords};
\item the outer geometry of $\wZ^{m+1}$ on scale $\ge \length_{m+1}$ above $T_\parab^{m+1}$ -- because the outer geometries of  $\overline Z$ and $\wZ^{m+1}$ are close (Lemma~\ref{lem:W^+:well grnd int});

\item the inner geometry of $\wZ^{m+1}$ with the estimates depending on $\bdelta$ -- see~\eqref{eq:1:prop:wZ:shallow scale} and~\eqref{eq:2:prop:wZ:shallow scale} in Theorem \ref{thm:wZ:shallow scale};  
\item the inner geometry of $\wZ^{m+1}$ with the estimates independent of $\bdelta$ on scale $\gg_{\bdelta} \length_{m+1}$, see~\eqref{eq:1:prop:wZ:shallow scale:impr} and~\eqref{eq:2:prop:wZ:shallow scale:impr} in  Theorem \ref{thm:wZ:shallow scale}.
\end{itemize}
Therefore, the families $\Fam^+_{\wZ^{m+1}}(N^{m+1},M^{m+1})$ and $\Fam^-_{\wZ^{m+1}}(N^{m+1},M^{m+1})$ have width $\asymp 1$. Since these families (after a slight adjustment) separate $I^{m+1}$ from $\partial O_{J^{m+1}}$ in $O_{J^{m+1}}\setminus \partial \wZ^m$, most of the curves in the vertical family 
${\FamG\coloneqq \Fam( O_{J^{m+1}}\setminus I^{m+1})}$ intersect ${M^{m+1}\cup N^{m+1}}$ (assuming $\Width(\FamG)\gg 1$). Moreover, we should expect that a typical curve in $\FamG$ submerges $\asymp \nu$ times into $N\cup M$ because we have a control of the inner and outer geometries on scale $\ge \length_{m+1}$. Combined with Lemma~\ref{lem:trad width to width+}, this would have implied the existence of an interval $S$ with $\Width^+_{\lambda}(S) \succeq   \nu$. Our proof gives a somewhat weaker estimate $\log \Width^+_{\lambda}(S) \succeq  \log \nu$ (i.e.~\eqref{eq:2:lmm:welding}) that is sufficient for our purposes.

The main step in Corollary~\ref{cor:regul} is construction of annuli $A(\alpha^m_i), A(\beta^m_i)$ around channel and dams. The annuli are of the form illustrated on Figure~\ref{fig:lem:A:encl form}: $A=O_\ell \cup \YY \cup O_\rho \setminus \XX$. Rectangles $\YY^\pm, \XX$ are constructed using Theorems~\ref{thm:par fjords} and~\ref{thm:wZ:shallow scale}. To assemble all such rectangles in a chain around $\wZ^{m+1}$ we need the property that $\RR$ contains central subrectangles. If this is not the case, then Lemma~\ref{lem:central cond} implies Case~\eqref{case:2:thm:regul} of the corollary. If $\RR$ contains central subrectangles, then applying the Welding Lemma to construct $O_\ell, O_\rho$, we obtain either Case~\eqref{case:1:thm:regul}  or Case~\eqref{case:3:thm:regul} of the corollary.

\subsection{Proof of the Welding Lemma} Since the endpoints of $M,N$ are within $\CP_{m+1}$, these intervals are well-grounded rel $\wZ^{m+1}$. By Lemma~\ref{lem:W^+:well grnd int}, we have $\Width^+_{\wZ^{m+1}}(N^{m+1},M^{m+1})\asymp 1$ . Since $|N|, |I|, |M|\gg_\bdelta \length_{m+1}$, we have $\Width^-_{\wZ^{m+1}}(N^{m+1},M^{m+1})\asymp 1$ by Theorem~\ref{thm:wZ:shallow scale}, Equations \eqref{eq:1:prop:wZ:shallow scale:impr} and \eqref{eq:2:prop:wZ:shallow scale:impr}.

Let $\RR, \sp\Width(\RR)=K\coloneqq 1/\varepsilon$ be the vertical lamination of the annulus ${O_{J^{m+1}}\setminus I^{m+1}}$. Let us assume that $K\gg 1$. Let us write \[F\approx _\ell G\sp\sp\text{ if both }\sp\sp F\asymp G \sp\text{ and } \sp F=G+O(\ell) \sp\text{ hold.}\]

To simplify notations, we will omit below the upper index ``$m+1$'' for intervals. All intervals will be in $\partial \wZ^{m+1}$. Let us write
\[N=[x,a],\sp I=[a,b],\sp M=[b,y] \sp \in \partial \wZ^{m+1},\sp\sp N\le I\le M \text{ in }T.\]

Let $N_1$ and $M_1$ be middle $1/3$ well-grounded subintervals of $N$ and $M$:  
\[ \dist (x,N_1) \approx_{\length_{m+1}} |N_1|\approx_{\length_{m+1}}\dist ( N_1 ,a)\asymp \nu\length_{m+1}/3,\]
\[ \dist (b,M_1) \approx_{\length_{m+1}} |M_1|\approx_{\length_{m+1}}\dist (M_1,y)\asymp \nu\length_{m+1}/3.\]


\begin{figure}[t!]

\[
\begin{tikzpicture}[scale=1.3]

\draw[line width=0.4mm] (-4.5,0)--(4.5,0);

\filldraw[blue] (-4,0) circle (0.04 cm);
\node[above left,blue] at (-4,0) {$x$};

\filldraw[blue] (4,0) circle (0.04 cm);
\node[above right,blue] at (4,0) {$y$};

\filldraw[red] (-1,0) circle (0.04 cm);
\node[above ,red] at (-1,0) {$a$};
\filldraw[red] (1,0) circle (0.04 cm);
\node[above ,red] at (1,0) {$b$};

\filldraw[red] (-3,0) circle (0.04 cm);
\filldraw[red] (-2,0) circle (0.04 cm);

\draw[red, line width=0.6mm] (-3,0) --(-2,0);
\node[below,red] at (-2.5,0) {$N_1$};
\draw[red, dashed, line width=0.3mm] (-3,0) edge[bend left=60](3,0);
\draw[red, dashed, line width=0.3mm] (-3,0) edge[bend right=60](3,0);
\draw[red, dashed, line width=0.3mm] (-2,0) edge[bend left=60](2,0);
\draw[red, dashed, line width=0.3mm] (-2,0) edge[bend right=60](2,0);

\node[above,red] at (0,1) {$\Fam_+$};

\node[below,red] at (0,-1) {$\Fam_-$};

\filldraw[red] (3,0) circle (0.04 cm);
\filldraw[red] (2,0) circle (0.04 cm);
\draw[red, line width=0.6mm] (3,0) --(2,0);
\node[below,red] at (2.5,0) {$M_1$};

\draw[red, line width=0.6mm] (-1,0) --(1,0);
\node[below,red] at (0,0) {$I$};

\draw[blue, line width=0.4mm] (-4,0) edge[bend left=60](4,0);
\draw[blue, line width=0.4mm] (-4,0) edge[bend right=60](4,0);

\node[above,blue] at (1,2) {$\partial O_{J}$};

\end{tikzpicture}
\]

\caption{Since the families $\Fam_+$ and $\Fam_-$ separate $I$ from $\partial O_{J}$ (compare with Figure~\ref{fig:O_J}), most of the curves in the family $\RR$ intersect $N_1\cup M_1$.}
\label{fig:def:I_1}
\end{figure}
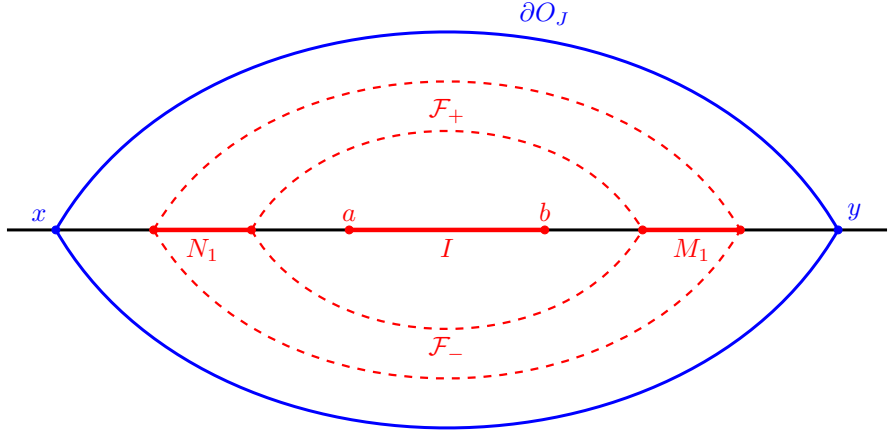

\begin{claim5}
\label{cl5:I_1}
At least $0.99 K$ curves in $\RR$ intersect $N_1\cup M_1$ before intersecting \newline${\lfloor N_1, M_1\rfloor^c\cup \partial^\out O_J}$, see Figure~\ref{fig:def:I_1}.
 \end{claim5}
\begin{proof}
Consider the outer and inner geodesic rectangles (see~\S\ref{sss:GeodRect}) \[\Fam_+\subset \wC\setminus \wZ^m,\sp \sp \Fam_-\subset \wZ^m,\sp\sp \partial ^{h,0}\Fam_+=\partial ^{h,0}\Fam_- = N_1,\sp  \partial ^{h,1}\Fam_+=\partial ^{h,1}\Fam_- = M_1\] between $N_1$ and $M_1$.  Since $|M_1|\asymp |N_1|\asymp \dist(M_1,N_1)\gg_{\bdelta} \nu\length_{m+1}$, Theorems~\ref{thm:par fjords} and~\ref{thm:wZ:shallow scale} imply that
$\Width(\Fam_-)\asymp \Width(\Fam_+)\asymp 1.$
Since $\Fam_-\cup \Fam_+$ separate $I$ from $\partial O_{[x,y]}$, most of the curves in $\RR$ must intersect $N_1\cup M_1$ before intersecting $\lfloor N_1\cup M_1\rfloor^c \cup \partial^\out O_J$.
\end{proof}

 Let $X,Y\subset\partial \wZ^{m+1}$ be a pair of intervals and let $A$ be one of the complementary intervals to $X,Y$. We denote by $\Fam_A(X,Y)$ the subfamily of $\Fam(X,Y)$ consisting of curves that are disjoint from $\partial \wZ^{m+1}\setminus \big(X\cup A\cup Y\big)$.

Let $L_a$ be the shortest complementary interval between $N_1$ and $I$, and let $L_b$ be the shortest complementary interval between $I$ and $M_1$. Claim~\ref{cl5:I_1} implies that either \[\Width _{L_a}(N_1, I)\ge 0.49 K\sp\sp\text{ or }\sp\sp \Width _{L_b}(I,M_1)\ge 0.49 K.\]
Setting $I_0\coloneqq I$ and $I_1, L_1$ to be either $N_1, L_a$ or $M_1, L_b$, we obtain $\Fam _{L_1}(I_1, I_0) \ge 0.49 K$. Note that $|L_1|\approx_{\length_{m+1}} |I_1|$. We can now proceed by induction:

\begin{claim5}
\label{cl5:I_n}
There is a sequence of grounded intervals \[I_1,I_2,\dots, I_{k},\sp\sp |I_t|\asymp \nu \length_{m+1}/3^t, \sp\sp k\asymp \log \nu, \sp\sp |I_k|\asymp \length _{m+1} \]
such that for every $I_t$ there is a $j(t)\in \{0,1,\dots, t-1\}$ with \[ \Width _{L_t}(I_t, I_{j(t)})\ge  0.49 \cdot 1.9^{t} K,\sp\sp\sp\sp
|L_t| \approx_{\length_{m+1}} |I_t|\asymp \nu \length_{m+1}/3^t,\]
where $L_t$ is the shortest complementary interval between $I_t$ and $I_{j(t)}$. Moreover,
\begin{equation}
\label{eq:cl5:I_n}
\Width _{L_t}(I_t, I_{j(t)}) - \Width^+ _{L_t}(I_t, I_{j(t)})  \ge  0.48 \cdot 1.9^{t} K \sp\sp\text{ for }t<k
\end{equation} 

and $\dist(I_{k-1}, I_k)\ge 0.6 | I_k|\ge 4\length_{m+1}$.

 \end{claim5}
\begin{proof} 
Assume that $I_t$ is constructed for $t\ge 1$. Set \[I_{t+1}\subset L_t \sp\sp\text{ with }\sp\sp |I_{t+1}|\approx_{\length_{m+1}} \dist(I_{t+1}, I_t)\approx_{\length_{m+1}} \dist(I_{t+1}, I_{j(t)})\]
to be a middle $1/3$-subinterval of $L_t$, see Figure~\ref{fig:def:I_n}. Let us show that either
\begin{equation}
\label{eq:prf:cl5:I_n}
 \Width_{L_a}(I_{t+1}, I_t)\ge 0.49 \cdot 1.9^{t+1} K\sp\sp\text{ or }\sp\sp \Width_{L_b}(I_{t+1}, I_{j(t)})\ge 0.49 \cdot 1.9^{t+1} K,
\end{equation} 
 where $L_a$ and $L_b$ are the shortest complementary intervals between $I_{t+1}, I_t$ and $I_{t+1}, I_{j(t)}$. This would finish the construction of $I_{t+1}$ with $j(t+1)\in \{t, j(t)\}$.

Consider a grounded interval $X\subset \partial \wZ^{m+1}$ attached to $I_t$ so that $|I_t|\approx _{\length_{m+1}} |X|$ and $I_t$ is between $X$ and $I_{t+1}$ in $T$. As in Claim~\ref{cl5:I_1}, consider the outer and inner geodesic rectangles \[\Fam_+\subset \wC\setminus \wZ^m,\sp \sp \Fam_-\subset \wZ^m,\sp\sp \partial ^{h,0}\Fam_+=\partial ^{h,0}\Fam_- = X,\sp  \partial ^{h,1}\Fam_+=\partial ^{h,1}\Fam_- = I_{t+1}\] between $X$ and $I_{t+1}$. By Theorems~\ref{thm:par fjords} and~\ref{thm:wZ:shallow scale}:
\begin{itemize}
\item if $|I_{t+1}|\gg_{\bdelta} 1$, then $\Width(\Fam_-)\asymp \Width(\Fam_+)\asymp 1$;
\item otherwise $\Width(\Fam_-)\asymp_{\bdelta} 1\asymp_\bdelta \Width(\Fam_+)$.
\end{itemize}

In the first case,  after removing $O(1)$ curves, the family $\Fam _{L_t}(I_t, I_{j(t)})$ overflows consequently $\Fam_{L_a}(I_{t+1}, I_t)$ and then $\Fam_{L_b}(I_{t+1}, I_{j(t)})$. The Equation~\eqref{eq:prf:cl5:I_n} follows.

In the second case, we have $t \gg_\bdelta 1$ because $\nu \gg _\bdelta 1$. Therefore, we can still remove $O_{\bdelta} (1)\ll  \Width _{L_t}(I_t, I_{j(t)})$ curves from  $\Fam _{L_t}(I_t, I_{j(t)})$; the remaining family overflows consequently $\Fam_{L_a}(I_{t+1}, I_t)$ and then $\Fam_{L_b}(I_{t+1}, I_{j(t)})$. The Equation~\eqref{eq:prf:cl5:I_n} follows. 

Note that we also established~\eqref{eq:cl5:I_n} for $t$. The induction can be proceed until $|I_t|> 20 \length_{m+1}$.

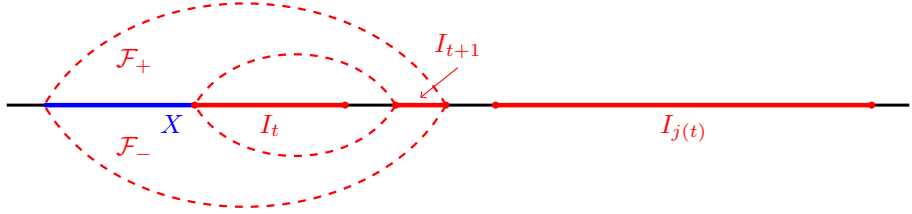
\begin{figure}[t!]

\[
\begin{tikzpicture}

\draw[line width=0.4mm] (-6.5,0)--(5.5,0);

\draw[blue, line width=0.6mm] (-6,0) --(-4,0);
\node[blue,below] at(-4.3,0) {$X$};

\filldraw[red] (0,0) circle (0.04 cm);
\filldraw[red] (5,0) circle (0.04 cm);
\draw[red, line width=0.6mm] (0,0) --(5,0);

\node[below,red] at (2.5,0) {$I_{j(t)}$};

\filldraw[red] (-4,0) circle (0.04 cm);
\filldraw[red] (-2,0) circle (0.04 cm);
\draw[red, line width=0.6mm] (-4,0) --(-2,0);
\node[below,red] at (-3,0) {$I_{t}$};

\filldraw[red] (-0.66,0) circle (0.04 cm);
\filldraw[red] (-1.33,0) circle (0.04 cm);
\draw[red, line width=0.6mm] (-0.66,0) --(-1.33,0);
\node[above,red] at (-0.5,0.5) {$I_{t+1}$};

\draw[red] (-0.5,0.5) edge[->] (-1,0.1);

\draw[red, dashed, line width=0.3mm] (-1.33,0) edge[bend left=60](-4,0);
\draw[red, dashed, line width=0.3mm] (-0.66,0) edge[bend left=60](-6,0);
\draw[red, dashed, line width=0.3mm] (-1.33,0) edge[bend right=60](-4,0);
\draw[red, dashed, line width=0.3mm] (-0.66,0) edge[bend right=60](-6,0);

\node[below,red] at (-4.8,-0.3) {$\Fam_-$};
\node[above,red] at (-4.8,0.3) {$\Fam_+$};

\end{tikzpicture}
\]

\caption{The interval $I_{t+1}$ is a middle $1/3$ interval between $I^t$ and $I^t_{j(t)}$. Since $\Fam_-$ and $\Fam_+$ separate $I^t$ and $I_{j(t)}$, most of the curves in $\Fam (I_t, I_{j(t)})$ must intersect $I_{t+1}$.}
\label{fig:def:I_n}
\end{figure}
\end{proof}
The Welding Lemma now follows from Lemma~\ref{lem:trad width to width+} applied to~\eqref{eq:cl5:I_n} with $t=k-1$.\qed

\subsection{Proof of Corollary~\ref{cor:regul}} \label{ss:prf of cor:regul}Write $K\coloneqq \Width(\RR)$ and let $\RR^\new$ be the rectangle obtained from $\RR$ by removing the outermost $K/2$ buffer. If $\RR^\new$ is not central, then Lemma~\ref{lem:central cond} implies Case~\eqref{case:2:thm:regul} of the corollary.

Assume that $\RR^\new$ is central and write $T=[a_0,a_1]$ with $a_0< \partial ^{h,0} \RR^\new<\partial ^{h,1}\RR^\new< a_1$. By Theorem~\ref{thm:par fjords},
\begin{equation}
\label{eq:Weld:Prf:1}
\log \frac{\dist (a_0, \partial ^{h,0}\RR^\new)}{\length_{m+1}} ,\sp  \log \frac{\dist (\partial ^{h,1}\RR^\new, a_1)}{\length_{m+1}} \succeq K  
\end{equation}
because the removed outermost buffer from $\RR$ has width $K/2$. As in~\S\ref{sss:regul within rect}, let $\RR^\new_{-j}$ for $j<\qq_{m+1}$ be the pullback of $\RR^\new$ along $f^j\colon \overline Z\selfmap$. Then every $\RR^\new_{-j}$ is based on a certain $T_{i}=[a_i,a_{i+1}]\in \Dbb_m$, $i=t(j)$, where $T_i$ are enumerated from left-to-right. We write $\RR^\new_{i(j)}=\RR^\new_{-j}$

Fix a big $S\gg 1$ independent of (and much smaller than) $K$. We can select intervals $X_i\subset \partial ^{h,0} \RR^\new_i$ and $Y_i\subset \partial^{h,1}\RR^\new_i$ such that 
\begin{enumerate}[label=\text{(\Alph*)},font=\normalfont,leftmargin=*]
\item \label{cond:1:WeldLmm:prf}the endpoints of $ X_i,Y_i$ are in $\CP_{m+1}$;
\item all $X_i$ are obtained by spreading around $X_0$ and, similar, all $Y_i$ are obtained by spreading around $Y_0$;
\item \label{cond:2:WeldLmm:prf}$0.99<\frac {|X_i|}{ |Y_i|}=\frac {|X_0|}{ |Y_0|}<1.01$;
\item \label{cond:3:WeldLmm:prf}$S \dist(X_i , a_i)<|X_i|<(S+1) \dist(X_i , a_i)$;
\item \label{cond:4:WeldLmm:prf}$S \dist(Y_i , a_{i+1})<|Y_i|<(S+1) \dist(Y_i , a_{i+1})$;
\item \label{cond:5:WeldLmm:prf} the geodesic rectangle $\FamG_i\subset \wC\setminus Z$ between $X_i, Y_i$ splits $\RR^\new_i$ into two rectangles with width $\asymp K$. 
 \end{enumerate}
 
 \noindent Here, to achieve~\ref{cond:4:WeldLmm:prf} and~\ref{cond:5:WeldLmm:prf} we use the property that $\RR^\new$ is central. Property~\ref{cond:6:WeldLmm:prf} follows from Theorem~\ref{thm:par fjords} and~\ref{cond:3:WeldLmm:prf},~\ref{cond:4:WeldLmm:prf}~\ref{cond:5:WeldLmm:prf}. By construction (by~\eqref{eq:Weld:Prf:1} and Theorem~\ref{thm:par fjords}), we also have
\begin{enumerate}[label=\text{(\Alph*)},start = 6,font=\normalfont,leftmargin=*]
\item \label{cond:6:WeldLmm:prf} $\log|X_i|, \log |Y_i| \succeq K$.  
 \end{enumerate}

 By~\ref{cond:2:WeldLmm:prf},~\ref{cond:3:WeldLmm:prf},~\ref{cond:4:WeldLmm:prf} and Theorem~\ref{thm:par fjords}, we have $\Width(\FamG_i)\asymp \log S\gg 1$.

 Let $X_i^{m+1},Y^{m+1}_i$ be the projections of $X_i,Y_i$ onto $\partial \wZ^{m+1}$. Similar, let $\FamG_i^{m+1}$ be the restriction of $\FamG_i$ onto $\wC\setminus \intr \wZ^{m+1}$. By \eqref{eq:lem:RR vs RR^m}, $\Width(\FamG^{m+1}_i)\asymp \log S \gg 1$.

Let $\FamH_i\subset \wZ^{m+1}$ be the geodesic rectangles between $Y^{m+1}_{i-1}$ and $X^{m+1}_{i}$. By~\ref{cond:2:WeldLmm:prf},~\ref{cond:3:WeldLmm:prf},~\ref{cond:4:WeldLmm:prf}, \ref{cond:6:WeldLmm:prf}, and Theorem~\ref{thm:wZ:shallow scale}, we have $\Width(\FamH_i)\asymp \log S\gg1$. Below we will construct relevant objects for $\wZ^m$ to satisfy~\S\ref{sss:regul within rect}.  

\subsubsection{Channels $\alpha_i$ and dams $\beta_i$, see Figure~\ref{Fg:wZ_Z_m}} Let us select points $x_i\in X_i \cap \CP_{m+1}, \ y_i\in Y_{i}\cap \CP_{m+1}$ such that the hyperbolic geodesics $\alpha_i=[y_{i-1},x_{i}]_{\wZ^{m+1}}$ and $\beta_i=[x_i, y_i]_{\wC\setminus Z}$ split every $\FamH_i$ and $\FamG_i^{m+1}$ into two subrectangles of width $\asymp \log S$ respectively. Moreover, we can choose genuine subrectangles $\FamG^{m+1}_{i,a}, \FamG^{m+1}_{i,b}, \FamG^{m+1}_{i,c}$ in $\FamG^{m+1}_i$ and genuine subrectangles $\FamH_{i,a}, \FamH_{i,b}, \FamH_{i,c}$ in $\FamH_i$  such that
\begin{itemize}
\item all subrectangles have width $\asymp S$;
\item $\FamG^{m+1}_{i,b}$ is between $\FamG^{m+1}_{i,a}$ and $\FamG^{m+1}_{i,c}$ and contains $\beta_i$ in the middle, i.e.~$\beta_i$ splits $\FamG^{m+1}_{i,b}$ into two subrectangles of width $\asymp S$;
\item $\dist(\partial^h \FamG^{m+1}_{i,b}, \partial^h  \FamG^{m+1}_{i,a}), \dist(\partial^h \FamG^{m+1}_{i,b}, \partial^h \FamG^{m+1}_{i,c})\gg \length_{m+1}$;
\item $\FamH_{i,b}$ is between $\FamH_{i,a}$ and $\FamH_{i,c}$ and contains $\alpha_i$ in the middle, i.e.~$\alpha_i$ splits $\FamH_{i,b}$ into two subrectangles of width $\asymp S$;
\item $\dist(\partial^h \FamH_{i,b}, \partial^h \FamH_{i,a}), \dist(\partial^h \FamH_{i,b}, \partial^h \FamH_{i,c})\gg \length_{m+1}$.
\end{itemize}  

Following~\eqref{eq:dfn UU(wZ)}, we denote by $f^k_*(\alpha_i), f^k_*(\beta_i)$ either the $f^k$-images of $\alpha_i, \beta_i$ if $k>0$ or the lifts under $f^{-k}$ starting and ending at $\partial Z$ if $k<0$, where $|k|\le \qq_{m+1}$. Similarly, $f^k_* (\FamG_{i}), f^k_* (\FamH_{i})$ are defined.

For every $|k|\le \qq_{m+1}$, the image $f^{k}_*(\beta_i)$  is in a certain $\FamG^{m+1}_j$ because at most $O(1)$ curves in $f^{k}_* \big(\FamG_{i,b}\big)$ can cross $\FamG^{m+1}_{j,a},\FamG^{m+1}_{j,c}$. Similarly, $f^{k}_*(\alpha_i)\cap \wZ^m$ is in a certain $\FamH_j$ because $f^{k}_*\big(\FamH_{i,b}\big)$ is disjoint from $\partial^h\FamH_{j,a}\cup \partial^h \FamH_{j,c} $ by Lemma~\ref{lem:HypGeod in wZ^m},~\ref{Case3:lem:HypGeod in wZ^m}, and hence at most $O(1)$ curves in $f^{k}_* \big(\FamH_{i,b}\big)$ can cross $\FamH_{j,a},\FamH_{j,c}$. By Lemma~\ref{lem:HypGeod in wZ^m},~\ref{Case3:lem:HypGeod in wZ^m}, $f^k_*(\alpha_i)$ can intersect only components of $f^k_*\left(\wZ^{m+1}\right)$ (see~\eqref{eq:dfn UU(wZ)}) that are close to the endpoints of $f^k_*(\alpha_i)$. Therefore, the orbits  $f^k_*(\alpha_i), f^k_*(\beta_i)$ with $|k|\le \qq_{m+1}$ is within $\bigcup_{i}  \big(\FamG_i^{m+1}\cup \FamH^{m+1}_i \big)$. This verifies Assumption~\ref{ass:wZ:2}.

\subsubsection{Collars $A(\alpha_i), A(\beta_i)$} \label{sss:WL:collars}By~\ref{cond:5:WeldLmm:prf}, we can choose well-grounded enlargements of intervals \[X_i=X_{i,0}\subset X_{i,1}\subset X_{i,2}\subset X_{i,3}\subset X_{i,4}\subset \partial^{h,0}\RR^\new_i,\]\[ Y_i=Y_{i,0}\subset Y_{i,1}\subset Y_{i,2}\subset Y_{i,3}\subset Y_{i,4}\subset \partial^{h,1}\RR^\new_i\]
such that for every $t\in \{1,2,3,4\}$ and every $i$
\begin{itemize}
\item $X_{i,t}\setminus X_{i,t-1}$ consists of a pair of intervals $X_{i,t}^{\pm}$ with $X^+_{i,t}<X_{i,t-1}<X^-_{i,t}$ such that $|X^+_{i,t}|\asymp_S |X^-_{i,t}|\asymp_S \dist(X^+_{i,t}, a_i)\asymp_S \dist(X^-_{i,t}, a_i)\asymp_S |X_i| $;
\item $Y_{i,t}\setminus Y_{i,t-1}$ consists of a pair of intervals $Y_{i,t}^\pm$ with $Y^-_{i,t}<Y_{i,t-1}<Y^+_{i,t}$ such that $|Y^-_{i,t}|\asymp_S |Y^+_{i,t}|\asymp_S \dist(Y^-_{i,t}, a_{i+1})\asymp_S \dist(Y^+_{i,t}, a_{i+1})\asymp_S |Y_i|$.
\end{itemize}
Taking the projection of the new intervals onto $\partial \wZ^m$ and using Theorems~\ref{thm:par fjords} and~\ref{thm:wZ:shallow scale}, we obtain that
\begin{itemize}
\item the geodesic rectangles $\FamG^+_{i,t}$ and $\FamG^-_{i,t}$ of $\wC\setminus \wZ^{m+1}$ between $(X^+_{i,t})^{m+1}$, $(Y^+_{i,t})^{m+1}$ and between $(X^-_{i,t})^{m+1}$, $(Y^-_{i,t})^{m+1}$ have width $\asymp_S 1$;
\item the geodesic rectangles $\FamH^+_{i,t}$ and $\FamH^-_{i,t}$ of $\wZ^{m+1}$ between $(Y^+_{i-1,t})^{m+1}$, $(X^+_{i,t})^{m+1}$ and between $(Y^-_{i-1,t})^{m+1}$, $(X^-_{i,t})^{m+1}$ have width $\asymp_S 1$.
\end{itemize}

Applying Welding Lemma~\ref{lmm:welding} with
\[J=N\#I\#M =X^{m+1}_{i,t}, \sp\text{ and }\sp  I=X_{i,t-1}^{m+1}; \sp\sp  \sp J=Y^{m+1}_{i,t}\sp\sp\text{ and }\sp  I=Y_{i,t-1}^{m+1}, \]
, we obtain either Case~\eqref{case:3:thm:regul} of the corollary, or:
\begin{equation}
\label{eq:mod O:bounds}
\mod(O_{X^{m+1}_{i,t}}\setminus X_{i,t-1}^{m+1}),\sp \mod(O_{Y^{m+1}_{i,t}}\setminus Y_{i,t-1}^{m+1}) \ge \varepsilon= \varepsilon(S).
\end{equation}

We now construct collars \[A^\inn(\alpha_i),\sp A^\out(\alpha_i),\sp  A^\inn(\beta_i), \sp A^\out(\beta_i)\]
 as annuli bounded by hyperbolic geodesics of $\wC\setminus \overline Z$ and of $\wZ^{m+1}$  such that their outer boundaries pass through the endpoints of \[Y^{m+1}_{i-1,2}\cup X^{m+1}_{i,2},\sp\sp Y^{m+1}_{i-1,4}\cup X^{m+1}_{i,4},\sp X^{m+1}_{i,2}\cup Y^{m+1}_{i,2},\sp X^{m+1}_{i,4}\cup Y^{m+1}_{i,4}\] while their inner boundaries pass through the endpoints of \[Y^{m+1}_{i-1,0}\cup X^{m+1}_{i,0},\sp\sp Y^{m+1}_{i-1,2}\cup X^{m+1}_{i,2},\sp X^{m+1}_{i,0}\cup Y^{m+1}_{i,0},\sp X^{m+1}_{i,2}\cup Y^{m+1}_{i,2}\] 
 respectively. The moduli of the collars are bounded by Lemma~\ref{lem:A:encl form} (see also Figure~\ref{fig:lem:A:encl form}) by $\bdelta=\bdelta(S)$ because we have bounds on the width of $\FamG^\pm_{i,t}, \FamH^\pm _{i,t}$ and the moduli bounds~\eqref{eq:mod O:bounds}.

This verifies Assumptions~\ref{ass:wZ:collars} and~\ref{ass:wZ:int patt}. Assumption~\ref{ass:wZ:CombSpace} follows from~\eqref{eq:Weld:Prf:1}. Assumption~\ref{ass:wZ:bDelta} follows from Theorems~\ref{thm:par fjords} and~\ref{thm:wZ:shallow scale}. 

Extra protections $\XX^m_i$ for Assumption~\ref{ass:wZ:EtraProt} can be selected as subrectangles of $\RR_i^\new$. Assumption~\ref{ass:wZ:Linking} and conditions in~\S\ref{sss:regul within rect}  hold by construction.

\begin{rem}
\label{rem:selections of:lmm:welding:cond} Let us comment that we select the comparison ``$\asymp$''  in~\eqref{eq:mod O:bounds} for~\eqref{lmm:welding:cond} of the Welding Lemma~\ref{lmm:welding} so that channels $\alpha_i$ and dams $\beta_i$ are in the middles of sufficiently wide (and hence almost invariant under $f^{\qq_{m+1}}$) rectangles $\FamG_i, \FamH_i$ between $X_i^{m+1},Y^{m+1}_i\subset \partial \wZ^{m+1}$. As a result, the collars $A(\alpha_i), A(\beta_i)$ have $\mod(\ ) \succeq \varepsilon$, where $\varepsilon$ depends on the width of $\FamG_i, \FamH_i$.
\end{rem}

\qed


\part{Covering and Calibration lemmas}
\label{part:CovCalibr}

\section{Covering and Lair Lemmas}\label{s:Covring+Lair Lmms}
In the section we will prove the following theorem that can be characterized by the principle ``if the life is bad now, then it will be worse tomorrow"\footnote{Compare with Kahn's principle: ``If the life is bad now, then it was even worse yesterday.''}:

\begin{ampthm}
\label{thm:SpreadingAround}
There are increasing functions \[\blambda_\bbt,\sp  \bK_\bbt \sp\sp \text{ for }\tt>1\sp\sp\text{ with }\sp \blambda_\bbt,  \ \bK_\bbt\ \underset{t\to \infty}\longrightarrow \ +\infty \] such that the following holds. Suppose that there is a combinatorial interval
\[ I\subset \partial Z\sp\sp \text{ such that }\sp  \Width^+_{\blambda_\bbt}(I) \eqqcolon K \ge \bK_\bbt \sp \text{ and }\sp |I|\le |\theta_0|/(2\blambda_\bbt). 
\]
Consider a geodesic pseudo-Siegel disk $\wZ^m$, where $m$ is the level of $I$.  Then there is a grounded rel $\wZ^m$ interval \[J\subset \partial Z \sp\sp \text{ such that } \sp \Width^+_{\blambda_\bbt} (J) \ge \bbt K \sp \text{ and }\sp |J|\le |I|.
\]
\end{ampthm}
\subsubsection{Motivation and outline} \label{sss:outline of CovL LairL} Recall from~\S\ref{sss:renorm til} that a forward orbit of a combinatorial interval up to the first return almost tiles $\partial Z$. If a combinatorial interval $I\subset \partial Z$ witnesses a big degeneration, say that $I$ is $[K,\lambda]^+$-wide with $K\gg_{\lambda}1$, then, using the Covering Lemma, we spread this degeneration around $\partial Z$ and obtain an almost tiling $I_k$ of $\partial Z$ so that, \emph{roughly} $I_k$ is $[C K, \lambda]$-wide for an absolute $C>0$. (Covering Lemma~\ref{lem:CovLmm} has two possibilities; we are omitting the ``local'' Case~\eqref{case:eq:ColAssum} in this outline.)  The constant $C$ is independent of $\lambda$; the $\lambda$ influences only the degeneration threshold $\bK_{\bbt}\gg_\lambda 1$. In short, Snake-Lair Lemma~\ref{lem:Hive Lemma} states that if $\lambda\gg _{C, \bbt} 1$, then $\lambda$ ``beats'' $C$ and produces a $[\bbt K, \lambda]^+$-wide interval $J$ on a deeper scale. More precisely, since wide families $\Fam_\lambda(I_k)$ combinatorially block each other, they must submerge under each other resulting in long snakes. Then Snake Lemmas~\ref{lem:SnakeWithCheckpts} and \ref{lem:sneaking} are applicable.

A key technical issue is that the new wide interval $J$ may be far from being combinatorial. Namely, the resulting wide family $\Fam^+_\lambda(J)$ can be within a wide non-winding parabolic rectangle -- such rectangles exist and are described by Theorem~\ref{thm:par fjords}. To deal with this issue, we apply the Covering and Snake-Lair Lemmas to the pseudo-Siegel disk $\wZ^m$ instead of $\overline Z$. Pseudo-Siegel disks are almost invariant up to $\sim \qq_{m+1}$ iterates~\S\ref{sss:Pullbacks of wZ^m} -- this is sufficient to spread the degeneration around using the Covering Lemma. Lemma~\ref{lem:W+:ground inter} allows us to trade $\Width^+$ wide families between $\overline Z$ and $\wZ^m$.

In~Section \ref{s:MainThms}, we will inductively construct  (from the deep to shallow scales) $\wZ^m$ so that it absorbs ``most'' of the non-winding parabolic rectangles. Then the Calibration Lemma will replace $J$ with a combinatorial interval on a deeper scale.

\subsection{Applying the Covering Lemma}\label{ss:ApplyingCoveringLemma}
As in~\S\ref{sss:reg interv}, we will denote by $I^m$ the projection of a regular interval $I\subset \partial Z$ onto $\partial \wZ^m$.
\begin{lem}
\label{lem:SpredAroundWidth}
For every $\bkappa>1$ and $\lambda>10$, there is $\bK_{\lambda,\bkappa}>1$ and $C_\kappa$ (independent of $\blambda$) such that the following holds. Suppose that there is a combinatorial interval \[I\subset\partial Z\sp\sp \text{ such that }
\sp\Width^+_{\lambda+2}(I)=K \ge \bK_{\lambda,\bkappa},\sp\sp |I|\le \theta / (2\lambda+4),\sp\sp m =\Level(I)\]
and such that one of the endpoints of $I$ is in $\CP_m$. Let $\wZ^m$ be a geodesic pseudo-Siegel disk (see~\S\ref{sss:regul within rect}), and \[ I_s\subset \partial Z,\sp\sp I_s=f^{i_s}(I),\sp\sp s\in \{0,1,\dots, \qq_{m+1}-1\}\] be the intervals obtained by spreading around $I=I_0$ (as in \S\ref{sss:spread around}). Then every interval $I_s$ is well-grounded rel $\wZ^{m+1}$ and its projection $I^m_s\subset \partial \wZ^m$ is  
\begin{enumerate}
\item either $[\bkappa K,10]$-wide; 
\item or $[ C_{\bkappa} K,\lambda]$-wide.
\end{enumerate} 
\end{lem}

\begin{proof}

Since one of the endpoints of $I$ is in $\CP_m$, all intervals $I_s,\sp s <\qq_{m+1}$ are well grounded, see the Remark~\ref{rem:easy cond:well ground}. 

We will start the proof by introducing appropriate branched covering restrictions of the $f^{i_s}$ with uniformly bounded degrees. Then we will apply the Covering Lemma. The condition  ``$ |I|\le |\theta_0|/(2\lambda+4)$'' will be used in removing slits.

\subsubsection{Projections onto $\wZ^m$} Let us first approximate $10 I$ and $\lambda I$ with well-grounded intervals. Choose intervals $L$ and $T$ whose endpoints are in $\CP_m$
such that
\[ 10 I \subset T \subset 12 I \sp\sp \text{ and }\sp\sp  \lambda I \subset L\subset (\lambda +2)I .\]
Applying $f^{i_s}$ to $T$ and $L$ we obtain the intervals $T_s$ and $L_s$ respectively satisfying 
\[ 10 I_s \subset T_s \subset 12 I_s \sp\sp \text{ and }\sp\sp  \lambda I_s \subset L_s\subset (\lambda +2)I_s .\] 
Then $T_s, L_s$ are well-grounded rel $\wZ^m$, see~\eqref{eq:rem:easy cond:well ground}.

\subsubsection{Covering structure around $f^{i_s}\mid I$} 
\label{sss:CovStrZ}
 Observe first that $I\subset \partial Z$ contains at most one critical point of $f^{i_s}$ because the map $f^{i_s}\colon I\to I_s$ realizes the first landing of points in $I$ onto $I_s$, see Lemma~\ref{lem:FirstLand}.

 Since $ |I|\le |\theta_0|/(2\lambda+4)<1/2$, the interval $(L_s)^c =\partial Z\setminus L_s$ has length greater than $1/2$. Consider a simple arc $\gamma_s\subset \C\setminus \overline Z$ connecting $(L_s)^c$ to $\infty$; we will specify $\gamma_s$ in~\S\ref{sss:removing gamma}. Then 
\begin{equation}
\label{eq:BrStr:V}
V\coloneqq \C\setminus (\gamma_s\cup (L_s)^c)
\end{equation}
is an open topological disk. Define $U_{-s}$ to be the pullback of $V$ along $f^{i_s}\mid I$. We obtain a branched covering 
\begin{equation}
\label{eq:UtoV}
f^{i_s}\colon U_{-s} \to V.
\end{equation}
\begin{lem}
\label{lem:f:UtoV}
The degree of~\eqref{eq:UtoV} is at most $4^{\lambda+2}$. 
\end{lem}
\begin{proof}
Let us present~\eqref{eq:UtoV} as the composition of branched coverings 
\[U_{-s}=X_0\overset{f}{\longrightarrow} X_1\overset{f}{\longrightarrow} X_2 \overset{f}{\longrightarrow} \dots \overset{f}{\longrightarrow} X_n=V.\] 
Observe that $X_j \cap \partial Z$ is the interior of the interval $f^{j}(L)$. The map $f\colon X_j\to X_{j+1}$ has degree $2$ if and only if $f^{j}(L)$ contains $c_0$ in its interior. Since $f^{i_s}\colon I\to I_s$ is the first landing, there are at most $2(\lambda+2)$ moments $t\in \{0,1,\dots, i_s\}$ such that $(\lambda+2)f^t(I)\supset  f^t(L)\ni c_0$. The lemma follows.
\end{proof}
\end{proof}

\subsubsection{Covering structure around $f^{i_s}\mid I^m$} Consider the projection $L^m_s\subset \partial \wZ^m$ of $L_s$. Similar to~\S\ref{sss:CovStrZ}, we choose a simple arc $\gamma^m_s\subset \wC\setminus \wZ^m$ connecting
 \[(\wL^m_s)^c =\partial \wZ^m\setminus \wL^m_s\sp\sp\text{ and }\sp \infty.\]  
By Lemma~\ref{lem:Pullback of wZ^m}, $\wZ^m$ has the conformal pullback  $f^{i_s}\colon  \wZ^m_{-s} \to \wZ^m$ such that $\wZ^m_{-s}$ is also a pseudo-Siegel disk. We denote by $I^{m,-s}$ the projection of $I$ onto $\wZ^m_{-s}$. Then $ I^m_s=f^{i_s}(I^{m,-s})$ is the projection of $I$ onto $\wZ^m$.

 Similar to \eqref{eq:UtoV}, we define the branched covering
\begin{equation}
\label{eq:UtoV:m}
f^{i_s}\colon U^m_{-s} \to V^m\coloneqq  \C\setminus \left(\gamma^m_s\cup \big(\wL^m_s\big)^c\right),
\end{equation}
where $U^m_{-s}$ is the pullback of $V^m$ along $f^{i_s}\colon I^{m,-s}\to I^m_s$.
\begin{lem}
\label{lem:f:UtoV:m}
The degree of~\eqref{eq:UtoV} is at most $4^{\lambda+2}$. \qed
\end{lem}
\begin{proof}
All critical values of $f^{i_s}$ are in $\partial Z\cap \partial \wZ^m$ and we can repeat the argument of Lemma~\ref{lem:f:UtoV}.
\end{proof}

\subsubsection{Covering Lemma}
The Covering Lemma was proven in~\cite{KL}; for our convenience we will state it in terms of the width $\Width(A)$ instead of $\mod(A)=1/\Width(A)$ for an annulus $A$. We will also state the Collar Assumption~\eqref{case:eq:ColAssum}  as one of the alternatives.

\begin{lem}[Covering Lemma]
\label{lem:CovLmm}
 Fix some $\kappa >1$. Let $U \supset  \Lambda' \supset \Lambda$ and
$V \supset B'\supset B$ be two nests of Jordan disks. Let \[f : (U,\Lambda'
, \Lambda) \to (V, B'
, B)\]  be a branched covering between the respective disks, and let
$D = \deg(U \to  V )$, $d = \deg(\Lambda' \to B'
)$. Then there is a $K_1>0 $ (depending on $\kappa$ and $D$) such that the following holds. If 
\[ \Width( U\setminus \Lambda) >K_1 ,\] then either
\begin{enumerate}
\item $\Width(B' \setminus B) > \kappa \Width(U \setminus \Lambda),$ \label{case:eq:ColAssum} or
\item $\Width( V \setminus B) >\left(2\kappa d^2 \right)^{-1}\Width (U\setminus \Lambda).$
\end{enumerate}
\end{lem}

Consider~\eqref{eq:UtoV:m}  and recall that $I^{m,-s}$ to be the projection of $I$ onto $\wZ^m_{-s}$. We denote by $T^m_s\subset \partial \wZ^m$ the projection of $T_s$. Set  
\begin{itemize}
\item $B\coloneqq I^m_s$;
\item $\Lambda$ to be the connected component of $f^{-i_s}(I^m_s)$ containing $I^{m,-s}$;
\item $B'\coloneqq  \C\setminus \left(\gamma^m_s\cup \big(T^m_s\big)^c\right)$;
\item $\Lambda'$ to be the connected component of $f^{-i_s}(B')$ containing $
I^{m,-s}$. 
\end{itemize}

By Lemma~\ref{lem:f:UtoV:m} applied to the case $\lambda=12$, the degree of $f\colon \Lambda' \to B'$ is at most $d\coloneqq 4^{12}$. Clearly, 
\[\Width(U^m_{-s}\setminus \Lambda) \ge \Width^+_{\wZ^m_{-s}}(I^{m,-s}) \ge \Width^+_{Z}(I) -O(1)= K-O(1).\]

 Applying the Covering Lemma to 
 \[f^{i_s}\colon (U^m_{-s},\Lambda' , \Lambda)\to (V^m,B',B)\] with $\kappa=3\bkappa $, we obtain that either
\begin{itemize}
\item $\Width(B'\setminus B) \ge 3\bkappa  K$; or
\item $\Width(V\setminus  B) \ge C_{\bkappa} K$ otherwise.
\end{itemize}

\subsubsection{Removing $\gamma_s$}
\label{sss:removing gamma} It remains to remove $\gamma_s$ from 
\[V^m=  \C\setminus \left(\gamma^m_s\cup \big(L^m_s\big)^c\right) \sp\sp \text{ and }\sp\sp B'=  \C\setminus \left(\gamma^m_s\cup \big(T^m_s\big)^c\right)\]
without decreasing much $\Width(B'\setminus B)$ and $\Width(V\setminus  B)$.

Consider the outer harmonic measure of $\partial \wZ^m$ -- it is the harmonic measure of $\wC\setminus \partial \wZ^m$ relative $\infty$. If the outer harmonic measure of $L^m_s$ is less than $2/3$, then we can choose $\gamma^m_s$ so that the width of curves in $V^m$ connecting $B$ and $\gamma_s$ is $O(1)$. We obtain
\[ \Width_{10} (I^m_s)\ge \Width(B'\setminus  B)- O(1),\sp\sp \Width_{\lambda} (I^m_s)\ge \Width(V\setminus  B)- O(1)\]
and the lemma follows.

Consider the remaining case when the outer harmonic measure of $L^m_s$ is bigger than $2/3$. Let $\wZ^m_\bullet$ be the pullback of $\wZ^m$ under $f$; i.e.~$\wZ^m_\bullet$ is the pseudo-Siegel disk so that $f\colon \wZ^m_\bullet\to \wZ^m$ is conformal. Let $I'^m, L'^m,\subset \partial \wZ^m_\bullet$ be the preimages of $I^m_s,L^m_s$ under $f\colon \wZ^m_\bullet\to \wZ^m$. By Lemma~\ref{lem:W^+:well grnd int}, the outer harmonic measures of $L'^m\subset \partial \wZ^m_\bullet$ and $L^m_s\subset \partial \wZ^m$ are very close to the outer harmonic measures of $L',L_s\subset \partial Z$, where $L'$ is the projection of $L'^m$ onto $\overline Z$. Since $L',L_s$ are disjoint, the outer harmonic measure of $L'\subset \partial \wZ^m_\bullet$ is less than $2/3$. Repeating the above argument for $I'^m$, we obtain that 
either 
\begin{itemize}
\item $I'^m$ is $[2\bkappa K,10]$-wide;  or
\item $I'^m$ is $[ C_{\bkappa} K,\lambda]$-wide
\end{itemize}   
relative $\wZ^m_\bullet$. Applying $f$ which has the global degree $2$, we obtain (see~\eqref{eq:Width:degree d}) that either
\begin{itemize}
\item $I^m_s$ is $[\kappa K,10]$-wide;  or
\item $I^m_s$ is $[ C_{\bbt} K/2,\blambda_\bbt]$-wide.
\end{itemize}

\subsection{Lair of snakes}
For our convenience, we enumerate intervals clockwise in the following lemma.
\begin{lairlmm}
\label{lem:Hive Lemma}
For every $\bbt>2$ there are $\bkappa,\blambda,\bK\gg_{\bdelta}1$ such that the following holds. Suppose that $\wZ^m$ is a pseudo-Siegel disk with $\length_m< \blambda/4$. Let \[I_{n+1}=I_0, I_1,\dots , I_n\subset \partial \wZ^m, \sp\sp\sp |I_k| = \length_m,\sp  \dist(I_k, I_{k+1})\le \length_m\]
be a sequence of well-grounded intervals enumerated clockwise such that every $I_s$ is one of the following two {\bf types}
\begin{enumerate}
\item either $I_s$ is $[\bkappa K,10]$-wide, \label{case:I_k:1}
\item or $I_s$ is $[ C_{\bkappa} K,\blambda]$-wide,\label{case:I_k:2}
\end{enumerate} 
where $K\ge \bK$ and $C_\kappa$ is a constant (from Lemma~\ref{lem:SpredAroundWidth}) independent of $\lambda$. Then there is a $[\bbt K,3 \blambda]^+$-wide interval $J\subset \partial Z$ grounded rel $\wZ^m$ with $|J|<|I|$.
\end{lairlmm}

\begin{proof}
The first three claims below show that families of Type~\eqref{case:I_k:2} appear with certain frequency. The last three claims amplify their width (by the snake lemmas).
We assume that $\bK \gg\blambda \gg \bkappa\gg \bbt$. The first claim follows immediately from Lemma~\ref{lem:trad width to width+}. 

\begin{claim}
Lemma~\ref{lem:Hive Lemma} holds if there is a Type~\eqref{case:I_k:1} interval $I_j$ such that \[\Width_{10}(I_j)-\Width^+_{10}(I_j)\ge \bkappa  K/2.\] \qed
\end{claim}

We assume from now on that for every Type~\eqref{case:I_k:1} interval $I_j$ we have
\begin{equation}
\label{eq:Type1inerv:ass1}
\Width^+_{10}(I_j)\ge \bkappa K/2.
\end{equation}

Let us enlarge every $I_i$ into a well-grounded interval $\wI_i\subset \partial \wZ^m$ by adding to $I_i$ the interval between $I_i$ and $I_{i+1}$ if $I_i$ and $I_{i+1}$ are disjoint. Since the distances between the $I_i$ and $I_{i+1}$ are $\le \length_{m+1}$ (see~\S\ref{sss:renorm til}), we have $|\wI_i|\le 2 \length_m$.

\begin{claim}
\label{cl:seq:I_i}
There is a sub-sequence \begin{equation}
\label{eq:seq:I_i  I_i+lambda}
\wI_{i}, \wI_{i+1},\dots, \wI_{i+\blambda/20}\subset\partial \wZ^m, \sp\sp\sp |\wI|\le 2 \length_m
\end{equation}
such that  every interval $\wI_j$ in~\eqref{eq:seq:I_i  I_i+lambda} is not $[3,\blambda/4]^+$ wide.
\end{claim} 
 \begin{proof}
Suppose converse. Then $(\wI_k)$ has a sub-sequence $(L_k)$ with $L_k=\wI_{\ell(k)}$ such that $\ell(k)<\ell(k+1)\le \ell(k)+\blambda/20$ and such that every $L_k$ is $[3,\blambda/4]^+$-wide. This is impossible by Lemma~\ref{lem:no inf quasi additivity}  (such families would block each other).
\end{proof}

\begin{claim}
There is $\bbk$ depending on $\bbt$ and $\bkappa$ but not on $\blambda$ such that Lemma~\ref{lem:Hive Lemma} holds if~\eqref{eq:seq:I_i  I_i+lambda} has $\bbk$ consecutive Type~\eqref{case:I_k:1} intervals in~\eqref{eq:seq:I_i  I_i+lambda}. 
\end{claim}

\begin{proof}
Suppose that~\eqref{eq:seq:I_i  I_i+lambda} has a consecutive sequence of Type~\eqref{case:I_k:1} intervals $I_a,I_{a+1},\dots, I_b$ with $b-a= \bbk -1$. Consider the intervals 
\[ I_{a,b}\coloneqq \lfloor I_a, I_b\rfloor\sp\sp\sp \text{ and }\sp\sp\sp L\coloneqq \bigcap_{j=a}^b \left(\frac \blambda 4 I_j \right)^c\]
and observe that $\Width^+(I, L)= O(b-a)$ by Claim~\ref{cl:seq:I_i}. Since the $\Fam^+_{10}(I_j), \Fam^+_{10}(I_{j+1})$ have small overlaps (they block each other), there is a rectangle \[  \RR \subset \Fam\coloneqq \Fam^+_{10}(I_a)\cup \Fam^+_{10}(I_{a+1})\cup \dots \cup \Fam^+_{10}(I_b) \sp\sp\text{ with }\sp \Width(\RR)\succeq (b-a)\kappa K.\]
Let $J_a,J_b$ be two intervals forming $L\setminus I_{a,b}$. We assume that $J_a<I_{a,b}<J_b< L$. Since at most $O(b-a)$ curves in $\RR$ land at $L$, we can select a subrectangle $\RR_2$ in $\RR$ with $\Width(\RR)\succeq (b-a)\kappa K=\bbk \kappa K$ such that, without loss of generality, $\RR_2$ is lands at $J_a$.

By removing $O(1)$-buffer, we can assume that $\RR_2$ skips over $I_{a-1}\subset (10I_a)$. Since Type~\eqref{case:I_k:1} intervals block each other, $\RR_2$ goes above a  Type~\eqref{case:I_k:2} interval $I_x\subset J_a$. The claim now follows by applying Sneaking Lemma~\ref{lem:sneaking} to $\RR_2$ and $\Fam(I_x)$.
\end{proof}

We may now assume that among 
$\bbk$ consecutive intervals in Sequence~\eqref{eq:seq:I_i  I_i+lambda} there is at least one Type~\eqref{case:I_k:2} interval. Let us enumerate Type~\eqref{case:I_k:2} intervals in Sequence~\eqref{eq:seq:I_i  I_i+lambda} as
\[I_{i_0},I_{i_2},\dots ,I_{i_{s}},\sp\sp i_j<i_{j+1}<i_j + \bbk,\]
where $s \ge \blambda/(22\bbk).$

Let us enlarge $I_{i_t}$ to well grounded intervals $\widetilde I_{i_t}\supset I_{i_t}$ such that
\begin{itemize}
\item $\widetilde I_{i_t}$ ends where $ I_{i_{t+1}}$ starts; and
\item $\widetilde I_{i_0}$ and $\widetilde I_{i_s}$ have length between $\blambda/4+1$ and $\blambda/4+3$.
\end{itemize}

It follows from Claim~\ref{cl:seq:I_i} that most of the curves in $\Fam^+_\blambda (\widetilde I_{i_t})$ do not bypass $\widetilde I_{i_0}\cup \widetilde I_{i_s}$:
\begin{claim}
Write $L\coloneqq \lfloor \widetilde I_{i_0},  \widetilde I_{i_s}\rfloor $. Then for every $i\in \{1,\dots, s-1\}$, we have
 \[  \Width^+(\widetilde I_{i_t}, L) =O(\bbk) \sp\sp\sp\text{ for } 1\le t\le s-1.\]
\qed
\end{claim}

 
Choose a big $T\gg 1$ (but still much smaller than $\bK$; the $T$ depends on $\rr,\bbg$ from Claims~\ref{cl:5},~\ref{cl:6}). We consider the following fundamental arc diagram $\mathfrak G$:
\begin{itemize}
\item the vertices of $\mathfrak G$ are the intervals $\widetilde I_{i_j}$ for $j\in \{1,2,\dots,s-1\}$.
\item there is an edge between $\widetilde I_{i_a}$ and $\widetilde I_{i_b}$ if and only if \[|a-b|\ge 2\sp\sp \text{ and }\sp\sp \Width^+(\widetilde I_{i_a},\widetilde I_{i_b})\ge T.\]
 
\end{itemize}

\begin{figure}[t!]
\[\begin{tikzpicture}[scale=1.4]

\draw (-6.3,0) -- (3.2,0);

\filldraw (-6,0) circle (0.04 cm);

\filldraw (-3,0) circle (0.04 cm);
\filldraw (-2.5,0) circle (0.04 cm);
\filldraw (-2,0) circle (0.04 cm);
\filldraw (-1.5,0) circle (0.04 cm);
\filldraw (-1,0) circle (0.04 cm);
\filldraw (-0.5,0) circle (0.04 cm);
\filldraw (0,0) circle (0.04 cm);
\filldraw (0.5,0) circle (0.04 cm);

\filldraw (1,0) circle (0.04 cm);
\filldraw (1.5,0) circle (0.04 cm);
\filldraw (2,0) circle (0.04 cm);
\filldraw (2.5,0) circle (0.04 cm);
\filldraw (3,0) circle (0.04 cm);

\node [below] at (-6,0){$\widetilde I_a$}; 
\node [below] at (-2.9,0){$\widetilde I_{b_1}$}; 

\node [below] at (-1.3,0){$\widetilde I_{b_{f-\bbg}}$};
\node [below] at (0.1,0){$\widetilde I_{b_f}$};
\node [below] at (1.6,0){$\widetilde I_{b_{f+\bbg}}$};
\node [below] at (3.1,0){$\widetilde I_{b_{\btau}}$};

\draw [red]  (-6,0) edge[bend left]  (-3,0); 
\draw [red]  (-6,0) edge[bend left]  (-2.5,0);
\draw [red]  (-6,0) edge[bend left]  (-2,0);
\draw [red]  (-6,0) edge[bend left]  (-1.5,0);
\draw [red]  (-6,0) edge[bend left]  (-1,0);
\draw [red]  (-6,0) edge[bend left]  (-0.5,0); 
\draw [red]  (-6,0) edge[bend left]  (-0,0); 
\draw [red]  (-6,0) edge[bend left]  (0.5,0); 
\draw [red]  (-6,0) edge[bend left]  (1,0); 
\draw [red]  (-6,0) edge[bend left]  (1.5,0); 
\draw [red]  (-6,0) edge[bend left]  (2,0); 
\draw [red]  (-6,0) edge[bend left]  (2.5,0); 
\draw [red]  (-6,0) edge[bend left]  (3,0); 
\end{tikzpicture}\]
\caption{Illustration to the argument in Claim~\ref{cl:5}.}
\label{Fg:snake thorugh I_a}
\end{figure}
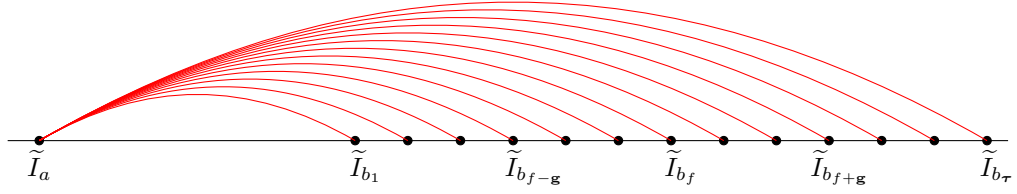

\begin{claim}
\label{cl:5}
There is a constant $\rr=\rr(\bbt)$ depending on $\bbt$ and $\bkappa$ but not on $\blambda$ such that Lemma~\ref{lem:Hive Lemma}  holds if $\mathfrak G$ has a vertex with degree $\rr$.
\end{claim}
\begin{proof}
For our convenience, we replace $\rr$ with $2\rr$ and we also introduce a big constant $\bbg$ with $\rr\gg \bbg \gg 1$.

  Assume that $\widetilde I_a$ is a vertex of $\mathfrak G$ with degree at least $2\rr$. Without loss of generality, we assume that $\rr$ neighbors of $I_a$ in $\mathfrak G$ are on the right of $I_a$; we enumerate these neighbors as $\widetilde  I_{b_1},\widetilde  I_{b_2},\dots , \widetilde  I_{b_\rr}$, see Figure~\ref{Fg:snake thorugh I_a}. We will show below that either there is an $\widetilde I_{b_f}$ such that $\Fam_\blambda (I_{b_k})$ submerges many times in the $I_{b_k}$, or the family $\Fam^+\left(\widetilde I_a, \lfloor \widetilde I_{b_2}, \widetilde I_{b_\rr}\rfloor\right)$ has width $\gg_{\bbt} K$. In the former case, we will apply the Snake Lemma with toll barriers. In the latter case, we will use the Sneaking Lemma.

    Consider $\widetilde  I_{b_f}$, where $f\in \{\bbg, \bbg+1,\dots, \rr-\bbg\}$, see Figure~\ref{Fg:snake thorugh I_a}. Observe that the width of curves in $\Fam_\blambda(I_{b_f})$ omitting $\widetilde I_{b_{f-\bbg}} \cup \widetilde I_{b_{f+\bbg}}\cup \widetilde I_a$ is at most $O_{\blambda}(1)$ because \[\Fam^+(\widetilde I_a, \widetilde I_{b_{f-\bbg}}),\sp\sp \Fam^+(\widetilde I_a, \widetilde I_{b_{f+\bbg}}),\sp\sp \Fam^-(\widetilde I_{b_{f-\bbg}}, \widetilde I_{b_{f+\bbg}}) \] have $\succeq_\blambda 1$ width. We orient curves in $\Fam_\blambda(\widetilde I_{b_f})$ from $\widetilde I_{b_f}$ towards $(\blambda \widetilde I_{b_f})^c$.
  
{\bf Case (A)}. Suppose there is an $f\in \{\bbg, \bbg+1,\dots, \rr-\bbg\}$ such that a $\frac 1 2 C_{\bkappa} K$ part of $\Fam_\blambda(I_{b_f})$, call it $\Fam$, intersects $\widetilde I_{b_{f-\bbg}} \cup \widetilde I_{B_{f+\bbg}}$ before intersecting $\widetilde I_a$. By the Small Overlapping Principle~\S\ref{sss:sm overl princ}, there are pairwise disjoint simple closed arcs \[\ell_j\in \Fam^+(\widetilde I_a, \widetilde I_{f+j})\sp\sp\text{ for }\sp  1\le |j-f|\le \bbg\] 
such that at most $O_\blambda(1)$ curves in $\Fam$ intersect $\bigcup_{1\le |j-f|\le \bbg} \ell_j$. Removing $O_\blambda(1)$ curves in $\Fam$ we obtain a lamination $\Fam^\new\subset \Fam$ whose curves are disjoint from any $\ell_j$. 

Suppose $\ell_j$ lands at $x_j\in \widetilde I_{b_j}$. Since $\dist(x_j,x_{j+2})\ge \length_m$, we can remove $O_\blambda(1)$-curves from $\Fam^\new$ so that every curve in the new family $\Fam^\New$ intersects 
\[\lfloor x_{b_{f-2j -2}},x_{b_{f-2j }}\rfloor \cup \lfloor x_{b_{f+2j}},x_{b_{f+2j+2 }}\rfloor \]  before intersecting $\lfloor x_{b_{f-2j -2}},x_{b_{f+2j+2 }}\rfloor^c$. For $|j|\le \bbg/2$, define the arc $\beta'_j\subset  \ell_{f-2j}\cup \widetilde I_a \cup \ell_{f+2j}$ to be the concatenation of $\ell_{f-2j}$, followed by the subarc of $ \widetilde I_a $, and followed by $ \ell_{f+2j}$. Note that $\beta'_j$ is a simple arc connecting $\widetilde I_{f-2j}$ and $\widetilde I_{f+2j}$. Moreover, $\beta'_j$ is disjoint from $\beta'_k$ away from $\widetilde I_a$. Let us slightly move the arcs $\beta'_j$ away from $\widetilde I_a$ so that the new arcs $\beta_j$ are pairwise disjoint and so that at most $O(1)$ curves in $\Fam^\New$ intersect any of $\beta_j$. We denote by $\Fam^\NEW$ the family obtained from $\Fam^\New$ by removing curves intersecting at least one $\beta_j$. Case (A) now follows from Snake Lemma~\ref{lem:SnakeWithCheckpts} applied to $\Fam^\NEW$, $\Width(\Fam^\NEW)\ge \frac 1 2 K -O_\blambda(1)$ with toll barriers $\beta_j$, $j\le \bbg/2$.

 {\bf Case (B)}. If {Case (A)} never occurs, then $C_\kappa K /2$-wide part of $\Fam^+(\widetilde I_a, \widetilde I_{b_f})$ is disjoint from $\Fam^+(\widetilde I_a, \widetilde I_{b_{f\pm \bbg}})$. Applying the Parallel Law~\ref{sss:ParLaw}, we obtain that \[\Fam\coloneqq \Fam^{+}\left(\widetilde I_{a}, \lfloor  \widetilde I_{b_2} ,\widetilde I_{b_\bbk} \rfloor\right)\sp\sp\text{ has width }\sp \succeq\frac{\rr}{\bbg} C_\bkappa K\gg_\bbt K .\]
 We now apply Sneaking Lemma~\ref{lem:sneaking} to $\Fam$ and $\Fam_\lambda (I_{b_1})$.
\end{proof}

\begin{claim}
\label{cl:6}
Suppose that the degree of every vertex in $\mathfrak G$ is bounded by $\rr$ from Claim~\ref{cl:5}.  For every $\bbg\gg 1$, if $\blambda\gg_\bbg 1$, then the following holds. There is an interval $\widetilde I_{s}$ such that $\Fam_\blambda(\widetilde I_s)$ contains a lamination $\Fam, \sp \Width(\Fam)\ge C_\bkappa K-O_\blambda(1)$ that has toll barriers $\ell_1, \ell_2, \dots, \ell_\bbg$.
\end{claim}
\begin{proof}
Consider the dual graph $\mathfrak G^\vee$:
\begin{itemize}
\item vertices of $\mathfrak G^\vee$ are faces of $\mathfrak G$,
\item edges of $\mathfrak G^\vee$ are orthogonal to edges of $\mathfrak G$.
\end{itemize}
We denote by $X$ the outermost vertex of  $\mathfrak G^\vee$ corresponding to the unbounded face. Since $\blambda\gg_\bbg 1$, it is easy to check that either:
\begin{enumerate}
\item there is a simple path in $\mathfrak G^\vee$ of length $3\rr \bbg$ starting at $X$, or
\item there is a face $Y$ of  $\mathfrak G$ (a vertex of $\mathfrak G^\vee$) containing at least $3\bbg$ vertices of $\mathfrak G$.  
\end{enumerate}

In the first case, we can choose edges $\tilde \ell_i, i\le \bbg$ of $\mathfrak G$ such that each $\tilde \ell_i,$ connects $\widetilde I_{a_i}$ and $\widetilde I_{b_i}$ with 
\begin{equation}
\label{eq:a_i b_i}
a_{i+1}+1<a_i ,\sp\sp  a_i+1 <s < b_i-1,\sp\sp b_i<b_{i+1}-1.
\end{equation}
 Since $\Fam^+(\widetilde I_{a_i}, \widetilde I_{b_i})\ge T$, Small Overlapping Principle~\S\ref{sss:sm overl princ} implies that $\Fam_\blambda(I_s)$ contains a lamination $\Fam$ with $\Width(\Fam)\ge C_\bkappa K-O_\blambda(1)$ that is disjoint from pairwise disjoint curves $\ell_i\in  \Fam^+(\widetilde I_a, \widetilde I_b)$. Since at most $O_\blambda(1)$ curves in $\Fam$ can pass under $\widetilde I_{a_{i}+1}$ or $\widetilde I_{b_{i}+1}$ (Squeezing Lemma~\ref{lem:squeezing}), we can remove $O_\blambda(1)$ curves from $\Fam$ such that $\ell_i$ are toll barriers for the new lamination $\Fam^\new$ with $\Width(\Fam^\new)\ge C_\bkappa K-O_\blambda(1)$.

Consider the second case. We can choose intervals $\widetilde I_s, \widetilde I_{a_i}, \widetilde I_{b_i} \sp i\le \bbg$ on the boundary of the face $Y$ such that~\eqref{eq:a_i b_i} holds. Since $Y$ is a face of the arc diagram $\mathfrak G$, most curves (up to  $O_\blambda(1)$-width) in $\Fam_\blambda(\widetilde I_s)$ intersect $I_{a_i}\cup I_{b_i}$ before intersecting $I_{a_{i+1}}\cup I_{b_{i+1}}$. Therefore, we can select pairwise disjoint arcs $\ell_i\in \Fam^+(\widetilde I_{a_i},\widetilde I_{b_i})$ such that at most $O_\blambda(1)$ curves in $\Fam_\blambda(\widetilde I_s)$ intersect $\bigcup_i \ell_i$. Since at most $O_\blambda(1)$ curves in $\Fam_\blambda(\widetilde I_s)$ can pass under $\widetilde I_{a_{i}+1}$ or $\widetilde I_{b_{i}+1}$, the family $\Fam_\blambda(\widetilde I_s)$ contains a lamination $\Fam$ with $\Width(\Fam)\ge C_\bkappa K-O_\blambda(1)$ such that $\ell_i$ are toll barriers for $\Fam$.
\end{proof}

The lemma now follows from Snake Lemma~\ref{lem:SnakeWithCheckpts} applied to $\Fam$ and toll barriers $\ell_1, \ell_2, \dots, \ell_\bbg$.
 \end{proof}

\begin{proof}[Proof of Theorem~\ref{thm:SpreadingAround}]  Consider a $[\bK,\blambda_\bbt]^+$ combinatorial interval $I\subset \partial Z$ of level $m$ as in the statement of Theorem~\ref{thm:SpreadingAround}. There are two level $m$ combinatorial intervals $I_a,I_b$ such that $I\subset I_a\cup I_b$ and at least one of the endpoints of $I_a$ and of $I_b$ is in $\CP_m$. Then either $I_a$ or $I_b$ is $[\bK/2,\blambda_\bbt/2]^+$-wide. The theorem now follows from Lemmas \ref{lem:SpredAroundWidth} and \ref{lem:Hive Lemma}.
\end{proof}

\subsection{Bounded type regime} Recall that we are considering eventually golden-mean rotation numbers $\theta=[0;a_1,a_2,\dots]$ with $a_n=1$ for $n\ge \bbn_\theta$. 

\begin{cor}
\label{cor:bds:deep case}
There are absolute constant $\bK,\bbn>2$ such that for every $\theta$ we have
\[\Width^+_3(I)\le \bK \sp\sp\sp\text{ for every interval }I\subset \partial Z \text{ with }|I|\le \length_{\max\{\bbn_\theta, \bbn\}} .\]
\end{cor}
\begin{proof}
It is sufficient to prove $\Width^+_{\blambda} (I)\le \bK$ for some $\blambda\ge 3$ and $\bK$; the case $\blambda=3$ follows by spiting $I$ and increasing $\bK$. For a sufficiently big $\bbt\gg 1$, let $\blambda=\blambda_\bbt$ and $\bK=\bK_\bbt$ be the constants from Theorem~\ref{thm:SpreadingAround}. Set $\bbn$ to be the integer part of $ 2 \log_2 (2\blambda)+2$; then $\length_\bbn \le \frac{|\theta_0|}{2\blambda _\bbt}$ by~\eqref{eq:length decrease}.

Assume converse: $\Width^+_\blambda(I)=K\ge \bK$ for some $I$ with $|I|\le \length_{\max\{\bbn_\theta, \bbn\}}$. Then $I$ contains a combinatorial subinterval $I'$ with $|I'|\asymp |I|$ such that $\Width^+_\blambda(I')\succeq K$. Applying Theorem~\ref{thm:SpreadingAround} with $\bZ^m= \overline Z$, we find an interval $I_2$ with $\Width(I_2) \succeq \bbt K\gg K$ and $|I_2|\le |I|$. Continuing the process,  we obtain a sequence of interval $I_k$ with $|I_k|\le |I_{k-1}|$ such that $\Width^+_{\blambda}(I_k)\to +\infty$. This is impossible.
\end{proof}

\section{The Calibration Lemma} \label{s:CalibrLmm}Recall from \S\ref{sss:ext div famil}  that the diving family $\Fam^+_{\lambda, \div, m}(I)\subset \Fam^+_{\lambda}(I)$ consists of curves intersecting (or diving into) $\filled_m\setminus \overline Z$.

\begin{caliblem}
\label{lmm:CalibrLmm}
There is an absolute constant $\bchi> 1$ such that the following holds for every $\lambda\ge 10$. Let $\wZ^{m+1}$ be a geodesic pseudo-Siegel disk and consider an interval $T\subset\partial Z$ in the diffeo-tiling $\Dbb_m$. If there are intervals \[I\subset T,\sp\sp L\subset \partial Z \sp\sp \text{ such that }\sp\sp \length_{m+1} \le |I|\le \length_{m},\sp\sp \dist(I,L^c)\ge \lambda \length
_{m+1} \]
\[\text{and }\sp\sp\sp \Width^+_{\div, m}(I, L^c) = \bchi K\gg_\lambda 1,\] then 
\begin{enumerate}[label=\text{(\Roman*)},font=\normalfont,leftmargin=*]
\item \label{Cal:Lmm:Concl:1} either there is a $[K,\lambda]^+$-wide level-$(m+1)$ combinatorial interval,
\item \label{Cal:Lmm:Concl:2} or there is a $[\bchi^{1.5} K,\lambda]^+$-wide interval $I'\subset \partial Z,\sp |I'|<\length_{m+1}$ grounded rel $\wZ^{m+1}$.
\end{enumerate}
\end{caliblem}

In applications, we will often take $L=\lambda I$.

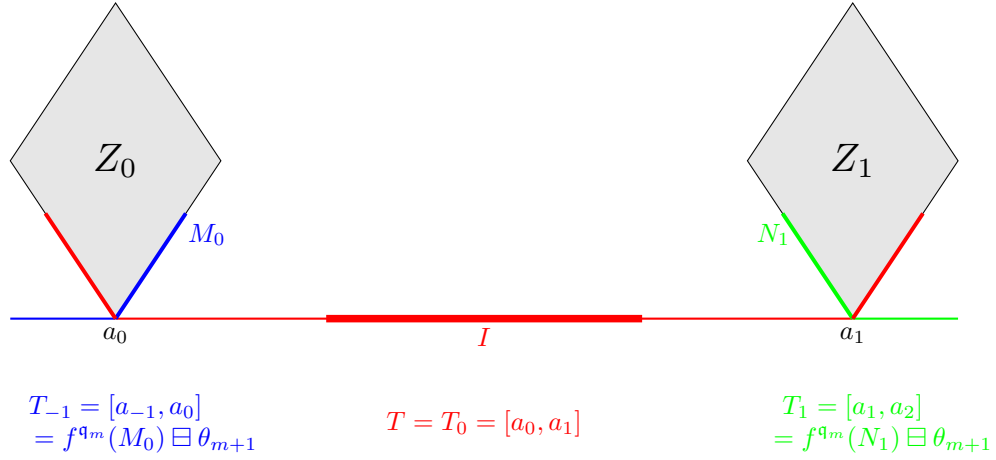
\begin{figure}[t!]
\[\begin{tikzpicture}[scale=1.4]

\draw[red,line width =0.8] (-7,0) --(0,0); 
\draw[blue,line width =0.8] (-8,0) --(-7,0);
\draw[green,line width =0.8] (0,0) --(1,0);

\node[below] at (-7,0){$a_0$};
\node[below] at (-0,0){$a_1$};
\node[scale=1.5] at (-7,1.5){$Z_0$};
\node[scale=1.5] at (-0,1.5){$Z_1$};

\node[green] at (0,-1){$
\begin{matrix}
T_1=[a_1,a_2]\\
\sp \sp\sp =f^{\qq_m}(N_1)\boxminus \theta_{m+1}
\end{matrix}
$};

\node[blue] at (-7,-1){$
\begin{matrix}
T_{-1}=[a_{-1},a_0]\\
\sp \sp\sp =f^{\qq_m}(M_0) \boxminus \theta_{m+1}
\end{matrix}
$};

\draw[fill=black, fill opacity=0.1] (0,0) -- (1,1.5 )--(0,3)--(-1,1.5)--(0,0);

\draw[shift={(-7,0)},fill=black, fill opacity=0.1] (0,0) -- (1,1.5 )--(0,3)--(-1,1.5)--(0,0);

\draw[green,line width =1.5] (0,0) -- (-2/3,1 );
\draw[red,line width =1.5] (0,0) -- (2/3,1 );


\node[right,green] at (-1, 0.8) {$N_1$};

\draw[blue,line width =1.5] (-7,0) -- (-7+2/3,1);
\draw[red,line width =1.5] (-7,0) -- (-7-2/3,1);

\node[right,blue] at (-6.4, 0.8) {$M_0$};





\draw[red,line width=1mm] 
(-5,0)--(-2,0);

\node[red,below] at(-3.5,0){$I$};
\node[red] at(-3.5,-1){$T=T_0=[a_0,a_1]$};

\end{tikzpicture}\]
\caption{Various intervals on $\partial Z$ and $\partial Z_i$.}
\label{Fg:CalibrLmm1}
\end{figure}

\subsubsection{Outline and Motivation}\label{sss:outline of Calibr} The Calibration Lemma allows us to trade $\Width^+_{\lambda, \div}$-wide intervals into $\Width^+_{\lambda}$-wide combinatorial intervals. In Section~\ref{s:MainThms}, we will construct a pseudo-Siegel disk $\wZ^{m+1}$ to absorb most of the external families. Therefore, the Calibration Lemma can be inductively applied if Case (\RN{2}) happens.

 The main idea of the proof is illustrated on Figure~\ref{Fig:cl4:1+1=1}. If Case (\RN{1}) does not happen, then we should expect a wide rectangle $\PP$ that essentially overflows its two conformal pullbacks $\PP_-, \PP_+$; this leads to the ``$1\le 1\oplus 1=0.5$" contradiction. To construct such a wide rectangle $\PP$, we will first spread the family $\Fam_{\lambda,\div, m}^+(I)$ around $\Dbb_m$ (using univalent push-forwards~\S\ref{sss:univ push forw}). Next we will find a wide rectangle $\PP$ between two neighboring intervals of $\Dbb_m$. Finally either the argument illustrated on  Figure~\ref{Fig:cl4:1+1=1} is applicable, or the roof of $\PP$ is shorter than $\length_{m+1}$ -- this leads to Case (\RN{2}).


\subsection{Proof of Calibration Lemma~\ref{lmm:CalibrLmm}} 
We split the proof into several subsections. We assume that $\bchi\gg1$ is sufficiently big.

\subsubsection{Bubbles $Z_i$} We enumerate intervals in $\Dbb_m$ as $T_i, \ T_0=T$ from left-to-right. Denote by $a_i$ the common endpoint of $T_{i-1}$ and $T_{i}$. We recall from~\S\ref{sss:diff tilings} that $a_i$ is a critical point of generation $t_i\le \qq_{m+1}$. Thus there is a bubble $Z_i$ attached to $a_i$, see Figure~\ref{Fg:CalibrLmm1}, with the first landing $f^{t_i}\colon Z_i \to Z$.

Let us pullback the diffeo-tiling $\Dbb_m$ to a partition of $\partial Z_i$ under \[f^{-\qq_m}\circ f^{\qq_m}\colon \partial Z\to \partial Z_i\sp\sp\sp\text{ (equivalently, under $f^{-t_i}\circ f^{t_i}\colon \partial Z\to \partial Z_i$)}.\]  We specify the following elements of the new partition:
\begin{itemize}
\item $N_i$ is the preimage of $T_{i}\boxplus \theta_{m+1}=f^{\qq_m}(T_{i+1})$ under $f^{\qq_m}\mid \partial Z_i$;
\item $M_i$ is the preimage of $T_{i-1} \boxplus \theta_{m+1}=f^{\qq_m}(T_{i})$ under $f^{\qq_m}\mid \partial Z_i$.
\end{itemize}
The point $a_i$ is the common endpoint of $M_i,N_i$.

By the following properties $N_i,M_i$ are intervals of $\filled_m=f^{-\qq_{m+1}}(\overline Z).$

\begin{claim4}
The interiors of $M_i,N_{i+1}$ contain no branched points of $f^{-\qq_{m+1}}(\partial Z)$. 
\end{claim4}
\begin{proof}
By definition, the interiors of the $T_i$ contain no critical points of $f^{\qq_{m+1}}$. Therefore, 
the interiors of the $f^{\qq_{m+1}}(T_i)$ contain no critical values of $f^{\qq_{m+1}}$.  Applying $f^{-\qq_{m+1}}$, we obtain a required claim for the $M_i, N_i$.
\end{proof}

\subsubsection{Rectangle $\RR_i$} \label{sss:RR_i}By Lemma~\ref{lmm:ext famil},~\ref{Cl:C:lmm:ext famil}, most of the width of $\Fam^+_{\div, m}(I, L^c)$  is within two rectangles; thus we can select a rectangle $\RR\subset \Fam^+_{\div, m}(I, L^c)$ satisfying 
\[\Width(\FamG) =\bchi K/2 -O(1) .\]
We orient the vertical curves in $\RR$ from $I$ to $L^c$. By shrinking $I$, we can assume that $\partial ^{h,0} \RR =I$. For an interval $J\subset I$, we denote by $\RR_J\subset \RR$ the genuine subrectangle  consisting of vertical curves of $\RR$ starting at $J$.

 {\bf Assuming that ~\ref{Cal:Lmm:Concl:1} does not hold}, we obtain 
\begin{enumerate}[label=\text{(\Alph*)},font=\normalfont,leftmargin=*]
\item\label{Cond:A:Cal_Lmm} For every combinatorial subinterval $J\subset I$ with $|J|=\length_{m+1}$, we have 
\[\Width(\RR_J) \le K.\]
\end{enumerate}
In particular, $|I|> \bchi \length_{m+1}/3\gg 1$. Let us present $I$ as a concatenation \[I=I_a\cup I_0\cup I_b  \sp\sp\text{ with } \sp\sp I_a<I_0 <I_b \] such that $\Width(\RR_{I_a})=\Width(\RR_{I_b})=6K$. In particular, $|I_a|,|I_b|\ge 6\length_{m+1}$ by~\ref{Cond:A:Cal_Lmm}.

We set
\[ \RR_0\coloneqq \RR_{I_0}, \sp\sp\text{ where }\sp\sp \Width(\FamG_0) =\bchi K/2-O(K).\]

Since $\RR_0$ is obtained from $\RR$ by removing sufficiently wide buffers, we can spread around $\RR_0$ using the univalent push-forward ~\eqref{eq:sss:univ push forw}. We denote by $\RR_i$ the resulting image of $\RR_0$ in $T_i$. As for $\RR$, we denote by $\RR_{i, J}$ the subrectangle of $\RR_i$ consisting of curves starting at $J\subset \partial ^{h,0} \RR_i$.

\begin{claim4}[\ref{Cond:A:Cal_Lmm} holds for all $\RR_i$]
\label{cl4:A for RR_i}
For every level $m+1$ combinatorial interval $J\subset \partial^{h,0} \RR_i $, we have $\Width(\RR_{i, J})\le K+O(1)$.  
\end{claim4}
\begin{proof}
Assume that $\Width(\RR_{i, J})\ge K +C$ for $C\gg 1$. Pushing $\RR_{i, J}$ forward under \eqref{eq:sss:univ push forw} towards $T_0$ we obtain the violation of~\ref{Cond:A:Cal_Lmm} for $\RR$.
\end{proof}

\subsubsection{Almost invariance of $\RR_i$} For every $\RR_{i,J}$, let us denote by $\RR^*_{i,J}$ the lift of $\RR_{i,J}$ under $f^{\qq_{m
+1}}$ starting at $J \boxminus \theta_{m+1}$.

\begin{claim4}
\label{cl4:RR is invar} Every $J\subset \partial^{h,0} \RR_i$ contains a subinterval $J^\new\subset J$ such that
\[\Width(\RR_{i,J^\new}) = \Width(\RR_{i,J})-O(K)\]
and such that $\RR_{i,J^\new}$ vertically overflows $\RR^*_{i,J}$.

Moreover, at most $O(K)$ curves of $\RR_{i}$ land at $T_i\boxplus \theta_{m+1}$.
\end{claim4}
\begin{proof}
Let $\RR_{i,J^\new}$ be the rectangle obtained from $\RR_{i,J}$ by removing two $3K$-buffers. Then the length of each of the intervals in $J\setminus J^\new$ is at least $2\length_{m+1}$ by~\ref{Cond:A:Cal_Lmm}, and at most $O(1)$ curves in $\RR_{i,J^\new}$ can cross the buffers of $\RR^*_{i,J}$ starting at $\big(J\boxminus \theta_{m+1}\big)\setminus J^\new$.

Similarly, up to $O(K)$, curves of $\RR^*_{i,J}$ are in $\RR_{i,J^\new}$. Since $\RR_{i,J^\new}$ is in $\Fam^+_{\div, m}$, we obtain that, up to  $O(K)$, curves of $\RR^*_{i,J}$ intersect $\partial Z_i\cup \partial Z_{i+1}$. Taking $J=\partial^{h,0}\RR_i$ and applying $f^{\qq_{m+1}}$, we obtain the second claim.
\end{proof}

\begin{claim4}[A rectangle between $T_{s}$ and $T_{s\pm1}$]
\label{cl4:RR_j}
There is an $\RR_s$ such that, up to removing $O(K)$ buffers, the roof of $\RR_s$ is in $\big(T_{s-1}\cup T_{s+1}\big)\boxplus \theta_{m+1}$. 

Moreover, up to removing $O(K)$ buffers from $\RR_s$, we can assume that \[\dist(\partial^{h,1} \RR_s, \{a_{s-1},a_{s+2}\})\ge 5\length_{m+1}.\]
\end{claim4}
\begin{proof}
Let us show that, up to removing $O(K)$ buffers, there is a rectangle $\RR_s$ such that $\partial^{h,1} \RR_s\subset \lfloor \partial ^{h,0} \RR_{s-1},\partial ^{h,0}\RR_{s+1} \rfloor\setminus (T_s\boxplus \theta_{m+1})$. This will imply the claim because $\partial^{h,0} \RR_j$ has distance $>5\length_{m+1}$ to $\{a_{j}, a_{j+1}\}$, see~\S\ref{sss:RR_i}.

 Assuming converse and using Claim~\ref{cl4:RR is invar} (the second part), we can choose in every $\RR_i$ a wide rectangle $\SS_i$ whose roof is outside of $\lfloor \partial ^{h,0} \RR_{i-1},\partial ^{h,0}\RR_{i+1} \rfloor$. This is a contraction by Lemma~\ref{lem:no inf quasi additivity}.
\end{proof}

\subsubsection{Proof of Calibration Lemma} We now fix a rectangle $\RR_s$ from Claim~\ref{cl4:RR_j}. We assume that most of the curves (up to $O(K)$) in $\RR_s$ land at $T_{s+1}$; the case of $T_{s-1}$ is similar. By removing $O( K)$ buffers from $\RR_s$, we obtain the new $\RR^\new_s$ with
\begin{equation} 
\label{eq:CL:prf:RR_s}
 \partial^{h,0}\RR_s^\new\subset T_s,\sp\sp  \partial^{h,1}\RR_s^\new\subset T_{s+1}\sp\sp \dist(\partial ^{h,1}\RR^\new_s , a_{s+2})\ge 5 \length_{m+1}.
\end{equation}
Let $\PP,\PP_1$ be the restrictions (see~\S\ref{sss:RestrOfFamil}) of $\RR^\new_s$ onto $\wZ^{m+1}$ and $f^{\qq_{m+1}}(\wZ^{m+1})$  respectively. Subrectangles $\PP_J$ of $\PP$ are defined in the same way as subrectangles $\RR_{i,J}$ for $\RR_i$.

\begin{claim4}[See Figure~\ref{Fig:cl4:1+1=1}]
\label{cl4:1+1=1}There is an interval $J\subset \partial ^{h,0} \PP$ well-grounded rel $\wZ^{m+1}$ with $\Width(\PP_{J})\ge \bchi^{0.9}K$ such that $|\partial^{h,1} \PP_{J} |\le \length_{m+1}/5$.
\end{claim4}
\begin{proof}Let $\PP_-$ and $\PP_+$ be the lifts of $\PP_1$ under $f^{\qq_{m+1}}$ such that $\PP_-$ starts in $T^{m+1}_s\boxplus \theta_{m+1}$ while $\PP_+$ lands $T^{m+1}_{s+1}\boxplus \theta_{m+1}$, where $T^{m+1}_s, T^{m+1}_{s+1}$ are the projections of $T_s, T_{s+1}$ onto $\partial \wZ^{m+1}$. Assume a required interval $J$ does not exist. Combining this assumption with Claim~\ref{cl4:A for RR_i}, we can remove $O(\bchi^{0.9}K)$ buffers from $\PP$ such that $\PP^\new$ consequently overflows $\PP_-$ and $\PP_+$.  This contradicts the Gr\"otzsch inequality:
\[\bchi K -O(\bchi^{0.9}K ) =\Width(\PP) -O(\bchi^{0.9}K )\le \Width(\PP_-) \oplus \Width(\PP_+)= 0.5 \bchi K.\] 

By removing $O(K)$-buffers from $\PP_J$, we can assume that $J$ is well-grounded.
\end{proof}

\begin{figure}[t!]
\[\begin{tikzpicture}[scale=1.4]

\draw[fill=black, fill opacity=0.1] (0,0) -- (1,1.5 )--(0,3)--(-1,1.5)--(0,0);

\node[above,red] at (-1.2,1.8) {$\PP $};

\node[red,below] at(4,0){$\partial^{h,1} \PP$}; 
\node[red,below] at(-4,0){$\partial^{h,0} \PP $};

\draw (-5.2,0) -- (6,0);

\draw[red,line width=1mm] (5,0)--(3,0)
(-5,0)--(-3,0);

\draw [red] (5,0)
 .. controls (3.5, 3) and (-3.5,3) ..
 (-5,0);

\draw [red] (3,0)
 .. controls (1.8, 1.1) and (-1.8,1.1) ..
 (-3,0);
 
 \draw[blue,dashed] (-4.6,0) 
 .. controls (-3.5, 1.4) and (-1.5,1.9) ..
 (-1+0.3,1.5+0.45);

   \draw[blue,dashed] (-2.6,0) 
 .. controls (-1.5, 0.5) and (-1,0.6) ..
 (-0.4, 0.6);

 \draw[blue,dashed] (5.4,0) 
 .. controls (4.5, 2) and (1.5,2.45) ..
 (1-0.7,1.5+1.05);

  \draw[blue,dashed] (3.4,0) 
 .. controls (2.5, 0.8) and (1,1) ..
 (0.7, 1.05); 
 
  \draw[blue, line width=0.6 mm] (-1+0.3,1.5+0.45) --(-1,1.5)--(-0.4, 0.6)
   (1-0.7,1.5+1.05)--(1,1.5)-- (0.7, 1.05);

\node[blue] at (-3,0.6){$\PP_-$};
\node[blue] at (3.2,0.9){$\PP_+$};
\end{tikzpicture}\]
\caption{Illustration to Claim~\ref{cl4:1+1=1}. If $\PP$ essentially overflows its univalent pullbacks $\PP_\pm$, then ${\Width(\PP_-)\oplus \Width(\PP_+)\ge \Width(\PP)-O(\bchi^{0.9}K)}$.}
\label{Fig:cl4:1+1=1}
\end{figure}
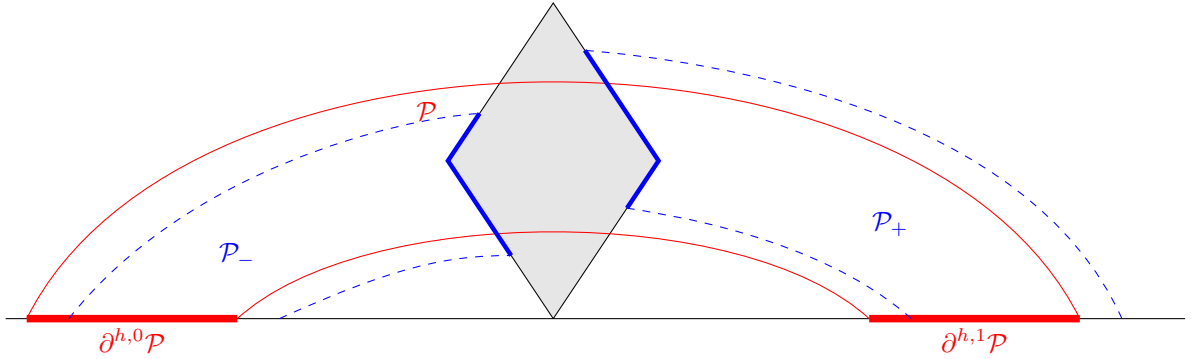

Let $\wZ_{s+1}$ be the pseudo-bubble (see~\S\ref{ss:Ps-bubles}) around $Z_{s+1}$ such that $f^{\qq_{m+1}}$ maps $\wZ_{s+1}$ onto $f^{\qq_{m+1}}\big(\wZ^{m+1}\big)$. We denote by $N^{m+1}_{s+1}$ the projection of $N_{s+1}$ onto $\partial \wZ_{s+1}$.

Consider an interval $J$ from Claim~\ref{cl4:1+1=1}. Write $X\coloneqq \partial^{h,1}\PP^1_{J}$, where $|X|\le \length_{m+1}/5$, and let $X^*\subset N^{m+1}_{s+1}$ be the lift of $X$ under $f^{\qq_{m+1}}$. We denote by $\PP^*_J,\  \partial^{h,1}\PP^*_J=X^*$ the full lift of $\PP_{1,J}$ under $f^{\qq_{m+1}}$.

By Claim~\ref{cl4:RR is invar}, we can remove $O(K)$ buffers from $\PP_J$ so that the new rectangle $\PP^\new_J$ overflows $\PP^*_J$. Let us denote by $V\subset X^*$ the subinterval of $X^*$ between the first intersections of $\partial^{v, \ell}\PP_J^\new, \partial^{v, \rho}\PP_J^\new$ with $X^*$. Set 
\begin{itemize}
\item $\Fam_1$ to be the restriction (see \S\ref{sss:short subcurves}) of $\Fam(\PP^\new_J)$ to $\Fam(\partial ^{h,0} \PP^\new_J, V)$; 
\item $\Fam_2$ to be the restriction of $\Fam(\PP^\new_J)$ to $\Fam(V,\partial ^{h,1} \PP^\new_J)$.
\end{itemize}
In other words:
\[ \Fam_1=\{\gamma_1\mid \gamma \in \Fam(\PP^\new_J)\}\sp\sp\text{ and }\sp\sp \Fam_2=\{\gamma_2\mid \gamma \in \Fam(\PP^\new_J)\},\]
where $\gamma_1$ and $\gamma_2$ are the shortest subarcs of $\gamma$ connecting $\partial ^{h,0} \PP^\new_J, V$ and $V,\partial ^{h,1} \PP^\new_J$ respectively. We remark that $\gamma_1$ lands at $V^+$ while $\gamma_2$ starts in $V^-$.


\begin{claim4}
We have $\Width(\Fam_2)\ge t^{1.7}K$.
\end{claim4}
\begin{proof}
By the Gr\"otzsch inequality $\Width(\PP^*_J) \oplus  \Width(\Fam_2)\ge \Width(\Fam_1) \oplus  \Width(\Fam_2)  \ge  \Width(\PP^\new_J)$, we have: \[\Width(\PP^*_J) \oplus  \Width(\Fam_2) = \bchi^{0.9} K \oplus \Width(\Fam_2)\ge  \Width(\PP^\new_J) =\bchi^{0.9}K -O(K),\]

\[ \frac{1}{ \bchi^{0.9} K}+ \frac{1}{ \Width(\Fam_2)} \le  \frac{1}{\bchi^{0.9}K -O(K)},\]
\[\Width(\Fam_2) \ge \frac{\bchi^{1.8}K^2}{ O(K)} \ge \bchi^{1.7}K\]
because $\bchi\gg 1$.
\end{proof}

Consider the rectangle $\FamG$ bounded by the leftmost and rightmost curves of $\Fam_2$:
\[ \FamG, \sp \partial^{h,0}\FamG\subset V^-, \sp \partial^{h,1}\FamG\subset \partial^{h,1} \PP^\new_J,\sp\sp  \Width(\FamG)\ge \Width(\Fam_2)\ge \bchi^{1.7}K.\]

 Applying Lemma~\ref{lem:rect rhrough PB} to $\FamG$, we obtain an interval $B\subset [(1+\lambda^{-2})V]\setminus V\subset \partial \wZ_{s+1}$ together with a lamination \[\FamQ\subset  \wC\setminus \intr \wZ_{s+1},\sp\sp\Width(\FamQ)\succeq \bchi^{1.7} K \] from $B$ to either $\big [ (\lambda B)^c\big]^\grnd\subset \partial \wZ_{s+1}$ (Case~\ref{c2:lem:rect rhrough PB}) or to $\partial ^{h,1}\PP^\new_J$ (Case~\ref{c1:lem:rect rhrough PB}). In both cases, $\FamQ$ is disjoint from $\intr \wZ^m$ as a restriction of a sublamination of $\PP^\new_J$. Write $N_{s+1}^{m+1}=[a_{s+1},b_{s+1}]$, where $a_{s+1}\in \partial Z$, see Figure~\ref{Fg:CalibrLmm1}. Observe that 
\begin{equation}
 \label{eq:CL:dist:B}
\dist_{\partial \wZ_{s+1}}(B, b_{s+1})\ge 4\length_{m+1}, \sp\sp \dist_{\partial \wZ_{s+1}\cup \partial \wZ^{m+1}}(B, \partial^{h,1}\PP^\new_J)\ge \length_{m+1}/5. 
\end{equation}
 Indeed, the first inequality follows from~\eqref{eq:CL:prf:RR_s} because $f^{\qq_{m+1}}(b_{s+1})=a_{s+2}$. The second inequality follows from the observation that if $X^*, |X^*|\le \length_{m+1}/5$ is close to $a_{s+1}$, then $X$ will be close to $a_{s+1}\boxplus \theta_{m+1}$.
 
 For every $\gamma\in \FamQ$, let $\gamma'$ be the first subarc of $\gamma$ connecting $B$ to $\partial \filled_m$. By \S\ref{sss:univ push forw} and~\eqref{eq:CL:dist:B}, $f^{\qq_{m+1}}$ injectively maps most of the $\{\gamma'\mid \gamma\in \FamQ\}$ into a sublamination of $\Fam\left(B_2, \big[(\lambda B_2)^c\big]^\grnd\right)$, where $B_2=f^{\qq_{m+1}}(B)\subset \partial f^{\qq_{m+1}}(\wZ^{m+1})$ is a grounded interval. By Lemma~\ref{lem:W+:ground inter}, the projection $B_2^\bullet$ of $B_2$ onto $\partial Z$ satisfies 
 \[ \Width_{\lambda}(B^\bullet_2)\succeq \bchi^{1.7} K \ge \bchi^{1.5} K;\]
 which is Case~\ref{Cal:Lmm:Concl:2} of the Calibration Lemma.

\part{Conclusions}
\label{part:conclus}

\section{Proof of the main result}
\label{s:MainThms}
Recall that we are considering eventually golden-mean rotation numbers:
\begin{equation}
\label{eq:cond on theta}
\theta=[0;a_1,a_2,\dots]\sp\sp \text{ with }\sp a_n=1\sp\sp\text{ for all }\sp n\ge \bbn_\theta.
\end{equation}

In this section we will establish the following results:

\begin{thm}
\label{mainthm:v1}
There is an absolute constant $\bK\gg 1$ such that $\Width^+_3(I)\le \bK$ for every combinatorial interval $I\subset \partial Z$ and every $\theta$ satisfying~\eqref{eq:cond on theta}. 
\end{thm}

\begin{thm}
\label{mainthm:v2}
There are absolute constants $\bN\gg 1$ and $\bK\gg 1$ such that for every $\theta$ satisfying~\eqref{eq:cond on theta}, there is a sequence of geodesic pseudo-Siegel disks~\S\ref{sss:regul within rect}
\[  \wZ^{\bbn_\theta} =\overline Z,\sp  \wZ^{\bbn_\theta-1} , \sp \wZ^{\bbn_\theta-2},\sp \dots, \sp \wZ^{-1}=\wZ \]
satisfying the following properties. If $I\subset \partial Z$ is a regular interval rel $\wZ^m$ with  $\length_{m+1}\le|I|\le\length_{m}$, then 
\begin{equation}
\label{eq:mainthm:v2}
 \text{either }\sp\sp |I|\ge \length_{m}/2\sp\sp\sp\text{ or }\sp\sp\sp |I|\le \bN \length_{m+1}
\end{equation}
and $\Width^{+}_3(I)\le \bK$.  
\end{thm}

\subsubsection{Outline and motivation}\label{sss:outline prd of main res} Let us note that Theorem~\ref{mainthm:v1} implies Theorem~\ref{thm:unform bounds}. Indeed, for every $\theta=[0;a_1,a_2,a_3, \dots]$, $|a_i|\le M_\theta$ of bounded type, define the approximating sequence
\[\theta_{n}\coloneqq [0;a_1,a_2,\dots, a_n,1,1,1,\dots]\to \theta.\]
Since $Z_{\theta_n}\to Z_{\theta}$ as qc disks (with dilatation depending on $M_\theta$), the estimate $\Width^+_3(I)\le \bK$ in Theorem~\ref{mainthm:v1} also holds for $Z_\theta$, where $\bK>1$ is independent of $M_\theta$.

We will prove Theorem~\ref{mainthm:v1} by induction from deep to shallow levels as follows. We will show that there are $\blambda \gg 1$ and $\bK\gg_\blambda 1$ such that the following properties hold for every level $m$:
\begin{enumerate}[label=\text{(\alph*)},font=\normalfont,leftmargin=*]
\item Existence of a geodesic pseudo-Siegel disk $\wZ^m$ so that
\[\Width^+_{\blambda,\ext, m} (I) =O(\sqrt \bK)\]
for every interval $I$ grounded rel $\wZ^m$ with $\length_{m+1}\le |I|\le \length_m$. (In other words, $\wZ^m$ consumes all but $O(\sqrt{\bK})$-external $\lambda$-separated rel $m$ families.)
\label{eq:case1:main prf}
\item $\Width_\blambda^+(I)\le (2\bchi) \bK$ for every grounded rel $\wZ^m$ interval $m$ with $|I|\le \length_m$, where $\bchi$ is the constant from Calibration Lemma~\ref{lmm:CalibrLmm}.\label{eq:case2:main prf}
\item $\Width_\blambda^+(I)\le  \bK$ for every combinatorial interval $I$ of level $\ge m$.\label{eq:case3:main prf}
\end{enumerate}

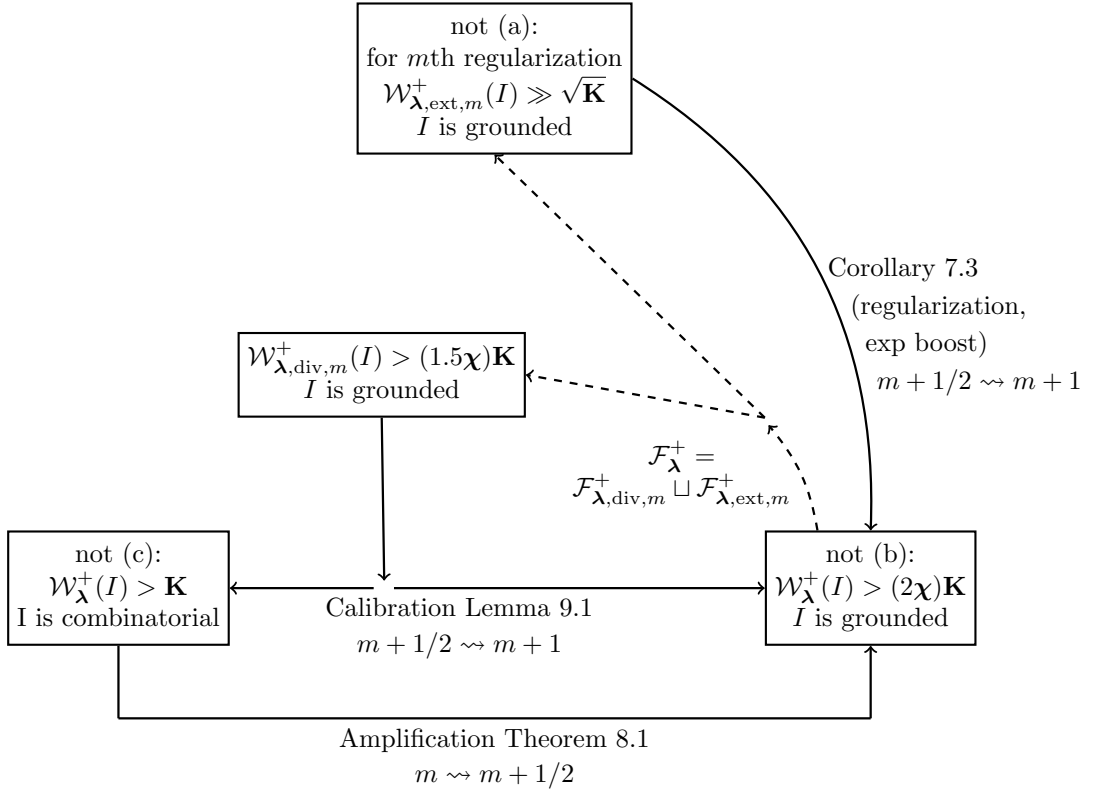
\begin{figure}[t!]

\[
\begin{tikzpicture}[line width=0.3mm]

\draw
         node[below] at (0,0)[block] (comb) {$\begin{matrix}
                 \text{not~\ref{eq:case3:main prf}:}
               \\ \Width^+_\blambda (I)> \bK
         \\
        $I$\text{ is combinatorial}
        \end{matrix}$};

               \coordinate (grnd) at (10,0);

     \draw
         node[below](Gr) at (grnd)[block]  {$\begin{matrix}
         \text{not~\ref{eq:case2:main prf}:}
         \\
          \text{$\Width_\blambda^+(I)> (2\bchi)\bK$} 
         \\
          \text{$I$ is grounded} 
        \end{matrix}$};

       \coordinate (regul) at (5,5);
        \draw
         node[above]  at (regul) [block] (Reg) {$\begin{matrix}
         \text{not~\ref{eq:case1:main prf}:}\\
        \text{for the $m$th regularization} 
         \\ \Width^+_{\blambda, \ext,m} (I)\gg \sqrt\bK
         \\
          \text{$I$ is grounded}             
          \end{matrix}$};

       \coordinate (diving) at (5-1.5,1.5);
        \draw
         node[above]  at (diving) [block] (Div) {$\begin{matrix}
           \text{$\Width_{\blambda,\div,m}^+(I)> (1.5\bchi)\bK$} 
         \\
          \text{$I$ is grounded} 
          \end{matrix}$};

       \draw (Reg.east) edge[bend left,->] node[right]{Corollary~\ref{cor:regul}} node[right,shift={(0.3,-0.5)}]{(regularization),} 
node[right,shift={(0.5,-1)}]{exp boost for} 
   node[right,shift={(0.65,-1.5)}]{$m+1/2 \leadsto m+1$}  (grnd);

    \draw  (comb.south) -- (0,-2.5);
     
    \draw (0,-2.5) edge node[below]{\text{Amplification Theorem~\ref{thm:SpreadingAround}}}
    node[below,shift={(0,-0.45)}]{$m \leadsto m+1/2$} (10,-2.5); 
          
     \draw (10,-2.5) edge[->] (Gr.south);
     
 \draw ([shift={(0.7,0)}]Gr.north west) edge[dashed, ->, bend right=15] node[left]{$\begin{matrix} \Fam_\blambda^+=\\ \Fam^+_{\blambda,\div,m}\sqcup \Fam^+_{\blambda,\ext,m}\end{matrix}$} ([shift={(0.05,1.4)}]Gr.north west);     
 
 \draw ([shift={(0,1.5)}]Gr.north west) edge[dashed, ->] (Reg.south);
 \draw ([shift={(0,1.5)}]Gr.north west) edge[dashed, ->] (Div.east);

\draw (Div.south) edge[->]([shift={(-3.56-1.5,0.1)}]Gr.west);

 \draw  ([shift={(-4.93,0.)}]Gr.west) edge[->]  (Gr.west);
\draw  ([shift={(-5.2,0.)}]Gr.west) edge[->]  (comb.east);

\node[below] at ([shift={(-3.56-0.5,0.)}]Gr.west) {Calibration Lemma~\ref{lmm:CalibrLmm}};
\node[below] at ([shift={(-3.56-0.5,-0.45)}]Gr.west) {$m+1/2 \leadsto m+1$};


\end{tikzpicture}
\]

\caption{Statements~\ref{eq:case1:main prf},~\ref{eq:case2:main prf}, and~\ref{eq:case3:main prf} are proved by contradiction: if one of the statements is violated on levels $m$ or $m+1/2$, then there will be even bigger violation on deeper scales. Here ``$m$'' indicates level $m$ combinatorial intervals, ``$m+1/2$'' indicates intervals with $\length_m\ge |I|\ge \length_{m+1}$, and ``$m+1$'' indicates intervals with $|I|\le \length_{m+1}$. The dashed arrows illustrate the decomposition $\Fam_\blambda^+(I)=\Fam^+_{\blambda,\div,m}(I)\sqcup \Fam^+_{\blambda,\ext,m}(I)$.}
\label{fig:main pattern}
\end{figure}


The proof of the induction step is illustrated on Figure~\ref{fig:main pattern}: 

\begin{itemize}
\item If a pseudo-Siegel disk $\wZ$ can not be constructed to satisfy Statement~\ref{eq:case1:main prf}, i.e.~to absorb all but $O(\sqrt{\bK})$-external rel $m$ families, then by the exponential boost in Corollary~\ref{cor:regul} there will be a degeneration of order $a^{\sqrt{\bK}} \gg (2\bchi)\bK$ on levels $\ge m+1$, where $a>1$ is fixed.

\item If Statement~\ref{eq:case2:main prf} is violated, then it follows from $\Fam_\blambda^+(I)={\Fam^+_{\blambda,\div,m}(I)\sqcup \Fam^+_{\blambda,\ext,m}(I)}$ that either $\wZ^m$ was not properly constructed, i.e.~the violation of Statement~\ref{eq:case1:main prf} with $\length _{m}\ge |I|\ge\length_{m+1}$, or there is a diving degeneration of order $> (1.5 \bchi)\bK$.

\item If there is a diving degeneration of scale $> (1.5 \bchi)\bK$ with $\length _{m}\ge |I|\ge\length_{m+1}$, then by Calibration Lemma~\ref{lmm:CalibrLmm}, either Statement~\ref{eq:case3:main prf} or Statement~\ref{eq:case2:main prf}  is violated on levels $\ge m+1$.

\item If Statement~\ref{eq:case3:main prf} is violated, then by Amplification Theorem~\ref{thm:SpreadingAround}, Statement~\ref{eq:case2:main prf} is violated with $|I|< \length_m$. 
\end{itemize} 
 
 After the induction, there might still be finitely many renormalization levels where Amplification Theorem~\ref{thm:SpreadingAround} is not applicable because of the condition $|I|\le |\theta_0|/(2\blambda_\bbt)$. The number of such levels is bounded in terms of $\blambda$; the estimates for these levels are established by increasing $\bK$.

Let us stress that regularizations on different levels do not interact much. Corollary~\ref{cor:regul}, Theorem~\ref{thm:SpreadingAround}, and Calibration Lemma~\ref{lmm:CalibrLmm}  are stated in terms of the outer geometry of the Siegel disk $Z$ with only indirect references to $\wZ^m$. Lemma~\ref{lem:W+:ground inter} implies that the outer geometries of $\wZ^m$ and $\overline Z$ are close rel grounded intervals independently of the number of regularizations. Our estimates for the inner geometry of $\wZ^m$ are also independent of the number of regularizations -- see the discussion in~\S\ref{sss:outline:s:NearRotatSystem} and in~\S\ref{sss:outline:s:Trad W to W+}.

To prove Theorem~\ref{mainthm:v2}, we will show that external families can not be unexpectedly narrow (Lemma~\ref{lem:est for bN}); otherwise, the dual family will be very wide -- contradicting the estimates established for Theorem~\ref{mainthm:v1}. Combined with the Parallel Law, this will imply the existence of the combinatorial threshold $\bN$. We will refer to $\bN$ in Theorem~\ref{mainthm:v2} as the \emph{high-type} condition. This a near-degenerate analogy of the high-type condition in the Inou-Shishikura theory~\cite{IS}.

\subsection{Proof of Theorem~\ref{mainthm:v1}}
We choose constants $\blambda, \bK$ and a parameter $\bbt$ from Theorem~\ref{thm:SpreadingAround} such that
\[ \bchi\ll  \bbt,\sp\sp\sp  \blambda\coloneqq \blambda_\bbt, \sp\sp \sp \bK_\bbt\ll \bK,\]
where $\bchi>1$ is the constant in Calibration Lemma~\ref{lmm:CalibrLmm} and $\blambda_\bbt,\bK_\bbt$ are from Theorem~\ref{thm:SpreadingAround}. Theorem~\ref{mainthm:v1} on levels $\ge \bbn_\theta$ follows from Corollary~\ref{cor:bds:deep case}. Let $\bbs=\bbs_\theta$ be the smallest so that $\length_\bbs\le |\theta_0|/(2\blambda+4)$. 

\begin{lem}[Induction]
\label{lem:ind:mainthm:v1}
There is a sequence of geodesic pseudo-Siegel disks 
\[  \wZ^{\bbn_\theta} =\overline Z,\sp  \wZ^{\bbn_\theta-1} , \sp \wZ^{\bbn_\theta-2},\sp \dots, \sp \wZ^{\bbs_\theta} \] 
with the following properties:
\begin{enumerate}[label=\text{(\Alph*)},font=\normalfont,leftmargin=*]
\item \label{C1:lem:ind:mainthm:v1} If $Z$ does not have an external level-$m$ parabolic rectangle of width $\sqrt{\bK}$, then $\wZ^m=\wZ^{m+1}$; otherwise $\wZ^m$ is the regularization of $\wZ^{m+1}$ constructed by Corollary~\ref{cor:regul}. Moreover, $\Width^+_{\blambda,\ext, m}(J)\le 2\sqrt {\bK}+O(1)$ for every interval $J\subset \partial Z$ grounded rel $\wZ^{m}$ with  $\length_{m+1}\le |J| \le \length_m$.

\item \label{C2:lem:ind:mainthm:v1}  For every interval $J\subset \partial Z$ grounded rel $\wZ^{m}$ with  $\length_{m+1}\le |J| \le \length_m$, we have $\Width^+_{\blambda}(J)\le 2\bchi \bK$.

\item \label{C3:lem:ind:mainthm:v1}  For every level-$m$ combinatorial interval $I\subset \partial Z$, we have $\Width^+_{\blambda}(I)\le \bK$. 
\end{enumerate}
\end{lem}
\begin{proof}
We proceed by induction from deep to shallow scales. The base case $m=\bbn_\theta$ follows from Corollary~\ref{cor:bds:deep case}. Let us assume that the lemma is true for levels $>m$. In particular, $\wZ^{m+1}$ is constructed. Let us verify the lemma for $m$.

Suppose that $Z$ has a level-$m$ external parabolic rectangle $\RR$ with $\Width(\RR)\ge \sqrt{\bK}$. Assume that $\RR$ is based on $T\in\Dbb_m$ and let $T'$ be as in~\S\ref{sss:diff tilings}. We replace $\RR$ with an outermost external rel $\wC\setminus \filled_m$ geodesic rectangle $\RR_\new$ based on $T'$ with $\Width(\RR_\new)\ge \sqrt K-O(1)$. In particular, $\RR_\new$ is non-winding. Let us apply Corollary~\ref{cor:regul} to $\RR_\new$. We {\bf claim} that Case~\eqref{case:1:thm:regul} of  Corollary~\ref{cor:regul} occurs.
\begin{proof}[Proof of the Claim]
 Assume Case~\eqref{case:3:thm:regul} of Corollary~\ref{cor:regul} occurs. We obtain an interval $I\subset \partial Z$ grounded rel $\wZ^{m+1}$ with $|I|\le \length_{m+1}$ such that $\log \Width^+_\blambda (I)\succeq \sqrt \bK$. Since $\bK\gg \bchi>1$, we have \[\Width^+_\blambda (I) \ge a^{\sqrt{\bK}}\gg 2 \bchi \bK,\sp\sp\sp \text{ where }\sp  a>1 \sp\text{represents ``$\succeq$''}\]contradicting the induction assumption that Statement (B) holds on levels $\ge m+1$.  
 
 Calibration Lemma~\ref{lmm:CalibrLmm} reduces Case~\eqref{case:2:thm:regul} of Corollary~\ref{cor:bds:deep case} to Case~\eqref{case:3:thm:regul}.
\end{proof}

By construction, $\sqrt\bK+O(1)$ bounds the width of level-$m$ external parabolic rectangles $\RR$ such that $\partial^{h} \RR $ is a pair of grounded rel $\wZ^m$ intervals -- wider rectangles are absorbed by $\wZ^m$. By splitting $J$ into at most $2$ intervals, we obtain the estimate $\Width^+_{\blambda,\ext, m}(J)\le 2\sqrt {\bK}+O(1)$. This proves Statement~\ref{C1:lem:ind:mainthm:v1}.

Let us verify Statement \ref{C2:lem:ind:mainthm:v1}. Assuming otherwise and using $\Width^+_{\blambda,\ext, m}(J)\le 2\sqrt {\bK}+O(1)$ (Statement \ref{C1:lem:ind:mainthm:v1}), we obtain $\Width^+_{\blambda,\div,m}(J)> 1.5 \bchi \bK$. Applying Calibration Lemma~\ref{lmm:CalibrLmm}, there would exist
\begin{itemize}
\item either a combinatorial $[1.5\bK,\blambda]^+$-wide level-$(m+1)$ combinatorial interval -- contradicting Statement \ref{C3:lem:ind:mainthm:v1} on level $m+1$,
\item or a $[1.5\bchi^{1.5} \bK,\blambda]^+$-wide interval $I'\subset \partial Z$ grounded rel $\wZ^{m+1}$ with $|I'|<\length_{m+1}$ -- contradicting Statement \ref{C2:lem:ind:mainthm:v1} on levels $\ge m+1$.
\end{itemize}

It remain to verify Statement \ref{C3:lem:ind:mainthm:v1}. Let us assume converse: $\Width^+_{\blambda}(I)> \bK$ for a combinatorial level $m$ interval $I\subset\partial Z$. Applying Theorem~\ref{thm:SpreadingAround}, we obtain a $[\bbt \bK,\blambda]^+$ wide interval $J$ grounded rel $\wZ^{m}$ with length $\le \length_{m}$. This contradicts Statement \ref{C2:lem:ind:mainthm:v1}.
\end{proof}

Since $\length_{n+2} < \length_n/2$, see~\eqref{eq:length decrease}, we have:\[\length_{2n}< \length_0 /2^n=|\theta_0|/2^n \le |\theta_0|/(2\blambda +4) \sp\sp\text{ if }\sp n >  \log_2 (2\lambda +4).\]
We obtain that $\bbs_\theta \le \bbs\coloneqq 2 \log_2 (2\lambda +4)$. Set $\bK_i\coloneqq (2\bchi)^i \bK$.

\begin{lem}[A few shallow levels]
\label{lem:ind:mainthm:shallow case}
The sequence of geodesic pseudo-Siegel disks in Lemma~\ref{lem:ind:mainthm:v1} can be continued with a sequence of geodesic pseudo-Siegel disks 
\[  \wZ^{\bbs_\theta-1},\sp \wZ^{\bbs_\theta-2}, \sp\wZ^{\bbn_\theta-3},\sp \dots, \sp \wZ^{-1} \coloneqq \wZ_{f_\theta},\sp\sp\sp \bbs_\theta\le \bbs\] 
with the following properties for $m<\bbs_\theta$:
\begin{enumerate}[label=\text{(\Alph*)},font=\normalfont,leftmargin=*]
\item \label{C1:lem:ind:mainthm:shallow case}
 If $Z$ does not have an external level-$m$ parabolic rectangle of width $\sqrt{(2\bchi)^{\bbs_{\theta}-m-1} \bK}$, then $\wZ^m=\wZ^{m+1}$; otherwise $\wZ^m$ is the regularization of $\wZ^{m+1}$ constructed by Corollary~\ref{cor:regul}.  Moreover, $\Width^+_{\blambda,\ext, m}(J)\le 2\sqrt { (2\bchi)^{\bbs_{\theta}-m-1}\bK}+O(1)$ for every interval $J\subset \partial Z$ grounded rel $\wZ^{m}$ with  $\length_{m+1}\le |J| \le \length_m$.

\item \label{C2:lem:ind:mainthm:shallow case} For every interval $J\subset \partial Z$ grounded rel $\wZ^{m}$ with  $\length_{m+1}\le |J| \le \length_m$, we have $\Width^+_{\blambda}(J)\le (2\bchi)^{\bbs_{\theta}-m} \bK$.
\end{enumerate}
\end{lem}
\begin{proof}
Statements \ref{C1:lem:ind:mainthm:shallow case} and \ref{C2:lem:ind:mainthm:shallow case} are proven in the same way as the corresponding statements in Lemma~\ref{lem:ind:mainthm:v1} where Statement \ref{C3:lem:ind:mainthm:v1} of Lemma~\ref{lem:ind:mainthm:v1} replaced with a weaker Statement \ref{C2:lem:ind:mainthm:shallow case} of Lemma~\ref{lem:ind:mainthm:shallow case}.
\end{proof}
 
\subsubsection{Proof of Theorem~\ref{mainthm:v1}} \label{sss:PrfOf mainthm:v1} We have shown in Lemmas~\ref{lem:ind:mainthm:v1} and~\ref{lem:ind:mainthm:shallow case} that there are absolute $\blambda \gg 1,\sp \bK\gg_\blambda 1$ such that $\Width^+_\blambda(I) \le \bK$ for every combinatorial interval $I$. We need to show that $\Width^+_3(I)\le  \bK_2$ for some $\bK_2$.

Assume $I$ is a level $m$ combinatorial interval. For simplicity, let us round up $\blambda$ to the smallest integer number. Choose the minimal $n>m$ such that $\length_{m}/\length_n>2\blambda+1$. We can decompose $I$ as a concatenation \[I=I_{-\blambda}\cup I_{-\blambda+1}\cup \dots I_{-1}\cup I_0\cup I_1\cup \dots \cup I_{\blambda}\]
so that for $k\not=0$ the interval $I_k$ is level-$n$ combinatorial while $I_0$ is a grounded rel $\wZ^{n}$ interval. By construction,
Then \[\Fam^+_{3} (I)\subset \bigcup_{j}\Fam^+(I_j, (3I)^c)\sp\sp\text{ and }\sp\sp \Fam^+(I_k, (3I)^c)\subset \Fam^+_\blambda(I_k)\sp\text{ for }k\not=0.\]
For $k\not=0$, we have $\Width^+_\blambda(I_k)\le \bK$. If $\Width^+(I_0, (3I)^c)\gg _\bK 1$, then applying Calibration Lemma~\ref{lmm:CalibrLmm} to $\Fam^+(I_0, (3I)^c)$, we obtain an interval $J$ grounded rel $\wZ^n$ with $|J|\le \length_n$ such that $\Width^+_\blambda(J)\gg_\bK 1$ -- contradicting the estimates in  Lemmas~\ref{lem:ind:mainthm:v1} and~\ref{lem:ind:mainthm:shallow case}. Therefore, $\Width_3(I)$ is bounded in terms of $\blambda$ and $\bK.$
\qed
 
\subsection{Proof of Theorem~\ref{mainthm:v2}}\label{ss:prf:mainthm:v2} Consider a renormalization level $m\ge -1$ with $\length_{m}/\length_{m+1}\gg 1$, and let $T=[v,w]$ be an interval in the diffeo-tiling $\Dbb_m$. As in~\S\ref{s:par fjords}, we assume that $v<w$ in $T$ and that $T'=[v',w]$ is $T\cap f^{\qq_{m+1}}(T)$ (with necessary adjustments for $m=-1$). 

For $k< \log_2 [\length_{m}/(20\length_{m+1})]$ we define $v_k,w_k\in T'$ to be the points at distance \newline ${10 (2^{k}-1)\length_{m+1}}$ from $v'$ and $w$ respectively with $v_k<w_k$ in $T'$. We set 
\[T^k\coloneqq [v_k,w_k]\subset T',\sp\sp X^{k+1}\coloneqq [v_k,v_{k+1}],\sp\sp Y^{k+1}\coloneqq  [w_{k+1},w_k]\]
i.e.~$T^{0}=T'$ and 
\[T^{k}=X^{k+1}\cup T^{k+1}\cup Y^{k+1},\sp\sp\sp|X^{k+1}|=|Y^{k+1}|=2^{k}10 \length_{m+1}.\]

\begin{lem}
\label{lem:est for bN}
For a constant $\bK$ in Theorem~\ref{mainthm:v1} and every above well-defined pair $X^k, Y^k$ with $k\ge 1$, we have
\[\Width^+_{\ext,m} (X^k,Y^k)\succeq_{\bK} 1.\]
\end{lem}
\begin{proof}
Assume converse; then we have the following estimate of the dual family:
\[\Width^+_{\filled_{m}} (T^k, \partial \filled_{m}\setminus T^{k-1})=K\gg_\bK 1.\]
Up to $O(1)$, the family $\Fam^+_{\filled_{m}} (T^k, \partial \filled_{m}\setminus T^{k-1})$  is within two rectangles $\RR_x,\RR_y$ in $\C\setminus \intr \filled_m$. Applying Lemma~\ref{lem:univ push}, we can push-forward $\Fam^+_{\filled_{m}} (T^k, \partial \filled_{m}\setminus T^{k-1})$ almost univalently under $f^{\qq_{m+1}}\colon \wC\setminus \filled_{m+1}\to \wC\setminus \overline Z$; we obtain that
\[\Width^+_{\overline Z}(T_k\boxplus \theta_{m+1}, \sp (T_{k-1}\boxplus \theta_{m+1})^c)\ge K-O(1).\]
Below we recognize three types of the curves in \[\Fam\coloneqq \Fam^+_{Z}(T_k\boxplus \theta_{m+1}, \sp (T_{k-1}\boxplus \theta_{m+1})^c)\] and prove that the width of each type curve family can be bounded in terms of $\bK$.

\emph{Curves diving into $\filled_{m}\setminus \overline Z$.} If the width of such curves is sufficiently big, then applying Calibration Lemma~\ref{lmm:CalibrLmm} to such curves, we obtain a sufficiently wide interval on deeper scale contradicting the estimates in Lemmas~\ref{lem:ind:mainthm:v1} and~\ref{lem:ind:mainthm:shallow case}.

\emph{Curves landing at $[v,v']$.} Since $[v,v']$ is combinatorial, the width of such curves is bounded by $\bK$.

\emph{External curves landing at $T'\cap (T_{k-1}\boxplus \theta_{m+1})^c$.} Note that $T'\cap (T_{k-1}\boxplus \theta_{m+1})^c$ consists of two intervals of length $\asymp 2^k\length_{m+1}$. Since the distance between \[T'\cap (T_{k-1}\boxplus \theta_{m+1})^c\sp\sp\text{ and }\sp\sp T_k\boxplus \theta_{m+1}\] is $\asymp 2^k\length_{m+1}$, the width of curves of this last type is $O(1)$ by Theorem~\ref{thm:par fjords}.
\end{proof}

Let us now choose a sufficiently big $\bN\gg_{\bK} 1$. Write $\bM\coloneqq \log_2 \bN/10^3. $. If $\length_{m}/\length_{m+1}\ge \bN/2$, then 
\[\Width^+_{\ext,m}(X^1\cup X^2\cup\dots \cup X^{\bM}, Y^1\cup Y^2\cup\dots \cup Y^{\bM}) \gg_\bK \bM \gg \bK\] by Lemma~\ref{lem:est for bN}.
Therefore, $\Fam^+_{\ext,m}(X^1\cup X^2\cup\dots \cup X^{\bM}, Y^1\cup Y^2\cup\dots \cup Y^{\bM})$ contains a parabolic external level $m$ rectangle of width $\sqrt{\bK}$ and the regularization happens within the orbit of such rectangle. This implies~\eqref{eq:mainthm:v2}.

The combinatorial threshold~\eqref{eq:mainthm:v2} implies that for every interval $I\subset\partial Z$ regular rel $\wZ^m$ with $\length_{m}\ge  |I|\ge \length_{m+1}$, there is a grounded rel $\wZ^m$ interval $I_\grnd \subset I$ such that $I\setminus I_\grnd$ is within $2\bN$ level $m+1$ combinatorial intervals. Therefore, the condition ``grounded rel $\wZ^m$" in  Lemmas~\ref{lem:ind:mainthm:v1} and~\ref{lem:ind:mainthm:shallow case} can be replaced with ``regular rel $\wZ^m$'' by possibly increasing $\bK$.

\section{Mother Hedgehogs and uniform quasi-conformality of $\wZ$}\label{s:MotherHedgehogs}
Recall that we are considering eventually golden-mean rotations numbers $\theta$,~\eqref{eq:cond on theta}.

\begin{thm}
\label{mainthm:v3}
There is an absolute constant $\bK\gg 1$ such that the pseudo-Siegel disk $\wZ_f=\wZ^{-1}$ in Theorem~\ref{mainthm:v2} is $\bK$-qc.
\end{thm}

Recall that a hull $Q\subset \C$ is a compact connected full set.
The {\em  Mother Hedgehog}~\cite{Ch} for a neutral polynomial $f_\theta$
is an invariant hull containing both the fixed point $0$ and the critical
point $c_0(f)$.

\begin{thm}
\label{mainthm:v4}
 Any neutral quadratic polynomial $f=f_\theta$, $\theta\not\in \Q$,
  has a Mother Hedgehog $H_f\ni c_0(f)$ such that $f\colon  H_f \to H_f$ is a homeomorphism.
\end{thm}

\subsubsection{Outline of the section}
\label{sss:outline MotherHedg} Since $\wZ=\wZ^{-1}$ is obtained from a qc disk $\overline Z$ by adding finitely many fjords bounded by hyperbolic geodesics in $\wC\setminus \overline Z$, the resulting pseudo-Siegel disk $\wZ$ is a qc disk. To show that $\wZ$ is uniformly $\bK$-qc, we will introduce a nest of tilings on $\partial \wZ$ as follows:
\begin{equation}
\label{eq:nest of til on wZ}
\TT(\wZ)\coloneqq \operatorname{Projection}_{\wZ}(\Dbb)\cup \bigcup_{\beta^{m}_i\subset\partial \wZ^m} \TT(\beta^m_i),
\end{equation}
where:
\begin{itemize}
\item $\operatorname{Projection}_{\wZ}(\Dbb)$ is the projection onto $\wZ$ the nest of diffeo-tilings $\Dbb=[\Dbb_n]_{n\ge -1}$~\eqref{eq:diffeo nest}, where intervals completely submerged into $\wZ$ are removed;
\item $\TT(\beta^m_i)=[\TT_{n}(\beta^m_i)]_{n\ge m+1}$ is an appropriate nest of tilings on dams, see~\S\ref{ss:nest of til:beta}.  
\end{itemize}

The combinatorial threshold $\bN$ will imply that $\TT(\wZ)$ has $2\bN$-bounded combinatorics: each level-$n$ interval consists of at most $2\bN$ intervals of level $n+1$. Using Theorem~\ref{mainthm:v1}, we will show that  $\TT(\wZ)$ has uniformly bounded outer geometry: neighboring intervals in $\TT^n(\wZ)$ have comparable outer harmonic measures. And using Theorem~\ref{thm:wZ:shallow scale}, we will show that $\TT(\wZ)$ has uniformly bounded inner geometry. This will conclude that $\wZ$ is uniformly $\bK$-qc as a result of quasisymmetric welding, see  Lemma~\ref{lmm:qc tiling condit}. 

Let us comment about the construction of the $\TT(\beta^m_i)$. Every dam $\beta^m_i$ connects two points in $\CP_{m+1}$, call them $x$ and $y$. For every $n\ge m+1$, we can consider four level-$n$ intervals of $\Dbb_n$ adjacent to $x,y$; we call these intervals the \emph{$n$th foundation} of $\beta^m_i$. Our estimates imply that intervals in the $n$th and $(n+1)$th foundations have comparable outer harmonic measures. This fact allows us to introduce a nest of tilings $\TT(\beta^m_i)=[\TT_{n}(\beta^m_i)]_{n\ge m+1}$ comparable with the foundations of $\beta^m_i$ on all levels $\ge m+1$. We view $\beta^m_i$ as a sole interval in $\TT_{m+1}(\beta^m_i)$. Since every dam $\beta^m_i$ is protected by a wide rectangle $\XX^m_i$ (Assumption~\ref{ass:wZ:EtraProt}), different dams are geometrically faraway and their nest of tilings do not interact much in~\eqref{eq:nest of til on wZ}.

Theorem~\ref{mainthm:v4} follows from Theorem~\ref{mainthm:v3} by taking Hausdorff limits of bounded-type Siegel disks:

\begin{proof}[Proof of Theorem~\ref{mainthm:v4} using Theorem~\ref{mainthm:v3}]
For every $\theta\in \R\setminus \Z$, consider a sequence of eventually golden-mean rotation numbers $\theta_n$ converging to $\theta$. Let $Z_{\theta_n}\subset \wZ_{\theta_n}$ be the Siegel disk and a $\bK$-qc pseudo-Siegel disk of $f_{\theta_n}$. By passing to a subsequence, we can assume that $Z_{\theta_n}$ has a Hausdorff limit $H_{f_\theta}=H_{\theta}$ and $\wZ_{\theta_n}$ has a qc limit $\wZ_{\theta}$. We obtain that $H_\theta$ is $f_\theta$ invariant, $c_0, \alpha\in H_\theta\subset \wZ_{\theta}$, and $f_\theta\colon  \wZ_{\theta}\to f_\theta \left(  \wZ_{\theta}\right)$ is a homeomorphism. Therefore, $f_\theta\mid H_\theta$ is a homeomorphism.  
\end{proof}

\subsubsection{Uniform Pseudo-Siegel bounds}\label{sss:uniform wZ bounds} Let us briefly summarize one of the main outcomes of the paper. Pseudo-Siegel disks $\wZ^m_\theta$ are almost invariant up to $\qq_{m+1}$ iterations such that 
\begin{enumerate}
\item[\setword{(I)}{Item:wZ:I}] their internal geometry is that of rotational dynamics, see Theorem~\ref{thm:wZ:shallow scale} and~\S\ref{ss:peninsulas:geometry};
\item[\setword{(O)}{Item:wZ:O}] their exterior geometry is that of unicritical circle maps, see~\S\ref{sss:rem:crit circle maps and fjords} and Remark~\ref{rem:TT:wZ^m}.
\end{enumerate}
Thus, $\wZ^m_\theta$ can be viewed as a \emph{mating} of rotational and unicritical dynamics; i.e., pseudo-Siegel disks extend principles of the Douady-Ghys surgery to all rotation numbers. We remark that such an extension is only possible in the almost-invariant framework: the inner geometry of $\wZ^m_\theta$ can disappear in the limit $H_\theta=\bigcap _{m\ge -1}\wZ^m_\theta$; for instance, if $\theta$ is a Cremer number, then $H_\theta$ has empty interior.

By \emph{pseudo-Siegel bounds}, we refer to the \emph{existence} of disks $\wZ^m_\theta$ with appropriate uniform bounds on tilings, see Definition~\ref{defn:ps Siegel bounds} and Remark~\ref{rem:TT:wZ^m}. Theorem~\ref{mainthm:v2} constructs such $\wZ^m_\theta$ for eventually golden-mean numbers. Since $\wZ^{-1}_\theta$ is uniformly qc, it exists by continuity for all $\theta$ (Theorem~\ref{mainthm:v3}). The continuity (and existence) of $\wZ^m$ for all $\theta$ is justified in~\cite{DL:sector bounds}; see~\S\ref{ss:pseudo-Siegel bounds} for summary.

\subsection{Nests of tilings}
\label{ss:NestOfTiling}
We will use notations similar to \cite[\S15.1]{L-book}. Consider a closed qc disk $D\subset \C$. Let $\TT=(\TT_n)_{n\ge m}$ be a system of finite partitions of $\partial D$ into finitely many closed intervals such that $\TT_{n+1}$ is a refinement of $\TT_n$. We say that $\TT$ is a \emph{nest of tilings} if 
\begin{itemize}
\item the maximal diameter of intervals in $\TT_n$ tends to $0$ as $n\to \infty$, and
\item every interval in $\TT_n$ for $n\ge m$ decomposes into at least two intervals of $\TT_{n+2}$.
\end{itemize}
 Similarly, a nest of tilings is defined for a closed qc arc. (In the second condition, we require $\TT_{n+2}$ instead of $\TT_{n+1}$ because of Lemma~\ref{lem:Dbb_n:Dbb_n+2}.)

We say that a nest of tilings $\TT$ has \emph{$M$-bounded combinatorics} if every interval of $\TT_n$ consists of at most $M>1$ intervals of $\TT_{n+1}$.

For an interval $I\in \TT_n$, let $I_\ell, I_\rho\in \TT_n$ be two its neighboring intervals. We denote by $[3I]^c$ the closure of $\partial D\setminus (I_\ell\cup I\cup I_\rho)$.  We set \[[3I]^c\coloneqq \overline{\Gamma\setminus (I\cup I_\ell \cup I_\rho)},\] and we define:
\begin{itemize}
\item $\Fam^-_{3,\TT}(I)$ to be the family of curves in $D$ connecting $I$ and $[3I]^c$;
\item  $\Fam^+_{3,\TT}(I)$ to be the family of curves in $\wC\setminus D$ connecting $I$ and $[3I]^c$;
\item $\Width^\pm_{3,\TT}(I)=\Width\big(\Fam^\pm_{3,\TT}(I)\big)$.
\end{itemize}

We say that a nest of tilings $\TT$ has \emph{essentially $C$-bounded outer geometry} if for every $I\in \TT$ we have $\Width^+_3(I)\le C$. If, moreover, $\TT$ has $M$-bounded combinatorics, then we say that $\TT$ has \emph{$(C,M)$-bounded outer geometry}. Similarly, bounded and essentially bounded inner geometries are defined.

\begin{lem}
\label{lmm:qc tiling condit}
For every pair $C,M$, there is a $K_{C,M}>1$ such that the following holds. Let $D$ be a closed qc disk and $\TT$ be a nest of tilings of $\partial D$. If $\TT$ has $(C,M)$-bounded inner and outer geometries, then $D$ is a $K_{C,M}$ qc disk.
\end{lem}
\begin{proof}
Assume that $\TT=[\TT^n]_{n\ge -1}$. Then there are at least $4$ intervals in $\TT_3$. Let us choose base points $u\in \intr D$ and $v\in \wC\setminus D$ such that the inner and outer harmonic measure of every $I\in \TT_3$ with respect to $u$ and $v$ is less than $1/3$.
 
Consider conformal maps $h_-\colon (\intr D,u)\to (\Disk, 0)$ and $h_+\colon (\wC\setminus D,v)\to (\Disk, 0)$ and define \[\TT^-_n\coloneqq h_{-,*}(\TT_n)\sp\sp \text{ and }\sp\sp \TT^+_n \coloneqq  h_{+, *}(\TT_n)\]
to be the induced partitions on $\Circle=\partial \Disk$. The assumptions on the harmonic measures and the width imply that the diameter of every $I\in \TT_n^-\cup \TT_n^+$ is comparable to the diameters of two neighboring intervals in the same tiling -- see the estimates in Lemma~\ref{lmm:W^- I J}. Therefore, $h_+\circ h_-^{-1}$ is quasisymmetric with the dilatation bounded in terms $C$ and $M$. The curve $\partial D$ is a $K_{C,M}$-qc circle as the result of a qc welding.
\end{proof}

\subsection{Estimates for $\Dbb_n$} \emph{Let us for the rest of this section view $\bK$ in Theorems~\ref{mainthm:v1} and~\ref{mainthm:v2} as $\bK=O(1)$.} In particular, the main estimate in Theorem~\ref{mainthm:v1}  takes form $\Width^+_3(I)\le \bK=O(1)$ for every combinatorial interval $I\subset \partial Z$. We will need the following estimates:

\begin{lem}
\label{lem:Dbb_n:est}
 For every diffeo-tiling \S\ref{sss:diff tilings} $\Dbb_m$ consisting of at least $4$ intervals and every interval $I_j\in \Dbb_m$ the following holds. Write $L_j\coloneqq  I_{j-1} \cup I_j\cup I_{j+1}$, and let $ I^m_j,  L^{m,c}_j $ be the projections of $I_j, L_j^c$ onto $\partial \wZ^m$. Then
 \begin{enumerate}[label=\text{(\Roman*)},font=\normalfont,leftmargin=*]
 \item \label{E1:lem:Dbb_n:est} $\Width^+_{3,\Dbb_m}(I_j)\coloneqq \Width_Z^+\big(I_j, L_j^c\big)\asymp 1$;
 \item \label{E2:lem:Dbb_n:est} $\Width_{3,\widehat \Dbb^m_m}(I_j)=\Width\left(\Fam_{3,\widehat \Dbb^m_m}(I_j)\right)\coloneqq \Width_{\wZ^{m}}\big( I^m_j,  L^{m,c} \big)\asymp 1$
 \end{enumerate}
(where $\widehat \Dbb^m_m$ denotes the projection of $\Dbb_m$ onto $\wZ^m$);
 \begin{enumerate}[label=\text{(\Roman*)},start=3,font=\normalfont,leftmargin=*]
\item \label{E3:lem:Dbb_n:est}  for an interval $V\in \Dbb_{m+n}$ such that $V\subset I_j$ and $V$ is attached to one of the endpoints of $I_j$, we have $\Width_Z^+\big(V, L^c \big)\asymp_n 1$.
\end{enumerate}
\end{lem}
\begin{proof}
It follows from Lemmas~\ref{lem:ind:mainthm:v1} and~\ref{lem:ind:mainthm:shallow case} by splitting $I_j$ as in \S\ref{sss:PrfOf mainthm:v1} that $\Width^+_{3,\Dbb_m}(I_j)\succeq  1$. Since this holds for all the $I_j$, we obtain Statement~\ref{E1:lem:Dbb_n:est}.  

Lemma~\ref{lem:trad width to width+} reduces Statement~\ref{E2:lem:Dbb_n:est} to Lemmas~\ref{lem:ind:mainthm:v1},~\ref{lem:ind:mainthm:shallow case}. 

Choose a point $w\in \wC\setminus \overline Z$ such that the intervals $I_{j-1},I_j, I_{j+1}, L_j^c$ have comparable harmonic measures in $\wC\setminus \overline Z$ with respect to $w$. Let $V_{n}\in \Dbb_{m+n}, V_0=I_j$ be a sequence of nested intervals so that $V_{n}$ is attached to one of the endpoints of $I_j$. We {\bf claim} that $V_n$ and $V_{n-1}$ have comparable harmonic measures in $(\wC\setminus \overline Z, w)$; this will imply Statement~\ref{E3:lem:Dbb_n:est}.
\begin{proof}[Proof of the claim]
If $\length_{m+n}\asymp \length_{m+n-1}$, then the claim follows from Statement~\ref{E1:lem:Dbb_n:est}. 

Assume that $\length_{m+n-1}\gg \length_
{m+n}$. Let $V'_n\subset V_{n-1}\setminus V_n$ be the interval in $\Dbb_{m+n}$ attached to another endpoint of $V_{n-1}$. It follows from $\Width^+(X^1,Y^1)\succeq 1$ (in Lemma~\ref{lem:est for bN}) and Statement~\ref{E1:lem:Dbb_n:est} that $\Width^+(V_n,V'_n)\asymp 1$. This implies the claim.
\end{proof}
\end{proof}

For an interval $I_j\subset \Dbb_m$, let $\RR^+_{\dual}(I_j)$ be the geodesic rectangle (see~\ref{sss:GeodRect}) in $\wC\setminus Z$ between $I_{j-1}$ and $I_{j+1}$; i.e., $\partial^{h,0}\RR^+_\dual (I_j)=I_{j-1},$ $\partial^{h,1}\RR^+_\dual (I_j)=I_{j+1},$ and the vertical sides of $\RR^+_\dual (I_j)=I_{j-1}$ are the hyperbolic geodesics of $\wC\setminus \overline Z$. It follows from Lemma~\ref{lem:Dbb_n:est}, \ref{E1:lem:Dbb_n:est} that 
\begin{equation}
\label{eq:Width:R_dual}
\Width\left(\RR^+_{\dual}(I_j)\right)\asymp 1.
\end{equation}

\subsection{Nest of tiling of dams} 
\label{ss:nest of til:beta}
Consider a dam $\beta^m
_i\subset \partial \wZ$, and assume that it connects $x$ and $y$. We recall from Assumption~\ref{ass:wZ:Linking} that $x,y\in \CP_{m+1}$. Let us denote by $V=V(\beta^m_i)\ni \infty$ the unbounded component of $\wC\setminus (\overline Z\cup \beta^m_i)$ and by $U=U(\beta^m_i)\not\ni \infty$ the bounded component of $\wC\setminus (\overline Z\cup \beta^m_i)$. For every $n\ge m+1$, we specify, see Figure~\ref{Fig:intervals: I_x}:
\begin{itemize}
\item $I^n_V$ to be the interval in $\Dbb_n$ (see~\S\ref{sss:diff tilings}) adjacent to $x$ such that $I^n_V\subset \partial V$,
\item $I^n_U$ to be the interval in $\Dbb_n$ adjacent to $x$ such that $I^n_U\subset \partial U$,
\item $J^n_V$ to be the interval in $\Dbb_n$ adjacent to $y$ such that $J^n_V\subset \partial V$,
\item $J^n_U$ to be the interval in $\Dbb_n$ adjacent to $y$ such that $J^n_U\subset \partial U$.
\end{itemize}
We will refer to $I^n_U, I^n_V,J^n_U, J^n_V\in \Dbb_n$ as the \emph{$n$th foundation} of $\beta^m_i$.

\begin{figure}[t!]
\[\begin{tikzpicture}[line width=0.2mm] 

\draw (-6.5,0)--(-5,0)
(-3,0)--(3,0)
(6.5,0)--(5,0);
\draw[blue,line width=0.35mm] (-5,0)--(-3,0)
(3,0)--(5,0);

\filldraw[red] (-4,0) circle (0.04 cm);
\node[below,red] at(-4,0) {$x$};
\filldraw[red] (4,0) circle (0.04 cm);
\node[below,red] at(4,0) {$y$};
\draw[red] (-4,0) edge[bend left=50] node[above]{$\beta^m_i$}  (4,0);

\filldraw[blue] (-3,0) circle (0.04 cm);
\filldraw[blue] (-5,0) circle (0.04 cm);

\node[below,blue] at (-3.5,0) {$I^n_{U}$};
\node[below,blue] at (-4.5,0) {$I^n_{V}$};

\node at (0,1) {$U$};

\node at (-4,1.5) {$V$};

\filldraw[blue] (3,0) circle (0.04 cm);
\filldraw[blue] (5,0) circle (0.04 cm);

\node[below,blue] at (3.5,0) {$J^n_{V}$};
\node[below,blue] at (4.5,0) {$J^n_{U}$};

\end{tikzpicture}\]
\caption{Intervals $I^n_U, I^n_V, J^n_U,J^n_V$ form the foundation of $\beta_i^m$.}
\label{Fig:intervals: I_x}
\end{figure}
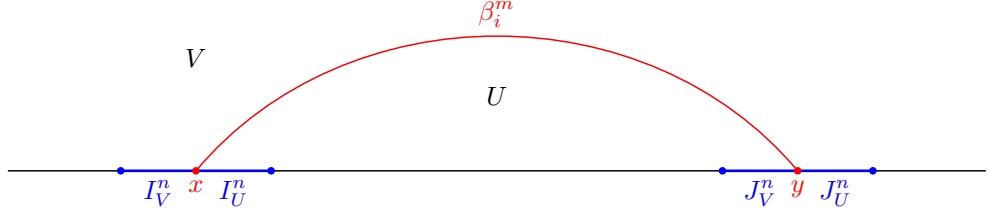

We say that intervals $A_1,A_2,\dots, A_s\subset \partial S^1$ with pairwise disjoint interiors are \emph{harmonically comparable with respect to $\Disk$} if $\Width^-_\Disk(I,J)\asymp 1$ for every pair of non-adjacent intervals
\[I, J \in \{A_j, j\le s\}\cup \{\text{connected components of }S^1\setminus \bigcup_{j=1}^s A_j\}.\]
In other words, all the $A_j$ as well as all their complementary intervals have comparable inner harmonic measure with respect to a certain point in $\Disk$. Similarly, the harmonic comparison is defined for intervals of $\partial Z$ rel $\wC\setminus \overline Z$. The following lemma is a consequence of Lemma~\ref{lem:Dbb_n:est}, Estimates~\ref{E1:lem:Dbb_n:est} and \ref{E3:lem:Dbb_n:est}.

\begin{lem}
\label{lem:harm compar of I,J}
For every $\beta^m_i$ and every $n> m+1$, we have:
\begin{itemize}
\item $I^{m+1}_U, I^{m+1}_V, J^{m+1}_U, J^{m+1}_V$ are harmonically comparable with respect to $\wC\setminus \overline Z$,
\item $I^{n}_U, I^{n}_V, (I^{n-1}_U\cup I^{n-1}_V)^c$ are harmonically comparable with respect to $\wC\setminus \overline Z$,
\item $J^{n}_U, J^{n}_V, (J^{n-1}_U\cup J^{n-1}_V)^c$ are harmonically comparable with respect to $\wC\setminus \overline Z$.
\end{itemize}
\end{lem}\qed

\begin{lem}[Existence of $\TT(\beta^m_i)$]
\label{lem:TT for beta^m_i}
There is an absolute $\bC>0$ such that for every $\beta^m_i$ there is a nest of tilings
 \begin{equation}\label{eq:lem:TT for beta^m_i}
\TT(\beta^m_i)=(\TT_{n})_{n\ge m+1}
\end{equation} with $10$-bounded combinatorics such that the following properties hold. \\
\noindent For $V=V(\beta^m_i)$ as above and every interval $I\in \TT_{n}\cup\{I^n_V,J^n_V\},$ $n\ge m+1$, let 
\begin{itemize}
\item $I^V_-, I_+^V$ be two neighboring intervals of $I$ in \[\TT_{n}(\beta^m_i)\ \cup \ \{\text{intervals of $\Dbb_n$ that are in $\partial V$} \};\]
\item $\RR^+_{\dual,V}(I_j)$ be the geodesic rectangle (similar to~\eqref{eq:Width:R_dual}) in $V$ between $I^V_-, I_+^V$;
\end{itemize}
\noindent For $U=U(\beta^m_i)$ and every interval $I\in \TT_{n}\cup\{I^n_U,J^n_U\},$ $n\ge m+1$, let 
\begin{itemize}
\item $I^U_-, I_+^U$ be two neighboring intervals of $I$ in \[\TT_{n}(\beta^m_i)\ \cup \ \{\text{intervals of $\Dbb_n$ that are in $\partial U$} \};\]
\item  $\Fam_{3,U}^-(I)$ be the family of curves in $U$ connecting $I$ to $\partial U\setminus (I^U_-\cup I\cup  I_+^U)$.
\end{itemize}

 \noindent Then
 \begin{enumerate}[label=\text{(\Alph*)},font=\normalfont,leftmargin=*]
\item \label{A:lem:TT for beta^m_i}$\Width(\RR^+_{\dual,V}(I_j))\asymp 1$;
\item\label{B:lem:TT for beta^m_i} $\Width_{3,U}^-(I)=\Width(\Fam_{3,U}^-(I))\asymp 1$.
\end{enumerate}
\end{lem}
\begin{proof}
Consider  \[S^1=\{z :\sp |z|=1\}\subset \C\sp\sp \text{ and } \sp\sp X=S^1\cup [-1,1]\subset \C. \]
And let  us consider a conformal map $h\colon \wC\setminus \overline Z\to \Disk$ mapping $x$ and $y$ to $-1$ and $1$ respectively. Since $\beta^m_i$ is a hyperbolic geodesic, we have  \[h(\beta^m_i)=[-1,1]\sp\sp\text{ and }\sp\sp h(\partial Z\cup \beta^m_i)=X.\]

Let $h_*(\Dbb_n)$ be the pushworward of the diffeo-tiling $\Dbb_n, n>m$ onto $S^1$ by $h$. By Lemma~\ref{lem:Dbb_n:est}, any two neighboring intervals in $h_*(\Dbb_n)$ have comparable diameters (uniformly over $n$). And by Lemma~\ref{lem:harm compar of I,J}, the following diameters are comparable:
\begin{itemize}
\item of $ h(I^{m+1}_U),h( I^{m+1}_V), h(J^{m+1}_U),h( J^{m+1}_V)$;
\item of $ h(I^{n}_U),h( I^{n}_V), h(I^{n+1}_U),h( I^{n+1}_V)$ for $n\ge m+1$;
\item of $ h(J^{n}_U),h( J^{n}_V), h(J^{n+1}_U),h( J^{n+1}_V)$ for $n\ge m+1$.
\end{itemize}

For $n\ge m+1$, let $\beta_I^n$ be the hyperbolic geodesic of $\Disk$ connecting the endpoints of $h\left(I^n_{U}\cup I^n_V\right)$, and let $x_I^n$ be the intersection of $\beta_I^n$ with $[-1,1]\subset X$. We define $I^n\coloneqq [-1,x^n_I]\subset [-1,1]$.

Similarly,  let $\ell_J^n$ be the hyperbolic geodesic of $\Disk$ connecting the endpoints of $h\left(J^n_{U}\cup J^n_V\right)$, and let $x_J^n$ be the intersection of $\ell_J^n$ with $[-1,1]\subset X$. We define $J^n\coloneqq [x^n_J,1]\subset [-1,1]$. By construction, the following diameters are comparable:
\begin{itemize}
\item of $ I^{m+1}, J^{m+1} ,h(I^{m+1}_U),h( I^{m+1}_V), h(J^{m+1}_U),h( J^{m+1}_V)$;
\item of $I^n,I^{n+1}, h(I^{n}_U),h( I^{n}_V), h(I^{n+1}_U),h( I^{n+1}_V)$ for $n\ge m+1$;
\item of $ J^n,J^{n+1},h(J^{n}_U),h( J^{n}_V), h(J^{n+1}_U),h( J^{n+1}_V)$ for $n\ge m+1$.
\end{itemize}

We can now easily extend $(I^n,J^n)_{n\ge m+1}$ to a tiling of $[-1,1]$ and then pull it back under $h$ to a required tiling of $\beta^m_i$.
\end{proof}

\subsection{Nest of tilings on $\partial \wZ$} \label{ss:NestofTilings}For $m\ge 0$, consider the diffeo-tiling $\Dbb_m$ of $\partial Z$. Since level $n\le m$ dams land at $\CP_{n+1}$ (Assumption~\ref{ass:wZ:Linking}), every interval $T\in \Dbb_m$ is either regular rel $\wZ$ or is inside a reclaimed fjord of generation $n< m$. We denote by $\Dbb'_m$ the set of regular rel $\wZ$ intervals in $\Dbb_m$. And we define $\widehat \Dbb'_m\equiv \Dbb'^{-1}_m$ to be the set of projections of intervals in $\Dbb'_m$ onto $\wZ\equiv \wZ^{-1} $. We define the \emph{diffeo-tiling} $\TT(\wZ)$ on $\partial \wZ$ to be union of the $\Dbb'_n$ and appropriated tilings on dams:
\begin{equation}
\label{eq:TT of wZ}
\TT(\wZ)\equiv \TT(\wZ^{-1})\coloneqq \big[\widehat \Dbb'_n\big]_{n\ge -1} \cup \bigcup_{\beta^m_i} \TT(\beta^m_i),\hspace{2cm} \text{ $\TT(\beta^m_i)$ is as in~\eqref{eq:lem:TT for beta^m_i}}
\end{equation} 

The following proposition combined with Lemma~\ref{lmm:qc tiling condit} implies Theorem~\ref{mainthm:v3}.
\begin{prop}\label{prop:wZ:tiling}There is an absolute $\bC\gg 1$ such that for every eventually golden-mean rotation number the nest of tilings $\TT(\wZ)$ in~\eqref{eq:TT of wZ} has $(2\bN, \bC)$-bounded inner and outer geometries.
\end{prop}
\begin{proof}
By construction, $\TT(\wZ)$ has $2\bN$ bounded combinatorics. Consider an interval $X\in \TT^n(\wZ)$. We need to show that $\Width^\pm _{3,\TT(\wZ)}(X)=O(1)$, where $\Width^\pm _{3,\TT(\wZ)}$ are defined in~\S\ref{ss:NestOfTiling}. We write 
\begin{itemize}
\item $X^\bullet\coloneqq X$ if $X$ is within a dam $\beta^m_i$, $m<n$;
\item $X^\bullet$ to be the projection of $X$ onto $\partial Z$ if $X$ is an interval in $\partial \wZ^n$.
\end{itemize}

Set either 
\begin{itemize}
\item $\RR\coloneqq \RR^+_{\dual,V}(X^\bullet)$ as in Lemma~\ref{lem:TT for beta^m_i} if  $X\in \TT_{n}\cup\{I^n_V,J^n_V\}$ for a dam $\beta^m_i$;
\item or, otherwise, $\RR\coloneqq \RR^+_\dual(X^\bullet) $ as in~\eqref{eq:Width:R_dual}.
\end{itemize}
 Let $\RR^n$ be the restriction (as in \S\ref{sss:RestrOfFamil}) of $\RR$ onto $\wC\setminus \wZ^n$. By~\eqref{eq:lem:RR vs RR^m}, we have $\Width(\RR^n)\asymp 1$. Since the curves in $\Fam^+ _{3,\TT(\wZ)}(X)$ cross $\RR^n$, we obtain \[\Width^+ _{3,\TT(\wZ)}(X)=O(1).\]
 
To show $\Width^- _{3,\TT(\wZ)}(X)=O(1)$, we will use the monotonicity of the width under the embeddings $\wZ^n, U(\beta^m_i)\subset \wZ=\wZ^{-1}$. Consider several cases.

Assume first that $X\subset \partial \wZ^n$ and $X$ is not a neighbor of any dam $\beta^m_i$, $m<n$. Then $\Width^- _{3,\TT(\wZ)}(X)\le \Width^-_{3,\widehat \Dbb_n^n(\wZ^n)}(X)=O(1)$ by Theorem~\ref{thm:wZ:shallow scale}, where $\widehat \Dbb_n^n$ is the prjection of $\Dbb_n$ onto $\wZ^n$.

In the remaining case, we have $X^\bullet\in \TT_{n}\cup\{I^n_V,J^n_V\},$ $n\ge m+1$ as in Lemma~\ref{lem:TT for beta^m_i} for a dam $\beta^m_i$. If $X$ is in the interior of $\beta^m_i$, then $\Width^-_{3,\TT(\wZ)}(X)=O(1)$ follows from Lemma~\ref{lem:TT for beta^m_i},~\ref{B:lem:TT for beta^m_i} because $\Fam^-_{3,\TT(\wZ)}(X)$ overflows $\Fam_{3,U}^-(X)$, where $U=U(\beta^m_i)\subset \wZ$. 

Assume finally that $X$ touches one of the endpoints of $\beta^m_i$. If $X\subset \beta^m_i$, then $\Fam^-_{3,\TT(\wZ)}(X)$ overflows 
\[ \Fam_{3,U}^-(X)\cup \Fam_{3,\widehat \Dbb^n_n}\left(\big[I^n_U(\beta^m_i)\big]^n \right)\cup \Fam_{3,\widehat \Dbb^n_n}\left(\big[J^n_U(\beta^m_i)\big]^n \right)\]
(see Lemma~\ref{lem:Dbb_n:est}, \ref{E2:lem:Dbb_n:est}; here $[\ ]^n$ denotes the projection onto $\wZ^n$);
otherwise 
$\Fam^-_{3,\TT(\wZ)}(X)$ overflows 
\[\Fam_{3,\widehat \Dbb^n_n}\left(X \right)\cup  \Fam_{3,U}^-\left(\big[I^n_U(\beta^m_i)\big]^n \right)\cup   \Fam_{3,U}^-\left(\big[J^n_U(\beta^m_i)\big]^n \right).\]
Lemma~\ref{lem:Dbb_n:est}, \ref{E2:lem:Dbb_n:est} and Lemma~\ref{lem:TT for beta^m_i}, \ref{B:lem:TT for beta^m_i} complete the proof.
\end{proof}

\begin{rem} \label{rem:TT:wZ^m}Similar to~\eqref{eq:TT of wZ}, we define the nest of tilings $\TT(\wZ^m)$ on $\partial \wZ^m$ to be
\begin{itemize}
\item $\big[\TT(\wZ^m)\big]_{n\ge m}\coloneqq \big[\widehat \Dbb'^m_n\big]_{n\ge m} \cup \bigcup_{\beta^m_i} \TT(\beta^m_i),$ and
\item  $\big[\TT(\wZ^m)\big]_{n< m}\coloneqq \big[\widehat \Dbb^m_n\big]_{n< m} $.
\end{itemize}
where  $\widehat \Dbb'^m_n, \widehat \Dbb^m_n$ are the projections onto $\partial \wZ^m$. Then  $\TT(\wZ^m)$ has \emph{essentially bounded geometry}:
\begin{itemize}
\item on the renormalization levels $\ge m$,  the nest of tilings $\TT(\wZ)$ has $(2\bN, \bC)$-bounded inner and outer geometries (it follows from Proposition~\ref{prop:wZ:tiling});
\item on the renormalization levels $<m$, the inner and outer geometries of $\wZ^m$ and $\overline Z$ are close and are described by the Log-Rules, see Theorem~\ref{thm:wZ:shallow scale}, \S\ref{ss:peninsulas:geometry}, and~\S\ref{sss:rem:crit circle maps and fjords}.
\end{itemize}
\end{rem}

\appendix

\section{Degeneration of Riemann surfaces}
Consider a compact Riemann surface $S\Subset \C$ with boundary. We assume that $\partial S$ consists of finitely many components. In this subsection, we will recall basic tools to detect degenerations of $S$, see~\cite{A,L-book,KL} for details. The discussion can be adjusted for open Riemann surfaces in $\C$ by considering their Caratheodory boundaries.

\subsection{Rectangles and Laminations}
Given two disjoint intervals $I,J\subset \partial S$ on the boundary of a Riemann surface, we denote by 
\begin{itemize}
\item $\Fam_S(I,J)$ the family of curves in $S$ connecting $I$ and $J$:
\begin{equation}
\label{eq:Fam I J}
 \Fam_S(I,J)\coloneqq \{\gamma\colon[0,1]\to S\mid \gamma(0) \in I, \sp \gamma(1)\in J\};
\end{equation}
\item $\Width_S(I,J)=\Width(\Fam_S(I,J))$ the extremal width between $I,J$ -- the modulus of the family $\Fam_S(I,J)$.
\end{itemize}
We will often write $\Fam^-(I,J)=\Fam_S(I,J)$ and $\Width^-(I,J)=\Width_S(I,J)$ when the surface $S$ is fixed. 

\subsubsection{Rectangles}
\label{sss:Rectan} A \emph{Euclidean rectangle} is a rectangle $E_x\coloneqq[0,x]\times [0,1] \subset \C$, where:
\begin{itemize}
\item $(0,0), (x,0), (x,1), (0,1)$ are four vertices of $E_x$,
\item $\partial^h E_x=[0,x]\times\{0,1\}$ is the horizontal boundary of $E_x$,
\item $\partial^{h,0} E_x\coloneqq [0,x]\times\{0\}$ is the \emph{base} of $E_x$,
\item $\partial^{h,1} E_x\coloneqq [0,x]\times\{1\}$ is the \emph{roof} of $E_x$,
\item $\partial^v E_x=\{0,x\}\times [0,1]$ is the \emph{vertical} boundary of $E_x$,
\item $\partial^{v,\ell} E_x\coloneqq \{0\}\times [0,1], \sp  \partial^{v,\rho} E_x\coloneqq \{x\}\times [0,1]$ is the \emph{left} and \emph{right vertical} boundaries of $E_x$; 
\item $\Fam(E_x)\coloneqq \{\{t\}\times[0,1]\mid t\in [0,x]\}$ is the \emph{vertical foliation} of $E_x$,
\item $\Fam^\full(E_x)\coloneqq \{\gamma\colon [0,1]\to E_x \mid \gamma(0)\in \partial ^{h,0}E_x,\ \gamma(1)\in \partial ^{h,1}E_x \}$ is the \emph{full family of curves} in $E_x$; 
\item $\Width(E_x)=\Width(\Fam(E_x))=\Width(\Fam^\full(E_x))=x$ is the \emph{width} of $E_x$,
\item $\mod (E_x)=1/\Width(E_x)=1/x$ the extremal length of $E_x$.
\end{itemize}

By a \emph{(topological) rectangle in $\C$} we mean a closed Jordan disk $\RR$ together with a conformal map $h\colon \RR\to E_x$ into the standard rectangle $E_x$.  The vertical foliation $\Fam(\RR)$, the full family $\Fam^\full(\RR)$, the base $\partial^{h,0}\RR$, the roof $\partial^{h,1}\RR$, the vertices of $\RR$, and other objects are defined by pulling back the corresponding objects of $E_x$. Equivalently, a rectangle $\RR\subset\C$ is a closed Jordan disk together with four marked vertices on $\partial \RR$ and a chosen base between two vertices.

A \emph{genuine subrectangle} of $E_x$ is any rectangle of the form $E'=[x_1,x_2]\times [0,1]$, where $0\le x_1<x_2\le x$; it is identified with the standard Euclidean rectangle $[0,x_2-x_1]\times [0,1]$ via $z\mapsto z- x_1$. A genuine subrectangle of a topological rectangle is defined accordingly. 

A \emph{subrectangle} of a rectangle $\RR$ is a topological rectangle $\RR_2\subset \RR$ such that $\partial^{h,0}\RR_2\subset \RR$ and $\partial^{h,1}\RR_2\subset \RR$. By monotonicity: $\Width(\RR_2)\le \Width(\RR)$.

Assume that $\Width(E_x)>2$. The \emph{left and right $1$-buffers} of $E_x$ are defined \[ B^\ell_1\coloneqq [0,1]\times [0,1]\sp\text{ and }\sp B^\rho_1\coloneqq [x-1,x]\times [0,1]\] respectively. We say that the rectangle 
\[E^\new_x\coloneqq [1,x-1]\times [0,1]=E_x\setminus \big( B^\ell_1 \cup B^\rho_1\big)\] is obtained from $E_x$ by \emph{removing $1$-buffers}. If $\Width(E_x)\le 2$, then we set $E_x^\new\coloneqq\emptyset$. Similarly, buffers of any width are defined.

\subsubsection{Annuli}
\label{sss:annuli} 
A \emph{closed annulus $A$ of modulus $1/x$} is a Riemann surface obtained from $E_x$ by gluing its vertical boundaries:
\[A\coloneqq E_x/_{\partial^{v,\ell} E_x\ni (0,y)\sim (x,y)\in \partial^{v,\rho} E_x\mid  \sp \forall y}\sp ,\sp \Width(A)= x,\sp \mod (A)\coloneqq 1/x.\]
Its interior $\intr (A)$ is an \emph{open annulus with modulus $x$}. The induced image of the vertical foliation $\Fam(E_x)$ is the \emph{vertical foliation $\Fam(A)$ of $A$}. The width of $\Fam (A)$ is equal to the width of all the curves in $A$ connecting its boundaries $\partial^{h,0} A, \partial^{h,1} A$ --  the induced images of the horizontal boundaries $\partial^{h,0} E_x, \partial^{h,1} E_x$. 


\subsubsection{Monotonicity and Gr\"otzsch inequality} \label{sss:Monot+Gr}We say a family of curves $\SS$ \emph{overflows} a family $\FamG$ if every curve $\gamma\in\SS$ contains a subcurve $\gamma'\in \FamG$. We also say that 
\begin{itemize}
\item a family of curves $\Fam$ \emph{overflows} a rectangle $\RR$ if $\Fam$ overflows $\Fam^\full(\RR)$;
\item a rectangle $\RR_1$ overflows another rectangle $\RR_2$ if $\Fam(\RR_1)$ overflows $\Fam^\full(\RR_2)$.
\end{itemize}

If $\Fam$ overflows a family or a rectangle $\FamG$, then $\FamG$ is wider than $\Fam$:
\begin{equation}
\label{eq:Width Monot}
\Width(\Fam) \le \Width(\FamG).
\end{equation}
If $\Fam$ overflows both $\FamG_1,\FamG_2$, and $\FamG_1,\FamG_2$ are disjointly supported, then the \emph{Gr\"otzsch} inequality states:
\begin{equation}
\label{eq:Grot}
\Width(\Fam) \le \Width(\FamG_1)\oplus \Width(\FamG_2),
\end{equation}
where $x\oplus y =(x^{-1}+y^{-1})^{-1}$ is the harmonic sum. 


\subsubsection{Parallel Law}\label{sss:ParLaw}
 For any families of curves $\FamG_1,\FamG_2$, we have:
\begin{equation}
\label{eq:ParLaw}
\Width(\FamG_1\cup \FamG_2)\le \Width(\FamG_1)+\Width(\FamG_2).
\end{equation}
If $\FamG_1,\FamG_2$ are disjointly supported, then
\begin{equation}
\label{eq:StrictParLaw}
\Width(\FamG_1\cup \FamG_2) = \Width(\FamG_1)+\Width(\FamG_2).
\end{equation}

\subsubsection{Restriction of families}
\label{sss:short subcurves} Consider a family of curves $\FamG$ connecting $X$ and $Y$. And suppose \[\widetilde X\supset X,\sp\sp  \widetilde Y\supset Y,\sp\sp \widetilde X\cap \widetilde Y=\emptyset\] are enlargements. Then every curve \[[\gamma\colon [0,1]\to \wC]\in \FamG\] 
has a \emph{unique first shortest} subcurve $\gamma'\subset \gamma$ connecting $\widetilde X$ and $\widetilde Y$: there is a minimal $t_1\ge 0$ for which there is a $t_2>t_1$ such that
\[ \gamma\big((t_1,t_2)\big)\subset \C\setminus (\widetilde X\cup \widetilde Y),\sp\sp\text{ and }\sp\sp\gamma(t_1)\in \widetilde X,\sp \gamma(t_2)\in \widetilde Y;\]
we set $\gamma'\coloneqq \gamma\mid[t_1,t_2]$. Define $\FamG^\new$ to be the family consisting of $\gamma'$ for all $\gamma\in \FamG$. Since $\FamG$ overflows $\FamG^\new$, we have (see~\S\ref{sss:Monot+Gr}):
\[\Width(\FamG^\new)\ge \Width(\FamG).\]
Note that if $\FamG$ is a lamination, then so is $\FamG^\new$.

Consider now the following generalization. For a lamination $\FamG$ and disjoint sets $X_1,X_2,\dots, X_m$ suppose that the following holds. Every curve $\gamma\in\FamG$ intersects all the $X_i$ and it intersects $X_i$ before intersecting any $X_{i+j}$ for $j>0$. Then every $\gamma\in \FamG$ contains disjoint subcurves $\gamma_1,\gamma_2,\dots, \gamma_{m-1}$ where $\gamma_i$ is the first shortest subcurve between $X_i$ and $X_{i_1}$. Setting $\FamG_i$ to be the set of $\gamma_i$ over all $\gamma\in \FamG$, we obtain that 
$\FamG$ \emph{overflows consequently} $\FamG_i$ and, by~\S\ref{sss:Monot+Gr}:  
\[\Width(\FamG)\le \Width(\FamG_1)\oplus \dots \oplus \Width(\FamG_{m-1}).\]
Note that $\FamG_i$ are disjoint laminations.

\subsubsection{Canonical rectangles}
\label{sss:can rect} 
Consider a closed Jordan disk $D\subset \C$ together with disjoint intervals $I,J\subset \partial D$. We denote by $\Fam^-(I,J),\Fam^+(I,J), \Fam(I,J)$ the families of curves in $D, \ \wC\setminus \intr D,\  \wC\setminus (I\cup J)$ connecting $I,J$. The widths of these families are denoted by $\Width^-(I,J),\Width^+(I,J), \Width(I,J)$.

We can view $D$ as a rectangle $\RR$ with $\partial^v\RR =I\cup J$. We call $\RR$ the \emph{canonical rectangle of $\Fam^-(I,J)$}; we have $\Width(\RR)=\Width^-(I,J)$. Similarly, viewing $\wC\setminus \intr D$ as a rectangle $\RR_2$ with $\partial^h \RR_2=I\cup J$, we obtain the \emph{canonical rectangle} $\RR_2$ of $\Fam^+(I,J)$; we have $\Width^+_D(I,J)=\Width(\RR_2)$.

Observe that $A\coloneqq \wC\setminus (I\cup J)$ is an open annulus; its Caratheodory boundary consists of $I^-\cup I^+$ and $J^-\cup J^+$, where $I^-, \ J^-$ are the sides of $I,J$ from $\intr D$ while $I^+, \ J^+$ are the sides of $I,J$ from $\wC\setminus  D$. The \emph{vertical family} $\FamH$ of $\Fam(I,J)$ consists of vertical curves of $A$ together with their landing points. We have $\Width(\FamH) =\Width(I,J)$.

\subsubsection{Innermost and outermost curves} \label{sss:inner outer order} It will be convenient for us to use the following inner-outer order on vertical curves in rectangles.
Consider a rectangle
\[ \RR\subset \wC,\sp\sp\sp\text{ with }\sp \partial ^h \RR\subset \partial D,\]
where $D$ is a closed Jordan disk, such that $\RR$ is disjoint from a complementary interval $N\subset \partial D$ between $\partial ^{h,0}\RR, \partial ^{h,1}\RR$. Let $N^-$ and $N^+$ be two sides of $N$ from the inside and outside of $D$. Consider a set of vertical curves $\{\ell_i\}_i\subset \Fam(\RR)$. The \emph{innermost} curve of $\{\ell_i\}_i$ is the curve $\ell_\inn$ 
separating $N^-$ from all remaining $\ell_i$ in $\wC\setminus (\partial^h \RR \cup N)$. The \emph{outermost} curve of $\{\ell_i\}_i$ is the curve $\ell_\out$ separating $N^+$ from all remaining $\ell_i$ in $\wC\setminus (\partial^h \RR \cup N)$.

\subsubsection{Laminations}\label{sss:App:Lam}  By a \emph{lamination} $\LL$ we mean a family of pairwise disjoint simple rectifiable arcs such that $\supp \LL$ is measurable. A \emph{sublamination} of $\LL$ is any collection $\FamH$ of arcs from $\FamG$ such that $\supp \FamH$ is measurable.

Laminations naturally appear as restrictions of rectangles -- see~\S\ref{sss:short subcurves}.  Note that a restriction of a rectangle is usually not a rectangle as discussed in~\cite[\S2.3]{KL}.  For convenience, a lamination $\FamG$ can often be replaced by a rectangle $\RR$ bounded by the left- and rightmost curves of $\FamG$; then $\Width(\RR)\ge \Width(\FamG)$.

All laminations in our paper will appear from rectangles using basic operations like restrictions and finite unions.

\subsubsection{Restrictions of sublaminations} \label{sss:Restr of Lam}Consider a lamination or a rectangle $\RR$, and let $\widetilde \SS$ be a sublamination of $\RR$ (or of $\Fam(\RR)$). Assume that $\widetilde \SS$ overflows a lamination $\SS$. Then we write
\begin{equation}
\label{eq:sss:Restr of Lam}
\Width(\RR | \SS) \coloneqq \Width(\widetilde \SS)\sp\sp\sp\sp  \left(\text{note that $\Width(\RR | \SS)\le \Width (\RR)$}\right). 
\end{equation}

\subsubsection{Splitting Rectangles}

\begin{lem}
\label{lem:splitting rectangle}
Consider a Jordan disk $D$ and let $I,J\subset \partial D$ be a pair of disjoint intervals. Consider an arc $\ell$ in the canonical rectangle of $\Fam^-_D(I,J)$, \S\ref{sss:can rect}. Suppose $\ell$ splits $I$ and $J$ into $I_1,I_2$ and $J_1,J_2$ enumerated so that the pairs $I_1,J_1$  and $I_2,J_2$ are on the same side of $\ell$. We denote by $D_1$ and $D_2$ connected components of $D\setminus \ell$ containing $I_1,J_1$ and $I_2,J_2$ on its boundaries respectively.
Then
\[ \Width^-_D(I_1,J_1)+\Width^-_D(I_2,J_2)-2\le  \Width^-_{D_1}(I_1,J_1)+ \Width^-_{D_2}(I_2,J_2)=\]
\[=\Width^-_D(I,J) \le \Width^-_D(I_1,J_1)+\Width^-_D(I_2,J_2).\]
\end{lem}
\begin{proof}
The last inequality is immediate. Let
\begin{itemize}
\item $\RR$ be the canonical rectangle of $\Fam^-_D(I,J)$;
\item $\RR_1,\RR_2$ be the canonical rectangles of $\Fam^-_{D_1}(I_1,J_1), \Fam^-_{D_2}(I_2,J_2)$;
\item $\widetilde \RR_1,\widetilde \RR_2$ be the canonical rectangles of $\Fam^-_D(I_1,J_1), \Fam^-_D(I_2,J_2)$.
\end{itemize}
Since $\ell$ splits $\RR$ into $\RR_1,\RR_2$, we have \[\Width^-_{D_1}(I_1,J_1)+ \Width^-_{D_2}(I_2,J_2)=\Width(\RR_1)+\Width(\RR_2)=\Width(\RR)=\Width^-_D(I,J).\]

By removing $1$-buffers from $\widetilde\RR_1,\widetilde \RR_2$, we obtain new disjoint rectangles $\widetilde \RR_1^\new,\widetilde \RR^\new_2$; since $\widetilde \RR_1^\new\sqcup \widetilde \RR^\new_2\subset \RR$, we have  \[ \Width(\widetilde \RR_1^\new)+\Width( \widetilde \RR^\new_2)\le  \Width(\RR).\]   
\end{proof}

\begin{lem}
\label{lem:splitting rectangle:2} Under the assumptions of Lemma~\ref{lem:splitting rectangle}, let $\FamG_1$ be the family of curves in $D$ connecting $I$ to $J$ such that every curve in $\FamG_1$ intersects $D_1$. Then 
\[\Width^-_{D_1}(I_1,J_1) \le \Width(\FamG_1)\le \Width^-_{D_1}(I_1,J_1)+2\]
\end{lem}
\begin{proof}
As in the proof of Lemma~\ref{lem:splitting rectangle}, let $\RR$ be the canonical rectangle of $\Fam^-_D(I,J)$ and $\RR_1$ be the canonical rectangle of $\Fam^-_{D_1}(I_1,J_1)$. Since $\RR_1$ is a genuine subrectangle of $\RR$, we can consider the genuine subrectangle $\RR_1^{+}$ of $\RR$ specified by \[\RR_1\subset \RR_1^+ \sp\sp\text{ and }\sp\sp\Width(\RR^+_1)= \Width(\RR_1)+1;\]
i.e.~$\RR_1$ is $\RR_1^+$ minus its one $1$-buffer $\BB$. The width of curves in $\FamG_1$ crossing $\BB$ is less than $1$; the remaining curves of $\FamG_1$ are in $\RR_1^+$.
\end{proof}

\subsubsection{Enclosed annuli}
\label{sss:EnclAnn} Let $A, B\subset \C$ be two closed annuli surrounding open disks $U$ and $V$ respectively. Assume that 
\begin{itemize}
\item $U \cup V$ is an open topological disk;
\item $A\cup U\cup B\cup V$ is a closed topological disk. 
\end{itemize}
Then the \emph{enclosed annulus} is
\[A\square B\coloneqq \big(A\cup U\cup B\cup V \big)\setminus  \big(U \cup V\big).\]
\begin{lem} 
\label{lem:square}
If $\mod (A), \mod (B)\ge 2\varepsilon$, then $\mod (A\square B) \ge \varepsilon$.
\end{lem}
\begin{proof}
Let $\gamma\colon [0,1] \to A\square B$ be a vertical curve of the annulus $A\square B$. Assume first that $\gamma(0)\in \partial^\inn A$. Then for some $t\in (0,1]$ we have
$\gamma(t)\in \partial^\out A$; i.e. $\gamma$ crosses $A$. Similarly, if $\gamma(0)\in\partial ^\inn B$, then $\gamma$ crosses $B$. By the Parallel Law~\S\ref{sss:ParLaw}, $\Width(A\square B)\le \Width(A)+\Width(B)\le 1/\varepsilon$.
\end{proof}

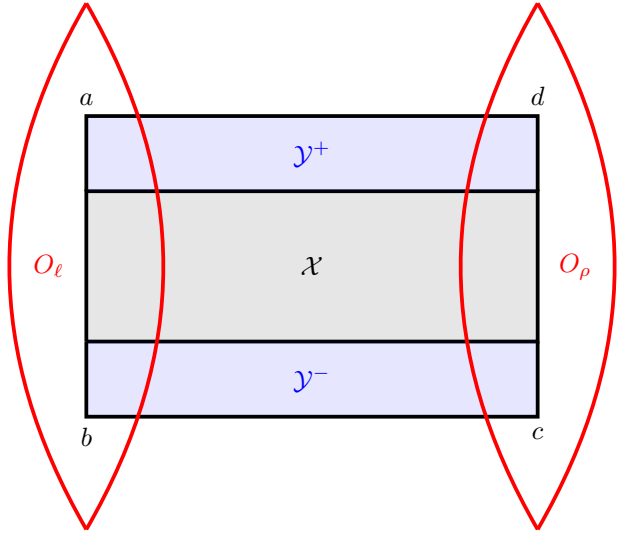
\begin{figure}

\[\begin{tikzpicture}[line width=0.05 cm]

\draw[fill, fill opacity=0.1] (0,0) -- (0,2)--(6,2)-- (6,0)--(0,0);

\draw[fill =blue, fill opacity =0.1] (0,2) -- (0,3)-- (6,3)--(6,2)
 (0,0) -- (0,-1)-- (6,-1)--(6,0); 
 \node[] at (3,1) {$\XX$};
\node[blue] at (3,2.5) {$\YY^+$};
\node[blue] at (3,-0.5) {$\YY^-$};

 \node[above] at (0,3) {$a$};
 \node[below] at (0,-1) {$b$};
 \node[above] at (6,3) {$d$};
 \node[below] at (6,-1) {$c$};

\node[red] at (-0.5,1) {$O_\ell$};
\node[red] at (6.5,1) {$O_\rho$};

\draw[red] (6,4.5) edge[bend right] (6,-2.5)
(6,4.5) edge[bend left] (6,-2.5)
(0,4.5) edge[bend right] (0,-2.5)
(0,4.5) edge[bend left] (0,-2.5);


\end{tikzpicture}\]
\caption{Assume a rectangle $\YY$ is a union $\YY^+\cup \XX\cup \YY^-$ such that $\Width(\YY^+)\asymp \Width(\YY^-)\asymp \mod (O_\ell \setminus [a,b])\asymp \mod (O_\rho \setminus [c,d])\asymp 1$. Then $\mod (O_\ell \cup \YY \cup O_\rho \setminus \XX) \succeq 1$, see Lemma~\ref{lem:A:encl form}.
}
\label{fig:lem:A:encl form}
\end{figure}

Let us consider the following construction which will be used in \S\ref{s:Welding}. Suppose, see Figure~\ref{fig:lem:A:encl form}:
\begin{itemize}
\item a rectangle $\YY$ is a union of its genuine subrectangles $\YY^+, \XX,\YY^-$ with disjoint interiors, where $\XX$ is between $\YY^+$ and $\YY^-$;
\item closed disks $O_\ell,O_\rho$ contain $\partial^{h,0} \YY=[a,b]$ and $\partial^{h,1}\YY=[c,d]$ respectively;
\item $O_\ell \cup \YY\cup O_\rho$ is a closed topological disk.
\end{itemize}

\begin{lem}
\label{lem:A:encl form}
For $O_\ell,\sp \YY=\YY^+\cup \XX\cup \YY^-,\sp O_\rho$ as above, if  
\[\Width(\YY^+)\asymp \Width(\YY^-)\asymp \mod (O_\ell \setminus [a,b])\asymp \mod (O_\rho \setminus [c,d])\asymp 1, \]
then $\mod (O_\ell \cup \YY \cup O_\rho \setminus \XX) \succeq 1$.
\end{lem}
\begin{proof}
Consider a vertical curve $\gamma$ of the annulus $O_\ell \cup \YY \cup O_\rho \setminus \XX$. 
\begin{itemize}
\item If $\gamma$ intersects $[a,b]$, then $\gamma$ crosses the annulus $O_\ell\setminus [a,b]$. 
\item If $\gamma$ intersects $[c,d]$, then $\gamma$ crosses the annulus $O_\rho\setminus [c,d]$. 
\item If $\gamma$ is disjoint from $[a,b]\cup [c,d]$, then $\gamma$ crosses either $\YY^+$ or $\YY^-$.
\end{itemize}
By the Parallel Law~\S\ref{sss:ParLaw}, $\mod (O_\ell \cup \YY \cup O_\rho \setminus \XX) \succeq 1$.
\end{proof}

\subsubsection{Geodesic Rectangles}
\label{sss:GeodRect}
Let $D$ be a closed Jordan disk, and consider two closed disjoint intervals $I,J\subset \partial D$. The \emph{geodesic rectangle} $\RR(I,J)$ in $\ovl D$ is a rectangle such that \[\partial^{h,0} \RR(I,J)=I,\sp\sp \partial^{h,1}\RR(I,J)=J,\]
and the vertical sides of $\RR(I,J)$ are the hyperbolic geodesics of $D$.

\begin{lem}
\label{lem:hyp rectangles}
Let $\RR$ with $\Width(\RR)>1$ be a rectangle in $\ovl D$ with $\partial ^{h,0}\RR=I$ and  $\partial ^{h,1}\RR=J$. Let $\RR^\new$ be the rectangle obtained from $\RR$ by removing two $1/2$-buffers on each side. Write $\partial^{h,0}\RR^\new=I^\new\subset I$ and $\partial ^{h,1}\RR^\new=J^\new\subset J$. Then
\[\RR\supset \RR(I^\new,J^\new)\sp\sp\text{ and }\sp\sp \RR^\new\subset \RR(I,J),\]
where $\RR(I^\new,J^\new), \RR(I,J)$ are geodesic rectangles as above.
\end{lem}
\noindent In particular, $\RR$ can be replaced with a geodesic subrectangle $\RR(I^\new,J^\new)$ so that $\Width(\RR)-\Width[\RR(I^\new,J^\new)]\le 1$.
\begin{proof}
We can assume that $\RR'=D=E_x$, where $E_x$ is a Euclidean rectangle, see~\S\ref{sss:Rectan}. Then the lemma follows by appropriately applying the following {\bf claim}: the hyperbolic geodesic $\gamma \subset E_x$ connecting $(0,0)$ and $(0,1)$ is within $E_{1/2}$ -- the left $1/2$-buffer of $E_x$. 

To prove the claim about $\gamma$, consider the right half-plan $\C_{>0}\coloneqq \{z \mid \Re z>0\}$. Then the hyperbolic geodesic $\widetilde \gamma \subset \C_{>0}$ connecting $(0,0)$ and $(0,1)$ is the semicircle orthogonal to the imaginary line; i.e. $\widetilde \gamma\subset E_{1/2}$. Since $E_x\subset \C_{>0}$, we also obtain that $  \gamma\subset E_{1/2}$. 
\end{proof}

It follows from Lemma~\ref{lem:hyp rectangles} that if $J\subset\partial D$ is a concatenation of subintervals $J_1\#J_2$ and $I\subset \partial D$, then
\begin{equation}
\label{eq:width: concat} \Width^-(I,J)=\Width^-(I,J_1)+\Width^-(I,J_2)-O(1).
\end{equation}

\subsection{Small overlapping of wide families} Many arguments in the near-degenerate regime are based on the principle that wide families have a relatively small overlap. 
\subsubsection{Non-Crossing Principle}
\label{sss:non cross princ}  Consider a closed Jordan disk $D$ and let \[\RR_1 , \RR_2\subset  D,\sp\sp \partial^h \RR_1 , \partial^h \RR_2\subset \partial D,\sp\sp \partial^h \RR_1 \cap \partial^{h} \RR_2 =\emptyset\] be two rectangles. If $\Width(\RR_1),\Width(\RR_2)> 1$, then $\RR_1, \RR_2$ \emph{do not cross-intersect}: there are vertical curves $\gamma_1\in \Fam(\RR_1)$ and $\gamma_2\in \Fam(\RR_2)$ with $\gamma_1\cap \gamma_2=\emptyset$. Indeed, assuming otherwise, we obtain
\[1/\Width (\RR_1) = \mod (\RR_1) \ge \Width(\RR_2)\] by monotonicity of the external length.

\subsubsection{Vertical boundaries}
\label{sss:sm overl princ} The following lemma is a slight generalization of \cite[Lemma 2.14]{KL}.

\begin{lem}
\label{lem:vet boundar} For every $\varepsilon>0$ the following holds. Consider rectangles \[\FamG, \sp \RR_1, \RR_2,\dots, \RR_n\subset \wC ,\sp\sp  \sp \Width(\FamG),\ \Width(\RR_i)>  8n+2\varepsilon\] such that the $\RR_i$ are pairwise disjoint. Then after removing buffers of width at most $4n+\varepsilon$, we can assume that the new rectangles $\FamG^\new, \RR^\new_1,\RR^\new_2,\dots,\RR^\new_n$ have disjoint vertical boundaries.
\end{lem}

\begin{proof} We need the following fact:

\begin{lem}[{\cite[Lemma 2.13]{KL}}]
\label{lem:inters of wide lamin} Consider two laminations $\Lambda,\FamG$ such that $\Lambda$ is a sublamination of the vertical foliation of a rectangle. If $\Width(\Lambda)>\kappa$ and $\Width(\FamG)\ge \kappa \ge1$, then there is a curve $\ell \in \FamG$ that intersects less than $\frac  1 \kappa \Width (\Lambda)$ of curves in $\Lambda$.
\end{lem}

 Let $\FamG_{\pm}, \RR_{\pm, n}$ be the buffers of width $4n+\varepsilon$ in $\FamG, \RR_n$. Applying Lemma~\ref{lem:inters of wide lamin}, we can select vertical curves $\gamma_{-,n}\in \RR_{-,n,} \gamma_{+,n}\in \RR_{+,n}$ so that each $\gamma_{\pm,n}$ intersects less than $\frac{1}{4n}\Width(\FamG_{-}\sqcup \FamG_{+})=\frac{1}{2n}\Width(\FamG_{\pm})$ curves in $\FamG_{-}\sqcup \FamG_{+}$.  Therefore, there are curves $\beta_-\in \FamG_-, \ \beta_+\in \FamG_+$ that are disjoint from all the $\gamma_{\pm,n}$. We set $\beta_{\pm}, \gamma_{\pm,n}$ to be the vertical boundaries of  $\FamG^\new, \RR^\new_1,\RR^\new_2,\dots,\RR^\new_n$.
\end{proof}

\subsubsection{Crossing an annulus}
\label{sss:cross ann}
Let $A\subset \C$ be an annulus and $\FamG$ be a family of curves such that every its curve starts in the unbounded component $U$ of $\C\setminus A$. Then at most $1/\mod A$ curves in $\FamG$ intersect the bounded component $O$ of $\C\setminus A$. Indeed, every curve $\gamma\in \FamG$ intersecting $O$ contains a subcurve $\gamma'$ connecting the inner and outer boundaries of $A$. The width of such $\gamma$ is at most $1/\mod A$.

\begin{lem} 
\label{lem:Rect in U} Let $D\subset \C$ be a closed Jordan disk and $A,\mu=\mod A$ be a closed topological annulus such that the bounded component $O$ of $\C\setminus A$ intersects $\partial D$. Then for every rectangle \[\RR\subset D\sp\sp \text{ such that }\sp \partial ^h \RR \subset D\setminus (A\cup O), \] after removing two $1/\mu$-buffers from $\RR$, the new rectangle $\RR^\new$ is disjoint from $O$.
\end{lem}

We will need the following topological property:

\begin{lem}
\label{lem:buffer:R O}
Let $D\subset \C$ be a closed Jordan disk together with a rectangle
\[\RR\subset  D,\sp\sp \partial^h \RR\subset \partial D. \]
Let $O\subset \C\setminus \partial^h\RR $ be a connected set intersecting $\partial D\setminus \partial ^h \RR$. If $\RR$ intersects $O$, then the set of vertical curves in $\RR$ intersecting $O$ forms either one or two buffers of $\RR$. If, moreover, $O$ intersects exactly one component of $\partial D\setminus \partial ^h \RR$, then the set of vertical curves in $\RR$ intersecting $O$ forms a buffer of $\RR$.
\end{lem}
\begin{proof}
If $\gamma_1,\gamma_2\in \Fam(\RR)$ are two curves disjoint from $O$, then all vertical curves of $\RR$ between $\gamma_1$ and $\gamma_2$ are also disjoint from $O$ -- otherwise $O$ would be enclosed by $\partial^h \RR\cup \gamma_1\cup \gamma_2$. Therefore, the set of vertical curves intersecting $O$ form one or two buffers. 

Assume there are two buffers. Then there will be a vertical curve $\gamma\in\Fam(\RR)\setminus \partial^{v} \RR$ that is disjoint from $O$. Since $O$ is disjoint from $ \gamma\cup \partial ^h \RR$ and since $O$ intersects both $\partial^{v,\ell} \RR, \partial^{v,\rho} \RR$, the set $O$ intersects both components of $\partial D\setminus \partial ^h \RR$.
\end{proof}
\begin{proof}[Proof of Lemma~\ref{lem:Rect in U}]
 At most $1/\mu$ vertical curves in $\RR$ can cross $A$ and all such curves form one or two buffers of $\RR$ by Lemma~\ref{lem:buffer:R O}. 
\end{proof}

\subsubsection{Push-forwards} Suppose $f\colon S_1\to S_2$ is a branched covering between Riemann surfaces of degree $d$. Let $\FamG$ be a family of curves in $S_1$. Then, see~\cite[Lemma 4.3]{KL}: 
\begin{equation}
\label{eq:Width:degree d}\Width(f[\FamG])\ge \frac 1 d \Width(\FamG).
\end{equation}

Covering Lemma~\cite{KL} (stated as Lemma~\ref{lem:CovLmm}) allows one to push-forward width of curves more efficiently.

\begin{lem}
\label{lem:univ push} Suppose $g\colon A\to B$ is a covering between either two closed annuli or between punctured disks. Let $\RR\subset A, \ \partial^h\RR\subset \partial A $ be a rectangle in $A$ such that $g$ maps $\partial^{h,0} \RR$ injectively onto $g\big(\partial^{h,0} \RR \big)$. Then after removing two $1$-buffers from $\RR$, the map $g$ is injective on the new rectangle $\RR^\new$.
\end{lem}
\begin{proof}
Write $D\coloneqq \deg g$. Since $g$ is a normal covering, $g$ has a group of deck transformations; we denote by $\RR_0=\RR, \RR_1,\dots, \RR_{D-1}$ the orbit of $\RR$ under the group of deck transformations. Since $\partial^{h,0} \RR_i$ are disjoint, all $\RR_i^\new$ are disjoint. (The last claim can be easily checked by lifting the $\RR_i$ to the universal cover.) Therefore, $g\mid \RR^\new_0$ is injective.
\end{proof}

\begin{figure}

\begin{tikzpicture}[scale=0.8]

\draw[shift={(0,-1.5)},line width =0.08 cm]
(3,0) -- (20,0);

\draw[shift={(0,-1.5)},line width =0.15 cm , red]
 (4, 0) -- (6,0)
 (8, 0) -- (10,0);
 
 \draw[shift={(9,-1.5)},line width =0.15 cm , red]
 (4, 0) -- (6,0)
 (8, 0) -- (10,0);
 
  \draw[shift={(0,-1.5)},line width =0.07 cm , blue]
(4, 0) edge[bend left =50](15,0)
(6, 0) edge[bend left ](13,0);

\draw[shift={(4,-1.5)},line width =0.07 cm , dashed, blue]
(4, 0) edge[bend left =50](15,0)
(6, 0) edge[bend left ](13,0); 
 
 \node[below,red] at (5,-1.5){$\partial^{h,0} \RR$};

\node[below,red] at (14,-1.5){$\partial^{h,1} \RR$};

\node[below,blue] at (9,0.7){$\RR$};

\node[below,blue] at (10.5+4.5,0.7){$\RR_1$};

\draw[line width =0.04 cm ] (9.5,0.6) edge[->,bend left] (14.5,0.6);


\end{tikzpicture}

\caption{If a rectangle $\RR$ cross-intersects its conformal image, then $\Width(\RR)\le 1$.}
\label{fig:cross inters}
\end{figure}
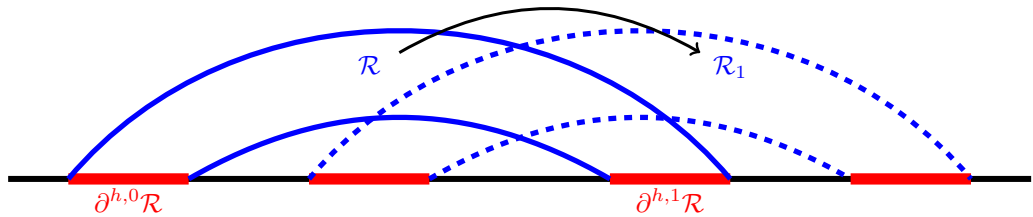

\subsection{Shift Argument} \label{sss:shift argum}  If a rectangle $\RR$ has a conformal shift $\RR_1$ cross-intersecting $\RR$, then $\Width(\RR)\le 1$, see Figure~\ref{fig:cross inters}. Often, a weaker condition is sufficient: $\partial^{h,0} \RR$ is disjoint from $\partial ^h \RR_1$. Let us provide details. Consider a rectangle
\[
\RR\subset \C\setminus Z\sp\sp\text{ such that }\sp\sp \partial^h \RR\subset \partial Z \]
that has a conformal pullback or push-forward 
\[
 \RR_1\coloneqq f^t(\RR) \overset{1:1}{\longleftarrow} \RR, \sp\sp\sp \RR_1\subset \C\setminus Z,\sp \sp\sp \partial^h \RR_1\subset \partial Z 
\]
for $t\in \Z$. Assume next that there is an interval $T\subsetneq \partial Z$ containing $\lfloor \partial^h \RR\rfloor\cup \lfloor \partial \RR^h_1\rfloor$ such that \[\partial^{h,0}\RR<\partial^{h,1}\RR,\sp\sp  \partial^{h,0}\RR_1<\partial^{h,1}\RR_1\sp\sp\text{ in }\sp T\]
and $f^t$ maps $\partial^{h,0}\RR, \partial^{h,1}\RR, \lfloor \partial ^h \RR\rfloor$ onto  $\partial^{h,0}\RR_1, \partial^{h,1}\RR_1, \lfloor \partial ^h \RR_1\rfloor$. We say that $\RR, \RR_1$ are \emph{linked} if
\begin{equation}
\label{eq:LinkCond}
\begin{matrix}
\text{either} \ \sp \partial ^{h,0} \RR_1 <\partial^{h,0} \RR< \partial^{h,1} \RR_1 \sp\sp\text{ or }\sp\sp \partial ^{h,0} \RR <\partial^{h,0} \RR_1< \partial^{h,1} \RR \\
\text{ or }\sp\sp\sp
\partial ^{h,0} \RR_1 <\partial^{h,1} \RR< \partial^{h,1} \RR_1 \sp\sp\text{ or }\sp\sp \partial ^{h,0} \RR <\partial^{h,1} \RR_1< \partial^{h,1} \RR
\end{matrix}
\end{equation}
holds.


\begin{lem}
\label{lem:shift argum}
If $\RR$ is linked to its conformal pullback or push-forward $\RR_1$ as above, then $\Width(\RR)\le 2$.
\end{lem}
\begin{proof}
Assume that $ \partial ^{h,0} \RR_1 <\partial^{h,0} \RR< \partial^{h,1} \RR_1 $ holds; the other cases are analogous. Assume that $\Width(\RR)> 2$. Let $\RR^\new$ be the rectangle obtained from $\RR$ by removing two $1$-buffers on each side. Set $\RR^\new_1\coloneqq f^t(\RR^\new)\subset \RR_1$. Since $\partial ^{h,0}\RR^\new\subset \partial^{h,0}\RR$ is disjoint from $\partial^h\RR^\new_1\subset \partial^{h}\RR_1$, the new rectangles $\RR^\new,\RR^\new_1$ are disjoint. Since $f^t$ maps $\partial^{h,0}\RR^\new, \partial^{h,1}\RR^\new, \lfloor \partial ^h \RR^\new\rfloor$ onto  $\partial^{h,0}\RR^\new_1, \partial^{h,1}\RR^\new_1, \lfloor \partial ^h \RR^\new_1\rfloor$ and all the intervals are in $T$, we obtain 
\[ \partial ^{h,0} \RR_1 <\partial^{h,0} \RR< \partial^{h,1} \RR_1<\partial ^{h,1}\RR \sp\sp\text{ in }\sp T;\] i.e., $\RR^\new, \RR^\new_1$ intersect. This is a contradiction.
\end{proof}

\end{document}